\documentclass{amsart}

\usepackage{amsmath,amssymb,amsthm,epic,eepic}
\usepackage[all]{xy}

\DeclareMathOperator{\Span}{Span}
\DeclareMathOperator{\Diff}{Diff}
\DeclareMathOperator{\Lie}{Lie}

\DeclareMathOperator{\SL}{SL}
\DeclareMathOperator{\GL}{GL}
\DeclareMathOperator{\SO}{SO}
\DeclareMathOperator{\U}{U}

\DeclareMathOperator{\GCD}{GCD}
\DeclareMathOperator{\LCM}{LCM}
\DeclareMathOperator{\Isom}{Isom}
\DeclareMathOperator{\Crit}{Crit}
\DeclareMathOperator{\Sing}{Sing}
\DeclareMathOperator{\Hess}{Hess}
\DeclareMathOperator{\curv}{curv}
\DeclareMathOperator{\Aut}{Aut}
\DeclareMathOperator{\grad}{grad}
\DeclareMathOperator{\id}{id}
\DeclareMathOperator{\pr}{pr}
\DeclareMathOperator{\loc}{loc}
\DeclareMathOperator{\dR}{dR}
\DeclareMathOperator{\Pic}{Pic}

\newtheorem{defn}{Definition}[section]
\newtheorem{thm}[defn]{Theorem}
\newtheorem{lem}[defn]{Lemma}
\newtheorem{prop}[defn]{Proposition}
\newtheorem{cor}[defn]{Corollary}
\newtheorem{exam}[defn]{Example}

\newenvironment{psmallmatrix}{\left( \begin{smallmatrix}}{\end{smallmatrix} \right)}

\begin{document}
\title[Five dimensional $K$-contact manifolds of rank $2$]
{Five dimensional $K$-contact manifolds of rank $2$}
\author{Hiraku Nozawa}
\address{Department of Mathematics, Faculty of Science, Ochanomizu University, 2-1-1 Ohtsuka, Bunkyo-ku, Tokyo 112-8610, Japan}
\email{nozawa@ms.u-tokyo.ac.jp}

\begin{abstract}
A $K$-contact manifold is a smooth manifold $M$ with a contact form $\alpha$ whose Reeb flow preserves a Riemannian metric $g$ on $M$. If $M$ is closed and connected, then the closure of the Reeb flow in the isometry group of $(M,g)$ is a torus. The dimension of the torus is called the rank of the $K$-contact manifold $(M,\alpha)$. A $K$-contact manifold $(M,\alpha)$ of rank $n$ has an $\alpha$-preserving $T^n$-action.

We study the geometry of closed $5$-dimensional $K$-contact manifolds of rank $2$. We show the following three theorems: i) A closed $5$-dimensional $K$-contact manifold of rank $2$ is obtained from a lens space bundle over a closed surface by a finite sequence of contact blowing up and down. ii) the isomorphism classes of closed $5$-dimensional $K$-contact manifolds of rank $2$ are determined by graphs of isotropy data which represent the combinatorial data of the $T^2$-actions. iii) A closed $5$-dimensional $K$-contact manifold of rank $2$ has a compatible Sasakian metric. We also give a sufficient condition for a closed $5$-dimensional $K$-contact manifold of rank $2$ to be toric.
\end{abstract}
\maketitle

\tableofcontents

\section{Introduction}

We study $5$-dimensional $K$-contact manifolds with interest in the geometry of $5$-dimensional Sasakian manifolds. A $K$-contact manifold is an odd dimensional manifold $M$ with a contact form $\alpha$ whose Reeb flow preserves a Riemannian metric $g$ on $M$. Main examples are Sasakian manifolds, more specifically, contact toric manifolds and links of isolated singularities of weighted homogeneous polynomials.

We focus on $5$-dimensional $K$-contact manifolds of rank $2$. A closed $K$-contact manifold $(M,\alpha)$ has rank $n$ if and only if the dimension of closures of generic orbits of the Reeb flow on $(M,\alpha)$ is an $n$-dimensional torus. The rank of a $(2n-1)$-dimensional closed $K$-contact manifold is less than $n+1$. In the case of dimension $5$, the other cases of rank $2$ correspond to objects in other geometry as follows: In the case of rank $1$, $K$-contact manifolds are $S^1$-orbibundles over symplectic orbifolds. In the case of rank $3$, there exists a one to one correspondence between the underlying contact manifolds and cones in $\mathbb{R}^{3}$ by the results of Lerman \cite{Ler2} and Boyer-Galicki \cite{BoGa2}. 

We state our results. Our first result is as follows: 
\begin{thm}\label{KarshonClassification}
Two closed $5$-dimensional $K$-contact manifolds of rank $2$ are isomorphic if and only if their graphs of isotropy data are isomorphic.
\end{thm}
Graphs of isotropy data of closed $5$-dimensional $K$-contact manifolds of rank $2$ are graphs with the combinatorial data of the $T^2$-actions determined by the Reeb flows. See Section \ref{ContactMomentMapsAndGraphsOfIsotropyData} for the definition of graphs of isotropy data (Definition \ref{KarshonGraphs}). By Theorem \ref{KarshonClassification}, closed $5$-dimensional $K$-contact manifolds of rank $2$ are classified by combinatorial objects. Theorem \ref{KarshonClassification} has the following corollary:
\begin{cor}\label{ReebVectorFieldsDetermineKContactStructures}
Two closed $5$-dimensional $K$-contact manifolds $(M_{0},\alpha_{0})$ and $(M_{1},\alpha_{1})$ of rank $2$ are isomorphic if and only if there exists a diffeomorphism $f \colon M_{0} \longrightarrow M_{1}$ such that $f_{*}R_{0}=R_{1}$ where $R_{i}$ is the Reeb vector field of $\alpha_{i}$ for $i=0$ and $1$.
\end{cor}

Our main theorem concerns the classification of closed $5$-dimensional $K$-contact manifolds of rank $2$ up to contact blowing up and down. Let $(M,\alpha)$ be a closed $5$-dimensional $K$-contact manifold of rank $2$. We denote the maximal component and the minimal component of the contact moment map by $B_{\max}$ and $B_{\min}$ respectively. Note that $B_{\min}$ and $B_{\max}$ are either closed orbits of the Reeb flow of $\alpha$ or $3$-dimensional submanifolds of $M$ (See Lemma \ref{MorseTheory}). Our main theorem is the following:
\begin{thm}\label{BlowDownToLensSpaceBundles}
\begin{enumerate}
\item If $(\dim B_{\max},\dim B_{\min})=(3,3)$, then $(M,\alpha)$ is obtained from a lens space bundle over a closed surface by a finite sequence of contact blowing up.
\item If $(\dim B_{\max},\dim B_{\min})=(1,3)$ or $(\dim B_{\max},\dim_{\min})=(3,1)$, then $(M,\alpha)$ is obtained from a lens space bundle over $S^2$ by performing contact blowing down once after a finite sequence of contact blowing up.
\item If $(\dim B_{\max},\dim B_{\min})=(1,1)$, then $(M,\alpha)$ is obtained from a lens space bundle over $S^2$ by performing contact blowing down twice after a finite sequence of contact blowing up.
\end{enumerate}
\end{thm}
\noindent Contact blowing up is defined as special cases of contact cuts defined by Lerman \cite{Ler1}. See Section \ref{ContactBlowingUpAndDown} for the definition of contact blowing up. Contact blowing down is the inverse operation of contact blowing up. We have the following corollary:
\begin{cor}
Any closed $5$-dimensional $K$-contact manifold of rank $2$ is obtained from a lens space bundle over a closed surface by a finite sequence of contact blowing up and down.
\end{cor}

We state two other results on the relation between geometry of $5$-dimensional $K$-contact manifolds of rank $2$ and toric or Sasakian geometry. Let $G$ be the closure of the Reeb flow in the isometry group $\Isom(M,g)$ for a Riemannian metric $g$ invariant under the Reeb flow. We denote the action of $G$ on $M$ by $\rho$. The identity components of the isotropy groups of $\rho$ at the maximal component and the minimal component of the contact moment map on $(M,\alpha)$ for $\rho$ by $G_{\max}$ and $G_{\min}$:
\begin{thm}\label{ToricCondition}
Assume that 
\begin{enumerate}
\item every closed orbit of the Reeb flow of $\alpha$ is isolated and 
\item there exists an $S^1$-subgroup $G'$ of $G$ such that the orbits of the action of $G'$ is transverse to $\ker \alpha$ on $M$ and both of $G' \times G_{\max}$ and $G' \times G_{\min}$ are isomorphic to $G$,
\end{enumerate}
then there exists an $\alpha$-preserving $T^3$-action on $(M,\alpha)$.
\end{thm}
\begin{thm}\label{CompatibleMetric}
A closed $5$-dimensional $K$-contact manifold $(M,\alpha)$ of rank $2$ has a Riemannian metric $g$ such that $(g,\alpha)$ is a Sasakian metric on $M$.
\end{thm}
Note that the condition (i) in Theorem \ref{ToricCondition} is equivalent to assume that $\dim B_{\min}=1$ and $\dim B_{\max}=1$ by Lemma \ref{MorseTheory}. The condition (ii) in Theorem \ref{ToricCondition} can be translated into a condition on the image of the contact moment map (See Lemma \ref{TranslationOfCondition2}). See Section \ref{Sasakian} for the definition of Sasakian metrics (Definition \ref{DefinitionOfSasakianMetrics}).

We explain our method of the proof of our theorems and the relation of our results with study on $4$-dimensional symplectic manifolds. Our main tool is the Morse theory for the contact moment maps for $T^2$-actions associated with $K$-contact forms. The Morse theory for moment maps was effectively used to study $4$-dimensional symplectic manifolds with hamiltonian $S^1$-actions by Audin \cite{Aud}, Ahara-Hattori \cite{HaA} and Karshon \cite{Kar}. The $5$-dimensional $K$-contact manifolds of rank $2$ are classified by graphs by Theorem \ref{KarshonClassification} as $4$-dimensional symplectic manifolds with hamiltonian $S^1$-actions (See Theorem 4.1 of Karshon \cite{Kar}). But the combinatorics of the classification are different from the $4$-dimensional case. Ahara-Hattori and Karshon showed that any $4$-dimensional symplectic manifold with a hamiltonian $S^1$-action can be obtained from $\mathbb{C}P^{2}$, Hirzebruch surfaces or ruled surfaces by a sequence of $S^1$-equivariant blowing up at the fixed points of the $S^1$-action (See \cite{HaA} and \cite{Kar}). The $S^1$-equivariant blowing up for $4$-dimensional symplectic manifolds at fixed points of the hamiltonian $S^1$-action corresponds to contact blowing up along closed orbits of the Reeb flow for $5$-dimensional $K$-contact manifolds of rank $2$. But the classification of $5$-dimensional $K$-contact manifolds of rank $2$ up to a finite sequence of contact blowing up along closed orbits of the Reeb flow and its inverse operation is more complicated even in the case of contact toric manifolds as we will see in Subsection \ref{ToricExamples}. If we allow to perform contact blowing up along lens spaces and its inverse operation in addition, we have a classification theorem as Theorem \ref{BlowDownToLensSpaceBundles}.

Note that Theorems \ref{KarshonClassification}, \ref{BlowDownToLensSpaceBundles}, \ref{ToricCondition} and \ref{CompatibleMetric} have a certain correspondence to Karshon's results \cite{Kar} for $4$-dimensional symplectic manifolds with hamiltonian $S^1$-actions. But our proof of the latter three theorems without using the correspondence between manifolds and graphs or the Duistermaat-Heckman measure is different from Karshon's method. Our proof of Theorems \ref{BlowDownToLensSpaceBundles} and \ref{ToricCondition} is based on the Morse theory of contact moment maps and combinatorial computation on Euler numbers of locally free $S^1$-actions on $3$-dimensional orbifolds. We use the complex orbifold theory to show Theorem \ref{CompatibleMetric}.

This paper has eight sections: In Sections 2 and 3, we state the definitions, examples and properties of fundamental objects in $5$-dimensional $K$-contact geometry. We show Theorem \ref{KarshonClassification} in Subsection \ref{Subsection : KarshonClassification}. Section 4 is devoted to prove Theorem \ref{ToricCondition}. We define contact blowing up and down in Section 5 to use in Sections 6 and 7. Theorem \ref{BlowDownToLensSpaceBundles} is proved in Section 6. Section 7 is devoted to prove Theorem \ref{CompatibleMetric}. In Section 8, we summarize and prove local normal form theorems in $K$-contact geometry, which are our fundamental tools throughout this paper.

The author expresses his gratitude to Emmanuel Giroux, Klaus Niederkr\"{u}ger, Patrick Massot and H\'{e}l\`{e}ne Eynard-Bontemps for having valuable discussion on this research project at \'{E}cole Normale Sup\'{e}rieure de Lyon. He is grateful to Hiroki Kodama, Yoshifumi Matsuda, Masashi Takamura, Atsushi Yamashita, Inasa Nakamura, Chikara Haruta, Naoki Katou, Ryo Ando, Tomoyuki Ishida, Jun Ishikiriyama, Naohiko Kasuya and Toru Yoshiyasu for attending the long seminar on this paper. He expresses his gratitude to Hiroki Kodama for the discussion on contact toric manifolds. He expresses his gratitude to Yoshifumi Matsuda for his encouragement to complete this paper. Finally he expresses his gratitude to his adviser Takashi Tsuboi.

\section*{Notation}

The set of positive real numbers is denoted by $\mathbb{R}_{>0}$. The standard coordinate of $\mathbb{R}_{>0}$ is written as $r$. The unit circle $\{ \zeta \in \mathbb{C} | |\zeta|=1 \}$ in the complex line is denoted by $S^1$. The standard coordinate of $S^1$ is written as $\zeta$. The disk $\{ (z_1,z_2,\cdots,z_n) \in \mathbb{C}^{n} | \sum_{i=1}^{n}|z_i|^2 \leq \epsilon^{2} \}$ of radius $\epsilon$ in $\mathbb{C}^{n}$ is denoted by $D^{2n}_{\epsilon}$. The standard coordinate of $D^{2n}_{\epsilon}$ is written as $(z_1,z_2,\cdots,z_n)$. We write $D^{2n}$ for $D^{2n}_{1}$. The unit sphere $\{ (z_1,z_2) \in \mathbb{C}^{2} | |z_1|^2 + |z_2|^2 = 1 \}$ in $\mathbb{C}^{2}$ is denoted by $S^3$. The standard coordinate of $S^3$ is written as $(z_1,z_2)$. We regard finite cyclic groups as subgroups of $\{ \zeta \in \mathbb{C} | |\zeta|=1 \}$.

For a topological group $H$, an $H$-action $\tau$ on a set $A$ and a $\tau$-invariant subset $B$ of $A$, we denote the cardinality of the kernel of $H \longrightarrow \Aut(B)$ by $I(\tau,B)$.

\section{Basic definitions and examples}

We define $K$-contact manifolds and give several examples of $K$-contact manifolds. We see that there exist $5$-dimensional $K$-contact manifolds of rank $2$ which have no $K$-contact structure of rank $3$ using a result of Lerman \cite{Ler4}.

\subsection{$K$-contact manifolds and its rank}

\begin{defn}($K$-contact manifolds and $K$-contact submanifolds)
An odd dimensional smooth manifold $M$ with a contact form $\alpha$ is called $K$-contact if there exists a Riemannian metric $g$ on $M$ preserved by the Reeb flow. A smooth contact submanifold of a $K$-contact manifold $(M,\alpha)$ invariant under the Reeb flow of $\alpha$ is called a $K$-contact submanifold of $(M,\alpha)$.
\end{defn}

The Reeb flow on the manifold $M$ with contact form $\alpha$ is the flow generated by the Reeb vector field of $(M,\alpha)$.

We define the rank of $K$-contact manifolds. Let $(M,\alpha)$ be a connected closed $K$-contact manifold. We take a Riemannian metric $g$ invariant under the Reeb flow on $(M,\alpha)$. Then the closure $G$ of the Reeb flow in the isometry group $\Isom(M,g)$ of $(M,g)$ is a torus, because $G$ is commutative and compact by the compactness of $\Isom(M,g)$. 

\begin{defn}(Rank of $K$-contact manifolds) The dimension of the closure $G$ of the Reeb flow in $\Isom(M,g)$ is called the rank of $(M,\alpha)$.
\end{defn}

We prepare a terminology for the action of $G$ determined by the Reeb flow.

\begin{defn}(Torus actions associated with $K$-contact forms) The action of the closure $G$ of the Reeb flow is called the torus action associated with the $K$-contact form $\alpha$. 
\end{defn}

The $G$-action $\rho$ has the following properties:
\begin{enumerate}
\item $\rho$ preserves $\alpha$.
\item The orbit of $\rho$ coincide with the closures of the orbits of the Reeb flow. In particular, the singular $S^1$-orbits of $\rho$ coincide with the closed orbits of the Reeb flow.
\item A smooth Lie group action $\tau$ on $M$ commutes with $\rho$ if and only if $\tau$ commutes with the Reeb flow on $(M,\alpha)$.
\item A smooth Lie group action which preserves $\alpha$ commutes with the Reeb flow and $\rho$.
\end{enumerate}
Hence the closures of the orbits of the Reeb flow are generically of the same dimension as $G$.

Finally, we remark that there exists a restriction on the rank of $K$-contact manifolds. Fix a point $x$ on $M$. Let $C$ be the orbit of $x$ of $\rho$. We show
\begin{lem}\label{RestrictionOfRank}
$T_{x}C \cap (\ker \alpha)_{x}$ is a Lagrangian subspace of the symplectic vector space $((\ker \alpha)_{x}, d\alpha_{x})$. In particular, the dimension of $C$ is less than or equal to $n$ for a $(2n-1)$-dimensional $K$-contact manifold $(M,\alpha)$.
\end{lem}
\begin{proof}
We put $X_{0}=R$. Let $X_{1}$, $X_{2}$, $\cdots$, $X_{k}$ be the infinitesimal actions of $\rho$ such that $\{X_{0 x}, X_{1 x}, X_{2 x}, \cdots, X_{k x}\}$ is a basis of $T_{x}C$. To show that $T_{x}C \cap (\ker \alpha)_{x}$ is a Lagrangian subspace of $((\ker \alpha)_{x}, d\alpha_{x})$, it suffices to show $d\alpha_{x}(X_{i x},X_{j x})=0$ for every $i$ and $j$. By $L_{X_{0}}\alpha=0$ and $[X_{0},X_{j}]=0$, we have $X_{0}(\alpha(X_{j}))=(L_{X_{0}}\alpha)(X_{j}) + \alpha([X_{0},X_{j}])=0$ for every $j$. Then $\alpha(X_{j})$ is constant on $C$, since $C$ is the closure of the orbit of $x$ of the Reeb flow. Since $X_{i}$ is tangent to $C$, we have $X_{i}(\alpha(X_{j}))=0$ for every $i$ and $j$. Since $\rho$ preserves $\alpha$, we have $L_{X_{i}}\alpha=0$. Then $d\alpha_{x}(X_{i x},X_{j x})=-X_{i}(\alpha(X_{j}))(x)$. Hence for every $i$ and $j$, we have $d\alpha_{x}(X_{i x},X_{j x})=-X_{i}(\alpha(X_{j}))(x)=0$. Hence Lemma \ref{RestrictionOfRank} is proved.
\end{proof}
Hence if we have an $\alpha$-preserving effective $T^n$-action $\tau$ on a $(2n-1)$-dimensional closed $K$-contact manifold $(M,\alpha)$, $\rho$ is a torus subaction of $\tau$.

\begin{lem}\label{PertubationOfReebVectorFields}
Let $(M,\alpha)$ be a closed $K$-contact manifold. Assume that there exists an $\alpha$-preserving effective torus action $\tau$ on $M$ whose product with $\rho$ is an effective $T^{m}$-action $\tau'$ on $M$. Then $M$ has a $K$-contact form of rank $m$.
\end{lem}
Note that since $\tau$ preserves $\alpha$, $\tau$ commutes with $\rho$. 
\begin{proof}
The Reeb vector field $R$ is an infinitesimal action of $\tau$. Let $\Omega$ be the set of infinitesimal actions $\overline{Y}$ of $\tau'$ such that the orbits of the flow generated by the vector field corresponding to $\overline{Y}$ are dense in the orbits of $\tau'$. Then $\Omega$ is dense in the set of infinitesimal actions of $\tau'$. By closedness of $M$, we can take an infinitesimal action $Y_{0}$ of $\tau'$ sufficiently close to $R$ so that $\alpha(Y_{0})$ has no zero on $M$. We put $\beta=\frac{1}{\alpha(X_{0})}\alpha$. Since $\ker \beta=\ker \alpha$, $\beta$ is a contact form. Since $\beta(Y_{0})=1$ and $L_{Y_{0}}\beta=0$, $Y_{0}$ is the Reeb vector field of $\beta$. Then the rank of $\beta$ is $m$.
\end{proof}

\subsection{Examples of $K$-contact manifolds}

\subsubsection{Circle bundles over symplectic manifolds} 
Let $(B,\omega)$ be a symplectic manifold. Assume that the cohomology class of $\omega$ is contained in the image of the canonical map $H^{2}(B ; \mathbb{Z}) \longrightarrow H^{2}(B ; \mathbb{R})$. Then there exists a principal $S^1$-bundle over $B$ whose Euler class is equal to $[\omega]$. There exists an $S^1$-connection form $\alpha$ such that $d\alpha=p^{*}\omega$ where $p$ is the projection from the total space $M$ of the principal $S^1$-bundle to $B$.  Then $(M,\alpha)$ is a $K$-contact manifold. The Reeb vector field $R$ of $(M,\alpha)$ generates the principal $S^1$-action on $M$. Hence the rank of $(M,\alpha)$ is $1$.

We see that $M$ has a $K$-contact form of rank $2$ if $(B,\omega)$ has a hamiltonian $S^1$-action $\sigma$. We denote a hamiltonian function of $\sigma$ by $h$ and the infinitesimal action of $\sigma$ by $X$. There uniquely exists a vector field $\tilde{X}$ on $M$ which satisfies $p_{*}\tilde{X}=X$ and $\alpha(\tilde{X})=0$. We put $Y= \tilde{X} - h R$. Then the flow generated by $Y$ preserves $\alpha$. Note that a function $h'=h+c$ is also a hamiltonian function of $\sigma$ for an arbitrary constant $c$. We can choose $c$ so that the orbits of the flow generated by $Y'= \tilde{X} - h' R $ are closed (See the third paragraph of the proof of Lemma \ref{LensSpacesAreRank2}). Hence $Y'$ generates an $S^1$-action $\tau$ on $M$ which preserves $\alpha$. Then we have $[Y',R]=0$. Hence $Y'$ and $R$ generate an effective $\alpha$-preserving $T^2$-action. By Lemma \ref{PertubationOfReebVectorFields}, $M$ has a $K$-contact form of rank $2$.

\subsubsection{Contact toric manifolds}\label{Subsubsection : ContactToricManifolds}
\begin{defn}(Contact toric manifolds) A $(2n-1)$-dimensional manifold with a contact structure $\xi$ and a $\xi$-preserving $T^n$-action is called a contact toric manifold. 
\end{defn}
If $n$ is greater than $2$, then there exists a one to one correspondence between equivariant isomorphism classes of contact toric manifolds and isomorphism classes of good cones in $\mathbb{R}^{n}$, which are polyhedra in $\mathbb{R}^{n}$ with certain combinatorial conditions (See Lerman \cite{Ler2} and Boyer-Galicki \cite{BoGa2}). Lerman \cite{Ler4} showed that the fundamental group of a contact toric manifold is finite and abelian. 

Let $\Omega$ be the subset of $\Lie(T^n)$ consisting of the elements whose infinitesimal actions are the Reeb vector fields of contact forms $\beta$ which satisfy $\ker \beta=\xi$. Then $\Omega$ is a connected cone in $\Lie(T^n)$. Since the flow generated by the infinitesimal action of an element of $\Omega$ preserves a $T^n$-invariant metric, the contact form whose Reeb vector field corresponds to an element of $\Omega$ is $K$-contact. $\Lie(T^n)$ has the lattice $\Lie(T^n)_{\mathbb{Z}}$ defined by the kernel of the exponential map $\Lie(T^n) \longrightarrow T^n$. Take a $\mathbb{Z}$-basis $\{\overline{X}_i\}_{i=1}^{n}$ of $\Lie(T^n)_{\mathbb{Z}}$. The rank of the $K$-contact form $\beta$ which corresponds to an element $X=\sum_{i=1}^{n}a_i \overline{X}_i$ of $\Omega$ is equal to the dimension of the vector space $\Span_{\mathbb{Q}}\{a_1,a_2,\cdots,a_n\}$ over $\mathbb{Q}$.

\subsubsection{Fiber join construction and contact fiber bundles}
The fiber join construction is due to Yamazaki \cite{Yam1}. The contact fiber bundle construction is due to Lerman \cite{Ler3}. The contact fiber bundle construction is a generalization of the fiber join construction.

Let $(N,\beta)$ be a $K$-contact manifold. Assume that $\beta$ is invariant under an action $\rho$ of a compact Lie group $G$. Let $B$ be a closed manifold and $\pi \colon E \longrightarrow B$ be a principal $G$-bundle over $B$. Assume that a $G$-connection form $\alpha$ satisfies the following condition: The composition of 
\begin{equation}
\xymatrix{ \ker \alpha \otimes \ker \alpha \ar[r]^<<<<<{\curv(\alpha)} & \Lie(G) \ar[r]^<<<<<{\Phi(x)} & \mathbb{R} }
\end{equation}
is nondegenerate at every point $x$ on $N$, where $\curv(\alpha)$ is the curvature form of $\alpha$ and $\Phi(x)$ is the value at the point $x$ of the contact moment map of $(N,\beta)$ for $\rho$ (See Definition \ref{ContactMomentMap}). Then we have a $K$-contact form on the total space of the bundle $E \times_{G} N$ associated with the principal $G$-bundle $\pi$ which is invariant under the induced action of $G$.

Let $S_{g}$ be an oriented closed surface of genus $g$. Let $L(p,q)$ be the lens space of type $(p,q)$. By the above construction, we obtain a $K$-contact form on $S_g \times L(p,q)$ of rank $2$ such that $\{x\} \times L(p,q)$ is a $K$-contact submanifold for each $x$ in $S_g$. If $g$ is positive, then the fundamental group of $S_g \times L(p,q)$ is not finite. These are the examples of the $K$-contact manifolds of rank $2$ which cannot have a $K$-contact structure of rank $3$, because $5$-dimensional contact toric manifolds have finite abelian fundamental groups by a result of Lerman \cite{Ler4}.

\subsubsection{Join construction}
The join construction is an orbifold version of the contact fiber bundle construction due to Wang-Ziller \cite{WaZi}. See also Boyer-Galicki-Ornea \cite{BoGaOr}.

Let $(M_1,\alpha_1)$ be a $3$-dimensional $K$-contact manifold of rank $1$. Let $(M_2,\alpha_2)$ be a $3$-dimensional $K$-contact manifold of rank $2$. If the diagonal action of $S^1$ on $M_1 \times  M_2$ is free, then we have a $K$-contact form on $M_1 \times _{S^1} M_2$ of rank $2$ induced from the $1$-form $\alpha_1 - \alpha_2$ on $M_1 \times  M_2$.

\subsubsection{Links of weighted homogeneous polynomials}
A polynomial $f$ in $\mathbb{C}[z_1,z_2,\cdots,z_n]$ is weighted homogeneous if $f$ satisfies
\begin{equation}\label{QuasiHomogeneous}
f(\lambda^{w_1}z_1,\lambda^{w_2}z_2,\cdots,\lambda^{w_n}z_n)=\lambda^{d}f(z_1,z_2,\cdots,z_n)
\end{equation}
for a vector $w=(w_1,w_2,\cdots,w_n)$ whose entries are positive integers, a positive integer $d$ and any $\lambda$ in $\mathbb{C}$. If the hypersurface in $\mathbb{C}^{n}$ defined by $f$ has an isolated singularity at $0$, then the link $L_{f}$ of $f$ at $0$ is defined by 
\begin{equation}
L_f=\{(z_1,z_2,\cdots,z_n) \in \mathbb{C}^{n} | f(z_1,z_2,\cdots,z_n)=0, \sum_{i=1}^{n} |z_i|^{2}=1 \}.
\end{equation}
$L_{f}$ is a smooth manifold with the $K$-contact form defined by the restriction of the contact form
\begin{equation}
\alpha_{w} = \sum_{i=1}^{n} \frac{\sqrt{-1}(z_{i}d\overline{z}_{i} - \overline{z}_{i}dz_{i})}{2w_i|z_i|^2}
\end{equation}
on the unit sphere $S^{2n-1}$ in $\mathbb{C}^{n}$. Note that the Reeb vector field of $\alpha_{w}$ generates the $S^1$-action defined by $\lambda \cdot (z_1,z_2,\cdots,z_n)= (\lambda^{w_1}z_1,\lambda^{w_2}z_2,\cdots,\lambda^{w_n}z_n)$ for $\lambda$ in $S^1$. The equation \eqref{QuasiHomogeneous} implies that the hypersurface defined by $f$ is invariant under the Reeb flow of $\alpha_{w}$. See \cite{BoGa3} and \cite{Tak}. 

We see that $L_f$ has a $K$-contact form of rank $2$ if 
\begin{equation}\label{QuasiHomogeneous2}
f(\lambda^{v_1}z_1,\lambda^{v_2}z_2,\cdots,\lambda^{v_n}z_n)=\lambda^{d'}f(z_1,z_2,\cdots,z_n)
\end{equation}
is satisfied for an element $v$ of $\mathbb{Z}^{n}$ which is not parallel to $w$ and an integer $d'$. Define an $S^1$-action $\sigma_{v}$ on $S^{2n-1}$ by $\lambda \cdot (z_1,z_2,\cdots,z_n)= (\lambda^{v_{1}}z_1,\lambda^{v_{2}}z_2,\cdots,\lambda^{v_{n}}z_n)$. Then $\sigma_{v}$ preserves $\alpha_{w}$. Moreover $L_f$ is preserved by $\sigma_{v}$. Hence $L_{f}$ has a $K$-contact form of rank $2$ by Lemma \ref{PertubationOfReebVectorFields}.

\section{Contact moment maps and graphs of isotropy data}\label{ContactMomentMapsAndGraphsOfIsotropyData}

We define the contact moment maps and show that the contact moment maps on $5$-dimensional $K$-contact manifolds of rank $2$ are Bott-Morse functions. We define graphs which represent combinatorial properties of the torus actions associated with $K$-contact forms. We will show Theorem \ref{KarshonClassification} which claims that graphs of isotropy data classify closed $5$-dimensional $K$-contact manifolds.

\subsection{Contact moment maps}

Let $(M,\alpha)$ be a closed connected $5$-dimensional $K$-contact manifold of rank $2$. We denote the Reeb vector field of $\alpha$ by $R$. Let $H$ be a Lie group and $\tau$ be an $\alpha$-preserving $H$-action on $M$. For each point $x$ on $M$, we define an element $\overline{\alpha}_x$ of the $\mathbb{R}$-dual space $\Lie(H)^{*}$ of $\Lie(H)$ by $\overline{\alpha}_x(\overline{Y})=\alpha_x(Y_x)$ for an element $\overline{Y}$ of $\Lie(H)$ where $Y$ is the infinitesimal action of $\overline{Y}$.

\begin{defn}\label{ContactMomentMap}(Contact moment maps) We define the contact moment map $\Phi_{\alpha}^{H}$ of $(M,\alpha)$ for $\tau$ by
\begin{equation}
\begin{array}{cccc}
\Phi_{\alpha}^{H} : & M & \longrightarrow & \Lie(H)^{*} \\
                & x & \longmapsto     & \overline{\alpha}_x.
\end{array}
\end{equation}
\end{defn}

Let $G$ be the closure of the Reeb flow in the isometry group $\Isom(M,g)$ for a Riemannian metric $g$ invariant under the Reeb flow. We denote the action of $G$ on $M$ by $\rho$. Let $\Phi_{\alpha}$ be the contact moment map for $\rho$. We fix a basis $\{\overline{R},\overline{X}\}$ of $\Lie(G)$ where $\overline{R}$ is the element of $\Lie(G)$ whose infinitesimal action is $R$. Let $\Phi_{\alpha}$ be the contact moment map for $\rho$. Then $\Phi_{\alpha}$ is presented as
\begin{equation}\label{CoordinatePresentation}
\Phi_{\alpha}(x)= \alpha_x(R_x) \overline{R}^{*} + \alpha_x(X_x) \overline{X}^{*}= \overline{R}^{*} + \alpha_x(X_x) \overline{X}^{*}
\end{equation}
\noindent where $\{\overline{R}^{*},\overline{X}^{*}\}$ is the basis of $\Lie(G)^{*}$ dual to $\{\overline{R},\overline{X}\}$ and $X$ is the infinitesimal action of $\overline{X}$. Hence the image of $\Phi_{\alpha}$ is contained in a $1$-dimensional affine space $\{v \in \Lie(G)^{*}| v(\overline{R})=1\}$ of $\Lie(G)^{*}$. We often fix $X$ as above and consider the function $\Phi=\alpha(X) \colon M \longrightarrow \mathbb{R}$. $\Phi$ is unique up to the multiplication by a real number. 
\begin{defn}(The maximal component and the minimal component of the contact moment map) We call the maximal component $B_{\max}$ and the minimal component $B_{\min}$ of $\alpha(X)$ the maximal component and minimal component of the contact moment map. 
\end{defn}
 Note that the maximal component and the minimal component can change if we change $X$.

$\Phi$ has the following fundamental properties.
\begin{lem}\label{MorseTheory}
\begin{enumerate}
\item Each level set of $\Phi$ is a union of orbits of $\rho$ and connected.
\item Let $\Sing \rho$ be the union of closed orbits of the Reeb flow on $(M,\alpha)$. Then we have $\Crit \Phi=\Sing \rho$.
\item Every connected component of $\Crit \Phi$ is an odd dimensional $K$-contact submanifold of $M$.
\item Except $B_{\max}$ and $B_{\min}$, every connected component of $\Crit \Phi$ is a closed orbit of the Reeb flow.
\end{enumerate}
\end{lem}

\begin{proof}
(i) Since we have
\begin{equation}\label{MomentMap}
Y(\alpha(X))=L_{Y}\alpha(X)+\alpha([Y,X])=0
\end{equation}
for an infinitesimal action $Y$ of $\rho$, $\alpha(X)$ is constant on the orbits of $\rho$. Hence the former part follows. Let $\sigma$ be an $S^1$-subaction of $\rho$ generated by a vector field $Z$ whose orbits are transverse to $\ker \alpha$. We define a $1$-form $\beta$ by $\frac{1}{\alpha(Z)}\alpha$. Then by the argument in the proof of Lemma \ref{PertubationOfReebVectorFields}, $\beta$ is the contact form with Reeb vector field $Z$ defining the contact structure $\ker \alpha$. $d\beta$ induces a symplectic form $\omega$ on the orbifold $M/\sigma$. $\rho$ induces an $\omega$-preserving $S^1$-action $\tau$ on $M/\sigma$. By $L_{Z}\beta=0$, we have
\begin{equation}\label{Hamiltonian}
d(\beta(Z))(Y)=-d\beta(Z,Y)
\end{equation}
for a vector field on $M/\sigma$. \eqref{Hamiltonian} implies that $\alpha(Z)$ is a hamiltonian function for the $S^1$-action $\tau$ on $(M/\sigma,\omega)$. The latter part follows from the connectivity of the fiber of symplectic moment maps for hamiltonian actions on symplectic orbifolds. We refer \cite{LeTo}.

(ii) For a point $x$ on $M$ and a vector $Y_x$ in $T_{x}M$, we have 
\begin{equation}\label{Nondegeneracy}
d(\alpha(X))_x(Y_x)=-d\alpha_x(X_x,Y_x)
\end{equation}
by $L_{X}\alpha=0$. The left hand side of \eqref{Nondegeneracy} is $0$ for every $Y$ in $T_{x}M$ if and only if $x$ is a critical point of $\Phi$. Since $d\alpha_x$ is nondegenerate on $T_xM/\mathbb{R}R_x$, the right hand side of \eqref{Nondegeneracy} is $0$ for every $Y$ in $T_{x}M$ if and only if $X_x$ is parallel to $R_x$. Hence the proof is completed.

We show (iii) and (iv). If the isotropy group $G_{\Sigma}$ of $\rho$ at $\Sigma$ is connected, then we have an open tubular neighborhood of $\Sigma$ which is diffeomorphic to $S^1 \times D^{4}_{\epsilon}$ such that $\alpha(X)$ is written in the standard coordinate as
\begin{equation}\label{EquationOfPhi}
\alpha(X)= w_1|z_1|^2 + w_2|z_2|^2 + c
\end{equation}
for real numbers $w_1,w_2$ and $c$ where $\Sigma=\{(\zeta,z_{1},z_{2}) \in S^1 \times D^{4}_{\epsilon}| z_1=0, z_2=0\}$ by the argument in Subsection \ref{ConnectedIsotropyGroupCases}. Then the connected component of $\Crit \Phi$ is $\{z_1=0\}$, $\{z_2=0\}$ or $\{z_1=0, z_2=0\}$. The first two cases occur only if $\Sigma$ is contained in the minimal component $B_{\min}$ or in the maximal component $B_{\max}$ of $\Phi$. Hence (iii) and (iv) are proved in this case. If $G_{\Sigma}$ is not connected, we have a finite cyclic covering of open tubular neighborhood of $\Sigma$ which is diffeomorphic to $S^1 \times D^{4}_{\epsilon}$ such that the pullback of $\alpha(X)$ is written as \eqref{EquationOfPhi} by Lemma \ref{FiniteCovering}. The rest of the argument is the similar to the previous case.
\end{proof}

\subsection{Chains of gradient manifolds}

We apply the Bott-Morse theory to the contact moment maps for the torus action associated with $K$-contact forms of rank $2$.

Let $(M,\alpha)$ be a $5$-dimensional $K$-contact manifold of rank $2$. We denote the Reeb vector field of $\alpha$ by $R$.
\begin{defn}(Riemannian metrics compatible with $K$-contact forms) We say a Riemannian metric $g$ on $M$ invariant under the Reeb flow is compatible with $\alpha$ if the following conditions are satisfied:
\begin{enumerate}
\item The element $J$ in $\Aut(\ker \alpha)$ determined by the equation $g(Jv,w)=d\alpha(v,w)$ for every $v$ and $w$ in $C^{\infty}(\ker \alpha)$ satisfies $J^{2}=-\id|_{\ker \alpha}$.
\item $\ker \alpha$ is orthogonal to $R$ with respect to $g$.
\end{enumerate}
\end{defn}
We show that there exists a compatible metric $g$ on any $K$-contact manifold $(M,\alpha)$ as follows: Fix a metric $g'$ invariant under the Reeb flow. Let $G$ be the closure of the Reeb flow in $\Isom(M,g')$. Then $\ker \alpha$ is a $G$-equivariant vector bundle over $M$ with an invariant symplectic structure $d\alpha|_{\ker \alpha}$. Define $A$ in $\Aut(\ker \alpha)$ by $d\alpha(Y,Z)=g'(AY,Z)$ for $Y$ and $Z$ in $C^{\infty}(\ker \alpha)$. $-A^{2}$ is positive symmetric with respect to $g$. Hence $\sqrt{-A^{2}}$ is well-defined. We put $J=(\sqrt{-A^{2}})A$. Then $J$ is a $G$-invariant complex structure compatible with $d\alpha$. The metric $g_1$ on $\ker \alpha$ defined by $g_1(Y,Z)=d\alpha(Y,JZ)$ is a $G$-invariant metric compatible with $d\alpha$. We obtain a metric compatible with $\alpha$ by extending $g_{1}$ to $TM \otimes TM$ so that $R$ is orthogonal to $\ker \alpha$.

We fix a metric $g$ on $M$ compatible with $\alpha$. Let $G$ be the closure of the Reeb flow in $\Isom(M,g)$. Let $\rho$ be the action of $G$ on $M$. We fix an element $\overline{X}$ of $\Lie(G)$ which is not parallel to the element corresponding to $R$. We denote the function $\alpha(X)$ on $M$ by $\Phi$. Then we have
\begin{equation}
\begin{array}{ccc}\label{gradientflow}
\grad(\Phi)=-JX, & [X,JX]=0, & [R,JX]=0.
\end{array}
\end{equation}
Hence we have a $(T^2 \times \mathbb{R})$-action $\sigma$ on $M$ by the product of $\rho$ and the gradient flow of $\Phi$ on $M$.

\begin{defn}(Gradient manifolds and their limit sets) An orbit of $\sigma$ is called a gradient manifold. A gradient manifold $L$ is called free if $L$ contains a $T^2$-orbit of $\rho$ consisting of points with trivial isotropy group. The $\alpha$-limit set of a gradient manifold $L$ is the union of $\alpha$-limit set of the orbit of the gradient flow of $\Phi$ contained in $L$. The $\omega$-limit set of $L$ is defined similarly.
\end{defn}

Then we have the following lemma:
\begin{lem}\label{MorseBottTheory}
\begin{enumerate}
\item Let $L$ be a gradient manifold in $(M,\alpha)$. Then $(L,\alpha|_{L})$ is a $K$-contact manifold of rank $2$.
\item The isotropy groups of $\rho$ at any two points in a gradient manifold coincide.
\item The $\alpha$-limit set of a gradient manifold is a closed orbit of the Reeb flow. The $\omega$-limit set of a gradient manifold is a closed orbit of the Reeb flow.
\item Let $\Sigma$ be a closed orbit of the Reeb flow. If $\Sigma$ is not contained in $B_{\max}$ nor in $B_{\min}$, then $\Sigma$ is the $\alpha$-limit set and the $\omega$-limit set of two different gradient manifolds. If $\Sigma$ is contained in $B_{\max}$ or $B_{\min}$ of dimension $3$, then $\Sigma$ is the $\omega$-limit set or the $\alpha$-limit set of a gradient manifold. If $\Sigma$ is contained in $B_{\max}$ or $B_{\min}$ of dimension $1$, then $\Sigma$ is the $\omega$-limit set or the $\alpha$-limit set of uncountably many gradient manifolds. 
\item If a limit set $\Sigma$ of a gradient manifold $L$ is not contained in $B_{\max}$ or $B_{\min}$ of dimension $1$, then the closure of $L$ is a smooth submanifold of $M$ near $\Sigma$. A closed orbit $\Sigma$ of the Reeb flow in $B_{\max}$ or $B_{\min}$ of dimension $1$ is contained in the closure of two gradient manifolds whose closures are smooth submanifolds near $\Sigma$.
\end{enumerate}
\end{lem}

\begin{proof}
(i) Let $J$ be the complex structure on $\ker \alpha$ defined by $d\alpha$ and $g$. By the first equation of \eqref{gradientflow}, $TL \cap \ker \alpha$ is spanned by $X$ and $JX$. Hence $TL \cap \ker \alpha$ is a symplectic subspace of $(\ker \alpha, d\alpha)$. Since $L$ is a union of the orbits of Reeb flow, $(L,\alpha|_{L})$ is a $K$-contact manifold. Since $L$ contains the orbits of the Reeb flow, the rank of $(L,\alpha|_{L})$ is $2$.

(ii) follows from the fact that $\rho$ commutes with the gradient flow of $\Phi$.

(iii) and (iv) follow from the facts that $\Phi$ is a Bott-Morse function and the Morse index of each connected component of $\Crit \Phi$ is even. We will show that $\Phi$ is a Bott-Morse function. By Lemma \ref{MorseTheory}, each connected component of $\Crit \Phi$ is a smooth submanifold of $M$ which is a union of closed orbits of the Reeb flow of $\alpha$. Let $\Sigma$ be a closed orbit of the Reeb flow of $\alpha$. By Lemma \ref{FiniteCovering} and \eqref{NormalFormOfPhi}, we have a finite covering of an open tubular neighborhood of $\Sigma$ diffeomorphic to $S^1 \times D^{4}_{\epsilon}$ such that $\Phi$ is written as
\begin{equation}
\alpha(X)= w_1|z_1|^2 + w_2|z_2|^2 + c
\end{equation}
in the standard coordinate on $S^1 \times D^{4}_{\epsilon}$ for some real numbers $w_{1}$ and $w_{2}$ such that $\Sigma$ is defined by the equation $z_{1}=0$ and $z_{2}=0$. We can assume that $w_{1} \geq w_{2}$. If both of $w_{1}$ and $w_{2}$ are nonnegative, then $\Sigma$ is contained in the minimal component of $\Phi$ and the Morse index is $0$. If $w_{1}$ is positive and $w_{2}$ is negative, then $\Sigma$ is a connected component of $\Crit \Phi$ of Morse index $2$. If both of $w_{1}$ and $w_{2}$ are negative, then $\Sigma$ is contained in the maximum component of $\Phi$ and the Morse index is $4$. In any case, the Hessian $\Hess_{\Phi}$ of $\Phi$ is nondegenerate on the transverse direction of the connected component of $\Crit \Phi$. Hence $\Phi$ is a Bott-Morse function and the Morse index of each connected component of $\Crit \Phi$ is even. Let $\{\phi_{t}\}$ be the flow generated by $\grad \Phi$. Then $\Hess_{\Phi}$ at $x_{0}$ is dual to the element of $\Aut(T_{x_{0}}M)$ defining the differential of the action of $\{-(D\phi_{t})_{x_0}\}_{t \in \mathbb{R}}$ on $T_{x_{0}}M$. In fact, for vector fields $Y$ and $Z$ locally defined near $x_{0}$, we have
\begin{equation}\label{Hessian}
\Hess_{\Phi} (Y_{x_{0}},Z_{x_{0}}) = (Y Z \Phi)(x_{0}) = Y_{x_{0}} d\Phi(Z)(x_{0}) = Y_{x_{0}}(g(\grad \Phi,Z))(x_{0}) = g([Y,\grad \Phi],Z)(x_{0}).
\end{equation}
The last equality follows from the chain rule and the equality $(\grad \Phi)_{x_{0}}=0$. Since the differential of the action of $\{-(D\phi_{t})_{x_0}\}_{t \in \mathbb{R}}$ on $T_{x_{0}}M$ by mapping $Y_{x_{0}}$ to $[\grad \Phi,Y]_{x_{0}}$, \eqref{Hessian} implies that the Hessian of $\Phi$ at $x_{0}$ is dual to the element of $\Aut(T_{x_{0}}M)$ defining the differential of the action of $\{-(D\phi_{t})_{x_0}\}_{t \in \mathbb{R}}$ on $T_{x_{0}}M$. Hence the dimension of $W^{s}_{\loc 1}$ is equal to the dimension of the maximal positive definite subspace of $T_{x_{0}}M$ with respect to $\Hess_{\Phi}$.

Then (iii) and (iv) follow from the stable manifold theorem for flows:
\begin{thm}\label{StableManifoldTheorem}
Let $Z$ be a smooth vector field on a smooth manifold $M$. Let $\{\phi_{t}\}$ be the flow generated by $Z$. Suppose that $x_{0}$ is a fix point of $\{\phi_{t}\}$. Then we have a decomposition $T_{x_{0}}M=W_{s} \oplus W_{c} \oplus W_{u}$ where $W_{s}$, $W_{c}$ and $W_{u}$ are the union of generalized eigenspaces of the linear action of $\{(D\phi_{t})_{x_0}\}_{t \in \mathbb{R}}$ on $T_{x_{0}}M$ with respect to eigenvalues of the absolute value less than $1$, the absolute value $1$ and the absolute value greater than $1$ respectively. Then there exist an open neighborhood $U$ of $x_{0}$, a distance $d$ on $U$ and a smoothly embedded disk $W^{s}_{\loc}$ in $M$ such that
\begin{enumerate}
\item $W^{s}_{\loc}=\{x \in U | \phi_{t}(U) \in U$ for every $t$ in $\mathbb{R}_{>0}$ and $d(\phi_{t}(x),0)$ tends to zero exponentially $\}$,
\item $W^{s}_{\loc}$ tangent to $W_{s}$ at $0$ and
\item $\phi_{1}|_{W^{s}_{\loc}}$ is a contraction map on $(W^{s}_{\loc},d)$.
\end{enumerate}
\end{thm}
\begin{proof}
By Theorem III.7 (1) in \cite{Shu}, we have an open neighborhood $U$, a distance $d$ on $U$ and a smoothly embedded disk $W^{s}_{\loc 1}$ in $M$ such that
\begin{enumerate}
\item $W^{s}_{\loc 1}=\{x \in U | \phi_{1}^{n}(U) \in U$ for every $n \geq 0$ and $d(\phi_{1}^{n}(x),0)$ tends to zero exponentially $\}$,
\item $W^{s}_{\loc 1}$ tangent to $W_{s}$ at $0$ and
\item $\phi_{1}|_{W^{s}_{\loc 1}}$ is an contraction map on $(W^{s}_{\loc 1},d)$.
\end{enumerate}
Since $\phi_{1}|_{W^{s}_{\loc 1}}$ is an contraction map on $(W^{s}_{\loc 1},d)$, we can assume that $U$ is relatively compact changing $W^{s}_{\loc 1}$ smaller. We put $W^{s}_{\loc}=\{x \in U | \phi_{t}(U) \in U$ for every $t$ in $\mathbb{R}_{>0}$ and $d(\phi_{t}(x),0)$ tends to zero exponentially $\}$. We put $||v||=d(v,0)$ for $v$ in $W^{s}_{\loc 1}$. To show Theorem \ref{StableManifoldTheorem}, it suffices to show that $W^{s}_{\loc}=W^{s}_{\loc 1}$. $W^{s}_{\loc}$ is contained in $W^{s}_{\loc 1}$ by definition. Take a point $y$ on $W^{s}_{\loc 1}$. We fix a constant $c$ such that $d(D\phi_{t}(v),0) \leq c$ for every $t$ in $[0,1]$ and $v$ in $\overline{U}$. For $t$ in $\mathbb{R}_{>0}$, we denote the greatest integer which is less than $n$ by $E(n)$. By the mean value theorem, we have
\begin{equation}
||\phi_{t}(y)|| = ||\phi_{t - E(t)}(\phi_{E(t)}(y)) - \phi_{t - E(t)}(\phi_{E(t)}(0))|| = \sup_{v \in \overline{U}}||D\phi_{t}(v)|| ||\phi_{E(t)}(y) - \phi_{E(t)}(0)|| \leq c ||\phi_{E(t)}(y)||.
\end{equation}
Since $||\phi_{E(t)}(y)||$ tends to zero exponentially when $t$ tends to $\infty$, $||\phi_{t}(y)||$ tends to zero exponentially. Hence we have $W^{s}_{\loc}=W^{s}_{\loc 1}$. The proof of Theorem \ref{StableManifoldTheorem} is completed.
\end{proof}

We show (v). If a limit set $\Sigma$ of a gradient manifold $L$ is not contained in $B_{\max}$ or $B_{\min}$ of dimension $1$, then the stable manifolds and the unstable manifolds of $\Sigma$ are of dimension $3$. Hence $L$ is the stable manifold or the unstable manifold of $\Sigma$, which has a smooth closure near $\Sigma$ by Theorem \ref{StableManifoldTheorem}. The latter part of (v) follows from Lemma \ref{TwoKContactSubmanifolds2}.
\end{proof}

We define chains of gradient manifolds.
\begin{defn}(Chains and its nontriviality) A chain is a finite ordered set of gradient manifolds $\{ L_1,L_2,\cdots,L_k \}$ of $\Phi$ which satisfies the following conditions: 
\begin{enumerate}
\item The $\alpha$-limit set of $L_1$ is a closed orbit of the Reeb flow contained in $B_{\min}$.
\item The $\omega$-limit set of $L_k$ is a closed orbit of the Reeb flow contained in $B_{\max}$.
\item The $\alpha$-limit set of $L_{i}$ and the $\omega$-limit set of $L_{i+1}$  coincide with each other $(i=1,2,\cdots,k-1)$.
\end{enumerate}
A chain is nontrivial if it contains more than one gradient manifold or a gradient manifold consisting of points with nontrivial isotropy group of $\rho$.
\end{defn}

\subsection{$K$-contact submanifolds}

We have the following:
\begin{lem}\label{KContactSubmanifolds}
\begin{enumerate}
\item We put $N=\{x \in M-\Crit \Phi| G_{x} \neq \{e\}\}$. Each connected component of $N$ is a gradient manifold whose closure in $M$ is a smooth submanifold diffeomorphic to a lens space.
\item A connected $3$-dimensional closed $K$-contact submanifold of $(M,\alpha)$ consisting of points with nontrivial isotropy groups of $\rho$ is $B_{\min}$, $B_{\max}$ or the closure of a gradient manifold.
\item The isotropy group $G_{L}$ of $\rho$ at a connected component $L$ of $N$ is cyclic.
\end{enumerate}
\end{lem}

\begin{proof}
(i) Let $L$ be a connected component of $N$. Take a point $x$ in $L$ and fix a metric $g$ on $M$. Let $L'$ be the gradient manifold containing $x$. Since two points of $L'$ have the same isotropy groups of $\rho$ by Lemma \ref{MorseBottTheory} (ii), $L'$ is contained in $L$. On an open tubular neighborhood of a limit set $\Sigma$ of $L'$, $\rho$ is written as 
\begin{equation}
(t_0,t_1) \cdot (\zeta,z_1,z_2) = (t_{0}^{n_0} \zeta, t_{0}^{n_1} t_{1}^{m_1} z_1, t_{0}^{n_2} t_{1}^{m_2} z_2)
\end{equation}
for some integers $n_0,n_1,n_2,m_1$ and $m_2$ by Lemma \ref{GeneralTorusActions} where $\Sigma$ is $\{(\zeta,z_1,z_2) \in S^1 \times D^4 | z_1=0, z_2 = 0 \}$. Both of $L'$ and $L$ should coincide with $\{(\zeta,z_1,z_2) \in S^1 \times D^4 | z_1 = 0, z_2 \neq 0 \}$ or $\{(\zeta,z_1,z_2) \in S^1 \times D^4 | z_1 \neq 0, z_2 = 0 \}$. Hence the closure of $L$ is a smooth submanifold. Since the closure of $L$ in $M$ is a $3$-dimensional $K$-contact manifold of rank $2$ by Lemma \ref{MorseBottTheory}, the closure of $L$ is diffeomorphic to a lens space by the classification theorem of contact toric manifolds by Lerman \cite{Ler2}.

(ii) A $K$-contact submanifold of $(M,\alpha)$ consisting of points with nontrivial isotropy groups is contained in the closure of $N$ or $\Crit \Phi$. Hence (ii) follows from (i) and Lemma \ref{MorseTheory}.

(iii) Since $G_{L}$ acts on a fiber of the normal bundle of $L$ in $M$ effectively and linearly, $G_{L}$ is isomorphic to a subgroup of $\SO(2)$. Hence $G_{L}$ is cyclic.
\end{proof}

\subsection{Combinatorial classification of closed $5$-dimensional $K$-contact manifolds of rank $2$}

We define graphs of isotropy data of closed $5$-dimensional $K$-contact manifolds of rank $2$ which represent combinatorial data of the $T^2$-actions associated with $K$-contact forms. We show that a closed $5$-dimensional $K$-contact manifold of rank $2$ can be recovered from the graph of isotropy data. Hence closed $5$-dimensional $K$-contact manifolds of rank $2$ are combinatorially classified. However, we do not know a good criterion to know whether a given graph can be realized by a closed $5$-dimensional $K$-contact manifolds of rank $2$ or not.

Karshon \cite{Kar} defined graphs for $4$-dimensional symplectic manifolds with  hamiltonian $S^1$-actions which have the combinatorial data of the $S^1$-actions and the hamiltonian functions for the $S^1$-actions. Karshon's uniqueness theorem claims that $4$-dimensional symplectic manifolds with hamiltonian $S^1$-actions can be recovered from the graph. Our results and arguments in this subsection are similar to those of Karshon in \cite{Kar}.

\subsubsection{Graphs of isotropy data}

Let $(M,\alpha)$ be a closed $5$-dimensional $K$-contact manifold of rank $2$. Let $G$ be the closure of the Reeb flow in $\Isom(M,g)$ for a metric $g$ on $M$ compatible with $\alpha$. Let $\rho$ be the action of $G$ on $M$. Let $\Sing \rho$ be the union of closed orbits of the Reeb flow of $\alpha$. The maximal and minimal components of the contact moment map for $\rho$ are denoted by $B_{\max}$ and $B_{\min}$ respectively. We fix an isomorphism $\psi$ from $G$ to $T^2$ and identify $\Lie(G)$ with $\Lie(T^2)$ by $\psi_{*}$. We define the graph of isotropy data of $(M,\alpha,\psi)$.
\begin{defn}\label{KarshonGraphs}(Graph of isotropy data of $(M,\alpha,\psi)$) The graph of isotropy data of $(M,\alpha,\psi)$ is the graph $\Gamma=(V,E)$ with data attached to the set $V$ of vertices and to the set $E$ of edges which we define as follows:
\begin{enumerate}
\item $V$ is defined to be $\pi_{0}(\Sing \rho)$.
\item $E$ is defined to be the set of $K$-contact lens spaces consisting of points with nontrivial isotropy groups of $\rho$. An edge $e$ connects two elements of $\pi_{0}(\Sing \rho)$ which contain the $\alpha$-limit set or the $\omega$-limit set of $e$.
\item A vertex $v$ has the datum of the isotropy group of $\rho$ at $v$.
\item An edge $e$ has the datum of the isotropy group of $\rho$ at $e$.
\item Let $B$ be $B_{\max}$ or $B_{\min}$ of dimension $3$. Then the corresponding vertex to $B$ is defined to be a fat vertex. We attach the following three data to a fat vertex:
\begin{enumerate}
\item The isotropy group of $\rho$ at $B$.
\item The Seifert invariant of the Seifert fibration with oriented base space on $B$ defined by the Reeb flow.
\item The Euler class of the symplectic normal bundle of $B$ in $M$.
\end{enumerate}
\item The graph $\Gamma$ has the datum of the element $\overline{R}$ in $\Lie(T^2)$ whose infinitesimal action is the Reeb vector field of $\alpha$.
\end{enumerate}
\end{defn}

\begin{figure}
\begin{equation}
\begin{picture}(46,81)(0,-10)
\put(20,90){\blacken\ellipse{5}{5}}
\put(40,70){\blacken\ellipse{5}{5}}
\put(50,30){\blacken\ellipse{5}{5}}
\put(0,40){\blacken\ellipse{5}{5}}
\put(20,0){\blacken\ellipse{20}{10}}
\path(0,40)(20,0)
\path(20,90)(40,70)(50,30)
\end{picture}
\end{equation}
\caption{A graph of isotropy data}
\label{graphofisotropydata}
\end{figure}
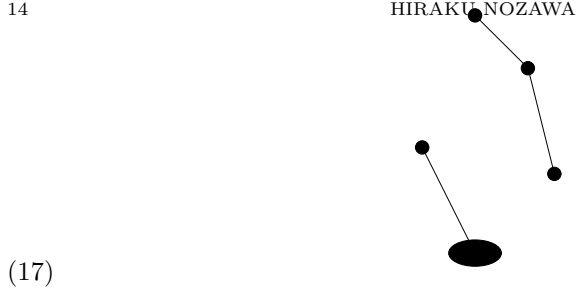

Note that the normal form imposes the following relation among the data by Lemma \ref{ClosedOrbitsInB}: Assume that $B_{\min}$ or $B_{\max}$ is of dimension $3$ and has an exceptional $S^1$-orbit $\Sigma$ of the Reeb flow. Let $L$ be the gradient manifold whose limit set contains $\Sigma$. Then the cardinality of the isotropy group of $\rho$ at $L$ is equal to the multiplicity of $\Sigma$ as an exceptional orbit of the Seifert fibration on $B_{\min}$.

Note that the graphs of isotropy data do not control the free gradient manifolds.

We define the isomorphism between graphs of isotropy data. Let $\Gamma_{0}=(V_{0},E_{0})$ and $\Gamma_{1}=(V_{1},E_{1})$ be two graphs of isotropy data.
\begin{defn}
A pair of maps $\phi_V \colon V_{0} \longrightarrow V_{1}$ and $\phi_E \colon E_{0} \longrightarrow E_{1}$ is called an isomorphism between graphs of isotropy data if the endpoints of $\phi_E(e)$ are $\phi_V(v)$ and $\phi_V(v')$ for every $e$ in $E_{0}$ where $v$ and $v'$ are endpoints of $e$, and the data attached to vertices and edges correspond to each other by an automorphism $A_{\phi}$ of $T^2$.
\end{defn}

\subsubsection{Recovering closed $5$-dimensional $K$-contact manifolds of rank $2$ from the graphs of isotropy data}\label{Subsection : KarshonClassification}  

Let $(M_i,\alpha_i)$ be a closed $5$-dimensional $K$-contact manifold of rank $2$ for $i=0$ and $1$. Let $R_i$ be the Reeb vector field of $\alpha_i$ for $i=0$ and $1$. Let $G_i$ be the closure of the Reeb flow in $\Isom(M_i,g_i)$ for a metric $g_i$ on $M_i$ compatible with $\alpha_i$ for $i=0$ and $1$. Let $\rho_i$ be the $G_i$-action on $M_i$ for $i=0$ and $1$. We fix an isomorphism $\psi_i$ from $G_i$ to $T^2$ for $i=0$ and $1$. We identify $G_i$ with $T^2$ by $\psi_i$ and $\Lie(G_i)$ with $\Lie(T^2)$ by $\psi_{i *}$. $\Phi_{\alpha_i} \colon M_i \longrightarrow \Lie(T^2)^{*}$ be the contact moment map for $\rho_{i}$ on $(M_i,\alpha_i)$ for $i=0$ and $1$. We denote the element of $\Lie(T^2)$ corresponding to the Reeb vector field of $\alpha_i$ by $\overline{R}_{i}$ for $i=0$ and $1$. Let $\Gamma_{i}$ be the graph of isotropy data of $(M_{i},\alpha_{i},\psi_{i})$ for $i=0$ and $1$. 

We have the following result similar to Proposition 4.3 of \cite{Kar}:
\begin{prop}\label{Karshon1}
If there exists an isomorphism $(\phi_V,\phi_E)$ between $\Gamma_{0}$ and $\Gamma_{1}$, then there exists a diffeomorphism $f \colon M_{0} \longrightarrow M_{1}$ such that
\begin{enumerate}
\item $f$ induces $(\phi_V,\phi_E)$ between $\Gamma_{0}$ and $\Gamma_{1}$,
\item $f_{*}R_{0}=R_{1}$ and 
\item $A_{\phi} \circ \Phi_{\alpha_1} \circ f = \Phi_{\alpha_0}$ where $A_{\phi}$ is the automorphism of $T^2$ associated with $(\phi_V,\phi_E)$.
\end{enumerate}
\end{prop}

\begin{proof}
By the definition of isomorphisms between graphs of isotropy data, we can assume that $\overline{R}_{0}=\overline{R}_{1}$ in $\Lie(T^2)$ and that the corresponding vertices and edges of $\Gamma_{0}$ and $\Gamma_{1}$ corresponded by $(\phi_V,\phi_E)$ have the same data. Fix $\overline{X}$ in $\Lie(T^2)$ which is not parallel to $\overline{R}_{0}=\overline{R}_{1}$. We define a function $\Phi_i$ on $M_i$ by $\Phi_i(x)=\alpha_i(X_i)(x)$ for a point $x$ on $M_i$ where $X_i$ is the infinitesimal action of $\overline{X}$ on $M_{i}$ for $i=0$ and $1$. Let $B_{\min i}$ and $B_{\max i}$ be the minimal and the maximal component of $\Phi_{i}$ for $i=0$ and $1$. 

We fix a metric $g_{i}$ on $M_{i}$ compatible with $\alpha_{i}$ for $i=0$ and $1$. Applying Lemma \ref{Linearization} to $B_{\max 0}$, we take $g_{0}$ so that 
\begin{enumerate}
\item for every isolated closed orbit $\Sigma$ of the Reeb flow of $\alpha_{0}$, $g_{0}$ is Euclidean with respect to a coordinate on a finite covering of an open neighborhood of $\Sigma$ which represent $\alpha_{0}$ as \eqref{NormalFormOfAlpha} and
\item the $\mathbb{C}^{\times}$-action defined by the product of the gradient flow of $\Phi_{0}$ and the action of the isotropy group of $\rho_{0}$ at $B_{\max 0}$ is locally isomorphic to a linear $\mathbb{C}^{\times}$-action on the normal bundle of $B_{\max 0}$ near $B_{\max 0}$.
\end{enumerate}
Note that the isotropy group of $\rho_{0}$ at $B_{\max 0}$ principally acts on the normal circle bundle of $B_{\max 0}$ by the effectiveness of $\rho_{0}$. By Lemma \ref{Perturbation} shown below, we can modify $g_i$ on the complement of an open neighborhood of the union of closed orbits of the Reeb flow so that $g_{i}$ is still compatible with $\alpha_{i}$ and one of limit sets of a free gradient manifold is contained in $B_{\min i}$ or $B_{\max i}$ for $i=0$ and $1$. 

By the assumption, Lemmas \ref{NormalFormOfSingularOrbits2} and \ref{NormalFormOfB}, there exists a diffeomorphism $f_{\min}$ which maps an open neighborhood of $B_{\min 0}$ to an open neighborhood of $B_{\min 1}$ and satisfies $f_{\min}^{*} \alpha_{1}=\alpha_{0}$.

We define a subset $N_{i}$ of $M_{i}$ to be the union of smooth $K$-contact lens spaces in $M_{i}$ which are the closures of gradient manifolds whose $\omega$-limit sets are not contained in $B_{\max i}$. Let $C^{j}_{0}$ be the closures of connected components of $N_{0} - B_{\min 0}$ in $M_0$ for $j=1$, $2$, $\cdots$, $k$. Let $C^{j}_{1}$ be the union of the $K$-contact lens spaces in $(M_1,\alpha_1)$ corresponding to the $K$-contact lens spaces contained in $C^{j}_{0}$ by $(\phi_{V},\phi_{E})$ for $j=1$, $2$, $\cdots$, $k$. By the assumption and Lemma \ref{NormalFormOfChains}, there exists an isomorphism $f^{j}$ from an open neighborhood of $C^{j}_{0}$ to an open neighborhood of $C^{j}_{1}$ such that $f^{j}$ coincides with $f_{\min}$ near $B_{\min 0}$ for each $j$ if the bottom closed orbit of $C^{j}_{0}$ is contained in $B_{\min 0}$. Hence we have a local diffeomorphism $f$ which satisfies the following:
\begin{enumerate}
\item $f$ maps an open neighborhood $W_{0}$ of $N_{0} \cup B_{\min 0}$ to an open neighborhood $W_{1}$ of $N_{1} \cup B_{\min 1}$,
\item $f$ induces $(\phi_V,\phi_E)$ and
\item $f^{*} \alpha_{1}=\alpha_{0}$.
\end{enumerate}

We modify the metric $g_{1}$ on $M_{1}$ compatible with $\alpha_{1}$ so that $f^{*}g_{1}=g_{0}$. Since one of the limit sets of every free gradient manifold with respect to $g_{0}$ is contained in $B_{\min 0}$ or $B_{\max 0}$, one of the limit sets of every free gradient manifold with respect to $g_{1}$ is also contained in $B_{\min 1}$ or $B_{\max 1}$. 

Let $\Sigma_{i}^{j}$ be the bottom closed orbit of $C_{i}^{j}$. Since the assumption of Lemma \ref{Octopus} is satisfied, we can extend $f$ near the gradient manifolds whose closure is not a smooth submanifold of $M_{i}$ and which connect $\Sigma^{j}_{i}$ to $B_{\min i}$ by Lemma \ref{Octopus}, so that
\begin{enumerate}
\item $f$ maps an open neighborhood $V_{0}$ of the union of $B_{\min 0}$ and the gradient manifolds whose $\omega$-limit sets are not contained in $B_{\max 0}$ to an open neighborhood $V_{1}$ of the union of $B_{\min 1}$ and the gradient manifolds whose $\omega$-limit sets are not contained in $B_{\max 1}$,
\item $f$ induces $(\phi_V,\phi_E)$ and
\item $f^{*} \alpha_{1}=\alpha_{0}$.
\end{enumerate}
We modify $g_{1}$ so that $g_{1}$ is compatible with $\alpha_{1}$ and satisfies $f^{*}g_{1}=g_{0}$.

By Lemma \ref{InvariantSubset} shown below, we have an open subset $U_{0}$ of $V_{0}$ which is an open neighborhood of the union of $B_{\min 0}$ and gradient manifolds whose $\omega$-limit sets are not contained in $B_{\max 0}$ such that $U_{0}$ satisfies $\phi^{t}_{0}(U_{0}) \subset U_{0}$ for every negative $t$ where $\{\phi^{t}_{0}\}$ is the gradient flow of $\Phi_{0}$ on $M_{0}$. We put $U_{1}=f(U_{0})$. Then $U_{1}$ contains the union of $B_{\min 0}$ and gradient manifolds whose $\omega$-limit sets are not contained in $B_{\max}$. $U_{1}$ satisfies that $\phi^{t}_{1}(U_{1}) \subset U_{1}$ for every negative $t$ where $\{\phi^{t}_{1}\}$ is the gradient flow on $M_{1}$, since $f$ is an isometry and satisfies $f^{*} \alpha_{1}=\alpha_{0}$. 

By the assumption, Lemmas \ref{NormalFormOfSingularOrbits2} and \ref{NormalFormOfB}, there exists a diffeomorphism $f_{\max}$ which maps an open neighborhood  $U_{\max 0}$ of $B_{\max 0}$ to an open neighborhood  $U_{\max 1}$ of $B_{\max 1}$ and satisfies $f_{\max}^{*} \alpha_{1}=\alpha_{0}$.

By Lemma \ref{Scaling} shown below, there exists a metric $g'_{0}$ on $M_{0}$ compatible with $\alpha_{0}$ which satisfies the following conditions:
\begin{enumerate}
\item $g'_{0}(X_{0},X_{0})$ is constant on the intersection of $M_{0}-U_{0}$ and each level set of $\Phi_{0}$,
\item $g'_{0}|_{U'_{\max 0}}=g_{0}|_{U'_{\max 0}}$ for an open neighborhood $U'_{\max 0}$ of $B_{\max 0}$ in $U_{\max 0}$ and
\item there exist a smooth function $h_{0}$ such that $h_{0}J_{0}X_{0}=J'_{0}X_{0}$ where $J_{0}$ and $J'_{0}$ are the CR structures determined by $g_{0}$ and $g'_{0}$ with $\alpha_{0}$, respectively.
\end{enumerate}
We take a metric $g'_{1}$ on $M_{1}$ compatible with $\alpha_{1}$ such that $g'_{0}=f^{*}g'_{1}$, $g'_{0}=f_{\max}^{*} g'_{1}$ and $g'_{1}(X_{0},X_{0})$ is constant on each level set of $\Phi_{1}$. Note that $U_{0}$ is invariant under the gradient flow with respect to $g'_{0}$ in the negative direction, because the orbits of the gradient flow with respect to $g'_{0}$ are same as the orbits of the gradient flow with respect to $g_{0}$ by the condition (iii). $U_{1}$ is also invariant under the gradient flow of $\Phi_{1}$ with respect to $g'_{1}$ in the negative direction.

Let $J'_{i}$ be the complex structure defined by $g'_{i}$ and $d\alpha_{i}$. Then the gradient flow of $\Phi_{i}$ with respect to $g'_{i}$ is generated by $J'_{i}X_{0}$ by \eqref{gradientflow}. By the condition (i), $\grad \Phi_{i}(\Phi_{i})$ is constant on the intersection of $M_{0}-U_{0}$ and each level set of $\Phi_{i}$. In fact, we have
\begin{equation}
\grad \Phi_{i}(\Phi_{i})=d\Phi_{i}(\grad \Phi_{i})=g'_{i}(\grad \Phi_{i},\grad \Phi_{i})=g'_{i}(J'_{i}X_{0},J'_{i}X_{0})=g'_{i}(X_{0},X_{0})
\end{equation}
by \eqref{gradientflow}. Hence the intersection of $M_{0}-U_{0}$ and the level sets of $\Phi_{0}$ are mapped to the intersection of $M_{0}-U_{0}$ and the level sets of $\Phi_{0}$ by the gradient flow of $\Phi_{0}$.

We extend $f$ on $M_{0}-B_{\max 0}$ by $f(x)=\phi^{\prime t}_{1} \circ f \circ \phi^{\prime -t}_{0} (x)$ where $\phi^{\prime -t}_{i}$ is the gradient flow of $\Phi_{i}$ with respect to $g'_{i}$ for $i=0,1$ and $t$ is taken so that $\phi^{\prime -t}_{0} (x)$ is contained in $U_{0}$. Then $f$ is well-defined and we have
\begin{enumerate}
\item $f$ maps $M_{0}-B_{\max 0}$ to $M_{1}-B_{\max 1}$,
\item $f$ induces $(\phi_V,\phi_E)$,
\item $\Phi_1 \circ f=\Phi_0$ and 
\item $f$ is $T^2$-equivariant
\end{enumerate}
by the construction of $g'_{0}$ and $g'_{1}$.

We consider the case where $B_{\max 0}$ and $B_{\max 1}$ are of dimension $3$. We will show that $f$ extends to $M_{0}$ smoothly. By the condition (ii) in the second paragraph of the proof of Proposition \ref{Karshon1}, the $\mathbb{C}^{\times}$-action defined by the product of the gradient flow of $\Phi_{i}$ and the action of the isotropy group of $\rho_{i}$ at $B_{\max i}$ is locally isomorphic to a linear $\mathbb{C}^{\times}$-action on the normal bundle of $B_{\max i}$ near $B_{\max i}$. Since $f$ is equivariant with respect to the $\mathbb{C}^{\times}$-actions defined by the product of the gradient flow of $\Phi_{i}$ and the action of the isotropy group of $\rho_{i}$ at $B_{\max i}$, $f$ can be regarded as a bundle map between vector bundles from the normal bundle of $B_{\max 0}$ to the normal bundle of $B_{\max 1}$. Hence $f$ extends to $M_{0}$ smoothly.

We consider the case where $B_{\max 0}$ and $B_{\max 1}$ are of dimension $1$. Applying Lemma \ref{GeneralTorusActions} to open neighborhoods of $B_{\max 0}$ and $B_{\max 1}$, there exists an interval $[v,v']$ in $\Phi_{0}(U_{\max 0})$ such that $\Phi_{0}^{-1}(w)$ is diffeomorphic to $S^1 \times S^3$ with a locally free $T^2$-action $\tau$ which has at most two exceptional $T^2$-orbits for every $w$ in $[v,v']$. We construct a diffeomorphism $f$ from $M_{0}$ to $M_{1}$ by connecting $f_{\max}|_{\Phi_{0}^{-1}([v',\infty))}$ and $f|_{\Phi_{0}^{-1}((-\infty,v])}$ on $\Phi_{0}^{-1}([v,v'])$ by an isotopy of $T^2$-equivariant diffeomorphisms on $S^1 \times S^3$. We fix a $T^2$-equivariant diffeomorphism $\Phi_{i}^{-1}([v,v']) \longrightarrow S^{1} \times S^{3} \times [v,v']$ for $i=0$ and $1$. We put $f_{v'}=f_{\max}|_{\Phi_{0}^{-1}(v')}$ and $f_{v}=f|_{\Phi_{0}^{-1}(v)}$. We denote the diffeomorphisms on the orbifold $(S^1 \times S^3)/\tau$ induced from $f_{v'}$ and $f_{v}$ by $\overline{f}_{v}$ and $\overline{f}_{v'}$. By the argument in the appendix of \cite{HaA} and Lemma B.2 of \cite{Kar}, the diffeomorphism group of $(S^1 \times S^3)/\tau$ is path-connected. Hence we can isotope $\overline{f}_{v'}$ to $\overline{f}_{v}$ by an isotopy $\{\overline{f}^{t}\}$ such that $\overline{f}^{0}=\overline{f}_{v}$ and $\overline{f}^{1}=\overline{f}_{v'}$. Hence we can isotope $f_{v'}$ to $f_{v}$ by an isotopy of $T^2$-equivariant diffeomorphisms. 

We have a diffeomorphism $f \colon M_{0} \longrightarrow M_{1}$ which satisfies 
\begin{enumerate}
\item $f$ induces $(\phi_V,\phi_E)$,
\item $\Phi_1 \circ f=\Phi_0$ and 
\item $f$ is $T^2$-equivariant with respect to $\id_{T^2}$.
\end{enumerate}
The condition (iii) implies $f_{*}R_{0}=R_{1}$ by the assumption. Then the condition (ii) and $f_{*}R_{0}=R_{1}$ imply $\Phi_{\alpha_{1}} \circ f=\Phi_{\alpha_{0}}$.
\end{proof}

We introduce our notation for three lemmas below. Let $(M,\alpha)$ be a $5$-dimensional closed $K$-contact manifold of rank $2$. Let $g$ be a metric on $M$ compatible with $\alpha$. $G$ denotes the closure of the Reeb flow in the isometry group $\Isom(M,g)$. We denote the action of $G$ on $M$ by $\rho$. Let $X$ be an infinitesimal action of an element $\overline{X}$ of $\Lie(G)$ where $\overline{X}$ is not parallel to the element corresponding to the Reeb vector field of $\alpha$. We put $\Phi=\alpha(X)$. The maximal component and the minimal component of $\Phi$ are denoted by $B_{\max}$ and $B_{\min}$, respectively.

The following lemma is similar to Lemma 3.6 in \cite{Kar}:
\begin{lem}\label{Perturbation}
Let $L$ be a free gradient manifold of $\Phi$ with respect to $g$. Assume that both of the $\alpha$-limit set $\Sigma^{0}$ of $L$ and the $\omega$-limit set $\Sigma^{1}$ of $L$ are not contained in $B_{\min} \cup B_{\max}$. Take a $T^2$-orbit $C$ of $\rho$ in $L$ and an open neighborhood $U$ of $C$. There exists a metric $g'$ compatible with $\alpha$ such that $\Sigma^{0}$ and $B_{\max}$ are connected by a gradient manifold,  $\Sigma^{1}$ and $B_{\min}$ are connected by a gradient manifold and $g|_{M-U}=g'|_{M-U}$. 
\end{lem}
\begin{proof}
Since the union $W$ of the orbits of the gradient flow whose limit sets are not contained in $B_{\min} \cup B_{\max}$ is closed in $M$, we can take an open neighborhood $V$ of $C$ in $U$ such that $V \cap W = L \cap W$. By Lemma \ref{NormalFormOfC}, we have a coordinate $T^2 \times [t^{0},t^{1}] \times \mathbb{R}^{2}$ on a tubular neighborhood in $V$ such that $\alpha=\frac{dx \cos t + dy \sin t}{a \cos t + b \sin t} + \frac{\sqrt{-1}}{2} (z_1 d\overline{z}_1 - \overline{z}_1dz_1)$ and $L$ is defined by the equation $z_1=0$ near $C$ where $(x,y)$ is the coordinate on $T^2$ induced from the Euclidean coordinate on $\mathbb{R}^{2}$, $t$ and $z_1$ are the standard coordinate on $[t^{0},t^{1}]$ and $\mathbb{R}^{2}$, respectively. Let $r$ be a nonnegative smooth function on $M$ invariant under $\rho$ whose support is contained in $V$ and whose restriction $r|_{C}$ to $C$ is positive. Then $r \frac{\partial}{\partial x_1}$ is a smooth vector field on $M$. We put $Z= \grad \Phi + r \frac{\partial}{\partial x_1}$. Since $d\alpha(Z,X)=d\alpha(\grad \Phi,X)=d\alpha(JX,X)$ is nowhere vanishing on $V$, we have a complex structure $J'$ which is compatible with $d\alpha$ and satisfies $J'X=Z$ and $J'|_{M-U}=J|_{M-U}$. The metric $g'$ defined by $J'$ and $\alpha$ satisfies the desired conditions.
\end{proof}

The following lemma is similar to Lemma 4.8 of \cite{Kar}:
\begin{lem}\label{InvariantSubset}
We assume that $g$ is Euclidean with respect to a coordinate on a finite covering of an open neighborhood of every isolated closed orbit of the Reeb flow of $\alpha_{0}$, where $\alpha_{0}$ is written as \eqref{NormalFormOfAlpha}. Let $V$ be an open neighborhood of the union of $B_{\min}$ and the gradient manifolds of $\Phi$ whose $\omega$-limit sets are not contained in $B_{\max}$. There exists an open subset $U$ of $V$ such that
\begin{enumerate}
\item $U$ is an open neighborhood of the union of $B_{\min}$ and gradient manifolds whose $\omega$-limit sets are not contained in $B_{\max}$ and
\item $U$ satisfies $\phi_{t}(U) \subset U$ for every negative $t$ where $\{\phi_{t}\}$ is the gradient flow of $\Phi$ on $M$.
\end{enumerate}

\end{lem}
\begin{proof}
We construct $U$ inductively. We denote gradient manifolds whose $\omega$-limit sets are not contained in $B_{\max}$ by $L^{1}$, $L^{2}$, $\cdots$, $L^{l}$ so that the $\alpha$-limit set of $L^{k}$ is contained in the closure of $B_{\min} \cup \cup_{j=1}^{k-1}L^{j}$ for $k=1$, $2$, $\cdots$, $l$.

We define an open neighborhood $U^{0}$ of $B_{\min}$ by $U^{0}=\{ x \in M | \Phi(x) < a \}$ where $a$ is a real number sufficiently close to the minimum value of $\Phi$. Then $U^{0}$ satisfies $\phi_{t}(U^{0}) \subset U^{0}$ for every negative $t$.

Assume that we have an open subset $U^{k-1}$ of $V$ which contains $B_{\min} \cup \cup_{j=1}^{k-1}L^{j}$ and satisfies $\phi_{t}(U^{k-1}) \subset U^{k-1}$ for every negative $t$. Let $\Sigma^{k}$ be the $\omega$-limit set of $L^{k}$. Take a sufficiently small open tubular neighborhood $V^{k}$ of $\Sigma^{k}$. We have a finite covering $\pi \colon \tilde{V}^{k} \longrightarrow V^{k}$ and a coordinate $(\zeta,z_1,z_2)$ on $\tilde{V}^{k}$ such that $\pi^{*}g$ is Euclidean and $\alpha$ is written as in \eqref{NormalFormOfAlpha} by the assumption. We can assume that $\pi^{-1}(L^{k})$ is defined by $z_1=0$. We define $C=\pi(\{ z_1=0, |z_2|=\epsilon \})$ for a sufficiently small positive number $\epsilon$. Since the $\alpha$-limit set of $L^{k}$ is contained in $U^{k-1}$, there exists a negative number $T$ such that $\phi_{T}(C)$ is contained in $U^{k-1}$ where $\{\phi_{t}\}_{t \in \mathbb{R}}$ is the gradient flow of $\Phi$. Hence there exists a positive number $\delta$ such that  $0 < \epsilon - \delta$ and $\phi_{T}(\pi(\{ |z_1|< \delta, \epsilon - \delta < |z_2| < \epsilon + \delta \}))$ is contained in $U^{k-1}$. Since $\pi^{*}g$ is Euclidean and $\pi^{*}\alpha$ is written as in \eqref{NormalFormOfAlpha}, we have 
\begin{equation}
\begin{array}{l}
\cup_{t<0}\phi_{t}(\pi(\{ |z_1|< \delta, |z_2| < \epsilon + \delta \})) \\ \subset \pi(\{ |z_1|< \delta, \epsilon - |z_2| < \epsilon + \delta \}) \cup \cup_{t<0}\phi_{t}(\pi(\{ |z_1|< \delta, \epsilon - \delta < |z_2| < \epsilon + \delta \})).
\end{array}
\end{equation} 
We put
\begin{equation}
U^{k}=U^{k-1} \cup \pi(\{ |z_1|< \delta, \epsilon - |z_2| < \epsilon + \delta \}) \cup \cup_{t<0}\phi_{t}(\pi(\{ |z_1|< \delta, \epsilon - \delta < |z_2| < \epsilon + \delta \})).
\end{equation}
Then $U^{k}$ satisfies $\phi_{t}(U^{k}) \subset U^{k}$ for every negative $t$ and contains $B_{\min} \cup \cup_{j=1}^{k}L^{j}$. Hence $U^{k}$ satisfies the desired conditions.
\end{proof}

\begin{lem}\label{Scaling}
Let $U$ be an open neighborhood of the union of the closed orbits of the Reeb flow of $\alpha$. Assume that $g(X,X)$ is constant on each level set of $\Phi$ in an open neighborhood $U_{\max}$ of $B_{\max}$. There exists a metric $g'$ on $M$ compatible with $\alpha$ which satisfies the following conditions:
\begin{enumerate}
\item $g'(X,X)$ is constant on the intersection of $M - U$ and each level set of $\Phi$,
\item $g|_{U'_{\max}}=g'|_{U'_{\max}}$ for an open neighborhood $U'_{\max}$ of $B_{\max}$ in $U$ and
\item there exists a smooth function $h$ such that $hJX=J'X$ where $J$ and $J'$ are the CR structures determined by $g$ and $g'$ with $\alpha$ respectively.
\end{enumerate}
\end{lem}
\begin{proof}
On an open neighborhood $A$ of $M-U$ invariant under $\rho$, $X$ and $JX$ are linearly independent. Hence we have a decomposition $\ker \alpha=E_1 \oplus E_2$ of $\ker \alpha$ into two $J$-invariant symplectic vector subbundles where $E_1=\mathbb{R}X \oplus \mathbb{R}JX$. We define a metric $g^{\prime \prime}$ on $\ker \alpha|_{A}$ by $g^{\prime \prime}|_{E_2}=g|_{E_2}$ on $M$, $g^{\prime \prime}|_{E_{1}}=g(X,X)g$ and setting $E_{1}$ is orthogonal to $E_2$ with respect to $g^{\prime \prime}$. $g^{\prime \prime}$ is compatible with $\alpha$ on $A$. Then we have a metric $g'$ such that  
\begin{enumerate}
\item $g^{\prime}|_{M-U}=g^{\prime \prime}|_{M-U}$, 
\item $g^{\prime}(X,X)$ is constant on the intersection of $M - U$ and each level set of $\Phi$ and
\item $g'|_{U'_{\max}}=g|_{U'_{\max}}$ for an open neighborhood $U'_{\max}$ of $B_{\max}$ in $U$.
\end{enumerate}
Hence $g'$ satisfies the desired conditions. 
\end{proof}

We have a result similar to Proposition 4.11 of \cite{Kar}:
\begin{prop}\label{Karshon2}
Let $M$ be a closed $5$-dimensional manifold. Let $R$ be a nowhere vanishing vector field on $M$. Assume that the flow generated by $R$ preserves a Riemannian metric $g$ on $M$ and that the closures of generic orbits of the flow generated by $R$ are of dimension $2$. Let $\alpha_{0}$ and $\alpha_{1}$ be two contact forms on $M$ with the same Reeb vector field $R$ and the same contact moment maps $\Phi_{\alpha} \colon M \longrightarrow \Lie(G)^{*}$, where $G$ is the closure of the flow generated by $R$ in $\Isom(M,g)$.
\begin{enumerate}
\item We put $\beta_t=(1-t)\alpha_{0}+t\alpha_{1}$ for $t$ in $[0,1]$. Then $d\beta_t$ is a symplectic form on $TM/\mathbb{R}R$ for every $t$ in $[0,1]$.
\item there exists a diffeomorphism $f$ of $M$ which satisfies $f^{*}\alpha_{1}=\alpha_{0}$.
\end{enumerate}
\end{prop}

\begin{proof}
The action of $G$ on $M$ is denoted by $\rho$. We show (i). It is suffices to show that $d\beta_t$ is nondegenerate on $TM/\mathbb{R}R$ for every $t$ in $[0,1]$. Let $x$ be a point on $M$.

We consider the case where the action $\rho$ is locally free near $x$. We fix an element $\overline{X}$ of $\Lie(G)$ which is not parallel to the element corresponding to $R$. We define the function $\Phi$ on $M$ by $\Phi(x)=\Phi_{\alpha}(x)(\overline{X})$. Let $X$ be the infinitesimal action of $\overline{X}$. Let $T_{x}M = \mathbb{R}R_{x} \oplus (T_{x}M/\mathbb{R}R_{x})$ be the orthogonal decomposition. We denote the orthogonal projection $T_{x}M \longrightarrow T_{x}M/\mathbb{R}R_{x}$ by $\pi$. Putting $v^{1}=\pi(X_{x})$, we have 
\begin{equation}\label{PairingWithv}
d\beta_{t}(v^{1},\pi(v))=\Phi_{*}v
\end{equation}
for $v$ in $T_{x}M$, because $\Phi_{*}v=d\Phi(v)=d\beta_{t}(\pi(X_{x}),\pi(v))=d\beta_{t}(v^{1},\pi(v))$. We take a symplectic basis $\{v^{1},v^{2},v^{3},v^{4}\}$ of $(T_{x}M/\mathbb{R}R_{x},(d\beta_{t})_{x})$ so that 
\begin{enumerate}
\item $\Phi_{*}v^{2}$ is nonzero and
\item $\Phi_{*}v^{j}=0$ for $j=3$, $4$.
\end{enumerate}
Let $\{v^{1 *},v^{2 *},v^{3 *},v^{4 *}\}$ be the basis of $(T_{x}M/\mathbb{R}R_{x})^{*}$ dual to $\{v^{1},v^{2},v^{3},v^{4}\}$. We put $(d\beta_{t})_{x}=\sum_{j,k} a_{t}^{jk} v^{j *} \wedge v^{k *}$. By \eqref{PairingWithv} and the condition (ii), we have $a_{t}^{1j}=0$ for $j=3$ and $4$. Hence we have $(d\beta_{t} \wedge d\beta_{t})_{x}=a_{t}^{12} a_{t}^{34} v^{1 *} \wedge v^{2 *} \wedge v^{3 *} \wedge v^{4 *}$. By \eqref{PairingWithv} and the condition (i), $a_{t}^{12}=d\beta_{t}(v^{1},v^{2})=\Phi_{*}v^{2}$ is nonzero. We show $a_{t}^{34}$ is nonzero. Since $d\beta_{t}(v^{3},v^{4})=(1-t)d\alpha_{0}(v^{3},v^{4})+td\alpha_{1}(v^{3},v^{4})$, it suffices to show that the signature of $d\alpha_{0}(v^{3},v^{4})$ and $d\alpha_{1}(v^{3},v^{4})$ are equal. That is, it suffices to show that the orientations of $T_{x}M/\mathbb{R}R_{x}$ determined by $d\alpha_{0} \wedge d\alpha_{0}$ and $d\alpha_{1} \wedge d\alpha_{1}$ are the same. Since the basic cohomology classes of $d\alpha_{0}$ and $d\alpha_{1}$ are the Euler classes of the isometric flow defined in \cite{Sar} which is determined by $R$, we have $[d\alpha_{0} \wedge d\alpha_{0}]_{B}=[d\alpha_{1} \wedge d\alpha_{1}]_{B}$ in $H_{B}^{4}(M/\mathcal{F})$ where $\mathcal{F}$ is the foliation defined by the orbits of the flow generated by $R$ and $H_{B}^{4}(M/\mathcal{F})$ is the basic cohomology of $(M,\mathcal{F})$ of degree $4$. Hence $a_{t}^{34}$ is nonzero and $(d\beta_{t})_{x}$ is symplectic for $t$ in $[0,1]$.

We consider the case where $x$ is a point on a singular $S^1$-orbit of $\rho$. Let $\sigma$ be the $S^1$-action on $M$ of the identity component of the isotropy group of $\rho$ at $x$. Fix a sufficiently small transversal $T$ to the Reeb flow which contains $x$ and is invariant under $\sigma$. Let $Y$ be the vector field generating $\sigma$. Then $(T,d\alpha_{t}|_{T})$ is a symplectic manifold with the hamiltonian $S^1$-action $\sigma$ and a hamiltonian function $\alpha_{t}(Y)$ for $t=0$ and $1$. In fact, we have $d(\alpha_{t}(Y))(Z)=d(\alpha_{t}(Y))(Z)-L_{Y}\alpha_{t}(Z)=-d\alpha_{t}(Y,Z)$ for a vector field $Z$ on $T$. Note that $\alpha_{0}(Y)=\alpha_{1}(Y)$. Hence $(T,d\alpha_{0}|_{T})$ and $(T,d\alpha_{1}|_{T})$ are symplectic manifolds with the same hamiltonian $S^1$-action $\sigma$ and the same hamiltonian function. Then $d\beta_{t}$ is symplectic at $x$ by Lemma 4.13 of \cite{Kar}.

We show (ii). Since $d\beta_t$ is a symplectic form on $T_{x}M/\mathbb{R}R_{x}$ for every $t$ in $[0,1]$ and every $x$ in $M$, $\ker \beta_{t}$ gives the isotopy of $\rho$-invariant contact structures from $\ker \alpha_{0}$ to $\ker \alpha_{1}$. Then we have a $\rho$-equivariant diffeomorphism $f$ on $M$ which satisfies $f_{*}(\ker \alpha_{0})=\ker \alpha_{1}$ by the equivariant version of Gray's stability theorem \cite{Gra}. Since $f$ is $\rho$-equivariant, $f_{*}R=R$. Hence we have $f^{*}\alpha_{1}=\alpha_{0}$.
\end{proof}

By Propositions \ref{Karshon1} and \ref{Karshon2}, we have Theorem \ref{KarshonClassification}. We state Theorem \ref{KarshonClassification} in a finer form:
\begin{cor}\label{Combinatorial}
If there exists an isomorphism $\phi$ between the graphs of isotropy data of $(M_{0},\alpha_{0},\psi_{0})$ and $(M_{1},\alpha_{1},\psi_{1})$, then there exists an isomorphism from $(M_{0},\alpha_{0})$ to $(M_{1},\alpha_{1})$ which induces $\phi$.
\end{cor}

We show Corollary \ref{ReebVectorFieldsDetermineKContactStructures} from Corollary \ref{Combinatorial}:

\begin{proof}
The ``only if'' part is clear. We show the ``if'' part. Let $f$ be a diffeomorphism $f \colon M_{0} \longrightarrow M_{1}$ such that $f_{*}R_{0}=R_{1}$. Fix an isomorphism $\psi \colon G_{1} \longrightarrow T^2$. Since $f_{*}R_{0}=R_{1}$, $f$ induces an isomorphism from the graph of isotropy data of $(M_{0},\alpha_{0},\psi \circ f_{*})$ to the graph of isotropy data of $(M_{1},\alpha_{1},\psi)$ where $f_{*} \colon G_{0} \longrightarrow G_{1}$ is the map induced from $f$. Hence Corollary \ref{ReebVectorFieldsDetermineKContactStructures} follows from Corollary \ref{Combinatorial}.
\end{proof}

\subsubsection{The realization problem of the given graphs of isotropy data}
In the case where the maximal component and the minimal component of $(M,\alpha)$ of dimension $3$, each nontrivial chain can be realized as a nontrivial chain in a closed $5$-dimensional contact toric manifold of rank $2$ by Lemma \ref{EmbeddingChains}. Hence we can construct a $K$-contact manifold which has the given graph of isotropy data $\Gamma$ by the construction of the fiber sum (See Definition \ref{DefinitionOfFiberSum}) if $\Gamma$ satisfies the following conditions:
\begin{enumerate}
\item $\Gamma$ has two fat vertices,
\item the genus of the Seifert invariants attached to two fat vertices coincide and
\item each path in $\Gamma$ connecting two fat vertices can be realized as a nontrivial chain in a closed $5$-dimensional contact toric manifold.
\end{enumerate}

\subsection{Nontrivial chains in contact toric manifolds}\label{ToricContactManifolds}
We see the relation between the nontrivial chains in $5$-dimensional contact toric manifolds of rank $2$ and the corresponding good cones in $\mathbb{R}^{3}$. A good cone in $\mathbb{R}^{3}$ is the image of the moment map of the symplectization of a $5$-dimensional contact toric manifold for the $T^3$-action, which determines the equivariant isomorphism class of $5$-dimensional contact toric manifolds by a result of Lerman \cite{Ler2}. 

Let $(M,\alpha)$ be a closed $5$-dimensional $K$-contact manifold of rank $2$. We denote the closure of the Reeb flow in $\Isom(M,g)$ for a metric $g$ compatible to $\alpha$ by $G$. The action of $G$ on $M$ is denoted by $\rho$. Assume that $M$ has an effective $\alpha$-preserving $T^3$-action $\tau$. Then $\rho$ is a $T^2$-subaction of $\tau$ by the remark after Lemma \ref{RestrictionOfRank}. The Reeb vector field $R$ of $(M,\alpha)$ is the infinitesimal action of an element $\overline{R}$ of $\Lie(T^3)$. We denote the contact moment map for $\tau$ by
\begin{equation}
\begin{array}{cccc}
\tilde{\Phi}_{\alpha} : & M & \longrightarrow & \Lie(T^3)^{*} \\
                & x & \longmapsto     & \overline{\alpha}_x.
\end{array}
\end{equation}
Note that the image of $\tilde{\Phi}_{\alpha}$ is contained in the $2$-dimensional affine subspace $A=\{v \in \Lie(T^3)^{*}| v(\overline{R})=1\}$ of $\Lie(T^3)^{*}$. The symplectization $(M \times \mathbb{R}_{>0},d(r\alpha))$ is a symplectic toric manifold and the image of $\tilde{\Phi}_{\alpha}$ is the intersection of $A$ and the symplectic moment map image of $(M \times \mathbb{R}_{>0},d(r^{2}\alpha))$. Since the image of the symplectic moment map image of $(M \times \mathbb{R}_{>0},d(r\alpha))$ is a convex cone by \cite{GuSt}, $\tilde{\Phi}_{\alpha}$ is a convex polyhedron.

We denote the contact moment map for $\rho$ by $\Phi_{\alpha}$ and the restriction map $\Lie(T^3)^{*} \longrightarrow \Lie(G)^{*}$ by $\pi$. Then we have $\pi \circ \tilde{\Phi}_{\alpha}= \Phi_{\alpha}$. The union of nontrivial chains of $(M,\alpha)$ is contained in the inverse image of the boundary of the image of $\tilde{\Phi}_{\alpha}$. Hence the number of nontrivial chains is at most $2$. If the minimal component $B_{\min}$ and the maximal component $B_{\max}$ of $\Phi_{\alpha}$ are of dimension $3$, $B_{\min}$ and $B_{\max}$ have $K$-contact structures of rank $2$. Hence $B_{\min}$ and $B_{\max}$ are diffeomorphic to lens spaces by the classification theorem of $3$-dimensional contact toric manifolds by Lerman \cite{Ler2}.

We prepare a lemma for later use.
\begin{lem}\label{MomentMapsAreSubmersive}
Let $H$ be a Lie group. Let $(M,\alpha)$ be a $K$-contact manifold with an $\alpha$-preserving $H$-action $\tau$. Assume that $R$ is an infinitesimal action of $\tau$. Let $\tilde{\Phi}_{\alpha}$ be a contact moment map for $\tau$. Then $\tilde{\Phi}_{\alpha}$ is a submersion on the union of the orbits of $\tau$ whose isotropy groups are trivial.
\end{lem}

\begin{proof}
Let $\{\overline{X}_{j}\}_{j=1}^{k}$ be elements of $\Lie(H)$ such that $\{\overline{R}, \overline{X}_{1}, \cdots, \overline{X}_{k}\}$ is a basis of $\Lie(H)$. Let $X_{j}$ be the infinitesimal action of $\overline{X}_{j}$. Identifying $\Lie(H)^{*}$ with $\mathbb{R}^{k+1}$ by the basis $\{\overline{R}^{*}, \overline{X}_{1}^{*}, \cdots, \overline{X}_{k}^{*}\}$ dual to $\{\overline{R}, \overline{X}_{1}, \cdots, \overline{X}_{k}\}$, we can write $\tilde{\Phi}_{\alpha}$ as
\begin{equation}\label{CoordinatePresetationOfMomentMap}
\begin{array}{rl}
\tilde{\Phi}_{\alpha}(x) & =\alpha(R)(x) \overline{R}^{*} + \sum_{j=1}^{k} \alpha(X_{j})(x) \overline{X}^{*}_{j} \\
 & =(1, \alpha(X_{1})(x), \alpha(X_{2})(x), \cdots, \alpha(X_{k})(x)).
\end{array}
\end{equation}
Take a point $x$ on $M$ such that the isotropy group of $\tau$ at $x$ is trivial. Then $\{R_{x}, X_{1 x}, \cdots, X_{k x}\}$ are linearly independent. Since $d\alpha$ is nondegenerate on $\ker \alpha$, there exist a vector $Y_{j}$ in $T_{x}M$ such that $d\alpha(X_{j x},Y_{j})$ is nonzero. Let $\pr_{j}$ be the $j$-th projection defined on $\Lie(H)^{*}$ with respect to the basis $\{\overline{R}^{*}, \overline{X}_{1}^{*}, \cdots, \overline{X}_{k}^{*}\}$. By \eqref{CoordinatePresetationOfMomentMap}, we have $d(\pr_{j} \circ \tilde{\Phi}_{\alpha})_{x}(Y_{j}) = d(\alpha(X_{j})(Y_{j})(x)=-d\alpha(X_{j x},Y_{j})$ is nonzero for each $j$. Hence $\tilde{\Phi}_{\alpha}$ is a submersion at $x$.
\end{proof}

\begin{figure}
\begin{equation}
\begin{picture}(178,96)(0,30)
\path(65,20)(65,140)
\path(67.000,132.000)(65.000,140.000)(63.000,132.000)
\path(-15,40)(145,40)
\path(137.000,38.000)(145.000,40.000)(137.000,42.000)
\path(0,10)(130,70)
\path(124,64)(130,70)(122,69)

\path(65,40)(0,120)
\path(65,40)(80,130)
\path(65,40)(130,100)
\thicklines
\path(32.5,80)(97.5,70)
\path(97.5,70)(73.5,91)
\path(32.5,80)(73.5,91)
\end{picture}
\end{equation}
\caption{A good cone.}
\end{figure}
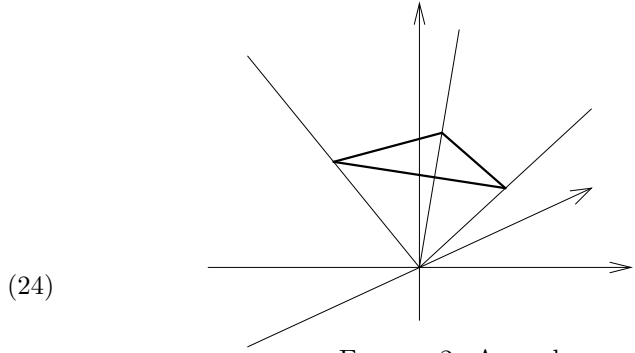

\section{A sufficient condition to be toric}\label{ToricSection}

We show Theorem \ref{ToricCondition} which gives a sufficient condition for a closed $5$-dimensional $K$-contact manifold of rank $2$ to be toric. Theorem \ref{ToricCondition} follows from Propositions \ref{TwoChains} and \ref{SufficientConditionToBeToric}.

Let $(M,\alpha)$ be a closed $5$-dimensional $K$-contact manifold of rank $2$. We denote the Reeb vector field of $\alpha$ by $R$. Let $G$ be the closure of the Reeb flow in $\Isom(M,g)$ for a metric $g$ compatible to $\alpha$. The action of $G$ on $M$ is denoted by $\rho$. We put $\Phi=\alpha(X)$ where $X$ is an infinitesimal action of $\rho$ which is not parallel to $R$. The maximal component and the minimal component of $\Phi$ are denoted by $B_{\max}$ and $B_{\min}$ respectively.

For a topological group $H$, an $H$-action $\tau$ on a set $A$ and a $\tau$-invariant subset $B$ of $A$, we denote the cardinality of the kernel of $H \longrightarrow \Aut(B)$ by $I(\tau,B)$.

\subsection{Estimate on the number of nontrivial chains}

We show
\begin{prop}\label{TwoChains}
If the conditions (i) and (ii) of Theorem \ref{ToricCondition} are satisfied, then the number of nontrivial chains in $(M,\alpha)$ is at most $2$.
\end{prop}

The condition (i) of Theorem \ref{ToricCondition} is satisfied if and only if $B_{\max}$ and $B_{\min}$ are isolated closed orbits of the Reeb flow by Lemma \ref{MorseTheory} (iv). The condition (ii) of Theorem \ref{ToricCondition} can be translated into a condition on the image of the contact moment map for $\rho$. We denote the kernel of $\exp \colon \Lie(G) \longrightarrow G$ by $\Lie(G)_{\mathbb{Z}}$. We put $\Omega =\{ v \in \Lie(G) | v \cdot x > 0$ for every $x$ in $\Phi_{\alpha}(M) \}$. Then $\Omega$ is an open convex cone in $\Lie(G)$ whose boundary has rational slopes. We denote the two primitive vectors in $\Lie(G)_{\mathbb{Z}}$ tangent to the boundary of $\Omega$ by $v_{\min}$ and $v_{\max}$.
\begin{lem}\label{TranslationOfCondition2}
The condition (ii) is satisfied if and only if there exists a vector $v$ in $\Omega \cap \Lie(G)_{\mathbb{Z}}$ such that both of $\{v,v_{\min}\}$ and $\{v,v_{\max}\}$ are $\mathbb{Z}$-bases of $\Lie(G)_{\mathbb{Z}}$.
\end{lem}
\begin{proof}
For an $S^1$-subaction $\sigma$ of $\rho$, let $X_{\sigma}$ the vector field generating $\sigma$. $v_{\sigma}$ denotes the vector in $\Lie(G)_{\mathbb{Z}}$ whose infinitesimal action is $X_{\sigma}$. For a point $x$ on $M$, we have $\Phi_{\alpha}(x)(v_{\sigma})>0$ if and only if $\alpha(X_{\sigma}) > 0$ by definition. Hence $v_{\sigma}$ is contained in $\Omega$ if and only if the orbits of $\sigma$ are positively transverse to $\ker \alpha$ on $M$. 

Note that for two vectors $v_{1}$ and $v_{2}$ in $\Lie(G)_{\mathbb{Z}}$, the product of $\sigma_{v_1}$ and $\sigma_{v_2}$ is isomorphic to $\rho$ if and only if $\{v_1,v_2\}$ is a $\mathbb{Z}$-basis of $\Lie(G)_{\mathbb{Z}}$. Note also that the infinitesimal action of $v_{\min}$ and $v_{\max}$ generates the action of $G_{\min}$ and $G_{\max}$.

We show the ``only if'' part. Assume that there exists an $S^1$-subgroup $G'$ of $G$ such that both of $G' \times G_{\max}$ and $G' \times G_{\min}$ are isomorphic to $G$. Let $v_{0}$ be the vector in $\Lie(G)_{\mathbb{Z}}$ whose infinitesimal action generates the action of $G'$. Since the orbits of the action of $G'$ is transverse to $\ker \alpha$, either of $v_{0}$ or $-v_{0}$ is contained in $\Omega$ by the argument in the first paragraph. On the other hand, by the remarks in the second paragraph, $\{v_{0},v_{\min}\}$, $\{v_{0},v_{\max}\}$, $\{-v_{0},v_{\min}\}$ and $\{-v_{0},v_{\max}\}$ are $\mathbb{Z}$-bases of $\Lie(G)_{\mathbb{Z}}$. Hence the ``only if'' part is proved.

We show the ``if'' part. Assume that there exists a vector $v$ in $\Omega \cap \Lie(G)_{\mathbb{Z}}$ such that $\{v,v_{\min}\}$ and $\{v,v_{\max}\}$ are $\mathbb{Z}$-bases of $\Lie(G)_{\mathbb{Z}}$. We denote the $S^1$-action on $M$ generated by the infinitesimal action of $v$ by $\sigma$. Then the orbits of the action of $G'$ is transverse to $\ker \alpha$ by the argument of the first paragraph, since $v$ is contained in $\Omega$. Both of $\{v,v_{\max}\}$ and $\{v,v_{\min}\}$ are $\mathbb{Z}$-bases of $\Lie(G)_{\mathbb{Z}}$ by the remarks in the second paragraph. Hence the ``if'' part is proved.
\end{proof}

To show Proposition \ref{TwoChains}, we use the Euler number of locally free $S^1$-actions on $3$-dimensional orbifolds.

\begin{defn}\label{Definition : EulerNumber}(Euler number of locally free $S^1$-actions on $3$-dimensional orbifolds) Let $\tau$ be a transversely oriented locally free $S^1$-action on a compact $3$-dimensional orbifold $W$. Then we define
\begin{equation}\label{DefinitionOfEulerNumber}
e(\tau,W) = \frac{1}{|G'|} \cdot e(\tau/G',W/(\tau|_{G'}))
\end{equation}
\noindent where $G'$ is a discrete subgroup of $S^1$ such that the effective $S^1$-action $\tau/G'$ on $W/(\tau|_{G'})$ induced from $\tau$ is free and $e(\tau/G',W/(\tau|_{G'}))$ is the Euler number of the oriented circle bundle over the topological surface $W/G'$ defined by the free $S^{1}/G'$-action $\tau/G'$ on $W/G'$.
\end{defn}

Note that the base surface of the $S^1$-action $\tau/G'$ on $W/(\tau|_{G'})$ is the oriented topological surface $W/\tau$. 

For the definition of Euler number of locally free $S^1$-actions on $3$-dimensional manifolds, we refer \cite{Aud2}. The Euler number $e(\tau,W)$ is shown to be independent of the choice of $G'$ and well-defined for $W$ and $\tau$ as in the case of locally free $S^1$-actions on $3$-dimensional manifolds.

\begin{lem}\label{EulerNumberOfSeifertFibrationOnS30}
Let $\sigma$ be the $S^1$-action on $S^3$ defined by
\begin{equation}
t \cdot(z_1,z_2)=(t^{m_1} z_1,t^{m_2} z_2)
\end{equation}
for nonzero coprime integers $m_1$ and $m_2$. We orient $S^3$ as the boundary of the unit ball in $\mathbb{C}^{2}$ with the orientation determined by the standard complex structure. The transverse orientation of $\sigma$ is defined from the orientation on $S^3$ and the orientation of the orbits of $\sigma$. Then we have
\begin{equation}
e(\sigma,S^3)=-\frac{1}{m_1 m_2}.
\end{equation} 
\end{lem}

\begin{proof}
Since $m_1$ and $m_2$ are coprime, the group generated by isotropy groups of $\tau$ at exceptional orbits is $\mathbb{Z}/m_1 m_2\mathbb{Z}$. We show the effective $S^1$-action on $S^3/(\sigma|_{\mathbb{Z}/m_1 m_2\mathbb{Z}})$ induced from $\sigma$ is isomorphic to the $S^1$-action $\sigma_{0}$ on $S^3$ defined by
\begin{equation}
t \cdot (z_1,z_2)=(t^{\frac{m_1 m_2}{|m_1 m_2|}} z_1,t^{\frac{m_1 m_2}{|m_1 m_2|}} z_2).
\end{equation}

We define a map $f \colon S^3 \longrightarrow S^3$ by $f(z_1,z_2)=(z_{1}^{m_2},z_{2}^{m_1})$. Then $f$ induces a map $\overline{f} \colon S^3/(\sigma|_{\mathbb{Z}/m_1 m_2\mathbb{Z}}) \longrightarrow S^3$. $\overline{f}$ is clearly surjective. We show $\overline{f}$ is injective. It suffices to show the fiber of $f$ of each point on $S^3$ coincides with an orbit of $\sigma|_{\mathbb{Z}/m_1 m_2\mathbb{Z}}$. Take a point $(w_{1},w_{2})$ on $S^3$. We consider the case where both of $w_1$ and $w_2$ are nonzero. Since $m_1$ and $m_2$ are coprime, the cardinality of $f^{-1}(w_{1},w_{2})$ is $|m_1m_2|$. Since the cardinality of the orbit of $\sigma$ of a point in $f^{-1}(w_{1},w_{2})$ is $|m_{1}m_{2}|$, $f^{-1}(w_{1},w_{2})$ is an orbit of $\sigma|_{\mathbb{Z}/m_1 m_2\mathbb{Z}}$. We consider the case where $w_{1}=0$. Then $w_{2}$ is nonzero, and the cardinality of $f^{-1}(w_{1},w_{2})$ is $|m_1|$. Since the cardinality of the orbit of $\sigma|_{\mathbb{Z}/m_1 m_2\mathbb{Z}}$ of a point in $f^{-1}(w_{1},w_{2})$ is $|m_{1}|$, $f^{-1}(w_{1},w_{2})$ is an orbit of $\sigma$. The proof in the case where $w_{2}=0$ is symmetric to the previous case. Hence $\overline{f}$ is a homeomorphism. 

$\sigma$ induces an $S^1$-action on $S^3/(\sigma|_{\mathbb{Z}/m_1 m_2\mathbb{Z}})=S^3$ defined by 
\begin{equation}
t \cdot (w_1,w_2)=(t^{m_1 m_2} w_1,t^{m_1 m_2} w_2).
\end{equation}
Hence the effective $S^1$-action on $S^3/(\sigma|_{\mathbb{Z}/m_1 m_2\mathbb{Z}})$ induced from $\sigma$ is $\sigma_{0}$. Since $e(\sigma_{0},S^3)=-\frac{m_1 m_2}{|m_1 m_2|}$ by the definition of the transverse orientation, we have
\begin{equation}
e(\sigma,S^3)=-\frac{1}{|m_1 m_2|}e(\sigma_{0},S^3)=-\frac{1}{m_1 m_2}
\end{equation}
by \eqref{DefinitionOfEulerNumber}.
\end{proof}

\begin{lem}\label{LocalComputation}
Assume that we have the effective $T^2$-action $\sigma$ on $S^1 \times S^3$ defined by
\begin{equation}\label{StandardT2Action}
(t_1,t_2)\cdot(\zeta,z_1,z_2)=(t_{1}^{a_0} t_{2}^{b_0} \zeta,t_{1}^{a_1} t_{2}^{b_1} z_1,t_{1}^{a_2} t_{2}^{b_2} z_2)
\end{equation}
\noindent in the standard coordinate where $a_{0}$ is positive. We orient $S^3$ as the boundary of the unit ball in $\mathbb{C}^{2}$ with the orientation determined by the standard complex structure. Let $\sigma_1$ and $\sigma_2$ be the $S^1$-actions obtained from $\tau$ by restricting to $S^1 \times \{1\}$ and $\{1\} \times S^1$ respectively. Then we have
\begin{equation}\label{Equation : Euler}
e(\sigma_2,(S^1 \times S^3)/\sigma_1)=-\frac{a_{0}}{(a_{0} b_{1} - a_{1} b_{0} )(a_{0} b_{2} - a_{2} b_{0} ) }.
\end{equation} 
\end{lem}

\begin{proof}
Let $\tilde{\sigma}_{1}$ and $\tilde{\sigma}_{2}$ be $S^1$-actions on $S^1 \times S^3$ defined by
\begin{equation}
t_1 \cdot(\xi,w_1,w_2)=(t_{1} \xi, w_1, w_2)
\end{equation}
and
\begin{equation}
t_2 \cdot(\xi,w_1,w_2)=(t_{2}^{b_{0}} \xi, t_{2}^{a_{0}b_{1}-a_{1}b_{0}} w_1, t_{2}^{a_{0}b_{2}-a_{2}b_{0}} w_2)
\end{equation}
respectively. Let $\tau$ be the $\mathbb{Z}/a_{0}\mathbb{Z}$-action on $S^1 \times S^3$ defined by
\begin{equation}
s \cdot(\xi,w_1,w_2)=(s \xi, s^{-a_1} w_1, s^{-a_2} w_2)
\end{equation}
for $s$ in $\mathbb{Z}/a_{0}\mathbb{Z}$ where we identify $\mathbb{Z}/a_{0}\mathbb{Z}$ with a subgroup of the complex numbers with absolute values $1$. We identify $(S^1 \times S^3)/\tau$ with $S^1 \times S^3$ by the diffeomorphism $f$ defined by $f([(\xi,w_1,w_2)])=(\xi^{a_{0}},\xi^{a_{1}}w_{1},\xi^{a_{2}}w_{2})$. Then $\tilde{\sigma}_{1}$ induces $\sigma_{1}$ on $(S^1 \times S^3)/\tau$. $\tilde{\sigma}_{2}$ induces the $S^1$-action $\sigma_{2}^{a_{0}}$ defined by
\begin{equation}
t_2 \cdot (\zeta,z_1,z_2)=(t_{2}^{a_0 b_0} \zeta, t_{2}^{a_0 b_1} z_1, t_{2}^{a_0 b_2} z_2).
\end{equation}
Hence the effective $S^1$-action on $(S^1 \times S^3)/\tau$ induced from $\tilde{\sigma}_{2}$ is $\sigma_{2}$.

The cardinality of the isotropy group of the action on $(S^1 \times S^3)/\tilde{\sigma}_{1}$ induced from $\tilde{\sigma}_{2}$ is $\GCD(|a_{0}b_{1}-a_{1}b_{0}|,|a_{0}b_{2}-a_{2}b_{0}|)$. The cardinality of the isotropy group of the action on $(S^1 \times S^3)/\sigma_{1}$ induced from $\sigma_{2}^{a_{0}}$ is $a_{0}$, since the cardinality of the isotropy group of the action on $(S^1 \times S^3)/\sigma_{1}$ induced from $\sigma_{2}$ is $1$ by the effectiveness of the $T^2$-action $\sigma$. Hence $(S^1 \times S^3)/\sigma_{1}$ with $S^1$-action $\sigma_{2}^{a_{0}}$ is equivariantly diffeomorphic to the quotient of $(S^1 \times S^3)/\tilde{\sigma}_{1}$ with $S^1$-action $\tilde{\sigma}_{2}$ by the $\mathbb{Z}/l\mathbb{Z}$-subaction of $\tilde{\sigma}_{2}$ where $l=\frac{a_{0}}{\GCD(|a_{0}b_{1}-a_{1}b_{0}|, |a_{0}b_{2}-a_{2}b_{0}|)}$. By \eqref{DefinitionOfEulerNumber}, we have
\begin{equation}\label{tildesigmaandsigma}
e(\tilde{\sigma}_{2},(S^1 \times S^3)/\tilde{\sigma}_{1}) = \frac{1}{\frac{a_{0}}{\GCD(|a_{0}b_{1}-a_{1}b_{0}|, |a_{0}b_{2}-a_{2}b_{0}|)}} e(\sigma_{2},(S^1 \times S^3)/\sigma_{1}).
\end{equation}

Since $(S^1 \times S^3)/\tilde{\sigma}_{1}=S^3$ and the Euler number of the $S^1$-action on $S^3$ defined by 
\begin{equation}
t \cdot (z_1,z_2)=(s^{m_1} z_1, s^{m_2} z_2)
\end{equation}
is $-\frac{\GCD(|m_{1}|, |m_{2}|)}{m_1 m_2}$ by Lemma \ref{EulerNumberOfSeifertFibrationOnS30}, we have
\begin{equation}\label{Covering}
e(\tilde{\sigma}_{2},(S^1 \times S^3)/\tilde{\sigma}_{1})=-\frac{\GCD(|a_{0} b_{1}- a_{1} b_{0} |, | a_{0} b_{2} - a_{2} b_{0} |)}{( a_{0} b_{1} - a_{1} b_{0} )( a_{0} b_{2} - a_{2} b_{0})}.
\end{equation}
By \eqref{tildesigmaandsigma} and \eqref{Covering}, we have \eqref{Equation : Euler}.
\end{proof}

We fix $S^1$-subactions $\rho_1$ and $\rho_2$ of $\rho$ so that the product of $\rho_{1}$ and $\rho_{2}$ is $\rho$ and the orbits of $\rho_{1}$ is positively transverse to $\ker \alpha$. Let $H_{\max}$ and $H_{\min}$ be level sets of $\Phi=\alpha(X)$ sufficiently close to $B_{\max}$ and $B_{\min}$, respectively. 

$\Phi$ is constant on the each orbit of $\rho$ by Lemma \ref{MorseTheory} (i). We fix transverse orientations of locally free orbits of $\rho$ in each level set of $\Phi$ as follows: Take a $4$-form $\omega$ on $M$ which satisfies $d\Phi \wedge \omega = \alpha \wedge d\alpha \wedge d\alpha$ on $M - B_{\min} - B_{\max} - \Crit \Phi$. We take a $2$-form $\omega'$ on $M$ which satisfies $\omega' \wedge \alpha_{1} \wedge \alpha_{2} = \omega$ on $M - \Crit \Phi$ where $\alpha_{j}$ is a $1$-form on $M$ which satisfies $\alpha_{j}(Y_{j})=1$ for the infinitesimal action $Y_{j}$ of $\rho_{j}$. Then $\omega'$ determines an transverse orientation of $\rho$ on each level set of $\Phi$. Note that the transverse orientation changes if we use $-\Phi$ instead of $\Phi$. We regard $H_{\min}$ as a normal $S^3$-bundle of $\Sigma_{\min}$. Let $F$ be a fiber in $H_{\min}$. $F$ is transverse to the orbits of $\rho$. We determine the signature of $\Phi$ so that the orientation of $F$ determined by the above process is equal to the orientation determined as the boundary of the symplectic normal bundle of $B_{\min}$.

\begin{lem}\label{LevelSetNearB}
Assume that $B_{\max}$ and $B_{\min}$ are closed orbits of the Reeb flow. Let $L_{\max}^{1}$ and $L_{\max}^{2}$ be two gradient manifolds whose closures are smooth near $B_{\max}$. Let $L_{\min}^{1}$ and $L_{\min}^{2}$ be two gradient manifolds whose closures are smooth near $B_{\min}$. We put $a_{\min}=I(\rho_1,B_{\min})$, $a_{\max}=I(\rho_1,B_{\max})$, $m_{\max}=I(\rho,L_{\max}^{1})$, $n_{\max}=I(\rho,L_{\max}^{2})$, $m_{\min}=I(\rho,L_{\min}^{1})$ and $n_{\min}=I(\rho,L_{\min}^{2})$. Then we have
\begin{equation}\label{LevelSetNearBmin}
e(\rho_2,H_{\min}/\rho_1)=-\genfrac{}{}{}{0}{a_{\min}}{m_{\min}n_{\min}}.
\end{equation}
and
\begin{equation}\label{LevelSetNearBmax}
e(\rho_2,H_{\max}/\rho_1)=\genfrac{}{}{}{0}{a_{\max}}{m_{\max}n_{\max}} 
\end{equation}
\end{lem}

\begin{proof}
By Lemma \ref{GeneralTorusActions}, we have a coordinate on an open neighborhood $U$ of $B_{\min}$ such that $\rho$ is written as
\begin{equation}
(s_1,s_2) \cdot (\zeta,z_1,z_2)=(s_{1}^{n_0} \zeta,s_{1}^{n_1} s_{2}^{m_1} z_1,s_{1}^{n_2} s_{2}^{m_2} z_2).
\end{equation}
Since $\Phi$ is written as \eqref{NormalFormOfPhi} on a finite covering of $U$, we can put $H_{\min}=\{(\zeta,z_1,z_2) \in S^1 \times D^4 | |z_1|^{2}+|z_{2}|^{2}=\epsilon\}$ for a small positive number $\epsilon$ and assume that $m_1$ and $m_2$ are positive. We can write $\rho_{1}$ and $\rho_{2}$ respectively as
\begin{equation}
t_{1} \cdot (\zeta,z_1,z_2)=(t_{1}^{a_{0}} \zeta, t_{1}^{a_{1}} z_{1}, t_{1}^{a_{2}} z_{2})
\end{equation}
and
\begin{equation}
t_{2} \cdot (\zeta,z_1,z_2)=(t_{2}^{b_{0}} \zeta, t_{2}^{b_{1}} z_{1}, t_{2}^{b_{2}} z_{2})
\end{equation}
so that $a_{0}$ is positive. Then we have
\begin{equation}\label{amin}
a_{0}=a_{\min}
\end{equation}
by the definition of $a_{\min}$. There exists an element $A$ of $\GL(2;\mathbb{Z})$ such that 
\begin{equation}\label{Transformation}
A \begin{pmatrix} a_{0} & a_{1} & a_{2} \\ b_{0} & b_{1} & b_{2} \end{pmatrix} = \begin{pmatrix} n_{0} & n_{1} & n_{2} \\ 0 & m_{1} & m_{2} \end{pmatrix}.
\end{equation}
Hence the signatures of $a_{0} b_{1} - a_{1} b_{0}$ and $a_{0} b_{2} - a_{2} b_{0}$ are equal. Since $L_{\min}^{1}$ and $L_{\min}^{2}$ are defined by the equations $z_{2}=0$ and $z_{1}=0$ respectively, we have $I(\rho,L_{\min}^{1})=|a_{0} b_{1} - a_{1} b_{0}|$ and $I(\rho,L_{\min}^{2})=|a_{0} b_{2} - a_{2} b_{0}|$. Then we have
\begin{equation}\label{Denominator}
m_{\min}n_{\min} = (a_{0} b_{1} - a_{1} b_{0})(a_{0} b_{2} - a_{2} b_{0}).
\end{equation}

By Lemma \ref{LevelSetNearB}, we have
\begin{equation}\label{CalculationOfEulerNumberOfH}
e(\rho_2,H_{\min}/\rho_1)=-\frac{a_{0}}{(a_{0} b_{1} - a_{1} b_{0} )(a_{0} b_{2} - a_{2} b_{0} ) }.
\end{equation}
By \eqref{amin}, \eqref{Denominator} and \eqref{CalculationOfEulerNumberOfH}, we have \eqref{LevelSetNearBmin}. 

The proof of \eqref{LevelSetNearBmax} is the same except the point where the signature of the right hand side of \eqref{CalculationOfEulerNumberOfH} changes because the transverse orientation is different from the previous case.
\end{proof}

\begin{lem}\label{LocalDifference}
Take two regular values $b$ and $b'$ of $\Phi$ so that $[b,b']$ contains a unique critical value $c$ of $\Phi$. We assume that the critical set in $\Phi^{-1}(c)$ is a closed orbit $\Sigma$ of the Reeb flow. Let $L$ and $L'$ be the gradient manifolds whose limit sets contain $\Sigma$. We put $k=I(\rho,L)$, $k'=I(\rho,L')$ and $a=I(\rho_1,\Sigma)$. Then we have
\begin{equation}\label{Difference0}
e(\rho_2,\Phi^{-1}(b')/\rho_1)-e(\rho_2,\Phi^{-1}(b)/\rho_1)=\frac{a}{kk'}.
\end{equation}
\end{lem}

\begin{proof}
By Lemma \ref{GeneralTorusActions}, we have a coordinate on an open neighborhood $U$ of $\Sigma$ such that $\rho$ is written as
\begin{equation}
(s_1,s_2) \cdot (\zeta,z_1,z_2)=(s_{1}^{n_0} \zeta,s_{1}^{n_1} s_{2}^{m_1} z_1,s_{1}^{n_2} s_{2}^{m_2} z_2).
\end{equation}
Since $\Phi$ is written as in \eqref{NormalFormOfPhi} on a finite covering of $U$ and $\Sigma$ is contained neither in $B_{\min}$ nor in $B_{\max}$, we can assume that $m_1$ is positive and $m_2$ is negative. We put $W=\{(\zeta,z_1,z_2) \in S^1 \times D^4 | |z_1|^{2}+|z_{2}|^{2}=\epsilon\}$ for a small positive number $\epsilon$. We can write $\rho_{1}$ and $\rho_{2}$ respectively as
\begin{equation}
t_{1} \cdot (\zeta,z_1,z_2)=(t_{1}^{a_{0}} \zeta, t_{1}^{a_{1}} z_{1}, t_{1}^{a_{2}} z_{2})
\end{equation}
and
\begin{equation}
t_{2} \cdot (\zeta,z_1,z_2)=(t_{2}^{b_{0}} \zeta, t_{2}^{b_{1}} z_{1}, t_{2}^{b_{2}} z_{2})
\end{equation}
so that $a_{0}$ is positive. Then we have
\begin{equation}\label{a}
a_{0}=a
\end{equation}
by the definition of $a$. Then the signatures of $a_{0} b_{1} - a_{1} b_{0} $ and $a_{0} b_{2} - a_{2} b_{0} $ are different by the equation \eqref{Transformation} and the fact that $m_1$ is positive and $m_2$ is negative. Hence we have 
\begin{equation}\label{Denominator2}
-k k'=(a_{0} b_{1} - a_{1} b_{0} ) (a_{0} b_{2} - a_{2} b_{0} ).
\end{equation}

By the assumption, the $S^1$-action $\rho_2$ on $\Phi^{-1}(b')/\rho_1$ is obtained from the $S^1$-action $\rho_2$ on $\Phi^{-1}(b)/\rho_1$ by the Dehn surgery which replaces a solid torus $S'$ with an exceptional orbit $(L' \cap \Phi^{-1}(b'))/\rho_1$ of $\rho_2$ to a solid torus $S$ with an exceptional orbit $(L \cap \Phi^{-1}(b))/\rho_1$ of $\rho_2$. Note that we can take $S$ and $S'$ so that $W$ is $S^1$-equivariantly diffeomorphic to $S \cup S'$. By the additivity of the Euler numbers with respect to the connected sum on the base spaces of the Seifert fibrations, the change of the Euler numbers under this operation in which we obtain $\Phi^{-1}(b)/\rho_1$ from $\Phi^{-1}(b')/\rho_1$ is equal to $e(\rho_{2},W/\rho_1)$. Hence we have \eqref{Difference0} by \eqref{a}, \eqref{Denominator2} and Lemma \ref{LocalComputation}.
\end{proof}

\begin{lem}\label{TotalDifference}
We denote the nontrivial chains in $(M,\alpha)$ by $C^1, C^2,\cdots,C^k$. Let $L_{1}^{j},L_{2}^{j},\cdots,L_{l(j)}^{j}$ be the gradient manifolds which form $C^j$ in the order of the subscripts. Let $\Sigma_{i i+1}^{j}$ be the closed orbit of the Reeb flow contained in both of the closures of $L_{i}^{j}$ and $L_{i+1}^{j}$ for $i=1,2,\cdots,l(j)-1$. We put $k_{i}^{j}=I(\rho,L_{i}^{j})$ for $i=1$, $2$, $\cdots$, $l(j)$, and $a_{i i+1}^{j}=I(\rho_1,\Sigma_{i i+1}^{j})$ for $i=1,2,\cdots,l(j)-1$. Then we have
\begin{equation}\label{SumOfEulerNumbers1}
e(\rho_2,H_{\max}/\rho_1) - e(\rho_2,H_{\min}/\rho_1)=\sum_{i,j} \frac{a_{i i+1}^{j}}{k_{i}^{j}k_{i+1}^{j}}.
\end{equation}
\end{lem}

\begin{proof}
We put $b_{\max}=\Phi(H_{\max})$ and $b_{\min}=\Phi(H_{\min})$. Then $b_{\min}$ and $b_{\max}$ are regular values of $\Phi$ sufficiently close to the minimum value and the maximum value of $\Phi$, respectively. When $b$ moves from $b_{\min}$ to $b_{\max}$, the change of $e(\rho_2,\Phi^{-1}(b)/\rho_1)$ occurs only when $b$ goes through a critical value of $\Phi$. By Lemma \ref{LocalDifference}, we can compute the contribution of each closed orbit of the Reeb flow. Since we obtain $e(\rho_2,H_{\max}/\rho_1)-e(\rho_2,H_{\min}/\rho_1)$ by summing up the contributions of closed orbits of the Reeb flow, we have the formula \eqref{SumOfEulerNumbers1}. 
\end{proof}

\begin{lem}\label{EulerNumberFormula}
Let $L_{1},L_{2},\cdots,L_{l}$ be the gradient manifolds which form a chain $C$ in the order of the subscripts. Let $\Sigma_{i i+1}$ be the closed orbit of the Reeb flow between $L_{i}$ and $L_{i+1}$ for $i=1,2,\cdots,l-1$. We put $k_i=I(\rho,L_i)$ for $i = 1,2,\cdots,l$, and $a_{i i+1}=I(\rho_1,\Sigma_{i i+1})$ for $i=1,2,\cdots,l-1$. Then we have
\begin{equation}\label{EulerNumberEquation}
\sum_{i=1}^{l-1} \frac{a_{i i+1}}{k_{i}k_{i+1}} = \frac{d}{\LCM(k_{1},k_{l})}
\end{equation}
for some positive integer $d$.
\end{lem}

\begin{proof}
We choose regular values $v_{1},v_{2},\cdots,v_{l}$ of $\Phi$ so that $v_{i}$ is contained in the image of $L_{i}$. By Corollary \ref{LocalComputation}, the contribution of $\Sigma_{i i+1}$ to the change of $e(\rho_{2},\Phi^{-1}(x)/\rho_{1})$ is $\frac{a_{i i+1}}{k_i k_{i+1}}$. Hence the total contribution  of closed orbits $\{\Sigma_{i i+1}\}_{i=1}^{l-1}$ in $C$ to the change of $e(\rho_{2},\Phi^{-1}(x)/\rho_{1})$ is equal to $\sum_{i=1}^{l-1} \frac{a_{i i+1}}{k_i k_{i+1}}$.

Choose a $\rho$-invariant open neighborhood $U$ of the chain $C$ which does not intersect other nontrivial chains. $(\Phi^{-1}(v_l) \cap U)/\rho_1$ is obtained from $(\Phi^{-1}(v_1) \cap U)/\rho_1$ by Dehn surgery which replaces a solid torus $S_{1}$ with an exceptional orbit $(\Phi^{-1}(v_1) \cap C)/\rho_1$ of $\rho_2$ with multiplicity $k_1$ by a solid torus $S_{l}$ with an exceptional orbit $(\Phi^{-1}(v_l) \cap C)/\rho_1$ of $\rho_2$ with multiplicity $k_l$. By the additivity of Euler numbers with respect to the connected sum on the base spaces of Seifert fibrations, the change of the Euler numbers under this operation by which we obtain $(\Phi^{-1}(v_{l}) \cap U)/\rho_1$ from $(\Phi^{-1}(v_{1}) \cap U)/\rho_1$ is equal to $e(\sigma,S_{1} \cup S_{l})$ where $\sigma$ is an $S^1$-action on an $3$-dimensional orbifold $S_{1} \cup S_{l}$ which is the union of two solid tori and has two exceptional orbits of multiplicity $k_{1}$ and $k_{l}$. Let $H$ be the subgroup of $S^1$ with $\LCM(k_{1},k_{l})$ elements. Since $H$ contains the isotropy group of the $S^1$-action on $S_{1} \cup S_{l}$, the effective $S^1$-action induced from $\sigma$ on $(S_{1} \cup S_{l})/(\sigma|_H)$ is free. Hence the denominator of $e(\sigma,S_{1} \cup S_{l})$ divides $\LCM(k_{1},k_{l})$ by \eqref{DefinitionOfEulerNumber}. Then the denominator of the total contribution of closed orbits $\{\Sigma_{i i+1}\}_{i=1}^{l-1}$ in $C$ to the change of $e(\rho_{2},\Phi^{-1}(x)/\rho_{1})$ divides $\LCM(k_{1},k_{l})$.

By the conclusions of the two previous paragraphs, we have
\begin{equation}\label{Difference}
\sum_{i=1}^{l-1} \frac{a_{i i+1}}{k_i k_{i+1}} = \frac{d}{\LCM(k_{1},k_{l})}.
\end{equation}
for some integer $d$. The positivity of $d$ follows from the positivity of the right hand side of \eqref{Difference}.
\end{proof}

\begin{lem}\label{RelationOfEulerNumbers}
We denote the nontrivial chains in $(M,\alpha)$ by $C^{1}, C^{2},\cdots,C^{k}$. Let $L_{1}^{j},L_{2}^{j},\cdots,L_{l(j)}^{j}$ be the gradient manifolds which form $C^{j}$ in the order of the subscripts. We put $k_{i}^{j}=I(\rho,L_{i}^{j})$ for $i=1,2,\cdots,l$. Then we have
\begin{equation}\label{SumOfEulerNumbers2}
e(\rho_2,H_{\max}/\rho_1) - e(\rho_2,H_{\min}/\rho_1)=\sum_{j} \frac{d^{j}}{\LCM(k_{1}^{j},k_{l(j)}^{j})}
\end{equation}
for some positive integers $d^j$.
\end{lem}

\begin{proof}
Let $\Sigma_{i i+1}^{j}$ be the closed orbit of the Reeb flow between $L_{i}^{j}$ and $L_{i+1}^{j}$. We put $a_{i i+1}^{j}=I(\rho_2,\Sigma_{i i+1}^{j})$ for $i=1,2,\cdots,l-1$. By Lemma \ref{TotalDifference}, we have 
\begin{equation}
e(\rho_2,H_{\max}/\rho_1) - e(\rho_2,H_{\min}/\rho_1)=\sum_{i,j} \frac{a_{i i+1}^{j}}{k_{i}^{j} k_{i+1}^{j}}.
\end{equation}
Applying Lemma \ref{EulerNumberFormula} to each chain $C^{j}$ and summing for $j=1,2,\cdots,k$, we obtain
\begin{equation}
\sum_{i,j} \frac{a_{i i+1}^{j}}{k_{i}^{j}k_{i+1}^{j}}=\sum_{j} \frac{d^{j}}{\LCM(k_{1}^{j}, k_{l(j)}^{j})}.
\end{equation}
Hence the equation \eqref{SumOfEulerNumbers2} is proved.
\end{proof}

\begin{cor}\label{RelationForEstimate}
Assume that the assumptions (i) and (ii) of Theorem \ref{ToricCondition} are satisfied. Then we have
\begin{equation}\label{SumOfEulerNumbers4} 
\frac{1}{\LCM(m_{\max},n_{\max})}+\frac{1}{\LCM(m_{\min},n_{\min})}=\sum_{j} \frac{d^{j}}{\LCM(k_{1}^{j},k_{l(j)}^{j})}
\end{equation}
for some positive integers $d^j$.
\end{cor}

\begin{proof}
We denote the lattice of $\Lie(G)$ which is the kernel of the exponential map $\Lie(G) \longrightarrow G$ by $\Lie(G)_{\mathbb{Z}}$. We define $\det \colon \wedge^{2} \Lie(G) \longrightarrow \mathbb{R}$ so that $\det({\wedge^{2} \Lie(G)_{\mathbb{Z}}})$ is contained in $\mathbb{Z}$. Let $\overline{X}$ be the primitive vector in $\Lie(G)_{\mathbb{Z}}$ which generates the $S^1$-subgroup $G'$ in the assumption (ii) of Theorem \ref{ToricCondition}. We denote the $S^1$-action of $G'$ on $M$ by $\sigma$. For a closed orbit $\Sigma$ of the Reeb flow, $Y_{\Sigma}$ denotes the primitive vector in $\Lie(G)_{\mathbb{Z}}$ whose infinitesimal action generates the action of the identity component of the isotropy group of $\rho$ at $\Sigma$. We have
\begin{equation}\label{Intersection}
I(\sigma,\Sigma)= I(\rho,\Sigma) \left| \det \begin{pmatrix} \overline{X} & Y_{\Sigma} \end{pmatrix} \right|.
\end{equation}
For, both sides are equal to the number of the intersection points of $G'$ and the isotropy group of $\rho$ at $\Sigma$ in $T^2$. Since
\begin{equation}
\det \begin{pmatrix} \overline{X} & Y_{B_{\max}} \end{pmatrix} =1, \det \begin{pmatrix} \overline{X} & Y_{B_{\min}} \end{pmatrix} =1
\end{equation}
by the assumption, we have
\begin{equation}\label{sigmaandrho}
I(\sigma,B_{\max})= I(\rho,B_{\max}), 
I(\sigma,B_{\min})= I(\rho,B_{\min})
\end{equation}
by \eqref{Intersection}. By Lemma \ref{GCD}, we have
\begin{equation}\label{GCD2}
I(\rho,B_{\min})=\GCD(m_{\min},n_{\min}),
I(\rho,B_{\max})=\GCD(m_{\max},n_{\max}).
\end{equation}
By Lemma \ref{LevelSetNearB}, \eqref{sigmaandrho} and \eqref{GCD2}, we have
\begin{equation}\label{LevelSetNearBmin2}
e(\rho_2,H_{\min}/\rho_1)=-\frac{I(\sigma,B_{\min})}{m_{\min}n_{\min}}=-\frac{I(\rho,B_{\min})}{m_{\min}n_{\min}}=-\frac{1}{\LCM(m_{\min},n_{\min})}
\end{equation}
and
\begin{equation}\label{LevelSetNearBmax2}
e(\rho_2,H_{\max}/\rho_1)=\frac{I(\sigma,B_{\max})}{m_{\max}n_{\max}}=\frac{I(\rho,B_{\max})}{m_{\max}n_{\max}}=\frac{1}{\LCM(m_{\max},n_{\max})}.
\end{equation}
Substituting \eqref{LevelSetNearBmin2} and \eqref{LevelSetNearBmax2} to the equation \eqref{SumOfEulerNumbers2} of Lemma \ref{RelationOfEulerNumbers}, we have \eqref{SumOfEulerNumbers4}.
\end{proof}

Note that \eqref{SumOfEulerNumbers4} in Corollary \ref{RelationForEstimate} is similar to the equation (5.12) in Lemma 5.11 of \cite{Kar}. We complete the proof of Proposition \ref{TwoChains} by using Corollary \ref{RelationForEstimate} and argument similar to Karshon's proof of Proposition 5.13 of \cite{Kar} as follows:

\begin{proof}
Assume that $(M,\alpha)$ has more than two nontrivial chains. Since $B_{\min}$ and $B_{\max}$ are contained in the closures of at most two gradient manifolds consisting of points with nontrivial isotropy group of $\rho$, one of the following occurs: 
\begin{enumerate}
\item there exists a nontrivial chain such that both of the top and bottom gradient manifolds are free or
\item there exist two nontrivial chains $C^{1}$ and $C^{2}$ such that
\begin{enumerate}
\item the top gradient manifold of $C^{1}$ and the bottom gradient manifold of $C^{2}$ consist of points with nontrivial isotropy group of $\rho$ and
\item the top gradient manifold of $C^{2}$ and the bottom gradient manifold of $C^{1}$ are free.
\end{enumerate}
\end{enumerate}
We show that both of (i) and (ii) cannot occur. 

First, we show that (ii) cannot occur. Let $m_{\max}$ be the cardinality of the isotropy group of $\rho$ at the top gradient manifold of $C^{1}$ and $n_{\min}$ be the cardinality of the isotropy group of $\rho$ at the bottom gradient manifold of $C^{2}$. If (ii) is true, then the left hand side of \eqref{SumOfEulerNumbers4} is less than or equal to $\frac{1}{m_{\max}} + \frac{1}{n_{\min}}$ and the right hand side of \eqref{SumOfEulerNumbers4} is greater than $\frac{1}{m_{\max}} + \frac{1}{n_{\min}}$. Hence it is contradiction.

Assume that (i) is true. By the equation \eqref{SumOfEulerNumbers4},
\begin{equation}\label{GreaterThanOne}
\frac{1}{\LCM(m_{\max},n_{\max})}+\frac{1}{\LCM(m_{\min},n_{\min})} > 1,
\end{equation}
since the right hand side of \eqref{SumOfEulerNumbers4} is greater than $1$. Then we can assume $m_{\max}=n_{\max}=1$, since if both of pairs $(m_{\max},n_{\max})$ and $(m_{\min},n_{\min})$ are not equal to $(1,1)$, the inequality \eqref{GreaterThanOne} is not satisfied. Then we have
\begin{equation}\label{Impossible1} 
1+\frac{1}{\LCM(m_{\min},n_{\min})} = \frac{d^1}{n_{\min}} + \frac{d^2}{m_{\min}} + d^3
\end{equation}
for some positive integers $d^1, d^2$ and $d^3$ by the equation \eqref{SumOfEulerNumbers4}. On the other hand, there exists no solution for \eqref{Impossible1}, since we have $1 \leq d^3$, $\frac{1}{\LCM(m_{\min},n_{\min})} \leq \frac{d^1}{n_{\min}}$ and $0 < \frac{d^2}{m_{\min}}$. It is contradiction. 

Hence the proof of Proposition \ref{TwoChains} is completed.
\end{proof}

\subsection{Two chains imply toric}

\begin{prop}\label{SufficientConditionToBeToric}
Assume that 
\begin{enumerate}
\item $B_{\max}$ and $B_{\min}$ are lens spaces or closed orbits of the Reeb flow,
\item the number of nontrivial chains in $(M,\alpha)$ is at most $2$.
\end{enumerate}
Then there exists an $\alpha$-preserving $T^3$-action on $M$.
\end{prop}

\begin{prop}\label{NearChains}
Assume that the conditions (i) and (ii) in Proposition \ref{SufficientConditionToBeToric} are satisfied. Let $C^1$ and $C^2$ be two chains in $(M,\alpha)$. Assume that there is no nontrivial chain except $C^1$ and $C^2$. There exist an open neighborhood $V$ of the union of $C^1$, $C^2$, $B_{\min}$ and $B_{\max}$ and an $\alpha$-preserving $T^3$-action $\tau$ on $V$. Moreover, the image of the contact moment map for $\tau$ is an open neighborhood of the boundary of a convex polygon in a $2$-dimensional affine subspace of $\Lie(T^3)^{*}$.
\end{prop}

\begin{proof}
By the condition (i) of Proposition \ref{SufficientConditionToBeToric}, Lemmas \ref{ToricActionNearSingularOrbits} and \ref{ToricActionNearB}, we have 
\begin{enumerate}
\item an open neighborhood $U_{\min}$ of $B_{\min}$ and an $\alpha$-preserving $T^3$-action $\tau_{\min}$ on $U_{\min}$ and
\item an open neighborhood $U_{\max}$ of $B_{\max}$ and an $\alpha$-preserving $T^3$-action $\tau_{\max}$ on $U_{\max}$.
\end{enumerate}
By Lemma \ref{ToricActionNearChains}, we have an open neighborhood $U^{i}$ of the chain $C^{i}$ and an $\alpha$-preserving $T^3$-action $\tau^{i}$ on $U^{i}$ for $i=1$ and $2$ which satisfy the following:
\begin{enumerate}
\item $U^{1} \cap U^{2}$ is contained in $U_{\min} \cup U_{\max}$,
\item $\tau^{i}$ preserves $\alpha$, 
\item the restriction of $\tau^{i}$ to $U^{i} \cap U_{\min}$ coincides with $\rho_{\min}$ and
\item the restriction of $\tau^{i}$ to $U^{i} \cap U_{\max}$ is conjugate to $\rho_{\max}$.
\end{enumerate}

We denote the $T^3$-action on $U_{\min} \cup U^{1} \cup U^{2}$ obtained from $\tau_{\min}$, $\tau^{1}$ and $\tau^{2}$ by $\tau^{3}$. Define $W=(U_{\min} \sqcup U^{1} \sqcup U^{2} \sqcup U_{\max}) / \sim$ where $x \sim y$ if both of $x$ and $y$ are the same points of $U^{2} \cap U_{\max}$, $U^{2} \cap U_{\min}$ or $U^{1} \cap U_{\min}$. Conjugating $\tau_{\max}$ so that $\tau_{\max}$ and $\tau_{2}$ are equal on $U^{2} \cap U_{\max}$, we obtain a $T^3$-action $\tau^{4}$ on $W$ from $\tau^{3}$ and $\tau_{\max}$. We denote the contact moment map for $\tau^{4}$ by $\tilde{\Phi}_{\alpha} \colon W \longrightarrow \Lie(T^3)^{*}$. Let $\Sigma_{0}$ be the top closed orbit of the Reeb flow in $C^{1}$. Let $\Sigma_{\max}$ and $\Sigma^{1}$ be the connected components of the inverse image of $\Sigma_{0}$ by the canonical map $W \longrightarrow U_{\min} \cup U^{1} \cup U^{2} \cup U_{\max}$ so that $\Sigma_{\max}$ is contained in the subset $U_{\max}$ of $W$ and $\Sigma^{1}$ is contained in the subset $U^{1}$ of $W$, respectively. We will show the former part of Lemma \ref{NearChains} by showing that the automorphism of $T^3$ which gives the conjugation of $\tau^{4}|_{U_{\max}}$ and $\tau^{4}|_{U^1}$ is the identity.

We show $\tilde{\Phi}_{\alpha}(\Sigma_{\max})=\tilde{\Phi}_{\alpha}(\Sigma^{1})$. We fix our notation. We identify $\Lie(T^3)$ with $\mathbb{R}^{3}$ so that the kernel of $\exp \colon \Lie(T^3) \longrightarrow T^3$ is identified with $\mathbb{Z}^{3}$ and $\Lie(G)$ is identified with $\{ \begin{psmallmatrix} v_1 \\ 0 \\ v_3 \end{psmallmatrix} \in \mathbb{R}^{3} | v_1,v_3 \in \mathbb{R} \}$. Since $\overline{R}$ is contained in $\Lie(G)$, we can rotate $\Lie(G)$ by an orthogonal matrix so that
\begin{equation}
\overline{R}=\begin{psmallmatrix} 0 \\ 0 \\ r \end{psmallmatrix}.
\end{equation}
Define $\overline{X}_{0}=\begin{psmallmatrix} 1 \\ 0 \\ 0 \end{psmallmatrix}$. Let $X_{0}$ be the infinitesimal action of $\overline{X}_{0}$ and define a function $\Phi'$ on $M$ by 
\begin{equation}
\Phi'=\alpha(X_{0}).
\end{equation}
$\overline{X}_{0}$ defines a linear function $\pi$ on $\Lie(T^3)^{*}$ by the canonical coupling, which coincides with the first projection. Note that $\pi \circ \tilde{\Phi}_{\alpha} = \Phi'$. Let $x_{\max}$ and $x_{\min}$ be the maximum and the minimum values of $\pi$ on the image of $\tilde{\Phi}_{\alpha}$. Let $A$ be the affine subspace $A=\{ v \in \Lie(T^3)^{*} | v(\overline{R})=1 \}$. Then $\tilde{\Phi}_{\alpha}(\Sigma_{\max})$ and $\tilde{\Phi}_{\alpha}(\Sigma^{1})$ are contained in $\pi^{-1}(x_{\max}) \cap A$ by the construction. Hence to show $\tilde{\Phi}_{\alpha}(\Sigma_{\max})=\tilde{\Phi}_{\alpha}(\Sigma^{1})$, it suffices to show the equality of the second coordinates of $\tilde{\Phi}_{\alpha}(\Sigma_{\max})$ and $\tilde{\Phi}_{\alpha}(\Sigma^{1})$. 

We will prove the equality of the second coordinates of $\tilde{\Phi}_{\alpha}(\Sigma_{\max})$ and $\tilde{\Phi}_{\alpha}(\Sigma^{1})$ computing the width of the level set of $\pi$ in the polyhedron surrounded by the image of $\tilde{\Phi}_{\alpha}$. Note that the width of the level set of $\pi$ in the domain surrounded by the image of $\tilde{\Phi}_{\alpha}$ corresponds to the value of the density function of the Duistermaat-Heckman measure in the case of symplectic manifolds with hamiltonian $S^1$-action (See \cite{Kar}). 

We define the width $\phi(x)$ of the level set of $\pi$ at a value $x$ by $\phi(x)=\pr_{2}(\tilde{\Phi}_{\alpha}(C^{2}) \cap \pi^{-1}(x)) - \pr_{2}(\tilde{\Phi}_{\alpha}(C^{1}) \cap \pi^{-1}(x))$ where $\pr_{2}$ is the second projection on $\Lie(T^3)^{*}$. We define the width $w_{\max}$ and $w_{\min}$ of $\tilde{\Phi}_{\alpha}(B_{\min})$ and $\tilde{\Phi}_{\alpha}(B_{\min})$ by $\sup_{z \in \Omega} \pr_{2}(z) - \inf_{z \in \Omega} \pr_{2}(z)$ for $\Omega=B_{\min}$ and $B_{\max}$. If $B_{\min}$ or $B_{\max}$ is a closed orbit of the Reeb flow, then $w_{\min}$ or $w_{\max}$ is $0$ respectively. It suffices to show $\phi(x_{\max})=w_{\max}$ to show the coincidence of the second coordinates of $\tilde{\Phi}_{\alpha}(\Sigma_{\max})$ and $\tilde{\Phi}_{\alpha}(\Sigma^{1})$. 

Let $L^{j}_{1}, L^{j}_{2}, \cdots, L^{j}_{l(j)}$ be the gradient manifolds which form $C^{j}$ in the order of the subscripts for $j=1$ and $2$. Recall that the value of $\Phi'$ on the $\omega$-limit set of $L^{j}_{i+1}$ is greater than the value of $\Phi'$ on the $\omega$-limit set of $L^{j}_{i}$ for $j=1$ and $2$. $\tilde{\Phi}_{\alpha}(L^{j}_{i})$ is contained in the intersection line of $A$ and a plane $P^{j}_{i}$ in $\Lie(T^3)^{*}$. We denote the primitive vector in $\mathbb{Z}^{3}$ with positive third coordinate defining the plane $P^{j}_{i}$ by $n^{j}_{i}$. Define $k^{j}_{i}=I(\rho,L^{j}_{i})$. 

We will prove that the second component of $n^{j}_{i}$ is equal to $(-1)^{j}k^{j}_{i}$ where $n^{j}_{i}$ is an element of $\Lie(T^3)$ whose infinitesimal action of $n^{j}_{i}$ generates the isotropic action of $\tau^{4}$ at $L^{j}_{i}$. Since the normal vector of $\Lie(G)$ in $\Lie(T^3)$ is equal to $\begin{psmallmatrix} 0 \\ 1 \\ 0 \end{psmallmatrix}$, the absolute value of the second component of $n^{j}_{i}$ is equal to the number of intersection points of $G$ and the $S^1$-subgroup of $T^3$ corresponding to $n^{j}_{i}$. The number of intersection points of $G$ and the $S^1$-subgroup of $T^3$ corresponding to $n^{j}_{i}$ is equal to $k^{j}_{i}$. The signature of the second component of $n^{j}_{i}$ is equal to $(-1)^{j}$ by the following two reason: 
\begin{enumerate}
\item the angle between the planes defined by $n^{j}_{i}$ and $n^{j}_{i+1}$ is smaller than $\pi$ by the local convexity of the image of the symplectic moment map defined on the symplectization (See Theorem 4.7 of \cite{GuSt}) and
\item the value of $\Phi'$ on the $\omega$-limit set of $L^{j}_{i+1}$ is greater than the value of $\Phi'$ on the $\omega$-limit set of $L^{j}_{i}$.
\end{enumerate}
 Hence the second component of $n^{j}_{i}$ is equal to $(-1)^{j}k^{j}_{i}$.

We denote the closed orbit of the Reeb flow between $L^{j}_{i}$ and $L^{j}_{i+1}$ by $\Sigma^{j}_{i i+1}$. We denote the value of $\pi$ on $\tilde{\Phi}_{\alpha}(\Sigma^{j}_{i i+1})$ by $x^{j}_{i i+1}$. Consider the union of segments $\cup_{j=1}^{l(1)}\tilde{\Phi}_{\alpha}(L^{1}_{i})$. Let $s^{j}$ be the slope of $\tilde{\Phi}_{\alpha}(L^{j}_{1})$ for $j=1$ and $2$. By Lemma \ref{ChangeOfSlope} below, the change of the slope of $\cup_{j=1}^{l(1)}\tilde{\Phi}_{\alpha}(L^{1}_{i})$ when $x$ goes through $x^{1}_{i}$ is $- \frac{1}{r} \big( \frac{1}{k^{1}_{i}k^{1}_{i+1}} \det \begin{psmallmatrix} n^{1}_{i} & n^{1}_{i+1} & \overline{R} \end{psmallmatrix} \big)$. Hence the slope of $\tilde{\Phi}_{\alpha}(L^{1}_{h})$ is 
\begin{equation}\label{SlopeOfC1}
s^{1} - \frac{1}{r} \sum_{i=1}^{h-1} \frac{1}{k^{1}_{i}k^{1}_{i+1}} \det \begin{pmatrix} n^{1}_{i+1} & n^{1}_{i} & \overline{R} \end{pmatrix}.
\end{equation}
Similarly, the change of the slope of $\cup_{j=1}^{l(1)}\tilde{\Phi}_{\alpha}(L^{1}_{i})$ when $x$ goes through $x^{2}_{i}$ is $- \frac{1}{r} \big( \frac{1}{k^{2}_{i}k^{2}_{i+1}} \det \begin{psmallmatrix} n^{2}_{i} & n^{2}_{i+1} & \overline{R} \end{psmallmatrix} \big)$. Hence the slope of $\tilde{\Phi}_{\alpha}(L^{2}_{h})$ is 
\begin{equation}\label{SlopeOfC2}
s^{2} - \frac{1}{r} \sum_{i=1}^{h-1} \frac{1}{k^{2}_{i}k^{2}_{i+1}} \det \begin{pmatrix} n^{2}_{i+1} & n^{2}_{i} & \overline{R} \end{pmatrix}.
\end{equation}
By Lemma \ref{ChangeOfSlope}, we have
\begin{equation}\label{InitialSlope}
s^{1}-s^{2}=\frac{1}{r} \Big( \frac{1}{k^{1}_{1}k^{2}_{1}} \det \begin{psmallmatrix} n^{1}_{1} & n^{2}_{1} & \overline{R} \end{psmallmatrix} \Big).
\end{equation}
Let $\hat{L}^{1}_{l(1)}$ be the connected component of the inverse image of $L^{1}_{l(1)}$ by the canonical map $W \longrightarrow M$ contained in the subset $U_{\max}$ of $W$. We denote the normal vector of $\tilde{\Phi}_{\alpha}(\hat{L}^{1}_{l(1)})$ in $\Lie(T^3)^{*}$ by $\hat{n}^{1}_{l(1)}$. We put $x^{j}_{-1 \, 0}=x_{\min}$ and $x^{j}_{l(j) l(j)+1}=x_{\max}$ for $j=1$ and $2$. By \eqref{SlopeOfC1}, \eqref{SlopeOfC2} and \eqref{InitialSlope}, we can compute $r \phi(x_{\max})$ as follows:
\begin{equation}\label{ComputationOfPhi}
\begin{array}{rl}
  & r \phi(x_{\max}) \\
= & r w_{\min} + \sum_{h} (x^{1}_{h h+1} - x^{1}_{h-1 h}) \big( r s^{1} - \sum_{i=1}^{h-1} \frac{1}{k^{1}_{i+1}k^{1}_{i}} \det \begin{pmatrix} n^{1}_{i+1} & n^{1}_{i} & \overline{R} \end{pmatrix} \big) \\
 & - \sum_{h} (x^{2}_{h h+1} - x^{2}_{h-1 h}) \big(  r s^{2} - \sum_{i=1}^{h-1} \frac{1}{k^{2}_{i+1}k^{2}_{i}} \det \begin{pmatrix} n^{2}_{i+1} & n^{2}_{i} & \overline{R} \end{pmatrix} \big) \\
= &  r w_{\min} + (x_{\max} - x_{\min}) \frac{1}{k^{1}_{1}k^{2}_{1}} \det \begin{psmallmatrix} n^{1}_{1} & n^{2}_{1} & \overline{R} \end{psmallmatrix} \\
  & - \sum_{j,h} (x^{j}_{h h+1}-x^{j}_{h-1 h}) \Big( \sum_{i=1}^{h-1} \frac{(-1)^{j-1}}{k^{j}_{i+1}k^{j}_{i}} \det \begin{pmatrix} n^{j}_{i+1} & n^{j}_{i} & \overline{R} \end{pmatrix} \Big) \\
= &  r w_{\min} + (x_{\max} - x_{\min}) \frac{1}{k^{1}_{1}k^{2}_{1}} \det \begin{pmatrix} n^{1}_{1} & n^{2}_{1} & \overline{R} \end{pmatrix} \\
  & - \sum_{j,i} (x_{\max} - x^{j}_{i i+1}) \frac{(-1)^{j-1}}{k^{j}_{i+1}k^{j}_{i}} \det \begin{pmatrix} n^{j}_{i+1} & n^{j}_{i} & \overline{R} \end{pmatrix} \\
= &  r w_{\min} \\
  & + x_{\max} \Big( \frac{1}{k^{1}_{1}k^{2}_{1}} \det \begin{pmatrix} n^{1}_{1} & n^{2}_{1} & \overline{R} \end{pmatrix} - \sum_{j,i} \frac{(-1)^{j-1}}{k^{j}_{i+1}k^{j}_{i}} \det \begin{pmatrix} n^{j}_{i+1} & n^{j}_{i} & \overline{R} \end{pmatrix} \Big) \\
  & - \Big( \frac{x_{\min}}{k^{1}_{1}k^{2}_{1}} \det \begin{pmatrix} n^{1}_{1} & n^{2}_{1} & \overline{R} \end{pmatrix} - \sum_{j,i} \frac{(-1)^{j-1} x^{j}_{i i+1}}{k^{j}_{i+1}k^{j}_{i}} \det \begin{pmatrix} n^{j}_{i+1} & n^{j}_{i} & \overline{R} \end{pmatrix} \Big) \\
=  & \Big(  r w_{\min} - \frac{1}{k^{1}_{1}k^{2}_{1}} \det \begin{pmatrix} n^{1}_{1} & n^{2}_{1} & x_{\min} \overline{R} - \overline{X}_{0} \end{pmatrix} \Big) \\
  & - \Big(  r w_{\max} + \frac{1}{k^{2}_{l(2)}k^{1}_{l(1)}} \det \begin{pmatrix} n^{2}_{l(2)} & \hat{n}^{1}_{l(1)} & x_{\max} \overline{R} - \overline{X}_{0} \end{pmatrix} \Big) \\
  & + x_{\max} \Big( \frac{1}{k^{2}_{l(2)}k^{1}_{l(1)}} \det \begin{pmatrix} n^{2}_{l(2)} & \hat{n}^{1}_{l(1)} & \overline{R} \end{pmatrix} + \frac{1}{k^{1}_{1}k^{2}_{1}} \det \begin{pmatrix} n^{1}_{1} & n^{2}_{1} & \overline{R} \end{pmatrix} \\
 &  \hspace{150pt} - \sum_{j,i} \frac{(-1)^{j-1}}{k^{j}_{i+1}k^{j}_{i}} \det \begin{pmatrix} n^{j}_{i+1} & n^{j}_{i} & \overline{R} \end{pmatrix} \Big) \\
  & - \Big( \frac{1}{k^{2}_{l(2)}k^{1}_{l(1)}} \det \begin{pmatrix} n^{2}_{l(2)} & \hat{n}^{1}_{l(1)} & \overline{X}_{0} \end{pmatrix} +  \frac{1}{k^{1}_{1}k^{2}_{1}} \det \begin{pmatrix} n^{1}_{1} & n^{2}_{1} & \overline{X}_{0} \end{pmatrix} \\
 & \hspace{150pt} - \sum_{j,i} \frac{(-1)^{j-1} x^{j}_{i i+1}}{k^{j}_{i+1}k^{j}_{i}} \det \begin{pmatrix} n^{j}_{i+1} & n^{j}_{i} & \overline{R} \end{pmatrix} \Big) \\
  & +  r w_{\max}.
\end{array}
\end{equation}
We show that the first line of the rightmost hand side of \eqref{ComputationOfPhi} is zero. If $B_{\min}$ is of dimension $1$, then we have $w_{\min}=0$ by the definition. We have $\det \begin{pmatrix} n^{1}_{1} & n^{2}_{1} & x_{\min} \overline{R} - \overline{X}_{0} \end{pmatrix}=0$ by Lemma \ref{VanishingDeterminant} below. Hence the first line is $0$. If $B_{\min}$ is of dimension $3$, then the first line is zero by Lemma \ref{WidthOfB} below. Similarly the second line is zero by Lemmas \ref{VanishingDeterminant} and \ref{WidthOfB}. The third and the fourth lines are zero by Lemma \ref{Formula} below. The fifth and the sixth lines are zero by Lemma \ref{Formula}. Hence we have $\phi(x_{\max})=w_{\max}$. Then $\tilde{\Phi}_{\alpha}(\Sigma_{\max})=\tilde{\Phi}_{\alpha}(\Sigma^{1})$ is proved.

We show $n^{1}_{l(1)}=\hat{n}^{1}_{l(1)}$. Let $P$ and $\hat{P}$ be the planes in $\Lie(T^3)^{*}$ defined by the normal vectors $n^{1}_{l(1)}$ and $\hat{n}^{1}_{l(1)}$ respectively. To show $n^{1}_{l(1)}=\hat{n}^{1}_{l(1)}$, it suffices to show $P \cap A=\hat{P} \cap A$. By $\tilde{\Phi}_{\alpha}(\Sigma_{\max})=\tilde{\Phi}_{\alpha}(\Sigma^{1})$, the intersection of both of lines $P \cap A$ and $\hat{P} \cap A$ on $A$ go through the same point $\tilde{\Phi}_{\alpha}(\Sigma_{\max})=\tilde{\Phi}_{\alpha}(\Sigma^{1})$. Hence to show $P \cap A=\hat{P} \cap A$, it suffices to show that $P \cap A$ and $\hat{P} \cap A$ are parallel. Since the slope of $P \cap A$ is
\begin{equation}
s^{1} - \frac{1}{r} \sum_{i=1}^{l(1)-1} \frac{1}{k^{1}_{i+1}k^{1}_{i}} \det \begin{pmatrix} n^{1}_{i+1} & n^{1}_{i} & \overline{R} \end{pmatrix}
\end{equation}
and the slope of $\hat{P} \cap A$ is
\begin{equation}
(s^{2} - \frac{1}{r} \sum_{i=1}^{l(2)-1} \frac{1}{k^{2}_{i+1}k^{2}_{i}} \det \begin{pmatrix} n^{2}_{i+1} & n^{2}_{i} & \overline{R} \end{pmatrix}) - \frac{1}{r} \big(\frac{1}{ k^{2}_{l(2)} k^{1}_{l(1)}} \det \begin{psmallmatrix} n^{2}_{l(2)} & \hat{n}^{1}_{l(1)} & \overline{R} \end{psmallmatrix} \big),
\end{equation}
their difference is $0$ by \eqref{InitialSlope} and Lemma \ref{Formula}. Hence $n^{1}_{l(1)}=\hat{n}^{1}_{l(1)}$ is proved. 

Then the matrix which gives the conjugation of $\tau^{4}|_{U_{\max}}$ and $\tau^{4}|_{U_{1}}$ is the identity, because it fixes a plane $\Lie(G)$ and a vector $n^{1}_{l(1)}=\hat{n}^{1}_{l(1)}$ which is not contained in the plane $\Lie(G)$. The proof of the former part of Lemma \ref{NearChains} is completed.

The latter part of Proposition \ref{NearChains} follows from the fact that the slopes of $\cup_{j=1}^{l(1)}\tilde{\Phi}_{\alpha}(L^{1}_{i})$ and $\cup_{j=1}^{l(2)}\tilde{\Phi}_{\alpha}(L^{2}_{i})$ are monotone. Then $\cup_{j=1}^{l(1)}\tilde{\Phi}_{\alpha}(L^{1}_{i})$ and $\cup_{j=1}^{l(2)}\tilde{\Phi}_{\alpha}(L^{2}_{i})$ are graphs of convex functions with respect to the first coordinate in $A$. Hence the width of each level set is positive and $\cup_{j=1}^{l(1)}\tilde{\Phi}_{\alpha}(L^{1}_{i}) \cup \cup_{j=1}^{l(2)}\tilde{\Phi}_{\alpha}(L^{2}_{i}) \cup \tilde{\Phi}_{\alpha}(B_{\min}) \cup \tilde{\Phi}_{\alpha}(B_{\max})$ is the boundary of a convex polyhedron.
\end{proof}

\begin{lem}\label{ChangeOfSlope}
Let $n=\begin{psmallmatrix} n(1) \\ n(2) \\ n(3) \end{psmallmatrix}$ and $n'=\begin{psmallmatrix} n'(1) \\ n'(2) \\ n'(3) \end{psmallmatrix}$ be two covectors in $\mathbb{R}^{3 *}$. Let $A$ be the affine subspace $\{ \begin{psmallmatrix} v(1) \\ v(2) \\ r \end{psmallmatrix} \in \mathbb{R}^{3} | v(1),v(2) \in \mathbb{R} \}$ where $r$ is a nonzero real number. Let $l$ be the intersection line of $A$ and the plane defined by $n$. Let $l'$ be the intersection line of $A$ and the plane defined by $n'$. We define the slope of $l$ by $\frac{y}{x}$ if a vector $\begin{psmallmatrix} x \\ y \\ 0 \end{psmallmatrix}$ is parallel to $l$. Then the difference of the slopes of $l$ and $l'$ is $\frac{1}{r} \big(-\frac{1}{n(2) n'(2)} \det \begin{psmallmatrix} n & n' & \overline{R} \end{psmallmatrix} \big)$ where $\overline{R}=\begin{psmallmatrix} 0 \\ 0 \\ r \end{psmallmatrix}$.
\end{lem}

\begin{proof}
$\overline{R} \times n=\begin{psmallmatrix} -rn(2) \\ rn(1) \\ 0 \end{psmallmatrix}$ and $\overline{R} \times n'= \begin{psmallmatrix} -rn'(2) \\ rn'(1) \\ 0 \end{psmallmatrix}$ are parallel to $l$ and $l'$. Hence the difference of the slopes is
\begin{equation}
\frac{-rn(1)}{rn(2)}-\frac{-rn'(1)}{rn'(2)}=-\frac{n(1) n'(2) - n'(1) n(2)}{n(2) n'(2)}=\frac{1}{r} \Big(-\frac{1}{n(2) n'(2)} \det \begin{psmallmatrix} n & n' & \overline{R} \end{psmallmatrix} \Big).
\end{equation}
\end{proof}

\begin{lem}\label{VanishingDeterminant}
Let $(M,\alpha)$ be a $5$-dimensional $K$-contact orbifold. Let $\Sigma$ be a closed orbit of the Reeb flow. Assume that we have an $\alpha$-preserving $T^3$-action $\tau$ on an open neighborhood $W$ of $\Sigma$. We denote the contact moment map for $\tau$ by $\tilde{\Phi}_{\alpha}$ defined on $W$. We fix an element $\overline{Y}$ of $\Lie(T^3)$. Let $L_{1}$ and $L_{2}$ be two $K$-contact manifolds such that $L_{j} \cap W$ consists of points with nontrivial isotropy group of $\tau$ and contains $\Sigma$. Let $n_{j}$ be the primitive vector in $\Lie(T^3)$ whose infinitesimal action generates the identity component of the isotropy group of $\tau$ at $L_{j} \cap W$. Then we have 
\begin{equation}\label{Equation : VanishingDeterminant}
 \det \begin{pmatrix} n_{1} & n_{2} & \overline{Y} - \tilde{\Phi}_{\alpha}(\Sigma)(\overline{Y}) \overline{R} \end{pmatrix}=0
\end{equation}
where $\Phi_{\alpha}(\Sigma)(\overline{Y})$ is the coupling of $\overline{Y}$ with the value of $\tilde{\Phi}_{\alpha}$ at $\Sigma$.
\end{lem}

\begin{proof}
Since $\tilde{\Phi}_{\alpha}(\Sigma)(\overline{Y} - \tilde{\Phi}_{\alpha}(\Sigma)(\overline{Y})\overline{R})=0$, the function $\alpha(Y -\tilde{\Phi}_{\alpha}(\Sigma)(\overline{Y})R)$ on $M$ is $0$ on $\Sigma$. Hence $Y -\tilde{\Phi}_{\alpha}(\Sigma)(\overline{Y})R$ is tangent to $\ker \alpha$ on $\Sigma$. Then $Y -\tilde{\Phi}_{\alpha}(\Sigma)(\overline{Y})R$ vanishes on $\Sigma$. Since $\{n_{1},n_{2}\}$ is a basis of the subspace of $\Lie(T^3)$ consisting of the vectors whose infinitesimal actions vanish on $\Sigma$, we have \eqref{Equation : VanishingDeterminant}.
\end{proof}

\begin{lem}\label{AttachingCorners}
Let $B$ be a $K$-contact submanifold of $(M,\alpha)$ diffeomorphic to a lens space. Assume that we have an $\alpha$-preserving $T^3$-action $\tau$ on an open neighborhood $W$ of $B$. We denote the contact moment map for $\tau$ by $\tilde{\Phi}$. Let $L^{1}$ and $L^{2}$ be two $K$-contact manifolds which consist of points with nontrivial isotropy group of $\tau$ and intersect $B$. Let $n^{j}$ be the primitive vector in $\Lie(T^3)$ whose infinitesimal action generates the identity component of the isotropy group of $\tau$ at $L^{j}$ for $j=1$ and $2$. Let $n_{B}$ be the primitive vector in $\Lie(T^3)$ whose infinitesimal action generates the identity component of the isotropy group of $\tau$ at $B$. Put $A=\{ v \in \Lie(T^3)^{*} | v(\overline{R})=1 \}$. We define cones $\Delta'$ and $\overline{\Delta}$ in $\Lie(T^3)^{*}$ by
\begin{equation}
\Delta'=\{ v \in \Lie(T^3)^{*} | v(n^{1}) \geq 0, v(n^{2}) \geq 0, v(n_{B}) \leq 0  \}
\end{equation}
and 
\begin{equation}
\overline{\Delta}=\{ v \in \Lie(T^3)^{*} | v(n^{1}) \geq 0, v(n^{2}) \geq 0 \}.
\end{equation}
Then there exist a $5$-dimensional $K$-contact orbifold $(U,\beta)$ and a $\beta$-preserving $T^3$-action $\sigma$ on $U$ such that 
\begin{enumerate}
\item the image of the contact moment map $\Psi$ of $(U,\beta)$ for $\sigma$ contains $\Delta' \cap A$ and
\item $(U-\Psi^{-1}(\Delta'),\beta|_{U-\Psi^{-1}(\Delta')})$ is isomorphic to $(V-B,\alpha|_{V-B})$ as $K$-contact manifolds for an open neighborhood $V$ of $B$ in $M$.
\end{enumerate}
\end{lem}

\begin{proof}
Let $Q^{j}$ be the subspace of $\Lie(T^{3})^{*}$ defined by $n_{j}$ for $j=0$ and $2$. Applying the Delzant construction for $K$-contact manifolds in Theorem 5.1 of \cite{BoGa} to $\overline{\Delta} \cap A$, we have a $5$-dimensional $K$-contact orbifold $(\overline{M},\overline{\alpha})$ of rank $2$ with an $\overline{\alpha}$-preserving $T^3$-action with the contact moment map $\overline{\Phi}$ such that 
\begin{enumerate}
\item the image of $\overline{\Phi}$ is $\overline{\Delta} \cap A$ and
\item $\overline{M}-\overline{\Phi}^{-1}(Q^{0} \cap Q^{2})$ is a smooth manifold.
\end{enumerate}
Take a small open neighborhood $W$ of $\Delta'$ in $\overline{\Delta}$. We show that $(\overline{\Phi}^{-1}(W),\overline{\alpha}|_{\Phi^{-1}(W)})$ satisfies the desired conditions on $(U,\beta)$. $(\overline{\Phi}^{-1}(W),\overline{\alpha}|_{\overline{\Phi}^{-1}(W)})$ satisfies the condition (i) by the construction.

By Lemma 4.9 and Proposition 5.2 of \cite{Ler2}, two $5$-dimensional $K$-contact toric manifolds are isomorphic if the images of the contact moment maps are the same convex subsets of $\Lie(T^{3})^{*}$. Since the images of the contact moment maps of $(\overline{\Phi}^{-1}(W - \Delta'),\overline{\alpha}|_{\overline{\Phi}^{-1}(W - \Delta')})$ and $(\tilde{\Phi}^{-1}(W) - B,\alpha|_{\Phi^{-1}(W) - B})$ are the same and convex, we have an isomorphism $f$ from $(\overline{\Phi}^{-1}(W - \Delta'),\overline{\alpha}|_{\overline{\Phi}^{-1}(W - \Delta')})$ to $(\tilde{\Phi}^{-1}(W) - B,\alpha|_{\Phi^{-1}(W) - B})$. Hence $(\overline{\Phi}^{-1}(W) - B,\overline{\alpha}|_{\Phi^{\prime -1}(W) - B})$ satisfies the condition (ii).
\end{proof}

\begin{lem}\label{WidthOfB}
We use the notation in the assumption of Lemma \ref{AttachingCorners}. We fix an element $\overline{Y}$ of $\Lie(G)$. The linear function on $\Lie(T^3)^{*}$ defined by the coupling with $\overline{Y}$ is denoted by $\pi$. Assume that $B$ is the minimal component of the function $\alpha(Y)$ where $Y$ is the infinitesimal action of $\overline{Y}$. We assume that $A$ is $\{ \begin{psmallmatrix} v(1) \\ v(2) \\ r \end{psmallmatrix} \in \mathbb{R}^{3} | v(1),v(2) \in \mathbb{R} \}$ where $r$ is a nonzero real number. We assume that the slope of $n^{1}$ is greater than the slope of $n^{2}$. We denote the width of $\Phi(B)$ as a level set of $\pi|_{A}$ by $w$. Then we have 
\begin{equation}\label{Equation : WidthOfB}
w=\frac{1}{r} \Big( -\frac{1}{n^{1}(2)n^{2}(2)} \det \begin{pmatrix} n^{1} & n^{2} & \Phi(B)(\overline{Y}) \overline{R} - \overline{Y} \end{pmatrix} \Big)
\end{equation}
where $v(2)$ denotes the second component of $v$ in $\mathbb{R}^{3}$ and $\Phi(B)(\overline{Y})$ is the coupling of $\overline{Y}$ with the value of $\Phi$ at $B$.
\end{lem}
Note that $\Phi$ is constant on $B$.
\begin{proof}
Take $(U,\beta)$ which satisfies the conditions (i) and (ii) in Lemma \ref{AttachingCorners}. We put $l^{j} = Q^{j} \cap A$ for $j=1$ and $2$. Let $Q^{j}$ be the subspace of $\Lie(T^{3})^{*}$ defined by $n_{j}$ for $j=1$ and $2$. We put $\{v_{0}\}=Q^{1} \cap Q^{2} \cap A$ and $\Sigma=\Psi^{-1}(v_{0})$. Then the difference of the slopes of $l^{1}$ and $l^{2}$ is $\frac{1}{r}\left(-\frac{1}{n^{1}(2)n^{2}(2)} \det \begin{pmatrix} n^{1} & n^{2} & \overline{R} \end{pmatrix}\right)$ by Lemma \ref{ChangeOfSlope}. Hence we have 
\begin{equation}\label{w1}
w = -\frac{1}{r}\big( \Phi(B)(\overline{Y})-\Psi(\Sigma)(\overline{Y}) \big) \frac{1}{n^{1}(2)n^{2}(2)} \det \begin{pmatrix} n^{1} & n^{2} & \overline{R} \end{pmatrix}.
\end{equation}
By Lemma \ref{VanishingDeterminant}, we have
\begin{equation}\label{w2}
\det \begin{pmatrix} n^{1} & n^{2} & \Psi(\Sigma)(\overline{Y}) \overline{R} \end{pmatrix} = \det \begin{pmatrix} n^{1} & n^{2} & \overline{Y} \end{pmatrix}.
\end{equation}
By \eqref{w1} and \eqref{w2}, we have \eqref{Equation : WidthOfB}.
\end{proof}

We prepare some notation for lens spaces which will be used in the following two lemmas. Let $\tau_{0}$ be an effective $S^1$-actions on $T^2 \times \{0\}$ defined by
\begin{equation}
t_{0} \cdot (x,y,0)= (x, t_{0}y, 0).
\end{equation}
Let $\tau_{1}$ be an $S^1$-actions on $T^2 \times \{1\}$ defined by
\begin{equation}
t_{1} \cdot (x,y,1)= (t_{1}^{p}x, t_{1}^{q}y, 1)
\end{equation}
for a pair of coprime integer $(p,q)$. We put $N=(T^2 \times [0,1])/ \sim$, where $(z_{1},u_{1}) \sim (z_{2},u_{2})$ if $u_{1}=u_{2}=j$ and $[(z_{1},u_{1})]=[(z_{2},u_{2})]$ in $(T^2 \times \{j\})/\tau_{j}$ for $j=0$ or $1$. For a pair of coprime integers $(p,q)$, let $\tau$ be a $\mathbb{Z}/p\mathbb{Z}$-action on $S^3$ defined by 
\begin{equation}
\xi \cdot (z_{1},z_{2}) = (\xi z_{1}, \xi^{-q} z_{2})
\end{equation}
for $\xi$ in $\mathbb{Z}/p\mathbb{Z}$ where $(\tilde{x},\tilde{y})$ is the standard coordinate on $S^3$ and we regard $\mathbb{Z}/p\mathbb{Z}$ as a subgroup of the group of complex numbers of absolute value $1$. We identify $L(p,q)=S^3/\tau$ with $N$ by the map induced from $f \colon S^3/\tau \longrightarrow T^2 \times [0,1]$ defined by $f(z_{1},z_{2}) =(\arg (z_{1})^{p}, \arg (z_{1})^{q} \arg z_{2}, |z_{1}|)$ on the complement of the singular orbits of $\tau$.

For a group $G'$, a subgroup $G''$ of $G'$ and a topological space $A$ with a $G''$-action, $G' \times_{G''} A$ denotes the quotient of $G' \times A$ by the $G''$-action defined by $g'' \cdot (g',x)=(g' g'', (g'')^{-1} \cdot x)$ for $g'$ in $G'$, $g''$ in $G''$ and $x$ in $A$.

\begin{lem}\label{EulerNumberOfLevelSetsNearB}
Assume that the minimal component $B$ of $\Phi$ is diffeomorphic to a lens space. Assume that we have an $\alpha$-preserving $T^3$-action $\tau$ near $B$. Let $H$ be a level set of $\Phi$ sufficiently close to $B$. Let $L_{1}$ and $L_{2}$ be two $K$-contact manifolds which intersect $B$ and consist of points with nontrivial isotropy group of $\tau$ near $B$. We put $k_{j}=I(\rho,L_{j})$ for $j=1$ and $2$. Let $n_{j}$ be the primitive vector in $\Lie(T^3)$ whose infinitesimal action generates the isotropy group of $\tau$ at $L_{j}$ for $j=1$ and $2$. Take $(U,\beta)$ which satisfies the conditions (i) and (ii) in Lemma \ref{AttachingCorners}. Let $Q^{j}$ be the subspace of $\Lie(T^{3})^{*}$ defined by $n_{j}$ for $j=1$ and $2$. We put $\{v_{0}\}=Q^{1} \cap Q^{2} \cap A$ and $\Sigma=\Psi^{-1}(v_{0})$. Let $\rho_{1}$ be an $S^1$-subaction of $\rho$ such that the restriction of the infinitesimal action of $\rho_{1}$ is a positive multiple of the Reeb vector field of $\beta$ on $\Sigma$. We fix an $S^1$-action $\rho_{2}$ of $\rho$ so that the product of $\rho_{1}$ and $\rho_{2}$ is equal to $\rho$. We fix a transverse orientation of orbits of $\rho$ in $H$ as Subsection 4.1 (See comments after Lemma \ref{LocalComputation}). Then we have
\begin{equation}\label{ComputationOfExtendedEulerNumber}
e(\rho_{2},H/\rho_{1}) = \frac{1}{k_{1}k_{2}} \det \begin{pmatrix} n_{1} & n_{2} & \overline{Y} \end{pmatrix} 
\end{equation}
where $\overline{Y}$ is the element of $\Lie(T^3)$ whose infinitesimal action generates $\rho_{1}$. If $B$ is the maximal component of $\Phi$ diffeomorphic to a lens space and $H$ is a level set of $\Phi$ sufficiently close to $B$, then we have
\begin{equation}\label{ComputationOfExtendedEulerNumber2}
e(\rho_{2},H/\rho_{1}) = - \frac{1}{k_{1}k_{2}} \det \begin{pmatrix} n_{1} & n_{2} & \overline{Y} \end{pmatrix}
\end{equation}
where $\overline{Y}$ is the element of $\Lie(T^3)$ whose infinitesimal action generates $\rho_{1}$.
\end{lem}

\begin{proof}
$H$ is the boundary of a $T^3$-invariant tubular neighborhood of $\Sigma$. Hence $H$ is a fiber bundle over $S^1$. A fiber $F$ in $H$ is diffeomorphic to a lens space. By the slice theorem for orbifolds (Proposition 2.3 of \cite{LeTo}), $H$ is $T^3$-equivariantly diffeomorphic to $T^{3} \times_{T^{3}_{\Sigma}} F$ where $T^{3}_{\Sigma}$ is the isotropy group of the $T^3$-action $\tau$ at $\Sigma$. By Lemma 3.13 of \cite{Ler2}, $T^{3}_{\Sigma}$ is connected. Hence $T^{3}_{\Sigma}$ is isomorphic to $T^2$. $\tau_{0}$ denotes the $T^3$-action on $S^1 \times F$ induced from the $T^3$-action on $S^1 \times T^{2} \times [0,1]$ defined by
\begin{equation}
(s_{1}, s_{2}, s_{3}) \cdot (\zeta,x,y,u)= (s_{1} \zeta, s_{2} x, s_{3} y, u).
\end{equation}
We will show that $T^{3} \times_{T^{3}_{\Sigma}} F$ is $T^3$-equivariantly diffeomorphic to $S^1 \times F$ with the $T^3$-action $\tau_{0}$. The $T^2$-equivariant diffeomorphism type of a lens space $L$ with a $T^2$-action is determined by the isotropy groups at two singular orbits by the slice theorem. In fact, $L$ is a union of two tubular neighborhoods $T^{2} \times_{T^{2}_{\Sigma_{1}}} D^{2}$ and $T^{2} \times_{T^{2}_{\Sigma_{2}}} D^{2}$ where $T^{2}_{\Sigma_{j}}$ is the isotropy group of the $T^2$-action on $L(p,q)$ at a singular orbit $\Sigma_{j}$ for $j=1$ and $2$. Since every $T^2$-equivariant diffeomorphism on $\partial(T^{2} \times_{T^{2}_{\Sigma_{2}}} D^{2})$ can be $T^2$-equivariantly extended to $T^{2} \times_{T^{2}_{\Sigma_{2}}} D^{2}$, the $T^2$-equivariant diffeomorphism type of $(T^{2} \times_{T^{2}_{\Sigma_{1}}} D^{2}) \cup (T^{2} \times_{T^{2}_{\Sigma_{2}}} D^{2})$ is determined by $T^{2}_{\Sigma_{1}}$ and $T^{2}_{\Sigma_{2}}$. The vectors in $\Lie(T^2)$ corresponding to the isotropy groups at two singular orbits in $F$ is written as $\begin{psmallmatrix} 1 \\ 0 \end{psmallmatrix}$ and $\begin{psmallmatrix} p \\ q \end{psmallmatrix}$ for a pair of coprime integers $(p,q)$ such that $p$ is positive with respect to an identification of $T^3_{\Sigma}$ with $T^2$. Then $F$ with the action of $T^{3}_{\Sigma}$ is $T^2$-equivariantly diffeomorphic to $L(p,q)$ with $T^2$-action induced from the $T^2$-action on $T^{2} \times [0,1]$ defined by
\begin{equation}
(s_{2}, s_{3}) \cdot (x,y,t)= (s_{2} x, s_{3} y, u).
\end{equation}
Then we have a $T^2$-equivariant diffeomorphism $F \longrightarrow L(p,q)$. $\{1\} \times F$ is a $T^2$-invariant transversal of a free $S^1$-subaction of $\tau$ in $H$. Similarly, $\{1\} \times L(p,q)$ is a $T^2$-invariant transversal of a free $S^1$-subaction of $\tau_{0}$. Hence $T^{3} \times_{T^{3}_{\Sigma}} F$ with $T^3$-action $\tau$ is $T^3$-equivariantly diffeomorphic to $S^1 \times L(p,q)$ with $T^3$-action $\tau_{0}$.

Hence we can assume that $H=S^1 \times L(p,q)$ and that $\tau$ is the $T^3$-action induced from the $T^3$-action on $S^1 \times T^{2} \times [0,1]$ defined by
\begin{equation}
(s_{1},s_{2},s_{3}) \cdot (\zeta,x,y,u)= (s_{1}\zeta, s_{2}x, s_{3}y, u).
\end{equation}
Then $\rho_{1}$ is written as 
\begin{equation}
t_{1} \cdot (\zeta, x, y, u)= (t_{1}^{a_{0}}\zeta, t_{1}^{a_{1}}x, t_{1}^{a_{2}}y, u)
\end{equation}
and $\rho_{2}$ is written as 
\begin{equation}
t_{2} \cdot (\zeta, x, y, u)= (t_{2}^{b_{0}}\zeta, t_{2}^{b_{1}} x, t_{2}^{b_{2}}y, u)
\end{equation}
for a positive integer $a_{0}$ and some integers $a_{1}$, $a_{2}$, $b_{0}$, $b_{1}$ and $b_{2}$.

We will show
\begin{equation}\label{ComputationOfa0}
-pa_{0} = \det \begin{pmatrix} n_{1} & n_{2} & \overline{Y} \end{pmatrix}. 
\end{equation}
Let $n_{B}$ be the vector in $\Lie(T^3)$ corresponding to the isotropy group $T^{3}_{B}$ of $\tau$ at $B$. By a theorem of Thornton \cite{Tho} (See Theorem \ref{LensSpaces} in this paper), if $B$ is diffeomorphic to $L(p',q')$, then $H$ is diffeomorphic to $S^1 \times L(p'',q')$ for some integer $p''$ which divides $p'$. Hence $B$ is diffeomorphic to $L(ep,q)$ for some positive integer $e$. We have $ep=-\det \begin{pmatrix} n_{1} & n_{2} & n_{B} \end{pmatrix}$ by Lemma \ref{GoodCone}. Hence \eqref{ComputationOfa0} is equivalent to 
\begin{equation}\label{ComputationOfa02}
\det \begin{pmatrix} n_{1} & n_{2} & a_{0}n_{B}+e\overline{Y} \end{pmatrix}=0. 
\end{equation}
$\{n_{1},n_{2}\}$ is a basis of the subspace of $\Lie(T^3)$ consisting vectors whose infinitesimal actions vanish on $\Sigma$. Note that the actions on $\Sigma$ induced from an $S^1$-action on $S^1 \times T^{2} \times [0,1]$ defined by
\begin{equation}
t \cdot (\zeta, x, y, u)= (t^{w_{0}} \zeta, t^{w_{1}} x, t^{w_{2}}y, u)
\end{equation}
is written as
\begin{equation}
t \cdot \zeta= t^{w_{0}} \zeta.
\end{equation}
We denote the action of the isotropy group of $\tau$ of $B$ by $\tau_{B}$. To show \eqref{ComputationOfa02}, it suffices to show that $\tau_{B}$ is induced from the $S^1$-action on $S^1 \times T^{2} \times [0,1]$ defined by
\begin{equation}
t \cdot (\zeta, x, y, u)= (t^{-e} \zeta, t^{c_{1}} x, t^{c_{2}}y, u)
\end{equation}
for some integers $c_{1}$ and $c_{2}$. We assume that $\tau_{B}$ is induced from the $S^1$-action on $S^1 \times T^{2} \times [0,1]$ defined by 
\begin{equation}
t \cdot (\zeta, x, y, u)= (t^{c_{0}} \zeta, t^{c_{1}} x, t^{c_{2}}y, u).
\end{equation}
To show $|c_{0}|=|e|$, we show that $B=H/\tau_{B}$ is diffeomorphic to $L(|c_{0}|p,q')$ for some $q'$. Since $\tau_{B}|_{\mathbb{Z}/|c_{0}|\mathbb{Z}}$ fixes the first $S^1$-component and acts on the $L(p,q)$-component freely, $H/(\tau_{B}|_{\mathbb{Z}/|c_{0}|\mathbb{Z}})$ is diffeomorphic to $S^1 \times L(|c_{0}|p,q')$ for some $q'$, and we have $H/\tau_{B}=(S^1 \times L(p,q))/\tau_{B}=\big((S^1 \times L(p,q))/(\tau_{B}|_{\mathbb{Z}/|c_{0}|\mathbb{Z}})\big)/\tau_{B}=\big(S^1 \times (L(p,q)/(\tau_{B}|_{\mathbb{Z}/|c_{0}|\mathbb{Z}})\big)/\tau_{B}$. Since $\tau_{B}|_{\mathbb{Z}/|c_{0}|\mathbb{Z}}$ acts freely on the first $S^1$-component of $S^1 \times (L(p,q)/(\tau_{B}|_{\mathbb{Z}/|c_{0}|\mathbb{Z}}))$, the injection $\iota \colon L(p,q)/(\tau_{B}|_{\mathbb{Z}/|c_{0}|\mathbb{Z}}) \longrightarrow \big(S^1 \times (L(p,q)/(\tau_{B}|_{\mathbb{Z}/|c_{0}|\mathbb{Z}})\big)/\tau_{B}$ defined by $\iota(z)=[(1,z)]$ is a diffeomorphism. Hence $H/\tau_{B}$ is diffeomorphic to $L(|c_{0}|p,q')$ for some $q'$. Hence we have $|c_{0}|=|e|$. We show that $c_{0}=-e$. The restriction of the infinitesimal action of $\rho_{1}$ to $\Sigma$ is a positive multiple of the Reeb vector field by the assumption. On the other hand, the infinitesimal action $Y_{B}$ of $\tau_{B}$ to $\Sigma$ is a negative multiple of the Reeb vector field. Then the signature of $a_{0}$ and $c_{0}$ are different. Since $a_{0}$ is positive, $c_{0}$ is negative. Since $e$ is positive, we have $c_{0}=-e$. Hence \eqref{ComputationOfa02} and \eqref{ComputationOfa0} are proved.

We prove Lemma \ref{EulerNumberOfLevelSetsNearB} in a way similar to the proof of Lemma \ref{LocalComputation}. Let $\tilde{\rho}_{1}$ and $\tilde{\rho}_{2}$ be $S^1$-actions on $S^1 \times L(p,q)$ induced from the $S^1$-actions on $S^1 \times T^{2} \times [0,1]$ defined by
\begin{equation}
t_1 \cdot (\xi, w_1, w_2, u)=(t_{1} \xi, w_1, w_2, u)
\end{equation}
and
\begin{equation}
t_2 \cdot (\xi, w_1, w_2, u)=(t_{2}^{b_{0}} \xi, t_{2}^{a_{0}b_{1}-a_{1}b_{0}} w_1, t_{2}^{a_{0}b_{2}-a_{2}b_{0}} w_2, u)
\end{equation}
respectively. Let $\tau$ be a $\mathbb{Z}/a_{0}\mathbb{Z}$-action on $S^1 \times L(p,q)$ induced from the $\mathbb{Z}/a_{0}\mathbb{Z}$-action on $S^1 \times T^{2} \times [0,1]$ defined by
\begin{equation}
s \cdot(\xi,w_1,w_2,u)=(s \xi, s^{-a_1} w_1, s^{-a_2} w_2,u)
\end{equation}
for $s$ in $\mathbb{Z}/a_{0}\mathbb{Z}$ where we identify $\mathbb{Z}/a_{0}\mathbb{Z}$ with a subgroup of the complex numbers with absolute values $1$. We identify $(S^1 \times L(p,q))/\tau$ with $S^1 \times L(p,q)$ by the map induced from $f \colon S^1 \times T^2 \times [0,1] \longrightarrow S^1 \times T^2 \times [0,1]$ defined by $f(\xi,w_1,w_2,u)=(\xi^{a_{0}},\xi^{a_{1}}w_{1},\xi^{a_{2}}w_{2},u)$. Then $\tilde{\rho}_{1}$ induces $\rho_{1}$ on $(S^1 \times L(p,q))/\tau$. $\tilde{\rho}_{2}$ induces an $S^1$-action $\rho_{2}^{a_{0}}$ on $(S^1 \times L(p,q))/\tau$ which is defined by
\begin{equation}
t_{2} \cdot (\zeta, x, y, \xi)= (t_{2}^{a_{0}b_{0}}\zeta, t_{2}^{a_{0}b_{1}} x, t_{2}^{a_{0}b_{2}}y, u).
\end{equation}
Hence the effective $S^1$-action on $(S^1 \times L(p,q))/\tau$ induced from $\tilde{\rho}_{2}$ is $\rho_{2}$.

The cardinality of the isotropy group of the $S^1$-action on $(S^1 \times L(p,q))/\tilde{\rho}_{1}$ induced from $\tilde{\rho}_{2}$ is $\GCD(|a_{0}b_{1}-a_{1}b_{0}|,|a_{0}b_{2}-a_{2}b_{0}|)$. Since the cardinality of the isotropy group of the $S^1$-action on $(S^1 \times L(p,q))/\rho_{1}$ induced from $\rho_{2}$ is $1$ by the effectiveness of $\rho$, the cardinality of the isotropy group of the $S^1$-action on $(S^1 \times L(p,q))/\rho_{1}$ induced from $\rho_{2}^{a_{0}}$ is $a_{0}$. Hence $S^1$-manifold $((S^1 \times L(p,q))/\rho_{1}, \rho_{2}^{a_{0}})$ is equivariantly diffeomorphic to the quotient of the $S^1$-manifold $((S^1 \times L(p,q))/\tilde{\rho}_{1},\tilde{\rho}_{2})$ by the $\mathbb{Z}/l\mathbb{Z}$-subaction of $\tilde{\rho}_{2}$ where $l=\frac{a_{0}}{\GCD(|a_{0}b_{1}-a_{1}b_{0}|,|a_{0}b_{2}-a_{2}b_{0}|)}$. By \eqref{DefinitionOfEulerNumber}, we have
\begin{equation}\label{tilderhoandrho}
e(\tilde{\rho}_{2},(S^1 \times L(p,q))/\tilde{\rho}_{1}) = \frac{1}{\frac{a_{0}}{\GCD(|a_{0}b_{1}-a_{1}b_{0}|, |a_{0}b_{2}-a_{2}b_{0}|)}} e(\rho_{2},(S^1 \times L(p,q))/\rho_{1}).
\end{equation}

Since $(S^1 \times L(p,q))/\tilde{\rho}_{1}=L(p,q)$ and the Euler number of the $S^1$-action on $L(p,q)$ induced from the $S^1$-action on $T^2 \times [0,1]$ defined by 
\begin{equation}
t \cdot (x,y)=(s^{m_1} x, s^{m_2} y)
\end{equation}
for a pair of integers $(m_{1},m_{2})$ is $-\frac{p\GCD(|m_{1}|,|m_{2}|)}{m_{1}(pm_{1}-qm_{2})}$ by Lemma \ref{EulerNumberOfSeifertFibrationOnLensSpaces}, we have
\begin{equation}\label{Covering2}
e(\tilde{\rho}_{2},(S^1 \times L(p,q))/\tilde{\rho}_{1})=-\frac{p\GCD(|a_{0}b_{1}-a_{1}b_{0}|,|a_{0}b_{2}-a_{2}b_{0}|)}{(a_{0}b_{1}-a_{1}b_{0})(p(a_{0}b_{1}-a_{1}b_{0})-q(a_{0}b_{2}-a_{2}b_{0}))}.
\end{equation}
By \eqref{tilderhoandrho} and \eqref{Covering2}, we have
\begin{equation}
e(\rho_{2},(S^1 \times L(p,q))/\rho_{1}) = -\frac{p a_{0}}{(a_{0} b_{1} - a_{1} b_{0})(p(a_{0}b_{1}-a_{1}b_{0})-q(a_{0}b_{2}-a_{2}b_{0}))}.
\end{equation}
Note that $k_{1}=|a_{0}b_{1}-a_{1}b_{0}|$ and $k_{2}=|p(a_{0}b_{1}-a_{1}b_{0})-q(a_{0}b_{2}-a_{2}b_{0})|$. Since $B$ is a minimal component of $\Phi$, the signatures of $a_{0}b_{1}-a_{1}b_{0}$ and $p(a_{0}b_{1}-a_{1}b_{0})-q(a_{0}b_{2}-a_{2}b_{0})$ are the same as in the case of \eqref{Denominator} and we have
\begin{equation}\label{k1k2}
k_{1} k_{2} = (a_{0} b_{1} - a_{1} b_{0})(p(a_{0}b_{1}-a_{1}b_{0})-q(a_{0}b_{2}-a_{2}b_{0})).
\end{equation}
By \eqref{ComputationOfa0} and \eqref{k1k2}, we have \eqref{ComputationOfExtendedEulerNumber}. The proof of \eqref{ComputationOfExtendedEulerNumber2} is similar.
\end{proof}

\begin{lem}\label{EulerNumberOfSeifertFibrationOnLensSpaces}
Let $\sigma$ be an $S^1$-subaction of a $T^2$-action on $L(p,q)$ induced from an $S^1$-action on $T^2 \times [0,1]$ defined by
\begin{equation}
t \cdot (x,y,u)= (t^{m_{1}} x,t^{m_{2}} y,u)
\end{equation}
for a pair of integers $(m_{1},m_{2})$. Then we have
\begin{equation}\label{EulerNumberOfsigma}
e(\sigma,L(p,q))=-\frac{p\GCD(|m_{1}|,|m_{2}|)}{m_{1}(pm_{1}-qm_{2})}. 
\end{equation}
\end{lem}

\begin{proof}
We define an $S^1$-action $\tilde{\sigma}$ on $S^3$ by 
\begin{equation}
t \cdot (z_1,z_2) = (t^{m_{1}} z_1,t^{pm_{1}-qm_{2}} z_2).
\end{equation}
Then $\tilde{\sigma}$ induces an $S^1$-action $p\sigma$ on $N$ which is induced from the $S^1$-action on $T^2 \times [0,1]$ which is the product of the trivial action on $[0,1]$ and the $S^1$-action $\sigma^{p}$ on the $T^2$-component defined by
\begin{equation}
t \cdot (x,y)= (t^{pm_{1}} x,t^{pm_{2}} y).
\end{equation}
Hence the effective $S^1$-action on $L(p,q)$ induced from $\tilde{\sigma}$ is $\sigma$.

The cardinality of the isotropy group of the $S^1$-action $\tilde{\sigma}$ on $S^3$ is $\GCD(|m_{1}|,|pm_{1}-qm_{2}|)$. The cardinality of the isotropy group of the action $\sigma^{p}$ on $L(p,q)$ is $p\GCD(|m_{1}|,|m_{2}|)$. Hence $S^1$-manifold $(L(p,q),\sigma^{p})$ is equivariantly diffeomorphic to the quotient of the $S^1$-manifold $(S^3,\tilde{\sigma})$ by the $\mathbb{Z}/l\mathbb{Z}$-subaction of $\tilde{\sigma}$ where $l=\frac{p\GCD(|m_{1}|,|m_{2}|)}{\GCD(|m_{1}|,|pm_{1}-qm_{2}|)}$. By \eqref{DefinitionOfEulerNumber}, we have
\begin{equation}\label{Covering3}
e(\tilde{\sigma},S^3) = \frac{1}{\frac{p\GCD(|m_{1}|,|m_{2}|)}{\GCD(|m_{1}|,|pm_{1}-qm_{2}|)}} e(\sigma,L(p,q)).
\end{equation}
Since $(S^1 \times S^3)/\tilde{\sigma}_{1}=S^3$ and the Euler number of the $S^1$-action on $S^3$ defined by 
\begin{equation}
t \cdot (z_1,z_2)=(s^{m_1} z_1, s^{m_2} z_2)
\end{equation}
is $-\frac{\GCD(|m_{1}|, |m_{2}|)}{m_1 m_2}$ by Lemma \ref{EulerNumberOfSeifertFibrationOnS30}, we have 
\begin{equation}\label{EulerNumberOfSeifertFibrationOnS31}
e(\tilde{\sigma}_{0},S^3)=-\frac{\GCD(|m_{1}|,|pm_{1}-qm_{2}|)}{m_{1}(pm_{1}-qm_{2})}.
\end{equation}
Then \eqref{EulerNumberOfsigma} follows from \eqref{Covering3} and \eqref{EulerNumberOfSeifertFibrationOnS31}.
\end{proof}

\begin{lem}\label{Formula}
We use the notation in the proof of Proposition \ref{NearChains}. Under the condition (ii) of Proposition \ref{SufficientConditionToBeToric}, we have
\begin{equation}\label{FirstFormula}
\begin{array}{l}
\frac{1}{k^{2}_{l(2)}k^{1}_{l(1)}} \det \begin{pmatrix} n^{2}_{l(2)} & \hat{n}^{1}_{l(1)} & \overline{Y} \end{pmatrix} + \frac{1}{k^{1}_{1}k^{2}_{1}} \det \begin{pmatrix} n^{1}_{1} & n^{2}_{1} & \overline{Y} \end{pmatrix} \\
\hspace{150pt} - \sum_{j,i} \frac{(-1)^{j-1}}{k^{j}_{i+1}k^{j}_{i}} \det \begin{pmatrix} n^{j}_{i+1} & n^{j}_{i} & \overline{Y} \end{pmatrix} = 0 
\end{array}
\end{equation}
for every $\overline{Y}$ in $\Lie(G)$ and
\begin{equation}\label{SecondFormula}
\begin{array}{l}
\frac{1}{k^{2}_{l(2)}k^{1}_{l(1)}} \det \begin{pmatrix} n^{2}_{l(2)} & \hat{n}^{1}_{l(1)} & \overline{X}_{0} \end{pmatrix} + \frac{1}{k^{1}_{1}k^{2}_{1}} \det \begin{pmatrix} n^{1}_{1} & n^{2}_{1} & \overline{X}_{0} \end{pmatrix} \\
\hspace{150pt} - \sum_{j,i}  \frac{(-1)^{j-1}x^{j}_{i i+1}}{k^{j}_{i+1}k^{j}_{i}} \det \begin{pmatrix} n^{j}_{i+1} & n^{j}_{i} & \overline{R} \end{pmatrix} = 0.
\end{array}
\end{equation}
\end{lem}

\begin{proof}
If $B_{\min}$ is diffeomorphic to a lens space, then we take $(U,\beta)$ which satisfies the conditions (i) and (ii) of Lemma \ref{AttachingCorners} for $B=B_{\min}$. we use the notation of Lemma \ref{AttachingCorners} substituting $B_{\min}$ to $B$. Let $Q^{j}$ be the subspace of $\Lie(T^{3})^{*}$ defined by $n_{j}$ for $j=1$ and $2$. We put $\{v_{0}\}=Q^{1} \cap Q^{2} \cap A$ and $\Sigma_{\min}=\Psi^{-1}(v_{0})$. We define $\Sigma_{\max}$ in a way similarly to $\Sigma_{\min}$ if $B_{\min}$ is diffeomorphic to a lens space. We put $\Omega=\{ \overline{Y} \in \Lie(G) |$ the restriction of the infinitesimal action of $\overline{Y}$ to $\Sigma_{\min}$, $\Sigma_{\max}$ or each closed orbit of the Reeb flow in $M$ is a positive multiple of $R. \}$. To show \eqref{FirstFormula} for every $\overline{Y}$ in $\Lie(G)$, it suffices to show \eqref{FirstFormula} for an element $\overline{Y}$ of $\Omega$, since the both sides of \eqref{FirstFormula} is linear and $\Omega$ contains a basis of $\Lie(G)$. We have
\begin{equation}\label{I1}
I(\sigma,\Sigma^{1}_{i})=\det \begin{psmallmatrix} n^{1}_{i+1} & n^{1}_{i} & \overline{Y} \end{psmallmatrix}
\end{equation}
and 
\begin{equation}\label{I2}
I(\sigma,\Sigma^{2}_{i})=-\det \begin{psmallmatrix} n^{2}_{i+1} & n^{2}_{i} & \overline{Y} \end{psmallmatrix}
\end{equation}
for $\overline{Y}$ in $\Omega$. In fact, $I(\sigma,\Sigma^{j}_{i})$ is equal to the number of intersection points of the isotropy group of $\Sigma^{j}_{i}$ and the $S^1$-subgroup corresponding to $\sigma$ in $T^3$. The right hand sides of \eqref{I1} and \eqref{I2} are equal to the number of intersection points of the isotropy group of $\Sigma^{j}_{i}$ and the $S^1$-subgroup corresponding to $\sigma$ in $T^3$, since $\{n^{j}_{i},n^{j}_{i+1}\}$ is a $\mathbb{Z}$-basis of the Lie algebra of the isotropy group of $\Sigma^{j}_{i}$. Note that the condition $\overline{Y}$ is an element of $\Omega$ determines the sign of the determinants in the right hand sides of \eqref{I1} and \eqref{I2}. 

If both of $B_{\min}$ and $B_{\max}$ are of dimension $1$, \eqref{FirstFormula}  for $\overline{Y}$ in $\Omega$ follows from equations \eqref{I1}, \eqref{I2} and Lemmas \ref{LevelSetNearB}, \ref{TotalDifference}. If $B_{\min}$ or $B_{\max}$ is diffeomorphic to a lens space, then \eqref{FirstFormula} for $\overline{Y}$ in $\Omega$ follows from equations \eqref{I1}, \eqref{I2} and Lemmas \ref{LevelSetNearB}, \ref{TotalDifference}, \ref{EulerNumberOfLevelSetsNearB}.

The equation \eqref{SecondFormula} follows from \eqref{FirstFormula} and Lemma \ref{VanishingDeterminant}. In fact, applying Lemma \ref{VanishingDeterminant} to each term of the left hand side of \eqref{SecondFormula}, we obtain
\begin{equation}\label{FourthFormula}
\begin{array}{l}
\frac{1}{k^{2}_{l(2)}k^{1}_{l(1)}} \det \begin{pmatrix} n^{2}_{l(2)} & \hat{n}^{1}_{l(1)} & \overline{X}_{0} \end{pmatrix} + \frac{1}{k^{1}_{1}k^{2}_{1}} \det \begin{pmatrix} n^{1}_{1} & n^{2}_{1} & \overline{X}_{0} \end{pmatrix} - \sum_{i,j}  \frac{(-1)^{j-1}x^{j}_{i i+1}}{k^{j}_{i+1}k^{j}_{i}} \det \begin{pmatrix} n^{j}_{i} & n^{j}_{i+1} & \overline{R} \end{pmatrix} \\
= \frac{1}{k^{2}_{l(2)}k^{1}_{l(1)}} \det \begin{pmatrix} n^{2}_{l(2)} & \hat{n}^{1}_{l(1)} & \overline{X}_{0} \end{pmatrix} + \frac{1}{k^{1}_{1}k^{2}_{1}} \det \begin{pmatrix} n^{1}_{1} & n^{2}_{1} & \overline{X}_{0} \end{pmatrix} - \sum_{i,j} \frac{(-1)^{j-1}}{k^{j}_{i}k^{j}_{i+1}} \det \begin{pmatrix} n^{j}_{i+1} & n^{j}_{i} & \overline{X}_{0} \end{pmatrix}.
\end{array}
\end{equation}
Then the right hand side of \eqref{FourthFormula} is zero by substituting $\overline{Y}=\overline{X}_{0}$ to \eqref{FirstFormula}.
\end{proof}

$\sigma_{0}$ denotes the $S^1$-action on the annulus $[0,1] \times S^1$ defined by the principal action on the $S^1$-component. We call $\sigma_{0}$ the standard $S^1$-action on $[0,1] \times S^1$.

\begin{lem}\label{ModificationOfTorusActions}
Assume that the conditions (i) and (ii) in Proposition \ref{SufficientConditionToBeToric} are satisfied. Let $C^1$ and $C^2$ be two chains in $(M,\alpha)$. Assume that there is no nontrivial chain except $C^1$ and $C^2$. Let $V$ be an open neighborhood of $C^1 \cup C^2 \cup B_{\min} \cup B_{\max}$ and an $\alpha$-preserving $T^3$-action $\tau$ on $V$. Assume that the image of the contact moment map $\tilde{\Phi}$ of $(V,\alpha|_{V})$ for $\tau$ is an open neighborhood of the boundary of a convex polygon in a $2$-dimensional affine subspace of $\Lie(T^3)^{*}$. Then there exist a $\tau$-invariant open neighborhood $W$ of $C^1 \cup C^2 \cup B_{\min} \cup B_{\max}$ in $V$ and a diffeomorphism $f \colon M - W \longrightarrow [0,1] \times [0,t_{0}] \times T^3$ for some positive number $t_{0}$ which satisfies the following:
\begin{enumerate}
\item the $T^2$-action $\rho$ associated with $\alpha$ is conjugated by $f$ to the product of the trivial action on $[0,1] \times [0,t_{0}] \times S^1$ and the principal action on $T^2$,
\item the $T^3$-action $\tau$ on $V-W$ is conjugated by $f$ to a $T^3$-action on an open neighborhood $U$ of the boundary of $[0,1] \times [0,t_{0}] \times T^3$ which is the product of the trivial action on $[0,1] \times [0,t_{0}]$ and the standard action on $T^3$ and
\item the level sets of $\Phi$ is mapped to the level sets of the projection $[0,1] \times [0,t_{0}] \times T^3 \longrightarrow [0,t_{0}]$.
\end{enumerate}
\end{lem}

\begin{proof}
We fix a Riemannian metric $g$ on $M$ so that $g$ is invariant under $\rho$ on $M$ and invariant under $\tau$ on an open neighborhood $V'$ of $C^1 \cup C^2 \cup B_{\min} \cup B_{\max}$ in $V$. We take a value $x_{0}$ of $\Phi$ sufficiently close to $\Phi(B_{\min})$. We put 
\begin{equation}
H= \Phi^{-1}(x_{0}) - S
\end{equation}
for a sufficiently small open tubular neighborhood $S$ of $(C^{1} \cup C^{2}) \cap \Phi^{-1}(x_{0})$ in $\Phi^{-1}(x_{0})$. Then $\overline{M - V'}$ is contained in $\cup_{0 \leq t \leq t'_{0}} \psi'_{t}(H)$ for some $t'_{0}$ where $\{\psi'_{t}\}$ is the gradient flow of $\Phi$ with respect to $g$. We put $W = \cup_{0 \leq t \leq t'_{0}} \psi_{t}(H)$. We will construct $f$ by modifying the gradient flow of $\Phi$. 

We will modify $g$ so that the norm of the gradient vector field of $\Phi$ is $1$ on $W$. We put a Riemannian metric $g_{1}$ on $M$ so that
\begin{equation}
g_{1}= \begin{cases}
 g & \text{on}\,\, (\mathbb{R}\grad_{g} \Phi)^{\perp} \otimes (\mathbb{R}\grad_{g} \Phi)^{\perp} \\
g(\grad_{g} \Phi, \grad_{g} \Phi) g & \text{on}\,\, (\mathbb{R}\grad_{g} \Phi) \otimes (\mathbb{R}\grad_{g} \Phi)
\end{cases}
\end{equation}
is satisfied on $W$. Note that $g_{1}$ is a metric, since $\grad_{g} \Phi$ is nowhere vanishing on $W$. We show $g_{1}(\grad_{g_{1}}\Phi,\grad_{g_{1}}\Phi)=1$. We can put $\grad_{g_{1}}\Phi=k\grad_{g}\Phi$ for a smooth nonnegative function $k$ on $M$. We have $g_{1}(\grad_{g_{1}}\Phi,\grad_{g_{1}}\Phi)=kg_{1}(\grad_{g_{1}}\Phi,\grad_{g}\Phi)=kd\Phi(\grad_{g}\Phi)=kg(\grad_{g}\Phi,\grad_{g}\Phi)$ by the definition of the gradient vector fields. Hence we have
\begin{equation}\label{Ratio1}
k=\frac{g_{1}(\grad_{g_{1}}\Phi,\grad_{g_{1}}\Phi)}{g(\grad_{g}\Phi,\grad_{g}\Phi)}.
\end{equation}
On the other hand, we have 
\begin{equation}
\begin{array}{rl}
g_{1}(\grad_{g_{1}}\Phi,\grad_{g_{1}}\Phi) & = g(\grad_{g} \Phi, \grad_{g} \Phi) g(\grad_{g_{1}}\Phi,\grad_{g_{1}}\Phi) \\
 & = g(\grad_{g} \Phi, \grad_{g} \Phi) k^{2} g(\grad_{g}\Phi,\grad_{g}\Phi)
\end{array}
\end{equation}
by the definition of $g_{1}$. Hence we have
\begin{equation}\label{Ratio2}
k=\frac{\sqrt{g_{1}(\grad_{g_{1}}\Phi,\grad_{g_{1}}\Phi)}}{g(\grad_{g}\Phi,\grad_{g}\Phi)}.
\end{equation} 
Then we have $g_{1}(\grad_{g_{1}}\Phi,\grad_{g_{1}}\Phi)=1$ by \eqref{Ratio1} and \eqref{Ratio2}.

The gradient flow of $\Phi$ with respect to $g_{1}$ maps the intersection of $W$ and level sets of $\Phi$ to the intersection of $W$ and level sets of $\Phi$, since $(\grad_{g_{1}}\Phi) \Phi=d\Phi(\grad_{g_{1}}\Phi)=g_{1}(\grad_{g_{1}}\Phi,\grad_{g_{1}}\Phi)=1$.

Since the contact moment maps are submersions on the union of the free orbits of the toric actions by Lemma \ref{MomentMapsAreSubmersive}, $\tilde{\Phi}|_{V - (C^1 \cup C^2 \cup B_{\min} \cup B_{\max})}$ is a submersion whose fibers are the orbits of $\tau$. Hence we have a diffeomorphism $\phi \colon H \longrightarrow [0,1] \times T^3$ which is $T^3$-equivariant with respect to $\tau$ and the principal $T^3$-action on the second component of $[0,1] \times T^3$. We identify $H$ with $[0,1] \times T^3$ by $\phi$. Then we have a diffeomorphism $\chi$ defined by
\begin{equation}
\begin{array}{cccc}
\chi \colon & \cup_{0 \leq t \leq t_{0}} \psi_{t}(H) & \longrightarrow & [0,1] \times [0,t_{0}] \times T^3 \\
 & \psi_{t}(s,\zeta_{1},\zeta_{2},\zeta_{3}) & \longmapsto & (t,s,\zeta_{1},\zeta_{2},\zeta_{3})
\end{array}
\end{equation}
where $(s,\zeta_{1},\zeta_{2},\zeta_{3})$ is the standard coordinate on $[0,1] \times T^3$, $\{\psi_{t}\}$ is the gradient flow of $\Phi$ with respect to $g_{1}$ and $t_{0}$ is taken so that $\overline{M - V'}$ is contained in $\cup_{0 \leq t \leq t_{0}} \psi_{t}(H)$. Since $g$ is $\rho$-invariant, $\rho$ commutes with the gradient flow $\{\psi_{t}\}$. Hence $\chi$ is $T^2$-equivariant with respect to $\rho$ and the principal $T^2$-action on the $T^2$-component of $[0,1] \times [0,t_{0}] \times S^1 \times T^2$. Since $g$ is $\tau$-invariant on an open neighborhood of $C^1 \cup C^2$, $\tau$ commutes with the gradient flow $\{\psi_{t}\}$. Hence $\chi$ is $T^3$-equivariant on $\chi^{-1}(\{(t,s,\zeta_{1},\zeta_{2},\zeta_{3}) \in [0,1] \times [0,t_{0}] \times T^3 | 0 \leq s \leq r_{0}, 1-r_{0} \leq s \leq 1\})$ for some $r_{0}$ with respect to $\tau$ and the principal $T^3$-action on the $T^3$-component of $[0,1] \times [0,t_{0}] \times T^3$. Hence the $T^2$-action $\chi \circ \rho \circ \chi^{-1}$ on $[0,1] \times \{0,t_{0}\} \times T^3$ is the principal $T^2$-action to the $T^2$-component of $[0,1] \times [0,t_{0}] \times T^3=[0,1] \times [0,t_{0}] \times S^1 \times T^2$. The restriction of the $T^3$-action $\chi \circ \tau \circ \chi^{-1}$ on an open neighborhood of $[0,1] \times \{0,t_{0}\} \times T^3$ is the principal $T^3$-action on the $T^3$-component. 

We put $D=([0,1] \times [0,t_{0}] \times T^3)/\rho$. $\tau$ induces an $S^1$-action $\overline{\tau}$ on an open neighborhood of the boundary of $D$. We show that $\overline{\tau}$ extends to $D$. By the construction of $g_{1}$, the level sets of $\Phi$ is mapped to the level sets of the second projection $[0,1] \times [0,t_{0}] \times T^3 \longrightarrow [0,t_{0}]$. Since the level sets of $\Phi$ is preserved by $\tau$, $\overline{\tau}$ preserves the level sets of the map $p|_{V/\rho} \colon V/\rho \longrightarrow [0,t_{0}]$ where $p$ is the map induced from the second projection $[0,1] \times [0,t_{0}] \times T^3 \longrightarrow [0,t_{0}]$. The level sets of $p \colon D \longrightarrow [0,t_{0}]$ are annuli. The coordinate on $D$ induced from the coordinate of $[0,1] \times [0,t_{0}] \times T^3$ gives a diffeomorphism from each level set of $p$ to $[0,1] \times S^1$. $\overline{\tau}$ is the standard rotation with respect to the coordinate induced from the coordinate of $[0,1] \times [0,t_{0}] \times T^3$ on each level set of $p$ near the boundary. Choose a small positive number $\epsilon$ so that $\overline{\tau}$ is defined on two level sets $p^{-1}(\epsilon)$ and $p^{-1}( t_{0} - \epsilon)$ of $p$. Since these $S^1$-actions on $p^{-1}(\epsilon)$ and $p^{-1}( t_{0} - \epsilon)$ are induced from the principal $T^3$-action on the principal $T^3$-bundle on the unit interval, they are isomorphic to the standard $S^1$-action $\sigma_{0}$ on the annulus $[0,1] \times S^1$. We fix diffeomorphisms
\begin{equation}
f_{t} \colon p^{-1}(t) \longrightarrow [0,1] \times S^1
\end{equation}
which are $S^1$-equivariant with respect to $\overline{\tau}|_{p^{-1}(t)}$ and $\sigma_{0}$ for $t=\epsilon$ and $t_{0} -\epsilon$. Let $\delta$ be a positive number which satisfies $\delta + \epsilon < \frac{1}{2}$. By Lemma B.2 of \cite{Kar}, there exists an isotopy $\{f^{t}\}_{\epsilon \leq t \leq \epsilon + \delta}$ connecting $f_{\epsilon}$ to $f_{\epsilon + \delta} = \id_{[0,1] \times S^1}$ in the diffeomorphism group of the annulus defined by a smooth map $F_{1} \colon p^{-1}([\epsilon,\epsilon+\delta]) \longrightarrow [0,1] \times S^1$ such that $f_{t}$ is a rotation on the $S^1$-component of each level set $p^{-1}(t)$ diffeomorphic to an annulus near the boundary for each $\epsilon \leq t \leq \epsilon + \delta$. Similarly by Lemma B.2 of \cite{Kar}, there exists an isotopy connecting $f_{t_{0} - \epsilon - \delta} = \id_{[0,1] \times S^1}$ to $f_{t_{0} - \epsilon}$ in the diffeomorphism group of the annulus defined by a smooth map $F_{2} \colon p^{-1}([t_{0}-\epsilon-\delta,t_{0}-\epsilon]) \longrightarrow [0,1] \times S^1$ such that $f_{t}$ is a rotation on the $S^1$-component of each level set $p^{-1}(t)$ diffeomorphic to an annulus near the boundary for each $t_{0}-\epsilon-\delta \leq t \leq t_{0}-\epsilon$. Then we can extend $\overline{\tau}$ to an $S^1$-action $\overline{\tau}_{1}$ on $D$ by defining
\begin{equation}
\overline{\tau}_{1}|_{p^{-1}(t)} = 
\begin{cases}
\overline{\tau}|_{p^{-1}(t)} & t_{0} - \epsilon < t \leq t_{0} \\
F_{2}^{-1}|_{p^{-1}(t)} \circ \sigma_{0} \circ F_{2}|_{p^{-1}(t)} & t_{0} - \epsilon - \delta < t \leq t_{0} - \epsilon \\
\sigma_{0} &  \epsilon + \delta < t \leq t_{0} - \epsilon - \delta \\
F_{1}^{-1}|_{p^{-1}(t)} \circ \sigma_{0} \circ F_{1}|_{p^{-1}(t)} & \epsilon < t \leq \epsilon + \delta \\
\overline{\tau}|_{p^{-1}(t)} & 0 \leq t \leq \epsilon.
\end{cases}
\end{equation}
Note that $\overline{\tau}_{1}$ coincides with $\overline{\tau}$ on an open neighborhood of the boundary of $D$, since every element in the isotopy $\{f_{t}\}_{\epsilon \leq t \leq \epsilon + \delta}$ and $\{f_{t}\}_{t_{0} - \epsilon - \delta \leq t \leq t_{0} - \epsilon}$ is a rotation on the $S^1$-component of each level set $p^{-1}(t)$ diffeomorphic to an annulus near the boundary.

We will show that there exists a $T^3$-action $\tau_{1}$ on $M$ which satisfies the following conditions:
\begin{description}
\item[(a)] $\tau_{1}$ commutes with $\rho$, 
\item[(b)] $\tau_{1}$ is an extension of $\tau$ and 
\item[(c)] $\Phi$ is constant on each orbit of $\tau_{1}$.
\end{description}
To show the existence of $\tau_{1}$, it suffices to show that the $T^3$-action $\chi \circ \tau \circ \chi^{-1}$ defined on an open neighborhood of the boundary of $[0,1] \times [0,t_{0}] \times T^3$ extends to $[0,1] \times [0,t_{0}] \times T^3$ so that the extension commutes with $\chi \circ \rho \circ \chi^{-1}$ and induces $\overline{\tau}_{1}$ on $D$. Let $\overline{Z}$ be a vector field on $D$ generating $\overline{\tau}_{1}$. Let $\pi$ be the projection $[0,1] \times [0,t_{0}] \times T^3 \longrightarrow ([0,1] \times [0,t_{0}] \times T^3)/ \chi \circ \rho \circ \chi^{-1}$. $\pi$ is the trivial $T^2$-bundle over $D$. Let $\sigma$ be the $S^1$-action on $[0,1] \times [0,t_{0}] \times T^3$ defined by the principal $S^1$-action on the first $S^1$-component of  $[0,1] \times [0,t_{0}] \times T^3=[0,1] \times [0,t_{0}] \times S^1 \times  S^1 \times S^1$. Then $\sigma$ induces $\overline{\tau}_{1}$ on an open neighborhood of the boundary of $D$. We take a vector field $Z$ on $[0,1] \times [0,t_{0}] \times T^3$ so that the following conditions are satisfied:
\begin{enumerate}
\item $Z$ is invariant by $\chi \circ \rho \circ \chi^{-1}$,
\item $\pi_{*}Z = \overline{Z}$ and
\item $Z$ coincides with an infinitesimal action of $\sigma$ on an open neighborhood of the boundary of $[0,1] \times [0,t_{0}] \times T^3$.
\end{enumerate}
For each $S^1$-orbit $\gamma$ of $\overline{\tau}_{1}$ in $D$, $Z$ defines a flat principal $T^2$-connection of $\pi|_{\pi^{-1}(\gamma)}$. Integrating $Z$, we have a holonomy of the flat bundle $\pi|_{\pi^{-1}(\gamma)}$ along $\gamma$ which is a left multiplication of an element $m(\gamma)$ of $T^2$. Since we have a smooth map defined by
\begin{equation}
\begin{array}{cccc}
m \colon & D/\overline{\tau}_{1} & \longrightarrow & T^{2} \\
         &     \gamma     & \longmapsto     & m(\gamma).
\end{array}
\end{equation}
If $\gamma$ is sufficiently close to the boundary of $D/\overline{\tau}_{1}$, we have $m(\gamma)=e_{T^2}$ where $e_{T^2}$ is the unit element of $T^2$, since $\tilde{Z}$ generates an $S^1$-action $\sigma$ on $\pi^{-1}(\gamma)$. Note that $\tilde{Z}$ generates an $S^1$-action on $\pi^{-1}(\gamma)$ which induces $\overline{\tau}_{1}$ on $D$ if and only if $m(\gamma)=e_{T^2}$. $D$ is a trivial $S^1$-bundle over $[0,1] \times [0,t_{0}]$. We fix an $S^1$-equivariant diffeomorphism $h \colon D \longrightarrow [0,1] \times [0,t_{0}] \times S^1$. $h$ induces a diffeomorphism $\overline{h} \colon D/\overline{\tau}_{1} \longrightarrow [0,1] \times [0,t_{0}]$. Since $\pi_{2}(T^2)=0$, we have a smooth map
\begin{equation}
q \colon ([0,1] \times [0,t_{0}]) \times [0,1] \longrightarrow T^2
\end{equation}
such that $q|_{([0,1] \times [0,t_{0}]) \times \{0\}}=m \circ \overline{h}^{-1}$ and the restriction of $q$ on an open neighborhood of the boundary of $([0,1] \times [0,t_{0}]) \times [0,1]$ is the constant map to $e_{T^2}$. We put $q(u,v,\theta)=(q_{2}(u,v,\theta),q_{3}(u,v,\theta))$. We take smooth functions $q'_{j}$ on $([0,1] \times [0,t_{0}]) \times [0,1]$ so that
\begin{equation}
\int_{0}^{\theta}q'_{j}(u,v,\eta)d\eta=q_{j}(u,v,\theta) 
\end{equation}
for $j=2$ and $3$. Then $q'_{j}$ induces a smooth map
\begin{equation}
q_{j} \colon ([0,1] \times [0,t_{0}]) \times S^1 \longrightarrow T^2
\end{equation}
on $([0,1] \times [0,t_{0}]) \times S^1 = ([0,1] \times [0,t_{0}]) \times [0,1] / ((u,v) \times \{0\} \sim (u,v) \times \{1\})$ for $j=2$ and $3$, since $q$ is constant near the boundary of $D$. We define a vector field on $[0,1] \times [0,t_{0}] \times T^3$ by
\begin{equation}
Z_{1 (s,t,\zeta_{1},\zeta_{2},\zeta_{3})} = Z - q'_{2} \circ h^{-1}(s,t,\zeta_{1})\frac{\partial}{\partial \zeta_{2}} - q'_{3} \circ h^{-1}(s,t,\zeta_{1})\frac{\partial}{\partial \zeta_{3}},
\end{equation}
where $(s,t,\zeta_{1},\zeta_{2},\zeta_{3})$ is the standard coordinate of $[0,1] \times [0,t_{0}] \times T^3$. Since $q'_{j}=0$ on an open neighborhood of the boundary of $[0,1] \times [0,t_{0}] \times T^3$, $Z_{1}=Z$ there. For each orbit $\gamma$ of $\overline{\tau}$, let $m_{1}(\gamma)$ be the element of $T^2$ whose left multiplication map is the holonomy map of the flat connection given by $Z_{1}$ on $\pi|_{\pi^{-1}(\gamma)}$. Since $m_{1}(\gamma) = m(\gamma) - q \circ h^{-1}(1)=0$ by the construction of $q$, we have $m_{1}(\gamma)=e_{T^2}$. Hence $Z_{1}$ generates an $S^1$-action which commutes with $\chi \circ \rho \circ \chi^{-1}$ and is an extension of $\sigma$.

Let $\tau_{1}$ be the $T^3$-action on $M$ which satisfies the above conditions (a), (b) and (c) in the third paragraph of the proof. We take a metric $g'$ on $M$ which is obtained by averaging $g_{1}$ by $\tau_{1}$. Then $g'(\grad_{g'}\Phi,\grad_{g'}\Phi)=1$ is satisfied. Let $\{\psi'_{t}\}$ be the gradient flow of $\Phi$ with respect to $g'$. Then we have a diffeomorphism $\chi'$ defined by
\begin{equation}
\begin{array}{cccc}
\chi' \colon & \cup_{0 \leq t \leq t_{0}} \psi_{t}(H) & \longrightarrow & [0,1] \times [0,t_{0}] \times T^3 \\
 & \psi'_{t}(s,\zeta_{1},\zeta_{2},\zeta_{3}) & \longmapsto & (t,s,\zeta_{1},\zeta_{2},\zeta_{3}).
\end{array}
\end{equation}
Since $g'$ is invariant under $\tau_{1}$, the desired conditions for $f$ is satisfied by $\chi'$.
\end{proof}

\begin{lem}\label{FarFromChains}
We denote the principal $T^3$-action on the $T^3$-component of $[0,1] \times [0,t_{0}] \times T^3$ by $\tau$. Let $\alpha$ be a $K$-contact form of rank $2$ on $[0,1] \times [0,t_{0}] \times T^3$ which satisfies the following conditions:
\begin{enumerate}
\item the $T^2$-action $\rho$ associated with $\alpha$ is the product of the trivial action on $[0,1] \times [0,t_{0}] \times S^1$ and the principal action on $T^2$, 
\item there exists a $\tau$-invariant open neighborhood $U$ of the boundary of $[0,1] \times [0,t_{0}] \times T^3$ and $\tau$ preserves $\alpha|_{U}$,
\item $[0,1] \times \{t\} \times T^3$ is a level set of the contact moment map for $\rho$ for each $t$ in $[0,t_{0}]$ and
\item the image of $U$ by the contact moment map for $\tau$ is an open neighborhood of the boundary in a convex subset in a $2$-dimensional affine space in $\Lie(T^3)^{*}$.
\end{enumerate}
Then we have an $\alpha$-preserving $T^3$-action $\tilde{\tau}$ on $[0,1] \times [0,t_{0}] \times T^3$ which coincides with $\tau$ on an open neighborhood of the boundary.
\end{lem}

\begin{proof}
We put $A=\{ v \in \Lie(T^3)^{*} | v(\overline{R})=1\}$. We fix a $1$-dimensional affine subspace $S$ of $A$ so that $A=S \oplus \ker (\pi|_{A})$ where $\pi$ is the restriction map $\Lie(T^3)^{*} \longrightarrow \Lie(G)^{*}$. Let $\tilde{\Phi} \colon U \longrightarrow S \oplus \ker (\pi|_{A})$ be the contact moment map for $\tau$. Let $\Phi \colon [0,1] \times [0,t_{0}] \times T^3 \longrightarrow S$ be the contact moment map for $\rho$. We have $\tilde{\Phi}(x)=(\Phi(x),\Psi(x))$ for some map $\Psi \colon [0,1] \times [0,t_{0}] \times T^3 \longrightarrow \ker (\pi|_{A})$ and every $x$ in $U$. By the conditions (i) and (ii), $\Psi$ and $\Phi$ can be written as $\Phi(x,y)=\overline{\Phi}(x)$ and $\Psi(x,y)=\overline{\Psi}(x)$ for $x$ in $[0,1] \times [0,t_{0}]$ and $y$ in $T^3$. Since $\tilde{\Phi}$ is a submersion by Lemma \ref{MomentMapsAreSubmersive}, $d\Phi \wedge d\Psi$ is nowhere vanishing on $U$. By the conditions (iii), (iv) and Lemma \ref{ExtensionOfPsi} below, we can extend $\overline{\Psi}$ to $[0,1] \times [0,t_{0}] \times T^3$ so that $d\overline{\Phi} \wedge d\overline{\Psi}$ is nowhere vanishing on $[0,1] \times [0,t_{0}]$. We extend $\tilde{\Phi}$ by $\tilde{\Phi}(x,y)=(\overline{\Phi}(x),\overline{\Psi}(x))$ for $x$ in $[0,1] \times [0,t_{0}]$ and $y$ in $T^3$. 

We take a $\tau$-invariant $1$-form $\beta$ on $[0,1] \times [0,t_{0}] \times T^3$ which satisfies $\beta|_{U'}=\alpha|_{U'}$ for an open neighborhood $U'$ of the boundary of $[0,1] \times [0,t_{0}] \times T^3$ and $\beta(Y)=\tilde{\Phi}(\overline{Y})$ for every element $\overline{Y}$ of $\Lie(T^3)$, where $Y$ is the infinitesimal action of $\overline{Y}$, in the following way: Let $F$ be the vector bundle on $[0,1] \times [0,t_{0}] \times T^3$ defined by the kernel of the differential map of the projection $[0,1] \times [0,t_{0}] \times T^3 \longrightarrow [0,1] \times [0,t_{0}]$. The equation $\beta(Y)=\tilde{\Phi}(\overline{Y})$ for every element $\overline{Y}$ of $\Lie(T^3)$ determines a $\tau$-invariant element $\beta_{0}$ of $C^{\infty}(F^{*})$. We take an inverse image $\beta_{1}$ of $\beta_{0}$ by the restriction map $C^{\infty}(T^{*}([0,1] \times [0,t_{0}] \times T^3)) \longrightarrow C^{\infty}(F^{*})$. By averaging $\beta_{1}$ by $\tau$, we obtain $\beta$ which satisfies the conditions.

We show that $\beta$ is a $K$-contact form whose Reeb vector field is $R$. By $\beta(R)=1$ and $L_{R}\beta=0$, we have
\begin{equation}\label{EquationOfBeta0}
\iota(R)d\beta=0.
\end{equation}
Hence it suffices to show that $\beta$ is a contact form. Take a point $x$ on $[0,1] \times [0,t_{0}] \times T^3$. Fix infinitesimal actions $Y_{1}$ and $Y_{2}$ of $\tau$ such that $\{R_{x}, Y_{1x}, Y_{2x}\}$ is linearly independent. Since $\tilde{\Phi}$ is a submersion, $d(\beta(Y_{1})) \wedge d(\beta(Y_{2}))$ is nowhere vanishing. Hence there exists a vector $Z_{1x}$ and $Z_{2x}$ in $T_{x}M$ such that $d(\beta(Y_{1x}))(Z_{2x})=0$, $d(\beta(Y_{2x}))(Z_{1x})=0$, $d(\beta(Y_{1x}))(Z_{1x})$ and $d(\beta(Y_{2x}))(Z_{2x})$ are nonzero. Since $d\beta(Y_{jx},Z_{kx})=-d(\beta(Y_{jx}))(Z_{kx})$ by $L_{Y_{j}}\beta=0$, we have
\begin{equation}\label{EquationOfBeta1}
d\beta(Y_{jx},Z_{kx}) \neq 0
\end{equation}
if $j=k$ and 
\begin{equation}\label{EquationOfBeta2}
d\beta(Y_{jx},Z_{kx}) = 0
\end{equation}
if $j \neq k$. Since $Y_{1x}$ and $Y_{2x}$ are tangent to an orbit of $\tau$, we have
\begin{equation}\label{EquationOfBeta3}
d\beta(Y_{1x},Y_{2x})=0
\end{equation}
by Lemma \ref{RestrictionOfRank}. By equations \eqref{EquationOfBeta0}, \eqref{EquationOfBeta1}, \eqref{EquationOfBeta2} and \eqref{EquationOfBeta3}, we have $\beta \wedge d\beta \wedge d\beta(R_{x},Y_{1x},Y_{2x},Z_{1x},Z_{2x})=-d\beta(Y_{1x},Z_{1x})d\beta(Y_{2x},Z_{2x})$. Hence $\beta \wedge d\beta \wedge d\beta$ is nowhere vanishing and $\beta$ is contact.

We show that $([0,1] \times [0,t_{0}] \times T^3,\alpha)$ and $([0,1] \times [0,t_{0}] \times T^3,\beta)$ are isomorphic by a diffeomorphism whose restriction to an open neighborhood of the boundary is the identity. We put $\alpha_{t}=(1-t)\alpha+t\beta$ for $t$ in $[0,1]$. Since the Reeb vector fields and the contact moment maps of $\alpha$ and $\beta$ are the same, $d\alpha_{t}$ induces a symplectic form on $TM/\mathbb{R}R$ by Proposition \ref{Karshon2} (i). Then $\ker \alpha_{t}$ is a $\rho$-invariant contact structure for every $t$. Then by the equivariant version of Gray's theorem \cite{Gra}, $\ker \alpha$ and $\ker \beta$ are isomorphic by a $\rho$-equivariant diffeomorphism $f$ whose restriction to an open neighborhood of the boundary is the identity. By the $\rho$-equivariance of $f$, we have $f_{*}R=R$. Hence we have $f^{*}\beta=\alpha$. 

By conjugating $\tau$ by $f$, we have a $T^3$-action $\tilde{\tau}$ which satisfies the desired conditions.
\end{proof}

\begin{lem}\label{ExtensionOfPsi}
Let $\overline{\Phi}$ be a second projection $[0,1] \times [0,t_{0}] \longrightarrow [0,t_{0}]$. Let $\overline{\Psi}$ be a smooth function defined on an open neighborhood $U$ of the boundary of $[0,1] \times [0,t_{0}]$. Assume that 
\begin{enumerate}
\item $d\overline{\Phi} \wedge d\overline{\Psi}$ is nowhere vanishing on $U$ and
\item the image of $(\overline{\Phi} \times \overline{\Psi})|_{U} \colon U \longrightarrow \mathbb{R}^{2}$ is an open neighborhood of the boundary of a convex subset in $\mathbb{R}^{2}$.
\end{enumerate}
Then there exists a smooth function $\overline{\Psi}_{1}$ on $[0,1] \times [0,t_{0}]$ such that 
\begin{enumerate}
\item $d\overline{\Phi} \wedge d\overline{\Psi}_{1}$ is nowhere vanishing on $[0,1] \times [0,t_{0}]$ and 
\item $\overline{\Psi}_{1}$ coincides with $\overline{\Psi}$ on an open neighborhood of the boundary of $[0,1] \times [0,t_{0}]$.
\end{enumerate}
\end{lem}

\begin{proof}
Note that the assumptions (i) and (ii) imply that $\overline{\Psi}$ is monotone increasing on each level set of $\overline{\Phi}$. We define a positive function $F$ on $U$ by $F(x)=\inf \{ \overline{\Psi}(1) - \overline{\Psi}(x), \overline{\Psi}(x) - \overline{\Psi}(0), t_{0} - \overline{\Phi}(x), \overline{\Phi}(x) \}$. For a positive real number $\epsilon$, we define $D(\epsilon)=\{ x \in U | F(x) \leq \epsilon \}$. $D(\epsilon)$ is an open neighborhood of the boundary of $[0,1] \times [0,t_{0}]$ by the assumption. We take a real number $\epsilon_{0}$ so that $\overline{D(2\epsilon_{0})}$ is contained in $U$. For $x$ in $[0,1] \times [0,t_{0}]$, let $x_{-}$ and $x_{+}$ be two points which satisfy the following: 
\begin{enumerate}
\item $\overline{\Phi}(x)=\overline{\Phi}(x_{-})$, $\overline{\Phi}(x)=\overline{\Phi}(x_{+})$,
\item $\overline{\Psi}(x_{-})= \epsilon + \overline{\Psi}(0)$ and
\item $\overline{\Psi}(x_{+})= \overline{\Psi}(1) - \epsilon$.
\end{enumerate}
Let $p$ be the first projection defined on $[0,1] \times [0,t_{0}]$. We define a function $\overline{\Psi}_{0}$ on $([0,1] \times [0,t_{0}]) - D(\epsilon_{0})$ by
\begin{equation}
\overline{\Psi}_{0}(x) = \Big( 1 - \frac{p(x_{+}) - p(x)}{p(x_{+}) - p(x_{-})} \Big) \Big( \overline{\Psi}(1) - 2\epsilon \Big) +  \frac{p(x_{+}) - p(x)}{p(x_{+}) - p(x_{-})} ( 2\epsilon + \overline{\Psi}(0)).
\end{equation}
Then $\overline{\Psi}_{0}$ is smooth and monotone increasing on each level set of $\Phi$. Let $b$ be a smooth function on $\mathbb{R}_{\geq 0}$ which is monotone increasing and satisfies 
\begin{equation}
b(t)=\begin{cases}
0, & 0 \leq t \leq \frac{4\epsilon_{0}}{3}  \\
1, & \frac{5\epsilon_{0}}{3} \leq t.
\end{cases}
\end{equation}
We extend $F$ to $[0,1] \times [0,t_{0}]$ defining $F(x)=2\epsilon_{0}$ for $x$ in $([0,1] \times [0,t_{0}]) - U$. We define a smooth function $\overline{\Psi}_{1}$ on $[0,1] \times [0,t_{0}]$ by
\begin{equation}
\overline{\Psi}_{1}(x) = (1 - b(F(x))) \overline{\Psi} + b(F(x)) \overline{\Psi}_{0}.
\end{equation}
Then $\overline{\Psi}_{1}$ is smooth and monotone increasing on each level set of $\Phi$. Since $\overline{\Psi}_{1}$ coincides with $\overline{\Psi}$ on $D(\epsilon)$, $\overline{\Psi}_{1}$ satisfies the desired conditions.
\end{proof}

We show Proposition \ref{SufficientConditionToBeToric}.

\begin{proof}
Take two chains $C^1$, $C^2$ in $(M,\alpha)$ so that there is no nontrivial chain except $C^{1}$ and $C^{2}$. By Lemma \ref{NearChains}, we have an $\alpha$-preserving $T^3$-action $\tau$ on an open neighborhood $V$ of $C^1 \cup C^2 \cup B_{\min} \cup B_{\max}$. By Lemma \ref{ModificationOfTorusActions}, there exist an open neighborhood $W$ of $C^1 \cup C^2 \cup B_{\min} \cup B_{\max}$ whose closure is contained in $V$ and a diffeomorphism $f \colon M - W \longrightarrow [0,1] \times [0,t_{0}] \times T^3$ such that the $K$-contact form $(f^{-1})^{*}\alpha$ on $[0,1] \times [0,t_{0}] \times T^3$ and the $T^3$-action $f^{-1} \circ \tau \circ f$ on $f(V \cap (M - W))$ satisfy the assumption of Lemma \ref{FarFromChains}. Then we have an $\alpha$-preserving $T^3$-action on $M$ by Lemma \ref{FarFromChains}.
\end{proof}

\section{Contact blowing up for $5$-dimensional $K$-contact manifolds}\label{ContactBlowingUpAndDown}

We give the definition of contact blowing up and down for $5$-dimensional $K$-contact manifolds as special cases of contact cuts defined by Lerman \cite{Ler1}.

\subsection{Contact blowing up along a closed orbit of the Reeb flow}

We define the contact blowing up along a closed orbit of the Reeb flow. Let $(M,\alpha)$ be a closed $5$-dimensional $K$-contact manifold of rank $2$. We denote the $T^2$-action associated with $\alpha$ by $\rho$. Take a closed orbit $\Sigma$ of the Reeb flow of $\alpha$. 

By Lemma \ref{ToricActionNearSingularOrbits}, we have an $\alpha$-preserving $T^3$-action $\tau$ on an open neighborhood of $\Sigma$. Let $\sigma$ be an $S^1$-subaction of $\tau$ which satisfies
\begin{enumerate}
\item $\{x \in M | \alpha(X)=0 \}$ is the boundary of an open tubular neighborhood $U$ of $\Sigma$ where $X$ is the vector field generating $\sigma$ and
\item $\sigma$ acts freely on $\{x \in M | \alpha(X)=0 \}$.
\end{enumerate}
Then we can define a smooth structure and a $K$-contact structure on $(M-\overline{U}) \cup \big( (\partial U)/\sigma \big)$ by Lerman \cite{Ler1}.
\begin{defn}(Contact blowing up along a closed orbit of the Reeb flow) We call the operation to obtain the $K$-contact manifold $(M-\overline{U}) \cup \big( (\partial U)/\sigma \big)$ from $M$ the contact blowing up along $\Sigma$. We call the inverse operation the contact blowing down along $(\partial U)/\sigma$ to $\Sigma$.
\end{defn}

We describe the contact blowing up in the case where the isotropy group of $\rho$ at $\Sigma$ is connected. In this case, the expression of normal forms is simple. In general cases, we have to take a finite covering of $\Sigma$ to write $\alpha$ by coordinates. 

By Lemma \ref{NormalFormOfSingularOrbits}, there exists a $\rho$-invariant tubular neighborhood $U$ such that $(U,\alpha|_{U})$ is isomorphic to $(S^1 \times D^{4}_{\epsilon}, \alpha_0)$ defined as follows:

$\alpha_{0}$ is the contact form on $S^1 \times D^{4}_{\epsilon}$ defined by
\begin{equation}
\alpha_{0}=\frac{1-\lambda_1 (m_1|z_1|^2 + m_2|z_2|^2)}{\lambda_0}d\zeta + \frac{\sqrt{-1}}{2} (z_{1}d\overline{z}_{1}-\overline{z}_{1}dz_{1}) + \frac{\sqrt{-1}}{2} (z_{2}d\overline{z}_{2}-\overline{z}_{2}dz_{2}),
\end{equation}
where $(m_1,m_2)$ is a pair of coprime integers, $\lambda_0$ and $\lambda_1$ are some real numbers which are linearly independent over $\mathbb{Q}$.

For positive numbers $r_1,r_2$ greater than $\epsilon^{-2}$, we define
\begin{equation}
\begin{array}{l}
D^{4}_{r_1,r_2}=\{ (z_1,z_2)  \in \mathbb{C}^{2} | r_1|z_1|^{2}+r_2|z_2|^{2} < 1 \} \\
S^{3}_{r_1,r_2}=\partial \overline{D^{4}_{r_1,r_2}}.
\end{array}
\end{equation}
If we define a vector field on $S^1 \times S^{3}_{r_1,r_2}$ by
\begin{equation}\label{X}
X=\lambda_0\frac{\partial}{\partial \zeta}+ \Big( \lambda_1 m_1 - r_1 \Big) \sqrt{-1} \Big( z_{1}\frac{\partial}{\partial z_{1}}-\overline{z}_{1}\frac{\partial}{\partial \overline{z}_{1}} \Big) + \Big( \lambda_1 m_2 - r_2 \Big) \sqrt{-1} \Big( z_{2}\frac{\partial}{\partial z_{2}}-\overline{z}_{2}\frac{\partial}{\partial \overline{z}_{2}} \Big),
\end{equation}
then the flow $\sigma$ generated by $X$ has the following properties:
\begin{enumerate}
\item $(\ker \alpha_0)|_{S^1 \times S^{3}_{r_1,r_2}}$ is invariant under $\sigma$,
\item the orbits of $\sigma$ are tangent to $(\ker \alpha_0)|_{S^1 \times S^{3}_{r_1,r_2}}$,
\item $\sigma$ commutes with $\rho$,
\item every orbit of $\sigma$ is closed if and only if $\frac{\lambda_1 m_1 - r_1}{\lambda_0}$ and $\frac{\lambda_1 m_2 - r_2}{\lambda_0}$ are rational, and
\item if every orbit of $\sigma$ is closed, then $X$ generates a free $S^1$-action if and only if $\GCD(a_0,a_1)=1$ and $\GCD(a_0,a_2)=1$, where $a_0$ is the least common multiple of the denominators of $\frac{\lambda_1 m_1 - r_1}{\lambda_0}$ and $\frac{\lambda_1 m_2 - r_2}{\lambda_0}$, $a_1$ and $a_2$ are integers defined by
\begin{equation}
a_1=a_0 \frac{\lambda_1 m_1 - r_1}{\lambda_0}, a_2 = a_0 \frac{\lambda_1 m_2 - r_2}{\lambda_0}.
\end{equation} 
\end{enumerate}

$X$ is characterized by the first two properties among vector fields which are linear with respect to the standard coordinate. Assume that the latter conditions in (iv) and (v) are satisfied. Then we have a smooth structure and a $K$-contact structure on $(M-(S^1 \times D^{4}_{r_1,r_2}))/\sigma$ by Lerman \cite{Ler1}. 

Note that $(S^1 \times S^{3}_{r_1,r_2})/\sigma$ is a lens space with two new closed orbits of the Reeb flow. The number of the closed orbits of the Reeb flow increases by $1$ by this operation. Figure \ref{Figure1} shows the change of the graph of isotropy data when we perform contact blowing up along a closed orbit $\Sigma$ of the Reeb flow.

\begin{figure}
\begin{equation}
\begin{picture}(204,99)(0,-10)
\put(20,40){\blacken\ellipse{4}{4}}
\path(-10,80)(20,40)(-10,0)
\path(60,40)(90,40)
\path(84,38)(90,40)(84,42)
\put(140,53.333){\blacken\ellipse{4}{4}}
\put(140,26.666){\blacken\ellipse{4}{4}}
\path(120,80)(140,53.333)(140,26.666)(120,0)
\end{picture}
\begin{picture}(219,90)(0,-10)
\put(30,15){\blacken\ellipse{30}{10}}
\path(30,70)(30,15)
\path(75,40)(105,40)
\path(100,38)(105,40)(100,42)
\put(150,15){\blacken\ellipse{30}{10}}
\put(150,40){\blacken\ellipse{4}{4}}
\path(150,70)(150,15)
\end{picture}
\end{equation}
\caption{Contact blowing up along $\Sigma$.}
\label{Figure1}
\end{figure}
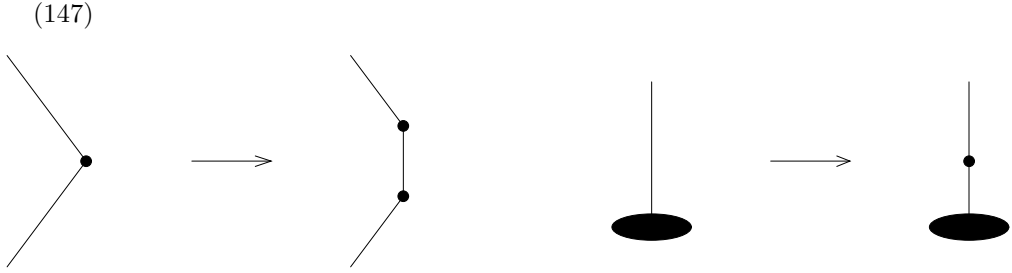

Figure \ref{Figure11} shows the change of the image of contact moment map for the standard $T^3$-action on $S^1 \times D^4$ with a standard $K$-contact form when we perform contact blowing up along $S^1 \times \{0\}$. Note that the image of the contact moment map for a $T^3$-action on a $5$-dimensional contact toric manifold is contained in a $2$-dimensional affine subspace of $\Lie(T^3)^{*}$ (See Subsection \ref{ToricContactManifolds}). Figure \ref{Figure11} is the picture in the $2$-dimensional affine space.

\begin{figure} 
\begin{equation}
\begin{picture}(60,75)(0,-10)
\dashline{4.000}(0,40)(0,0)(30,0)
\thicklines
\path(0,65)(0,35)(25,0)(60,0)
\end{picture}
\end{equation}
\caption{Contact blowing up $S^1 \times D^4$ along $S^1 \times \{0\}$.}
\label{Figure11}
\end{figure}
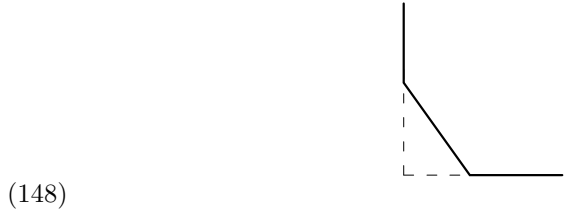

\subsection{Contact blowing up along a $K$-contact lens space}
We define contact blowing up along a lens space. Let $(M,\alpha)$ be a closed $5$-dimensional $K$-contact manifold of rank $2$. We denote the $T^2$-action associated with $\alpha$ by $\rho$. Let $L$ be a $K$-contact submanifold of $(M,\alpha)$ diffeomorphic to a lens space. 

By Lemma \ref{ToricActionNearGradientManifolds}, there exist an open neighborhood $U$ of $L$ and an $(\alpha|_{U})$-preserving $T^3$-action $\tau$ on $U$. Let $\sigma$ be an $S^1$-subaction of $\tau$ which satisfies the following conditions:
\begin{enumerate}
\item $\{x \in M | \alpha(X)=0 \}$ is the boundary of an open tubular neighborhood $U$ of $\Sigma$ where $X$ is the infinitesimal action of $\sigma$ and
\item $\sigma$ acts freely on $\{x \in M | \alpha(X)=0 \}$.
\end{enumerate}
Then we can define a smooth structure and a $K$-contact structure on $(M-\overline{U}) \cup \big( (\partial U)/\sigma \big)$ by Lerman \cite{Ler1}.
\begin{defn}(Contact blowing up along a $K$-contact lens space) We call the operation to obtain a $K$-contact manifold $(M-\overline{U}) \cup \big( (\partial U)/\sigma \big)$ from $M$ the contact blowing up along $L$. We call the inverse operation the contact blowing down along $(\partial U)/\sigma$ to $L$. 
\end{defn}

The number of the closed orbits of the Reeb flow does not change in this operation. Note that we can always perform contact blowing down along any $K$-contact lens space to some $K$-contact lens space. The situation is different from the contact blowing down along a $K$-contact lens space to a closed orbit of the Reeb flow as we will see the next subsection.

When we perform contact blowing up along lens spaces, the underlying graph of the graph of isotropy data does not change, but the attached data change.

Figure \ref{Figure2} shows the change of the image of contact moment map for a $T^3$-action on an open neighborhood of a $K$-contact lens space when we perform contact blowing up along the $K$-contact lens space. Note again that the image of the contact moment map for a $T^3$-action on a contact toric manifold is contained in a $2$-dimensional affine subspace of $\Lie(T^3)^{*}$ (See Subsection \ref{ToricContactManifolds}) and Figure \ref{Figure2} is the picture in the $2$-dimensional affine subspace.

\begin{figure} 
\begin{equation}
\begin{picture}(60,75)(0,-10)
\dashline{4.000}(0,45)(0,20)(20,0)(55,0)
\thicklines
\path(0,80)(0,45)(55,0)(90,0)
\end{picture}
\end{equation}
\caption{Contact blowing up along a $K$-contact lens space.}
\label{Figure2}
\end{figure}
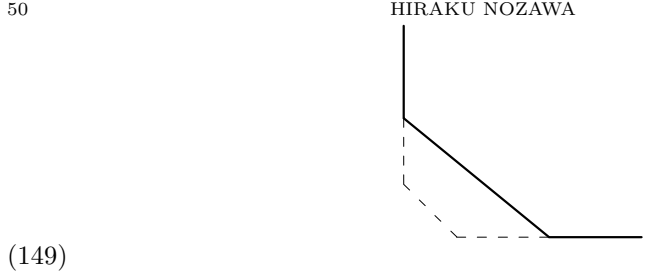

\subsection{Conditions to perform contact blowing down along a lens space to a closed orbit of the Reeb flow}

Let $(M,\alpha)$ be a $5$-dimensional $K$-contact manifold of rank $2$ and $L$ be a $K$-contact submanifold of $(M,\alpha)$ diffeomorphic to a lens space. We denote the torus action associated with $\alpha$ by $\rho$. We present two conditions under which we can perform contact blowing down along $L$ to a closed orbit of the Reeb flow. 

The first condition is a topological characterization which follows directly from the normal form theorem of $K$-contact submanifolds. The second condition is a sufficient condition in terms of the cardinality of the isotropy group of $\rho$. Note that the Euler class of the normal bundle of $L$ is well-defined since the normal bundle has a symplectic structure induced from $d\alpha$.

\begin{lem}\label{TopologicalCharacterization}
The followings are equivalent:
\begin{enumerate}
\item There exist a $K$-contact manifold $(\overline{M},\overline{\alpha})$ of rank $2$ and a closed orbit $\Sigma$ of the Reeb flow on $(\overline{M},\overline{\alpha})$ such that $(M,\alpha)$ is obtained by the contact blowing up along $\Sigma$ to $L$ from $(\overline{M},\overline{\alpha})$.
\item The total space of the normal $S^1$-bundle of $L$ in $M$ is diffeomorphic to $S^1 \times S^3$.
\item The Euler class of the normal $S^1$-bundle of $L$ in $M$ is a generator of $H^2(L;\mathbb{Z})$.
\end{enumerate}
\end{lem}

\begin{proof}
The proof of (ii) from (i) is clear, since the boundary of a tubular neighborhood of $\Sigma$ in $M$ is diffeomorphic to $S^1 \times S^3$. 

The equivalence of (ii) and (iii) follows from the following theorem of Thornton \cite{Tho}: Let $L(p,q)$ be the lens space of type $(p,q)$. Note that $H^2(L(p,q);\mathbb{Z})$ is isomorphic to $\mathbb{Z}/p\mathbb{Z}$ and $\GCD(p,e)$ is well-defined for $e$ in $H^2(L(p,q);\mathbb{Z})$.
\begin{thm}[Thornton \cite{Tho}]\label{LensSpaces}
The total space $N$ of the principal $S^1$-bundle over $L(p,q)$ with the Euler class $e$ in $H^2(L(p,q);\mathbb{Z})$ is homeomorphic to $S^1 \times L(\GCD(p,e),q)$. Hence $N$ is homeomorphic to $S^1 \times S^3$ if and only if the Euler class is a generator of $H^2(L(p,q);\mathbb{Z})$.
\end{thm}

We show (i) from (iii). Assume that the Euler class of the normal $S^1$-bundle of $L$ in $M$ is a generator of $H^2(L;\mathbb{Z})$. By Lemma \ref{ToricActionNearGradientManifolds}, there exists an $\alpha$-preserving $T^3$-action $\tau$ on an open neighborhood $W$ of $L$. Let $L^{0}$ and $L^{2}$ be two gradient manifolds whose closures intersect $L$ and are smooth near $L$. We put $L^{1}=L$. Let $n^{j}$ be the primitive vector in $\Lie(T^3)$ whose infinitesimal action generates the isotropy group of $\tau$ at $L^{j}$ for $j=0$, $1$ and $2$. Let $A=\{ v \in \Lie(T^3)^{*} | v(\overline{R})=1 \}$. Let $\tilde{\Phi}_{\alpha}$ be the contact moment map for $\tau$. Then $V=\tilde{\Phi}_{\alpha}(U)$ is an open neighborhood of $F^{1} \cap A$ in $\Delta \cap A$ where $\Delta$ is the cone defined by \begin{equation}
\Delta=\{ v \in \Lie(T^3)^{*} | v(n^{0}) \geq 0, v(n^{2}) \geq 0, v(n^{1}) \geq 0  \}
\end{equation}
and $F^{i}$ is the face of $\Delta$ defined by the normal vector $n^{i}$ for $i=0$, $1$ and $2$. We define the cones $\Delta'$ and $\tilde{\Delta}$ in $\Lie(T^3)^{*}$ by
\begin{equation}
\Delta'=\{ v \in \Lie(T^3)^{*} | v(n^{0}) \geq 0, v(n^{2}) \geq 0, v(n^{1}) \leq 0 \}
\end{equation}
and 
\begin{equation}
\tilde{\Delta}=\{ v \in \Lie(T^3)^{*} | v(n^{0}) \geq 0, v(n^{2}) \geq 0 \}.
\end{equation}

By Lemma \ref{AttachingCorners}, we have a $K$-contact orbifold $(\tilde{U},\tilde{\alpha})$ and an $\tilde{\alpha}$-preserving $T^3$-action $\sigma$ on $\tilde{U}$ such that 
\begin{description}
\item[(a)] the image of the contact moment map $\Phi$ for $\sigma$ contains $\Delta'$ and
\item[(b)] $(\tilde{U}-\Phi^{-1}(\Delta'),\tilde{\alpha}|_{\tilde{U}-\Phi^{-1}(\Delta')})$ is isomorphic to $(U-B,\alpha|_{U-B})$ as $K$-contact manifolds, where $U$ is an open neighborhood of $B$ in $M$.
\end{description}

By (iii) and Lemma \ref{GoodCone}, there exists a vector $v$ in $\Lie(T^3)^{*}_{\mathbb{Z}}$ such that $\det \begin{psmallmatrix} n^{0} & n^{2} & v \end{psmallmatrix} = 1$. Hence $(\tilde{U},\tilde{\alpha})$ is a smooth manifold, since the image of the contact moment map satisfies the Delzant condition. We obtain a $K$-contact manifold $(\overline{M},\overline{\alpha})$ by attaching $(M - L,\alpha|_{M - L})$ to $(\tilde{U},\tilde{\alpha})$ by the isomorphism between $(\tilde{U}-\Phi^{-1}(\Delta'),\tilde{\alpha}|_{\tilde{U}-\Phi^{-1}(\Delta')})$ and $(U-B,\alpha|_{U-B})$. Let $\Sigma$ be the closed orbit of the Reeb flow of $\overline{\alpha}$ defined by $\Sigma=\Phi^{-1}(F^{0} \cap F^{2})$. We can obtain $A \cap \Delta$ from $A \cap \Delta'$ by cutting off a triangle $A \cap (\Delta - \Delta')$ containing the vertex $A \cap F^{0} \cap F^{2}$. A contact blowing up along $\Sigma$ results the same change of the image of the contact moment maps $A \cap \Delta$ from $A \cap \Delta'$. Hence $(M,\alpha)$ is obtained from $(\overline{M},\overline{\alpha})$ by performing a contact blowing up along $\Sigma$.
\end{proof}

We give a sufficient condition for a gradient lens space to be blown down to a closed orbit of the Reeb flow in terms of the cardinality of the isotropy groups of $\rho$. Let $L^{1}$ be a gradient manifold of $(M,\alpha)$ whose closure is a smooth submanifold $N$ of $M$. We denote the $\alpha$-limit set of $L^{1}$ and the $\omega$-limit set of $L^{1}$ by $\Sigma^{0}$ and $\Sigma^{1}$, respectively. Let $L^{0}$ be the gradient manifold other than $L^{1}$ whose closure contains $\Sigma^{0}$ and is smooth near $\Sigma^{0}$. Let $L^{2}$ be the gradient manifold other than $L^{1}$ whose closure contains $\Sigma^{1}$ and is smooth near $\Sigma^{1}$. We put $k^{j}=I(L^{j},\rho)$ for $j=0$, $1$ and $2$. For a topological group $H$, an $H$-action $\tau$ on a set $A$ and a $\tau$-invariant subset $B$ of $A$, we denote the cardinality of the kernel of $H \longrightarrow \Aut(B)$ by $I(\tau,B)$.

\begin{lem}\label{Characterizationbyk}
There exist a $K$-contact manifold $(\tilde{M},\tilde{\alpha})$ of rank $2$ and a closed orbit $\Sigma$ of the Reeb flow on $(\tilde{M},\tilde{\alpha})$ such that $(M,\alpha)$ is obtained by the contact blowing up along $\Sigma$ from $(\tilde{M},\tilde{\alpha})$ if one of the following holds:
\begin{enumerate}
\item $I(\rho,L^{0})=I(\rho,L^{1})=1$ and the closure of $L^{2}$ is equal to $B_{\max}$.
\item The closure of $L^{0}$ is equal to $B_{\min}$ and $I(\rho,L^{1})=I(\rho,L^{2})=1$.
\end{enumerate}
\end{lem}

\begin{proof}
Assume that the condition (i) is satisfied. By Lemma \ref{TopologicalCharacterization}, it suffices to prove that the Euler class of the normal bundle of $N$ is a generator of $H^2(N;\mathbb{Z})$. We put $k^{j}=I(\rho,L^{j})$ for $j=0$ and $1$. 

We have an $\alpha$-preserving $T^3$-action $\tau$ on an open neighborhood of $N$ by Lemma \ref{ToricActionNearGradientManifolds}. A generic orbit of $\tau$ in $N$ is diffeomorphic to $T^2$. Let $C$ be a $T^2$-orbit of $\tau$ in $N$. Then by the classification of contact toric $3$-manifolds by Lerman \cite{Ler2}, $C$ divides $N$ into two solid tori. We obtain a Heegaard decomposition of $N$ of genus $1$ which is invariant under $\tau$ such that $C$ is the boundary of two solid tori. Cutting the total space $E$ of the normal bundle of $N$ along the union of fibers over $C$, we have a $\tau$-invariant decomposition of $E$ defined by
\begin{equation}
E= (S^1 \times D^2)_{0} \times \mathbb{R}^2 \cup_{\tilde{A}} (S^1 \times D^2)_{1} \times \mathbb{R}^{2}
\end{equation}
such that $\tau$ is written as 
\begin{equation}
t \cdot (\zeta,z_1,z_2)=(l_{0}^{j}(t) \zeta, l_{1}^{j}(t) z_{1}, l_{2}^{j}(t) z_{2})
\end{equation}
 on $(S^1 \times \partial D^2)_{j} \times \mathbb{R}^2$ for every $t$ in $T^3$ where $l_{0}^{j}$, $l_{1}^{j}$ and $l_{2}^{j}$ are homomorphisms from $T^3$ to $S^1$ for $j=0$ and $1$. The attaching map $\tilde{A} \colon (S^1 \times \partial D^2)_{0} \times \mathbb{R}^2 \longrightarrow (S^1 \times \partial D^2)_{1} \times \mathbb{R}^2$ is written as
\begin{equation}
\tilde{A} (\zeta,z_1,z_2) = (\zeta^{a} z_{1}^{b}, \zeta^{c} z_{1}^{d} , \zeta^{e_{0}} z_{1}^{f_{0}} z_{2})
\end{equation}
with respect to the standard coordinates. Note that the map $A_{0} \colon (S^1 \times \partial D^2)_{0} \longrightarrow (S^1 \times \partial D^2)_{1}$ defined by
\begin{equation}
A_{0} (\zeta,z_1) = (\zeta^{a} z_{1}^{b}, \zeta^{c} z_{1}^{d})
\end{equation} 
is the attaching map of the Heegaard decomposition of $N_{0}^{1}$ of genus $1$. The determinant of $A_{0}$ is $-1$. Since $\rho$ is a $T^2$-subaction of $\tau$, $\rho$ is written as 
\begin{equation}\label{RhoNearSigmaj}
(t_{j0},t_{j1}) \cdot (\zeta_{j}, z_{j1}, z_{j2})=(t_{j0}^{s_{j}} t_{j1}^{r_{j}} \zeta_{j}, t_{j0}^{q_{j}} t_{j1}^{p_{j}} z_{j1}, t_{j0}^{u_{j}} t_{j1}^{v_{j}} z_{j2})
\end{equation}
on $(S^1 \times \partial D^2)_{j} \times \mathbb{R}^2$ with respect to the standard coordinate for $j=0$ and $1$.

We put
\begin{equation}
h^{0} = \det \begin{pmatrix} q_{0} & p_{0} \\ u_{0} & v_{0} \end{pmatrix}, h^{1}=\det \begin{pmatrix} q_{1} & p_{1} \\ u_{1} & v_{1} \end{pmatrix}.
\end{equation}

We will show $h^{0}=0$ and $h^{1}=0$. By Lemma \ref{GCD} and $k^{0}=k^{1}=1$, we have $I(\rho,\Sigma^{0})=1$, that is, the isotropy group of $\rho$ at $\Sigma^{0}$ is connected. By \eqref{StandardTorusActionInTheCaseOfConnectedIsotropyGroup}, $\rho$ is written as
\begin{equation}\label{RhoNearSigma0}
(t'_{00},t'_{01}) \cdot (\zeta'_{0},z'_{01},z'_{02})=( t'_{00} \zeta'_{0}, t_{01}^{\prime p'_{0}} z'_{01}, t_{01}^{\prime v'_{0}} z'_{02})
\end{equation}
 on an open tubular neighborhood of $\Sigma^{0}$ where $(\zeta'_{0},z'_{01},z'_{02})$ is the coordinate on an open tubular neighborhood of $\Sigma^{0}$ defined by
\begin{equation}
\begin{pmatrix} \zeta_{0} \\ z_{01} \\ z_{02} \end{pmatrix} = \begin{pmatrix} 1 & 0 \\ x^{0} & U^{0} \end{pmatrix} \begin{pmatrix} \zeta'_{0} \\ z'_{01} \\ z'_{0} \end{pmatrix}
\end{equation}
for some $x^{0}$ in $\mathbb{Z}^{2}$ and some $U^{0}$ in $\GL(2;\mathbb{Z})$. By \eqref{RhoNearSigma0}, we have
\begin{equation}\label{h0}
T^{0} \begin{pmatrix} 1 & 0 & 0 \\ 0 & p'_{0} & v'_{0} \end{pmatrix} \begin{pmatrix} 1 & 0 \\ x^{0} & U^{0} \end{pmatrix} = \begin{pmatrix} s_{0} & q_{0} & u_{0} \\ 0 & p_{0} & v_{0} \end{pmatrix}
\end{equation}
for some $T^{0}$ in $\GL(2;\mathbb{Z})$ which maps $(t'_{00},t'_{01})$ to $(t_{00},t_{01})$. $h^{0}=0$ follows from \eqref{h0}.

By \eqref{RhoNearSigmaj}, $\Sigma^{1}$ is written as
\begin{equation}\label{RhoNearSigma1}
(t'_{10},t'_{11}) \cdot (\zeta_{1},z_{11},z_{12})=( t_{10}^{\prime s'_{1}} \zeta_{1}, t_{10}^{\prime q'_{1}} t_{11}^{\prime p'_{1}} z_{11}, t_{10}^{\prime u'_{1}} t_{11}^{\prime v'_{1}} z_{12})
\end{equation}
for a coordinate $(t'_{10},t'_{11})$ on $T^2$ and integers $s'_{1}$, $q'_{1}$, $p'_{1}$, $u'_{1}$ and $v'_{1}$. Note that $t'_{11}$ generates the identity component of the isotropy group of $\rho$ at $\Sigma^{1}$. Then we have 
\begin{equation}\label{s1}
|s_{1}|=I(\rho,\Sigma^{1})
\end{equation}
 by the definition of $I(\rho,\Sigma^{1})$. Since $L^{2}=B_{\max}$, the isotropy group of $\rho$ at $\Sigma^{1}$ fixes $L^{2}$. Hence we can assume that $p_{1}=1$ and $v_{1}=0$. Then \eqref{RhoNearSigma1} is written as
\begin{equation}\label{RhoNearSigma12}
(t'_{10},t'_{11}) \cdot (\zeta_{1},z_{11},z_{12})=( t_{10}^{\prime s_{1}} \zeta_{1}, t_{10}^{\prime q'_{1}} t_{11} z_{11}, t_{10}^{\prime u'_{1}} z_{12}).
\end{equation}
Since the cardinality of the isotropy group of the $T^2$-action on $T^2$ defined by
\begin{equation}
(t_{0},t_{1}) \cdot (\xi_{0},\xi_{1}) =  (t_{0}^{a_{11}} t_{1}^{a_{12}} \xi_{0}, t_{0}^{a_{21}} t_{1}^{a_{22}} \xi_{1})
\end{equation}
is equal to $|a_{11} a_{22} - a_{12} a_{21}|$, we have 
\begin{equation}
k^{1} = |s_{1}|.
\end{equation}
Hence we have $|s_{1}|=1$ by the assumption. By $|s_1|=1$ and \eqref{s1}, the isotropy group of $\rho$ at $\Sigma^{1}$ is connected. By the same argument that proving $h^{0}$, we can show $h^{1}=0$.

By Lemma \ref{LemmaEquationfork}, we have 
\begin{equation}\label{k}
\begin{array}{l}
(be-af)k^{1}=-k^{2}+ak^{0}-bh^{0}, \\
(de-cf)k^{1}=h^{1}+ck^{0}-dh^{0}.
\end{array}
\end{equation}
Substituting $k^{0}=h^{0}=h^{1}=0$ and $k^{1}=k^{2}=1$ to \eqref{k}, we have 
\begin{equation}
be-af=-1, de-cf=0.
\end{equation}
Hence we have $\begin{psmallmatrix} e \\ f \end{psmallmatrix}=\begin{psmallmatrix} -c \\ -d \end{psmallmatrix}$. Hence $f$ is coprime to $b$. By \eqref{tildeA} and the proof of Lemma \ref{GoodCone} below, $f$ modulo $b$ is the Euler class of the normal bundle of $N$ in $M$. Hence the proof is completed for the case (i). The proof for the case (ii) is symmetric.
\end{proof}

\subsection{Examples of contact toric manifolds}\label{ToricExamples}
We present examples of $5$-dimensional contact toric manifolds of rank $2$ without a gradient manifold which can be blown down to a closed orbit of the Reeb flow. Note that $4$-dimensional symplectic manifolds with hamiltonian $S^1$-actions which are not $\mathbb{C}P^{2}$, Hirzebruch surfaces or ruled surfaces always have a gradient two sphere which can be $S^1$-equivariantly blown down (See \cite{HaA} and \cite{Kar}). 

We show a lemma to compute the Euler class of the normal bundle of invariant lens spaces in contact toric manifolds by the normal vectors of the corresponding good cone: Let $(M,\alpha)$ be a toric $5$-dimensional $K$-contact manifold and $\Delta$ be the corresponding good cone in $\Lie(T^3)^{*}$. We identify $\Lie(T^3)^{*}$ with $\mathbb{R}^{3}$ so that the kernel of $\exp \colon \Lie(T^3) \longrightarrow T^3$ is identified with $\mathbb{Z}^{3}$. We denote the element of $\Lie(T^3)$ whose infinitesimal action is the Reeb vector field of $(M,\alpha)$ by $\overline{R}$.
\begin{lem}\label{GoodCone}
Let $P^1, P^2$ and $P^3$ be three faces of $\Delta$ and assume that $P^{2}$ is adjacent to both of $P^1$ and $P^3$ in $\Delta$. We denote the primitive normal vector of $P^i$ by $n^i$ for $i=1,2$ and $3$. Assume that $\det \begin{pmatrix} n^{1} & n^{2} & n^{3} \end{pmatrix}$ is positive. Let $L$ be the lens space in $M$ which is the inverse image of the intersection of $P^2$ and the affine subspace $\{v \in \mathbb{R}^{3}| v(\overline{R})=1\}$ by the contact moment map for the $T^3$-action. Then the cardinality $c$ of $\pi_1(L)$ and the Euler class $f$ of the normal bundle of $L$ in $M$ are given by
\begin{equation}
\begin{array}{l}
b = \det \begin{pmatrix} n^{1} & n^{2} & n^{3} \end{pmatrix} \\
f = \det \begin{pmatrix} n^{1} & n^{3} & l^{2} \end{pmatrix} \mod b\\
\end{array}
\end{equation}
where $l^{2}$ is an element of $\mathbb{Z}^{3}$ which satisfies $\det \begin{psmallmatrix} n^{3} & l^{2} & n^{2} \end{psmallmatrix}=1$. 
\end{lem}

\begin{proof}
We denote the normal bundle of $L$ in $M$ by $\nu$. We fix a Heegaard decomposition of genus $1$ of $L$ and compute $b$ and $f$ by $T^3$-equivariantly trivialising $\nu$ on each solid torus in $L$. Let $\Sigma_{12}$ and $\Sigma_{23}$ be the singular $S^1$ orbits of the $T^3$-action which are contained in $L$ where the image of $\Sigma_{ij}$ by the contact moment map for the $T^3$-action is $P^{i} \cap P^{j}$ for $(i,j)=(1,2)$ and $(2,3)$. Note that 
\begin{enumerate}
\item the $S^1$-action generated by $n^{2}$ is the principal $S^1$-action on $\nu$
\item an orbit of the $S^1$-action generated by $n^{1}$ is bounded by a disk $D_1$ near $\Sigma_{12}$ and
\item an orbit of the $S^1$-action generated by $n^{3}$ is bounded by a disk $D_2$ near $\Sigma_{23}$.
\end{enumerate}
We can take a Heegaard decomposition $L=V_1 \cup V_2$ of $L$ of genus $1$ so that $D_i$ is the meridian disk of the solid torus $V_i$ for $i=1$ and $2$. Then we can trivialize $\nu$ equivariantly on $V_i$ as follows: Let $l^{1}$ be an element of $\mathbb{Z}^{3}$ which satisfies $\det \begin{psmallmatrix} l^{1} & n^{1} & n^{2} \end{psmallmatrix}=1$. Let $l^{2}$ be an element of $\mathbb{Z}^{3}$ which satisfies $\det \begin{psmallmatrix} l^{2} & n^{2} & n^{3} \end{psmallmatrix}=1$. Then $\nu|_{V_i}$ is equivariantly isomorphic to $(S^1 \times D^2)_{i} \times \mathbb{R}^2$ where the $T^3$-action on $(S^1 \times D^2)_{i} \times \mathbb{R}^2$ is the product of the rotation for $i=1$ and $2$. Then there exists a tubular neighborhood $U$ of $L$ in $M$ such that $U$ is diffeomorphic to $((S^1 \times D^2)_{1} \times \mathbb{R}^2) \cup_{A} ((S^1 \times D^2)_{2} \times \mathbb{R}^2)$ where the identification $A \colon ((S^1 \times \partial D^2)_{1} \times \mathbb{R}^2) \longrightarrow ((S^1 \times \partial D^2)_{2} \times \mathbb{R}^2$ is written as
\begin{equation}
A (\zeta,z_1,z_2) = (\zeta^{a_{11}} z_{1}^{a_{12}} z_{2}^{a_{13}} , \zeta^{a_{21}} z_{1}^{a_{22}} z_{2}^{a_{23}}, \zeta^{a_{31}} z_{1}^{a_{32}} z_{2}^{a_{33}})
\end{equation}
where $a_{ij}$ is defined by 
\begin{equation}\label{EquationforA}
\begin{pmatrix} a_{11} & a_{12} & a_{13} \\ a_{21} & a_{22} & a_{23} \\ a_{31} & a_{32} & a_{33} \end{pmatrix}=\begin{pmatrix} l^{2} & n^{3} & n^{2} \end{pmatrix}^{-1} \begin{pmatrix} l^{1} & n^{1} & n^{2} \end{pmatrix}.
\end{equation}
Then the cardinality $b$ of $\pi_1(L)$ is the absolute value of $a_{12}$ and $f$ is equal to $a_{32} \mod b$. Hence Lemma \ref{GoodCone} is proved by computing the $(1,2)$-th and $(3,2)$-th entries of the right hand side of \eqref{EquationforA}.
\end{proof}

We will construct good cones corresponding to a contact toric manifold without a gradient manifold which can be blown down to a closed orbit of the Reeb flow by the method. Note that we treat open cones for simplicity of notation, though the image of the moment maps are closed cones with nonempty interior.

Let $\{n^{i}\}_{i=0}^{k+2}$ be the set of primitive vectors in $\mathbb{Z}^{3}$ and define cones $\Delta$ and $\Delta'$ by 
\begin{equation}
\Delta=\{v \in \mathbb{R}^{3} | n^{i} \cdot v > 0, 0 \leq i \leq k+2 \}
\end{equation}
We put $n^{k+3}=n^{0}$. We denote the plane defined by $n^{i}$ by $P^{i}$ for $i=0$, $1$, $\cdots$, $k+3$. We assume that $\det \begin{psmallmatrix} n^{0} & n^{1} & n^{2} \end{psmallmatrix} > 0$.

\begin{lem}\label{FacesOfDelta}
The set of the faces of $\Delta$ is $\{\overline{\Delta} \cap P^{i}\}_{i=0}^{k+2}$ and the pair $(\overline{\Delta} \cap P^{i},\overline{\Delta} \cap P^{i+1})$ is adjacent in $\overline{\Delta}$ for $i$ satisfying $0 \leq i \leq k+2$ if and only if $\det \begin{psmallmatrix} n^{i} & n^{i+1} & n^{j} \end{psmallmatrix} > 0$ for every $i$ and $j$ satisfying $0 \leq i \leq k+2$, $0 \leq j \leq k+2$ and $j \neq i$, $i+1$.
\end{lem}

\begin{proof}
$n^{i} \times n^{j}$ is tangent to $P^{i} \cap P^{j}$. Hence $\overline{\Delta} \cap P^{i}$ and $\overline{\Delta} \cap P^{i+1}$ are two adjacent faces of $\Delta$ if and only if the following two conditions are true:
\begin{enumerate}
\item $n^{i} \times n^{i+1}$ or $-n^{i} \times n^{i+1}$ is contained in $\overline{\Delta}$ and
\item $n^{i} \times n^{i+1}$ is not contained in the planes defined by $n^{j}$ for any $j$ not equal to $i$ nor $i+1$.
\end{enumerate}
Since
\begin{equation}
\overline{\Delta}=\{v \in \mathbb{R}^{3} | n^{i} \cdot v \geq 0, 0 \leq i \leq k+1 \},
\end{equation} 
$n^{i} \times n^{i+1}$ and $-n^{i} \times n^{i+1}$ are contained in $\overline{\Delta}$ if and only if $(n^{i} \times n^{i+1}) \cdot n^{j} \geq 0$ for every $j$ and $(n^{i} \times n^{i+1}) \cdot n^{j} \leq 0$ for every $j$, respectively. Since $(a \times b) \cdot c=\det \begin{psmallmatrix} a & b & c \end{psmallmatrix}$, $n^{i} \times n^{i+1}$ is not contained in the planes defined by $n^{j}$ for any $j$ not equal to $i$ nor $i+1$ if and only if $\det \begin{psmallmatrix} n^{i} & n^{i+1} & n^{j} \end{psmallmatrix}$ is nonzero for every $i$ and $j$ satisfying $j$ is equal to neither $i$ nor $i+1$. 

We show the ``if'' part. Assume that $\det \begin{psmallmatrix} n^{i} & n^{i+1} & n^{j} \end{psmallmatrix} > 0$ for every $i$ and $j$ satisfying $0 \leq i \leq k+2$, $0 \leq j \leq k+2$ and $j \neq i$, $i+1$. Then by the argument in the previous paragraph, $\overline{\Delta} \cap P^{i}$ and $\overline{\Delta} \cap P^{i+1}$ are two adjacent faces of $\Delta$ for every $i$ satisfying $0 \leq i \leq k+2$. Since every $n^{i}$ defines a face of $\Delta$ and $n^{i}$ cannot define two faces of $\Delta$ by the convexity, the set of faces of $\Delta$ is $\{\overline{\Delta} \cap P^{i}\}_{i=0}^{k+2}$. Hence the ``if'' part is proved.

We show the ``only if'' part. Assume that the set of the faces of $\Delta$ is $\{\overline{\Delta} \cap P^{i}\}_{i=0}^{k+2}$ and the pair $(\overline{\Delta} \cap P^{i},\overline{\Delta} \cap P^{i+1})$ is adjacent in $\overline{\Delta}$ for $i=0$, $1$, $\cdots$, $k+2$. By the argument in the first paragraph of the proof, for each $i$, we have $\det \begin{psmallmatrix} n^{i} & n^{i+1} & n^{j} \end{psmallmatrix} > 0$ for every $j$ satisfying $j \neq i$, $i+1$ or $\det \begin{psmallmatrix} n^{i} & n^{i+1} & n^{j} \end{psmallmatrix} < 0$ for every $j$ satisfying $j \neq i$, $i+1$. By the assumption $\det \begin{psmallmatrix} n^{0} & n^{1} & n^{2} \end{psmallmatrix} > 0$, we have $\det \begin{psmallmatrix} n^{0} & n^{1} & n^{j} \end{psmallmatrix} > 0$ for every $j$ satisfying $j \neq 0$, $1$. Since we have $\det \begin{psmallmatrix} n^{1} & n^{2} & n^{0} \end{psmallmatrix} > 0$, we have $\det \begin{psmallmatrix} n^{1} & n^{2} & n^{j} \end{psmallmatrix} > 0$ for every $j$ satisfying $j \neq 1$, $2$. Inductively we have $\det \begin{psmallmatrix} n^{i} & n^{i+1} & n^{i-1} \end{psmallmatrix} > 0$ for every $i$. Hence by the assumption, we have $\det \begin{psmallmatrix} n^{i} & n^{i+1} & n^{j} \end{psmallmatrix} > 0$ for every $i$ and $j$ satisfying $0 \leq i \leq k+2$, $0 \leq j \leq k+2$ and $j \neq i$, $i+1$.
\end{proof}

\begin{lem}\label{FacesOfDelta2}
If $\{n^{i}\}_{i=0}^{k}$ satisfies the following conditions:
\begin{enumerate}
\item $(n^{0} \times n^{i})(3) > 0$ for every $i$ satisfying $1 \leq i \leq k$ and
\item $(n^{i} \times n^{i+1})(3) > 0$ for every $i$ satisfying $0 \leq i \leq k-1$,
\end{enumerate}
then we have $(n^{i} \times n^{j})(3)>0$ for every $i$ and $j$ satisfying $i<j$.
\end{lem}

\begin{proof}
We show Lemma \ref{FacesOfDelta2} by the induction on $k$. The case where $k=2$ is clear by the assumption. Assume that Lemma \ref{FacesOfDelta2} is true for $k=s-1$. We show 
\begin{equation}\label{Righthanded6}
(n^{i} \times n^{s})(3)>0
\end{equation}
for every $i<s$. Note that $(v^{1} \times v^{2})(3)>0$ if and only if $\det_{\mathbb{R}^2} \begin{psmallmatrix} \pi(v^{1}) & \pi(v^{2}) \end{psmallmatrix} > 0$ where $\pi \colon \mathbb{R}^{3} \longrightarrow \mathbb{R}^{2}$ is the projection defined by $\pi(x_1,x_2,x_3)=(x_1,x_2)$. By the induction hypothesis, we have
\begin{equation}
\det_{\mathbb{R}^2} \begin{psmallmatrix} \pi(n^{i}) & \pi(n^{j}) \end{psmallmatrix} = (n^{i} \times n^{j})(3) > 0
\end{equation}
for every $i$ and $j$ satisfying $i < j < s$. By the assumption, we have
\begin{equation}
\det _{\mathbb{R}^2} \begin{psmallmatrix} \pi(n^{0}) & \pi(n^{s}) \end{psmallmatrix} = (n^{0} \times n^{s})(3) > 0.
\end{equation}
Hence $\{\pi(n^{i})\}_{i=1}^{s}$ is contained in the half space $H=\{v \in \mathbb{R}^{2} | \det_{\mathbb{R}^2} \begin{psmallmatrix} \pi(n^{0}) & v \end{psmallmatrix} > 0\}$. Consider linear functions $f_{i}$ on $\mathbb{R}^{2}$ defined by $f_{i}(v)=\det_{\mathbb{R}^2} \begin{psmallmatrix} v & \pi(n^{i}) \end{psmallmatrix}$. Let $g_1$, $g_2$ be linear functions on $\mathbb{R}^{2}$. Note that for any vectors $v^{1}$, $v^{2}$ and $v^{3}$ in $H$, if we have $g_{1}(v^{1}) < 0, g_{1}(v^{2})=0, g_{1}(v^{3}) > 0$ and $g_{2}(v^{1}) > 0$, $g_{2}(v^{3}) > 0$, then we have $g_{2}(v^{2}) > 0$. Since we have $f_{i}(v^{0}) > 0$, $f_{i}(v^{i}) = 0$, $f_{i}(v^{s-1}) > 0$ and $f_{s}(v^{0}) > 0$, $f_{s}(v^{s-1}) > 0$, we have $f_{s}(v^{i})>0$. Hence the proof of \eqref{Righthanded6} is completed.
\end{proof}

\begin{lem}\label{FacesOfDelta3}
If $\{n^{i}\}_{i=0}^{k+1}$ satisfies the following conditions:
\begin{enumerate}
\item $(n^{0} \times n^{i})(3) > 0$ for every $i$ satisfying $1 \leq i \leq k+1$,
\item $(n^{i} \times n^{i+1})(3) > 0$ for every $i$ satisfying $0 \leq i \leq k$ and
\item $\det \begin{psmallmatrix} n^{i} & n^{i+1} & n^{i+2} \end{psmallmatrix} > 0$ for every $i$ satisfying $0 \leq i \leq k-1$,
\end{enumerate}
then we have 
\begin{enumerate}
\item $(n^{i} \times n^{j})(3)>0$ for every $i<j$ and
\item $\det \begin{psmallmatrix} n^{h} & n^{i} & n^{j} \end{psmallmatrix} > 0$ for every $h<i<j$.
\end{enumerate}
\end{lem}

\begin{proof}
(i) follows from Lemma \ref{FacesOfDelta2}. We show (ii) by the induction on $k$. The case where $k=2$ is clear by the assumption $\det \begin{psmallmatrix} n^{0} & n^{1} & n^{2} \end{psmallmatrix} > 0$. Assume that Lemma \ref{FacesOfDelta2} is true for $k=s-1$. It suffices to show
\begin{equation}\label{Righthanded5}
\det \begin{psmallmatrix} n^{h} & n^{i} & n^{s} \end{psmallmatrix} > 0
\end{equation}
for every $h$ and $i$ satisfying $h<i<s$ to complete the induction. 

We show
\begin{equation}\label{Righthanded7}
\det \begin{psmallmatrix} n^{0} & n^{i} & n^{s} \end{psmallmatrix} > 0.
\end{equation}
for $1 \leq i \leq s-1$ by the induction on $i$. We consider the case where $i=s-1$. Note that $\frac{1}{n^{0} \times n^{s-1}(3)} n^{0} \times n^{s-1}$, $\frac{1}{n^{s-2} \times n^{s-1}(3)} n^{s-2} \times n^{s-1}$ and $\frac{1}{n^{s-1} \times n^{s}(3)} n^{s-1} \times n^{s}$ are contained in the $1$-dimensional affine space $H=\{v \in \mathbb{R}^{3} | v(3)=1, n^{s-1} \cdot v = 0\}$. For linear functions $g'_{1}$ and $g'_{2}$ on $\mathbb{R}^{3}$ and vectors $v^{1}$, $v^{2}$ and $v^{3}$ in $H$, if we have $g'_{1}(v^{1}) < 0$, $g'_{1}(v^{2}) = 0$, $g'_{1}(v^{3}) > 0$ and $g'_{2}(v^{1}) = 0$, $g'_{2}(v^{2}) > 0$, then we have $g'_{2}(v^{3}) > 0$. Let $f^{i}$ be the linear function on $\mathbb{R}^{3}$ defined by $f^{i}(v)=n^{i} \cdot v$. By (i), $n^{0} \times n^{s-1}(3)$, $n^{s-2} \times n^{s-1}(3)$ and $n^{s-1} \times n^{s}(3)$ are positive. Hence by the assumption, we have $f^{s-2}(\frac{1}{n^{0} \times n^{s-1}(3)} n^{0} \times n^{s-1}) < 0$, $f^{s-2}(\frac{1}{n^{s-2} \times n^{s-1}(3)} n^{s-2} \times n^{s-1}) = 0$ and $f^{s-2}(\frac{1}{n^{s-1} \times n^{s}(3)} n^{s-1} \times n^{s}) > 0$, $f^{0}(\frac{1}{n^{0} \times n^{s-1}(3)} n^{0} \times n^{s-1}) =0$, $f^{0}(\frac{1}{n^{s-2} \times n^{s-1}(3)} n^{s-2} \times n^{s-1}) > 0$, we have $f^{0}(\frac{1}{n^{s-2} \times n^{s-1}(3)} n^{s-2} \times n^{s-1})>0$. Since $n^{s-2} \times n^{s-1}(3)>0$, we have \eqref{Righthanded7} for $i=s-1$. 

Assume that \eqref{Righthanded7} is true for $i=t$. We consider the case of $i=t-1$. Note that $\frac{1}{n^{0} \times n^{t-1}(3)} n^{0} \times n^{t-1}$, $\frac{1}{n^{0} \times n^{t}(3)} n^{0} \times n^{t}$ and $\frac{1}{n^{0} \times n^{s}(3)} n^{0} \times n^{s}$ are contained in the $1$-dimensional affine space $H=\{v \in \mathbb{R}^{3} | v(3)=1, n^{0} \cdot v = 0\}$. For linear functions $g'_{1}$ and $g'_{2}$ on $\mathbb{R}^{3}$ and vectors $v^{1}$, $v^{2}$ and $v^{3}$ in $H$, if we have $g'_{1}(v^{1}) > 0$, $g'_{1}(v^{2}) = 0$, $g'_{1}(v^{3}) < 0$ and $g'_{2}(v^{1}) = 0$, $g'_{2}(v^{2}) < 0$, then we have $g'_{2}(v^{3}) < 0$. By (i), $n^{0} \times n^{t}(3)$, $n^{t-1} \times n^{t}(3)$ and $n^{t} \times n^{s}(3)$ are positive. By the assumption and the induction hypothesis, we have $f^{t}(\frac{1}{n^{0} \times n^{t-1}(3)} n^{0} \times n^{t-1}) > 0$, $f^{t}(\frac{1}{n^{0} \times n^{t}(3)} n^{0} \times n^{t}) = 0$ and $f^{t-1}(\frac{1}{n^{0} \times n^{s}(3)} n^{0} \times n^{s}) < 0$ and $f^{t}(\frac{1}{n^{0} \times n^{t-1}(3)} n^{0} \times n^{t-1}) = 0$, $f^{t}(\frac{1}{n^{0} \times n^{t}(3)} n^{0} \times n^{t}) < 0$. Hence we have $f^{t}(\frac{1}{n^{0} \times n^{s}(3)} n^{0} \times n^{s}) < 0$. Since $n^{0} \times n^{s}(3)>0$, we have \eqref{Righthanded7} for $i=t$. Hence \eqref{Righthanded7} is true for $1 \leq i \leq s-1$.

We change the coordinate of $\mathbb{R}^{3}$ by a matrix $A$ in $\SL(3;\mathbb{Z})$ so that $An^{s}=\begin{psmallmatrix} 0 \\ 0 \\ 1 \end{psmallmatrix}$. Then $\{An^{i}\}_{i=0}^{s}$ satisfies the assumptions of Lemma \ref{FacesOfDelta2} by \eqref{Righthanded7} and the assumption (iii). Hence we have 
\begin{equation}
\det \begin{psmallmatrix} An^{h} & An^{i} & An^{s} \end{psmallmatrix} = (An^{h} \times An^{i})(3) > 0
\end{equation}
for every $h$ and $i$ satisfying $h<i<s$ by Lemma \ref{FacesOfDelta2}. The proof is completed.
\end{proof}

\begin{lem}\label{LastVector}
Let $\{n^{i}\}_{i=0}^{k+1}$ be a set of primitive vectors in $\mathbb{Z}^{3}$ which satisfies $\det \begin{psmallmatrix} n^{i} & n^{i+1} & n^{j} \end{psmallmatrix} > 0$ for every $i$ and $j$ satisfying $0 \leq i \leq k$, $0 \leq j \leq k+1$ and $j \neq i$, $i+1$. Then there exists a vector $n^{k+2}$ in $\mathbb{Z}^{3}$ such that
\begin{enumerate}
\item $\det \begin{psmallmatrix} n^{k+1} & n^{k+2} & n^{j} \end{psmallmatrix} >0$ for $0 \leq j \leq k$ and $\det \begin{psmallmatrix} n^{k+2} & n^{0} & n^{j} \end{psmallmatrix} >0$ for $1 \leq j \leq k+1$ and
\item There exist vectors $l$ and $l'$ in $\mathbb{Z}^{3}$ such that $\det \begin{psmallmatrix} n^{k+1} & n^{k+2} & l \end{psmallmatrix}=1$ and
$\det \begin{psmallmatrix} n^{k+2} & n^{0} & l' \end{psmallmatrix}=1$.
\end{enumerate} 
\end{lem}

\begin{proof}
Let $v^{0}$ be a primitive normal vector of the plane spanned by $n^{0}$ and $n^{k}$ such that $\det \begin{psmallmatrix} n^{0} & n^{k} & v^{0} \end{psmallmatrix} > 0$. We show that any vector $n^{k+2}$ in $\mathbb{Z}^{3}$ satisfying $v^{0} \cdot n^{k+2}=1$ satisfies the condition (ii). We can assume that $v^{0}=\begin{psmallmatrix} 0 \\ 1 \\ 0 \end{psmallmatrix}$. Then we can put $n^{i}=\begin{psmallmatrix} a^{i} \\ 0 \\ b^{i} \end{psmallmatrix}$ and $n^{k+2}=\begin{psmallmatrix} x \\ 1 \\ z \end{psmallmatrix}$ where $(a^{i},b^{i})$ are coprime integers for $i=0$ and $k+1$. If we take integers $c^{i}$ and $d^{i}$ so that $a^{i}d^{i} - b^{i}c^{i}=1$ for $i=0$, $k+1$ and put $l=\begin{psmallmatrix} c^{0} \\ 0 \\ d^{0} \end{psmallmatrix}$ and $l'=\begin{psmallmatrix} -c^{k+1} \\ 0 \\ -d^{k+1} \end{psmallmatrix}$, we have $\det \begin{psmallmatrix} n^{k+1} & n^{k+2} & l \end{psmallmatrix}=1$ and $\det \begin{psmallmatrix} n^{k+2} & n^{0} & l' \end{psmallmatrix}=1$. 

We show that there exists a vector $n^{k+2}$ in $\mathbb{Z}^{3}$ satisfying $v^{0} \cdot n^{k+2}=1$ and the condition (i). We define cones $\Theta$ and $\Theta'$ in $\mathbb{R}^{3}$ by
\begin{equation}
\Theta=\{ v \in \mathbb{R}^{3} | \det \begin{psmallmatrix} n^{k+1} & v & n^{j} \end{psmallmatrix} > 0, 0 \leq j \leq k, \det \begin{psmallmatrix} v & n^{0} & n^{j} \end{psmallmatrix} >0, 1 \leq j \leq k+1 \}.
\end{equation}
and
\begin{equation}
\Theta'=\{ v \in \mathbb{R}^{3} | \det \begin{psmallmatrix} n^{k+1} & v & n^{j} \end{psmallmatrix} > 0, 1 \leq j \leq k, \det \begin{psmallmatrix} v & n^{0} & n^{j} \end{psmallmatrix} >0, 1 \leq j \leq k \}.
\end{equation}
To show the existence of a vector $n^{k+2}$ satisfying $v^{0} \cdot n^{k+2}=1$ and the condition (i), it suffices to show that $\{n \in \mathbb{Z}^{3} | v^{0} \cdot n=1\} \cap \Theta$ is not empty. To show that $\{n \in \mathbb{Z}^{3} | v^{0} \cdot n=1\} \cap \Theta$ is not empty, it suffices to show that $\{n \in \mathbb{Z}^{3} | v^{0} \cdot n=1\} \cap \Theta'$ is not empty, since $\{n \in \mathbb{Z}^{3} | v^{0} \cdot n=1\} \cap \Theta'$ is a subset of $\{n \in \mathbb{Z}^{3} | v^{0} \cdot n=1\} \cap \Theta$. 

We show that if $\{n \in \mathbb{Z}^{3} | v^{0} \cdot n=0\} \cap \Theta'$ is not empty, then $\{n \in \mathbb{Z}^{3} | v^{0} \cdot n=1\} \cap \Theta'$ is not empty. We assume that $\{n \in \mathbb{Z}^{3} | v^{0} \cdot n=0\} \cap \Theta'$ is not empty. Then there exists a sequence of vectors $\{w^{l}\}_{l \in \mathbb{N}}$ in $\{n \in \mathbb{Z}^{3} | v^{0} \cdot n=0\} \cap \Theta'$ such that the ball centered at $w^{l}$ of radius $l$ is contained in $\Theta'$, since $\{n \in \mathbb{Z}^{3} | v^{0} \cdot n=0\} \cap \Theta'$ is an open cone in $\{n \in \mathbb{Z}^{3} | v^{0} \cdot n=0\}$. We take a vector $u$ in $\mathbb{Z}^{3}$ such that $v^{0} \cdot u=1$. We take a positive integer $i_{0}$ so that $i_{0} > \sqrt{u \cdot u}$. Then $w^{i_{0}} + u$ is an element of $\{n \in \mathbb{Z}^{3} | v^{0} \cdot n=1\} \cap \Theta'$. Hence $\{n \in \mathbb{Z}^{3} | v^{0} \cdot n=1\} \cap \Theta'$ is not empty.

Since $\{n \in \mathbb{Z}^{3} | v^{0} \cdot n=0\} \cap \Theta'$ contains $n^{0} + n^{k+1}$, $\{n \in \mathbb{Z}^{3} | v^{0} \cdot n=0\} \cap \Theta'$ is not empty. The proof of Lemma \ref{LastVector} is completed.
\end{proof}

\begin{prop}\label{ConstructionOfToricManifolds}
For an arbitrarily positive integer $k$ greater than $1$, there exists a $5$-dimensional contact toric manifold with a $K$-contact form of rank $2$ and a nontrivial chain $C$ of length $k$ such that we cannot perform a contact blowing down along any lens space in $C$ to a closed orbit of the Reeb flow.
\end{prop}

\begin{proof}
It suffices to construct a cone $\Delta$ defined by normal vectors $\{n^{i}\}_{i=0}^{k+2}$ which satisfies the following conditions:
\begin{enumerate}
\item $(n^{0} \times n^{i})(3) > 0$ for $1 \leq i \leq k+1$,
\item $(n^{i} \times n^{i+1})(3)>0$ for $0 \leq i \leq k$,
\item $\det \begin{psmallmatrix} n^{i} & n^{i+1} & n^{i+2} \end{psmallmatrix} > 0$ for $0 \leq i \leq k-1$,
\item $\det \begin{psmallmatrix} n^{k+1} & n^{k+2} & n^{j} \end{psmallmatrix} > 0$ for $0 \leq j \leq k$ and $\det \begin{psmallmatrix} n^{k+2} & n^{0} & n^{j} \end{psmallmatrix} >0$ for $1 \leq j \leq k+1$,
\item there exists a vector $l^{i}$ in $\mathbb{Z}^{3}$ such that $\det \begin{psmallmatrix} n^{i} & n^{i+1} & l^{i} \end{psmallmatrix}=1$ for $0 \leq i \leq k+2$ where $n^{k+3}=n^{0}$,
\item If we put 
\begin{equation}\label{MatrixA}
\begin{pmatrix} n^{i+2} & l^{i+1} & n^{i+1} \end{pmatrix} = \begin{pmatrix}  n^{i} & l^{i} & n^{i+1} \end{pmatrix} \begin{pmatrix} a_i & b_i & 0 \\ c_i & d_i & 0 \\ e_i & f_i & 1 \end{pmatrix},
\end{equation}
then $e_i$ and $c_i$ is not coprime for $1 \leq i \leq k-2$.
\end{enumerate}
In fact, the conditions (i), (ii), (iii) and (iv) imply that $\det \begin{psmallmatrix} n^{i} & n^{i+1} & n^{j} \end{psmallmatrix} > 0$ for every $i$ and $j$ satisfying $0 \leq i \leq k+2$, $0 \leq j \leq k+2$ and $j \neq i$, $i+1$ by Lemma \ref{FacesOfDelta2}. Hence by Lemma \ref{FacesOfDelta}, the set of the faces of $\Delta$ is $\{\overline{\Delta} \cap P^{i}\}_{i=0}^{k+2}$ and the pair $(\overline{\Delta} \cap P^{i},\overline{\Delta} \cap P^{i+1})$ is adjacent in $\overline{\Delta}$ for $i=0$, $1$, $\cdots$, $k+2$. The condition (iv) implies that $\Delta$ is good. Hence we obtain a contact toric manifold $(M,\xi)$ such that the image of the moment map of the symplectization is $\Delta$ by the Delzant type construction of Boyer-Galicki \cite{BoGa2}. If we define the affine space $A$ by $A=\{ v \in \mathbb{R}^{3} | (n^{0} + \sqrt{2} n^{k+1}) \cdot v =1\}$, then $A$ determines a Reeb vector field on $(M,\xi)$ and hence a $K$-contact form $\alpha$. The rank of $(M,\alpha)$ is $2$ (See \ref{Subsubsection : ContactToricManifolds}). The condition (v) implies that we cannot perform a contact blowing down along any lens space in $M$ which is the inverse image of $\overline{\Delta} \cap P^{i}$ for $i=1$, $2$, $\cdots$, $k$ by Lemma \ref{TopologicalCharacterization} and \ref{GoodCone}. Hence the proof of Proposition \ref{ConstructionOfToricManifolds} is completed.

First, we construct $\{n^{i}\}_{i=0}^{k+1}$ which satisfies the conditions (i), (ii), (iii), (v) and (vi). We put
\begin{equation}
n^0=\begin{pmatrix}
1 \\ 0 \\ 1
\end{pmatrix},
n^1=\begin{pmatrix}
1 \\ 1 \\ 1
\end{pmatrix},
l^0=\begin{pmatrix}
0 \\ 0 \\ 1
\end{pmatrix}.
\end{equation}
Then the conditions (i), (ii) and (v) are satisfied for $n^{0}$ and $n^{1}$. Then conditions (iii) and (vi) are trivially satisfied for $n^{0}$ and $n^{1}$. Assume that we have $\{n^{i}\}_{i=0}^{s}$ for $s < k+1$ so that
\begin{enumerate}
\item $(n^{0} \times n^{i})(3) > 0$ for $1 \leq i \leq s$,
\item $(n^{i} \times n^{i+1})(3) > 0$ for $0 \leq i \leq s-1$,
\item $\det \begin{psmallmatrix} n^{i} & n^{i+1} & n^{i+2} \end{psmallmatrix} > 0$ for $0 \leq i \leq s-2$,
\item there exists a vector $l^{i}$ in $\mathbb{Z}^{3}$ such that $\det \begin{psmallmatrix} n^{i} & n^{i+1} & l^{i} \end{psmallmatrix}=1$ for $0 \leq i \leq s-1$,
\item If we put 
\begin{equation}
\begin{pmatrix} n^{i+2} & l^{i+1} & n^{i+1} \end{pmatrix} = \begin{pmatrix}  n^{i} & l^{i} & n^{i+1} \end{pmatrix} \begin{pmatrix} a_i & b_i & 0 \\ c_i & d_i & 0 \\ e_i & f_i & 1 \end{pmatrix},
\end{equation}
then $e_i$ and $c_i$ is not coprime for $0 \leq i \leq s-2$.
\end{enumerate}
We show that we can take $n^{s+1}$ and $l^{s+1}$ so that $\{n^{i}\}_{i=0}^{s+1}$ satisfies the conditions (i), (ii), (iii), (v) and (vi). We put $n^{s+1}=  an^{s-1} + cl^{s-1} + en^{s}$. Then the conditions (i), (ii), (iii) for $\{n^{i}\}_{i=0}^{s+1}$ are satisfied if
\begin{enumerate}
\item $a (n^{0} \times n^{s-1})(3) + c (n^{0} \times l^{s-1})(3) + e  (n^{0} \times n^{s})(3) > 0$,
\item $a (n^{s} \times n^{s-1})(3) + c (n^{s} \times l^{s-1})(3) > 0$ and
\item $c > 0$.
\end{enumerate}
If $a$ and $c$ are coprime and $c$ and $e$ are not coprime, we can take integers $b$, $d$ and $f$ and a vector $l^{s}$ so that the equation
\begin{equation}
\begin{pmatrix} n^{s+1} & l^{s} & n^{s} \end{pmatrix} = \begin{pmatrix}  n^{s-1} & l^{s-1} & n^{s} \end{pmatrix} \begin{pmatrix} a & b & 0 \\ c & d & 0 \\ e & f & 1 \end{pmatrix}
\end{equation}
and conditions (v), (vi) are satisfied. We can take such $a$, $c$ and $e$ as follows: We put $c=2$. We take a negative odd number $a$ of sufficiently large absolute value so that $a (n^{s} \times n^{s-1})(3) + c (n^{s} \times l^{s-1})(3) > 0$ is satisfied. Note that $(n^{s} \times n^{s-1})(3)$ is negative by the induction hypothesis. We take a positive even number $e$ sufficiently large so that $a (n^{0} \times n^{s-1})(3) + c (n^{0} \times l^{s-1})(3) + e (n^{0} \times n^{s})(3) > 0$ is satisfied. Note that $(n^{0} \times n^{s})(3)$ is positive by the induction hypothesis. Hence we can take $\{n^{i}\}_{i=0}^{k+1}$ so that the conditions (i), (ii), (iii), (v) and (vi) are satisfied. 

By Lemma \ref{LastVector}, we can take $n^{k+2}$ so that 
\begin{enumerate}
\item $\det \begin{psmallmatrix} n^{k+1} & n^{k+2} & n^{j} \end{psmallmatrix} >0$ for every $j$ not equal to $k+1$, $k+2$ and $\det \begin{psmallmatrix} n^{k+2} & n^{0} & n^{j} \end{psmallmatrix} >0$ for every $j$ not equal to $0$, $k+2$, 
\item There exist vectors $l$ and $l'$ in $\mathbb{Z}^{3}$ such that $\det \begin{psmallmatrix} n^{k} & n^{k+1} & l \end{psmallmatrix}=1$ and
$\det \begin{psmallmatrix} n^{k+1} & n^{0} & l' \end{psmallmatrix}=1$.
\end{enumerate}
Then $\{n^{i}\}_{i=0}^{k+2}$ satisfies the conditions (i), (ii), (iii), (iv), (v) and (vi).
\end{proof}

\begin{exam}
We put $n^{i}=\begin{pmatrix} 1 \\ i \\ i^{2} - i + 1 \end{pmatrix}$ for $0 \leq i \leq k+1$ and put $n^{k+2}=\begin{pmatrix} 1 \\ 1 \\ k+2 \end{pmatrix}$. Let $\Delta$ be the good cone defined by the normal vectors $\{n^{i}\}_{i=0}^{k+2}$. We put $A=\{ v \in \mathbb{R}^{3} | (n^{0} + \sqrt{2} n^{k+1}) \cdot v =1\}$. Let $(M,\alpha)$ be the $K$-contact manifold of rank $2$ determined by $\Delta$ and $A$. Then we cannot perform the contact blowing down along the $K$-contact lens space in $(M,\alpha)$ defined by the inverse image of a face of $\Delta$ by the contact moment map to a closed orbit of the Reeb flow. The $K$-contact lens space in $(M,\alpha)$ defined by the inverse images of the face defined by $n^{i}$ of $\Delta$ by the contact moment map is diffeomorphic to $\mathbb{R}P^{3}$ and has a trivial normal bundle in $M$ for $1 \leq i \leq k$. 
\end{exam}

\section{Classification up to contact blowing up and down}

We show Theorem \ref{BlowDownToLensSpaceBundles} which gives the classification of $5$-dimensional $K$-contact manifolds of rank $2$ up to finite times of contact blowing up and down. Theorem \ref{BlowDownToLensSpaceBundles} is shown by using Proposition \ref{Fibersum}, Lemma \ref{fib2} and Lemma \ref{EmbeddingChains} at the end of this section.

\subsection{Germs of chains}

We define germs of chains, which are of combinatorial nature and used for the classification of $5$-dimensional $K$-contact manifolds of rank $2$. 

\begin{defn}\label{DefinitionOfGermsOfChains}(Germs of chains) A germ of a chain is an equivalence class of connected open neighborhoods of chains in $5$-dimensional $K$-contact manifolds of rank $2$: An open neighborhood $U_1$ of a chain $C_1$ in a $5$-dimensional $K$-contact manifold $(N_1,\alpha_1)$ is defined to be equivalent to an open neighborhood $U_2$ of a chain $C_2$ in a $5$-dimensional $K$-contact manifold $(N_2,\alpha_2)$ if there exists an open neighborhood $V_1$ of $C_1$ in $U_1$ and an open neighborhood $V_2$ of $C_2$ in $U_2$ such that $(V_1,\alpha_1|_{V_1})$ is isomorphic to $(V_2,\alpha_2|_{V_2})$ as $K$-contact manifolds. 

A germ of a chain is defined to be trivial if it is the germ of a trivial chain.
\end{defn}

Note that the germ of a chain $C$ is determined by combinatorial data of $C$ and every germ of chains can be embedded in a contact toric manifold by the normal form theorem (See Lemma \ref{NormalFormOfChains} and \ref{ToricActionNearChains}).

We define the contact blowing up and down for germs of chains.

\begin{defn}(Contact blowing up and down for germs of chains) A germ of a chain $C$ is obtained by a contact blowing up from a germ of a chain $\overline{C}$ if $C$ and $\overline{C}$ can be embedded in $K$-contact manifolds $(N,\beta)$ and $(\overline{N},\overline{\beta})$ as chains, respectively, and $(N,\beta)$ is obtained from $(\overline{N},\overline{\beta})$ by a contact blowing up along a closed orbit of the Reeb flow in $\overline{C}$. Contact blowing down is defined to be the inverse operation of the contact blowing up.
\end{defn}

We define the fiber sum, which is a construction to obtain a $K$-contact manifold from germs of chains and a lens space bundle.

\begin{defn}\label{DefinitionOfFiberSum}(Fiber sums) Let $E$ be a lens space bundle over a closed surface $S$ with a $K$-contact structure of rank $2$ whose fibers are $K$-contact submanifolds of rank $2$. Let $\{C^{i}\}_{i=1}^{m}$ be germs of chains. Assume that 
\begin{enumerate}
\item $C^{i}$ has the minimal component and the maximal component of dimension $3$ for each $i=1$, $2$, $\cdots$, $m$ and
\item fibers of $E$ and general fibers of $C^{i}$ are isomorphic as $K$-contact manifolds.
\end{enumerate}
Choose $m$ points $\{x^{i}\}_{i=1}^{m}$ on $S$ and let $F^{i}$ be the fiber of $E$ over $x^i$ for $1 \leq i \leq m$. Then we can attach germs of chains $\sqcup_{i=1}^{m} C^i$ to $E-(\cup_{i=1}^{m}F^i)$ and obtain a $K$-contact manifold $N$ of rank $2$. We call $N$ the fiber sum of a lens space bundle $E$ with $\{C^{i}\}_{i=1}^{m}$.
\end{defn}

Note that lens spaces with a $K$-contact structure of rank $2$ are classified by the maximal and minimal values of the contact moment map and the element corresponding to the Reeb vector field in the Lie algebra of the closure of the Reeb flow in the isometry group by the classification of $3$-dimensional contact toric manifolds by Lerman \cite{Ler2}.

\subsection{Contact blowing up to obtain fiber sums}

We show the following Proposition which is used to show Theorem \ref{BlowDownToLensSpaceBundles}. 

\begin{prop}\label{Fibersum}
Let $(M,\alpha)$ be a closed $5$-dimensional $K$-contact manifold of rank $2$. 
\begin{enumerate}
\item If $(\dim B_{\min}, \dim B_{\max})=(1,1)$, $(M,\alpha)$ is isomorphic to the fiber sum of a lens space bundle over a closed surface with $k$ germs of chains for some $k$ after performing contact blowing up along a closed orbit of the Reeb flow twice.
\item If $(\dim B_{\min}, \dim B_{\max})=(1,3)$ or $(3,1)$, $(M,\alpha)$ is isomorphic to the fiber sum of a lens space bundle over a closed surface with $k$ germs of chains for some $k$ after performing contact blowing up along a closed orbit of the Reeb flow once.
\end{enumerate}
\end{prop}

Proposition \ref{Fibersum} is a consequence of the following Lemmas \ref{fib1} and \ref{fib2}. Let $(M,\alpha)$ be a closed $5$-dimensional $K$-contact manifold of rank $2$. We denote by $\rho$ the $T^2$-action associated with $\alpha$. We put $\Phi=\alpha(X)$ where $X$ is an infinitesimal action of $\rho$ which is not parallel to $R$. The minimal component and the maximal component of $\Phi$ are denoted by $B_{\min}$ and $B_{\max}$. Note that $B_{\min}$ and $B_{\max}$ are of dimension $1$ or $3$ by Lemma \ref{MorseTheory}.

\begin{lem}G\label{fib1}
If $B_{\min}$ is of dimension $1$, then we can perform a contact blowing up along $B_{\min}$ so that the new minimal component of the contact moment map is of dimension $3$.
\end{lem}

\begin{proof}
We show Lemma \ref{fib1} in the case where the isotropy group of $\rho$ at $B_{\min}$ is connected. General cases are reduced to this case by taking a finite covering of a neighborhood of $B_{\min}$ by Lemma \ref{FiniteCovering}. By the argument in Subsection \ref{ConnectedIsotropyGroupCases}, we have a tubular neighborhood $V$ of $B_{\min}$ such that $(V,\alpha|_{V})$ is isomorphic to $(S^1 \times D^{4}_{\epsilon},\alpha_{0})$ defined as follows:

$\alpha_{0}$ is defined by
\begin{equation}
\alpha_{0}=\frac{1-\lambda_1(m_1|z_1|^2+m_2|z_2|^2)}{\lambda_0}d\zeta+\frac{\sqrt{-1}}{2}(z_1d\overline{z}_1-\overline{z}_1dz_1)+\frac{\sqrt{-1}}{2}(z_2 d\overline{z}_2 - \overline{z}_{2} dz_2)
\end{equation}
for a pair of coprime integers $(m_1,m_2)$ and a pair of real numbers $(\lambda_0,\lambda_1)$ which are linearly independent over $\mathbb{Q}$. The $T^2$-action $\rho_{0}$ associated with $\alpha_{0}$ is written as
\begin{equation}
(t_0,t_1) \cdot (\zeta,z_1,z_2)=(t_{0}\zeta,t_{1}^{m_1}z_1,t_{1}^{m_2}z_2)
\end{equation}
and the $S^1$-subaction of $\{ (1, t_1)|t_1 \in S^1\}$ is the isotropic action of $\rho$ at $B_{\min}$. Note that $m_1$ and $m_2$ are nonzero, since $B_{\min}$ is of dimension $1$.

We perform a contact blowing up by cutting $S^1 \times D^{4}_{\epsilon}$ at $U=\{(\zeta,z_1,z_2) \in S^1 \times D^{4}_{\epsilon} | r_1|z_1|^2+r_2|z_2|^2=1 \}$ for positive numbers $r_1$ and $r_2$ greater than $\epsilon^{-2}$. To perform a contact blowing up along $B_{\min}$ to obtain a new minimal component of dimension $3$, it suffices to take real numbers $r_1$ and $r_2$ which satisfy the following conditions: There exists a real number $N_{0}$ such that
\begin{enumerate}
\item  $N_{0}\begin{psmallmatrix} \lambda_0 \\ \lambda_1 m_1 - r_1 \\ \lambda_1 m_2 - r_2 \end{psmallmatrix}$ is a vector in $\mathbb{Z}^{3}$ so that the vector field $X$ defined by the equation \eqref{X} generates an $S^1$-action $\sigma$ on $U$,
\item $\GCD(N_{0}\lambda_0, N_{0}(\lambda_1 m_1 - r_1))=1, \GCD(N_{0}\lambda_0, N_{0}(\lambda_1 m_2 - r_2))=1$ are satisfied so that the $S^1$-action $\sigma$ on $U$ generated by $X$ is free and
\item $N_{0}(\lambda_1 m_1 - r_1)m_2 - N_{0}(\lambda_1 m_2 - r_2)m_1=0$ is satisfied so that $\sigma$ is a subaction of $\rho$ and hence the new minimal component of the contact moment map is of dimension $3$.
\end{enumerate}
Note that the equation in the condition (iii) means $\rho$ degenerates to an $S^1$-action on the lens space $U/\sigma$ so that $U/\sigma$ becomes a new minimal component. 

We show that for any $\epsilon$ there exist $r_1$ and $r_2$ which are greater than $\epsilon^{-2}$ and satisfy the above conditions. We put $r_1=lm_1$ and $r_2=lm_2$ where $l$ is a real number. Then the condition (iii) is satisfied. $\begin{psmallmatrix} \lambda_0 \\ \lambda_1 m_1 - r_1 \\ \lambda_1 m_2 - r_2 \end{psmallmatrix}$ is tangent to $\begin{psmallmatrix} 1 \\ \frac{\lambda_1 - l}{\lambda_0} m_1 \\ \frac{\lambda_1 - l}{\lambda_0} m_2 \end{psmallmatrix}$. Hence the conditions (i) and (ii) are satisfied, if we can choose $l$ so that $\frac{\lambda_1 - l}{\lambda_0}$ is a rational number $\frac{u}{v}$ for coprime integers $u,v$ and $\GCD(v,m_1)=\GCD(v,m_2)=1$. It is possible to choose $l$ so that $\GCD(v,m_1)=\GCD(v,m_2)=1$ and $r_1$, $r_2$ are greater than $\epsilon^{-2}$ for any $\epsilon$. Then $r_1=lm_{1}$ and $r_2=lm_{2}$ satisfy the conditions (i), (ii) and (iii).
\end{proof}

\begin{lem}\label{fib2}
If both of $B_{\min}$ and $B_{\max}$ are of dimension $3$, then $(M,\alpha)$ is isomorphic to the fiber sum of a lens space bundle over a closed surface with a finite number of germs of chains.
\end{lem}

\begin{proof}
Let $C^{1},C^{2},\cdots,C^{m}$ be nontrivial chains of $(M,\alpha)$. By Lemma \ref{Linearization}, there exists a metric $g_{s}$ compatible with $\alpha$ on open neighborhoods of $B_{s}$ respectively such that the gradient flow of $\Phi$ with respect to $g_{s}$ is conjugate to an $\mathbb{R}$-subaction of the $\mathbb{C}^{\times}$-action on the normal bundles of $B_{s}$ defined by a linear complex structure for $s=\min$ and $\max$, respectively. We fix a metric $g$ on $M$ compatible with $\alpha$ so that the restriction of $g$ on an open neighborhood of $B_{s}$ coincides with $g_{s}$ for $s=\min$ and $\max$, respectively.

We show that closures of gradient manifolds with respect to $g$ are leaves of a smooth foliation $\mathcal{F}$ of $M - \sqcup_{i=1}^{m}C^{i}$. Since gradient manifolds are orbits of an effective $(\mathbb{R} \times T^2)$-action on $M - \sqcup_{i=1}^{m}C^{i} - B_{\max} - B_{\min}$, they are leaves of a smooth foliation on $M - \sqcup_{i=1}^{m}C^{i} - B_{\max} - B_{\min}$. Near $B_{s}$, the closures of gradient manifolds are mapped to fibers of the normal bundle of $B_{s}$ by the construction of $g$ for $s=\min$ and $\max$, respectively. Hence closures of gradient manifolds with respect to $g$ define a smooth foliation on $M - \sqcup_{i=1}^{m}C^{i}$.

We show that closures of gradient manifolds with respect to $g$ are fibers of a proper submersion. It suffices to show that the foliation $\mathcal{F}$ defined by closures of gradient manifolds has no nontrivial holonomy. Let $F$ be a leaf of $\mathcal{F}$ and fix a point $x_{0}$ on $F$. By Lerman's classification, $F$ is diffeomorphic to a lens space and the generator of $\pi_{1}(F,x_{0})$ is an orbit of an $S^1$-subaction of $\rho$. Since the isotropy group of $\rho$ at $F$ is trivial, $F$ has no nontrivial holonomy. Hence $\mathcal{F}$ has no nontrivial holonomy.

The closure of each gradient manifold is a $3$-dimensional $K$-contact submanifold of rank $2$ by Lemmas \ref{MorseTheory} and \ref{MorseBottTheory}. $3$-dimensional toric $K$-contact manifolds are classified by the element in $\Lie(T^2)$ corresponding to the Reeb vector field, the maximal value and the minimal value of the contact moment map by Lerman \cite{Ler2}. Hence they are isomorphic to each other.

We obtain a lens space bundle $E$ over a closed manifold with a $K$-contact form by replacing $\sqcup_{i=1}^{m}C^{i}$ to trivial germs of chains. Then $(M,\alpha)$ is isomorphic to a fiber sum of $E$ with $\sqcup_{i=1}^{m}C^{i}$.
\end{proof}

\subsection{Contact blowing down in contact toric manifolds}

We defined contact blowing up and down using local $T^3$-actions in the previous section. The computation is simpler, if we have global $T^3$-actions. We show a combinatorial lemma by using Dirichlet prime number theorem and apply it to perform contact blowing down for germs of chains for the proof of Theorem \ref{BlowDownToLensSpaceBundles}. 

We denote the standard inner product on $\mathbb{R}^{3}$ by $\cdot$. We identify $\mathbb{R}^{3}$ with the dual of $\mathbb{R}^{3}$ by the inner product. 

\begin{lem}\label{Toric1}
Let $x=\begin{psmallmatrix} x_1 \\ x_2 \\ x_3 \end{psmallmatrix}$ and $y=\begin{psmallmatrix} y_1 \\ y_2 \\ y_3 \end{psmallmatrix}$ be primitive vectors in $\mathbb{Z}^{3}$. Assume that there exists a primitive vector $z$ in $\mathbb{Z}^{3}$ which satisfies $\det \begin{psmallmatrix} x & y & z \end{psmallmatrix}=1$. Let $\Theta$ be a nonempty cone in $\mathbb{R}^{3}$ defined by 
\begin{equation}
\Theta=\{v \in \mathbb{R}^{3} | n^1 \cdot v > 0, n^2 \cdot v > 0, n^3 \cdot v < 0\}
\end{equation}
for primitive vectors $n^1,n^2,n^3$ in $\mathbb{Z}^{3}$ defined by
\begin{equation}
\begin{array}{l}
n^1= x \times y, n^2= y \times \begin{psmallmatrix} 0 \\ 0 \\ 1 \end{psmallmatrix}, n^3 = x \times \begin{psmallmatrix} 0 \\ 0 \\ 1 \end{psmallmatrix}.
\end{array}
\end{equation}
Assume that $y_1 x_3 - x_1 y_3$ and $y_1 x_2 - x_1 y_2$ are nonzero. Then there exists an element $t=\begin{psmallmatrix} t_1 \\ t_2 \\ t_3 \end{psmallmatrix}$ of $\Theta \cap \mathbb{Z}^{3}$ such that 
\begin{description}
\item[(a)] there exists a primitive vector $u$ in $\mathbb{Z}^{3}$ such that $\det \begin{psmallmatrix} x & t & u \end{psmallmatrix} = 1$,
\item[(b)] $\GCD(t_1,t_2)=1$ and $\GCD(t_1,t_3)=1$ are satisfied.
\end{description}
\end{lem}

\begin{proof}
We put $t=s_1 x + s_2 y + s_3 z$ for integers $s_1,s_2$ and $s_3$. We have
\begin{equation}
\begin{array}{l}
t \cdot n^1 = s_3 z \cdot (x \times y) = s_3, \\
t \cdot n^2 = s_1 x \cdot (y \times \begin{psmallmatrix} 0 \\ 0 \\ 1 \end{psmallmatrix}) + s_3 z \cdot (y \times \begin{psmallmatrix} 0 \\ 0 \\ 1 \end{psmallmatrix}), \\
t \cdot n^3 = s_2 y \cdot (x \times \begin{psmallmatrix} 0 \\ 0 \\ 1 \end{psmallmatrix}) + s_3 z \cdot (x \times \begin{psmallmatrix} 0 \\ 0 \\ 1 \end{psmallmatrix}).
\end{array}
\end{equation}
Note that $x \cdot (y \times \begin{psmallmatrix} 0 \\ 0 \\ 1 \end{psmallmatrix})=y_1 x_2 - x_1 y_2$. We define $\epsilon=1$ if $y_1 x_2 - x_1 y_2$ is positive and $\epsilon=-1$ if $y_1 x_2 - x_1 y_2$ is negative. $t$ is an element of $\Theta$ if $s_3=1$ and $\epsilon s_1$ and $\epsilon s_2$ are sufficiently large.

We translate the condition (a) in terms of $s_1,s_2$ and $s_3$. Define the matrix $A$ by $A=\begin{psmallmatrix} x & y & z \end{psmallmatrix}$. Then $t$ satisfies the condition (a) if and only if there exists a primitive vector $u'$ in $\mathbb{Z}^{3}$ such that $\det \begin{psmallmatrix} A^{-1}x & A^{-1}t & u' \end{psmallmatrix} = 1$. Since $A^{-1}x=\begin{psmallmatrix} 1 \\ 0 \\ 0 \end{psmallmatrix}$ and $A^{-1}t=\begin{psmallmatrix} s_1 \\ s_2 \\ s_3 \end{psmallmatrix}$, the latter condition is equivalent to $\GCD(s_2,s_3)=1$.

Hence $t=s_1 x + s_2 y + z$ satisfies the conditions if $\epsilon s_1$ and $\epsilon s_2$ are sufficiently large and $\GCD(t_1,t_2)=1$ and $\GCD(t_1,t_3)=1$. 

We take positive integers $s_{1}^{0}$ and $s_{2}^{0}$ so that $s_{1}^{0} x_1+ s_{2}^{0} y_1$ is coprime to $z_1$. By the assumption, $y_1 x_2 - x_1 y_2$ and $y_1 x_3 - x_1 y_3$ are nonzero. Hence we have a sufficiently large positive integer $c$ such that $c( \epsilon s_{1}^{0} x_1+ \epsilon s_{2}^{0} y_1) + z_1$ is equal to a prime number $q$ which is coprime to $y_1 x_2 - x_1 y_2$ and $y_1 x_3 - x_1 y_3$ by the Dirichlet prime number theorem. Then one of 
\begin{equation}
\begin{psmallmatrix}
s_1 \\ s_2 \\ s_3
\end{psmallmatrix}=
\begin{psmallmatrix}
c \epsilon s_{1}^{0} \\ c \epsilon s_{2}^{0} \\ 1
\end{psmallmatrix},
\begin{psmallmatrix}
c \epsilon s_{1}^{0} - y_1 \\ c \epsilon s_{2}^{0} + x_1 \\ 1
\end{psmallmatrix},
\begin{psmallmatrix}
c \epsilon s_{1}^{0} - 2y_1 \\ c \epsilon s_{2}^{0} + 2 x_1 \\ 1
\end{psmallmatrix}
\end{equation}
satisfies the required conditions.
\end{proof}

\begin{defn}(Delzant conditions and good cones) We say two primitive vectors $n^1$ and $n^2$ in $\mathbb{Z}^{3}$ satisfy the Delzant condition if there exists a vector $v$ in $\mathbb{Z}^{3}$ such that $\det \begin{psmallmatrix} n^1 & n^2 & v \end{psmallmatrix}=1$. We say two planes defined by primitive normal vectors $n^1$ and $n^2$ in $\mathbb{Z}^{3}$ satisfy the Delzant condition if $n^1$ and $n^2$ satisfy the Delzant condition. A closed cone $\overline{\Delta}$ in $\mathbb{R}^{3}$ with nonempty interior is a good cone if every pair of adjacent planes of $\overline{\Delta}$ satisfies the Delzant condition (See \cite{Ler2}).
\end{defn}

We obtain the following by Lemma \ref{Toric1}:
\begin{lem}\label{Toric2}
Let $P^0,P^1,P^2$ and $P^3$ be planes in $\mathbb{R}^{3}$ defined by primitive normal vectors $n^0,n^1,n^2$ and $n^3$ in $\mathbb{Z}^{3}$ respectively. Assume that $P^0 \cap \overline{\Delta}$, $P^1 \cap \overline{\Delta}$, $P^2 \cap \overline{\Delta}$ and $P^3 \cap \overline{\Delta}$ are four faces of a good cone $\overline{\Delta}$ so that $P^i \cap \overline{\Delta}$ and $P^{i+1} \cap \overline{\Delta}$ are adjacent in $\overline{\Delta}$ for $i=0,1$ and $2$. Let $\Theta$ be a cone in $\mathbb{R}^{3}$ defined by 
\begin{equation}
\Theta=\{v \in \mathbb{R}^{3} | (n^0 \times n^1) \cdot v > 0, (n^1 \times n^2) \cdot v > 0, (n^0 \times n^2) \cdot v < 0\}.
\end{equation}
Then we have a plane $Q$ defined by a primitive normal vector $t$ in $\Theta \cap \mathbb{Z}^{3}$ such that $(Q,P^0)$, $(Q,P^2)$ and $(Q,P^3)$ satisfy the Delzant condition.
\end{lem}

\begin{proof}
Note that two primitive vectors $v^1$ and $v^2$ satisfy the Delzant condition if and only if there exists an element $A$ of $\SL(3;\mathbb{Z})$ such that $A v^{1}= \begin{psmallmatrix} 0 \\ 1 \\ 0 \end{psmallmatrix}$ and  $A v^{2}= \begin{psmallmatrix} 0 \\ 0 \\ 1 \end{psmallmatrix}$. Hence we can assume that $n^{2}= \begin{psmallmatrix} 0 \\ 0 \\ 1 \end{psmallmatrix}$ and $n^{3}= \begin{psmallmatrix} 0 \\ 1 \\ 0 \end{psmallmatrix}$. Note that we have $(An^0 \times An^1) \cdot Av=\det \begin{psmallmatrix} An^0 & An^1 & Av \end{psmallmatrix}= \det \begin{psmallmatrix} n^0 & n^1 & v \end{psmallmatrix}=(n^0 \times n^1) \cdot v$. 

We put $x=n^0$, $y=n^1$. We show that the assumptions of Lemma \ref{Toric1} are satisfied. $x$ and $y$ satisfies the Delzant condition by the assumption. We have $y_1 x_3 - x_1 y_3=(n^0 \times n^1) \cdot n^3$ and $y_1 x_2 - x_1 y_2=- (n^0 \times n^1) \cdot n^2$. $n^0 \times n^1$ is a vector contained in $\overline{\Delta}$. Among the vectors in $\overline{\Delta}$, the point $v$ satisfying the equation $v \cdot n^i=0$ is contained in the face defined by $n^i$ by the convexity of $\overline{\Delta}$. Since $n^0 \times n^1$ is not contained in the face defined by $n^2$ nor $n^3$, $y_1 x_3 - x_1 y_3$ and $y_1 x_2 - x_1 y_2$ are nonzero. Hence the assumptions of Lemma \ref{Toric1} are satisfied.

Hence there exists a vector $t = \begin{psmallmatrix} t_1 \\ t_2 \\ t_3 \end{psmallmatrix}$ in $\Theta \cap \mathbb{Z}^{3}$ which satisfies $\GCD(t_1,t_2)=1$, $\GCD(t_1,t_3)=1$ and the Delzant condition with $n^{1}$. Then the plane $Q$ defined by $t$ satisfies the Delzant condition with $P^0$, $P^2$ and $P^3$.
\end{proof}

We deduce a useful lemma from Lemma \ref{Toric2} to perform contact blowing down for germs of chains. 

For an open cone $\Delta$, let $\overline{\Delta}$ be its closure. Let $k$ be an integer greater than $2$. Let $\overline{\Delta}$ be a good cone defined by primitive normal vectors $\{n^i\}_{i=0}^{k}$ in $\mathbb{Z}^{3}$. We put $n^{k+1}=n^{0}$. The plane defined by $n^i$ is denoted by $P^i$ for $0 \leq i \leq k+1$. We assume that the faces $P^{i} \cap \overline{\Delta}$ and $P^{i+1} \cap \overline{\Delta}$ are adjacent for $0 \leq i \leq k$. Let $\Delta_1$ be the cone defined by $\{n^{0}\} \cup \{n^{i}\}_{i=k'}^{k}$ for an integer $k'$ greater than $1$.

\begin{lem}\label{Toric4}
There exists a plane $Q$ defined by a primitive normal vector $l$ in $\mathbb{Z}^{3}$ which satisfies the following:
\begin{enumerate}
\item $Q$ intersects $\Delta_1$ and does not intersect $\overline{\Delta}$.
\item $(Q,P^0)$ and $(Q,P^k)$ satisfy the Delzant condition.
\item Let $\overline{\Delta}_2$ be the good cone defined by $\{n^{0}\} \cup \{l\} \cup \{n^{i}\}_{i=k'}^{k}$. Then the contact toric manifold corresponding to $\overline{\Delta}_{2}$ is obtained from the contact toric manifold corresponding to $\overline{\Delta}$ by a finite sequence of contact blowing down. 
\end{enumerate}
\end{lem}

\begin{proof}
We show the case where $k'=2$. If $k'=2$, our claim directly follows from Lemma \ref{Toric2}. In fact, we apply Lemma \ref{Toric2} by substituting $P^0$, $P^1$ and $P^2$ to $P^0$, $P^1$ and $P^2$ in Lemma \ref{Toric4}, we have a plane $Q$ which satisfies the conditions (i) and (ii). Note that we can ignore $P^3$. By the Delzant type construction of Boyer-Galicki \cite{BoGa2}, we have contact toric manifolds $(M,\xi)$ and $(M_{2},\xi_{2})$ such that the images of the symplectic moment maps of the symplectizations of $(M,\xi)$ and $(M_{2},\xi_{2})$ are $\Delta$ and $\overline{\Delta}_2$ respectively. $\overline{\Delta}_{2}$ is obtained from $\overline{\Delta}$ by attaching a rectangle which is an open neighborhood of $\overline{\Delta}_{2} \cap Q$ in $\overline{\Delta}_{2}$. We can perform a contact blowing down for $(M,\xi)$ along the lens space which is the inverse image of $\overline{\Delta} \cap P^1$ by the contact moment map of $(M,\xi)$ so that the image of the contact moment map change from $\overline{\Delta}$ to $\overline{\Delta}_{2}$. Then the contact toric manifold obtained after the contact blowing up is isomorphic to $(M_{2},\xi_{2})$ by a result of Lerman \cite{Ler2}, since the images of the contact moment maps are the same and convex. Hence $(M_{2},\xi_{2})$ is obtained from $(M,\xi)$ by a contact blowing up. $Q$ satisfies the condition (iii).

Assume that Lemma \ref{Toric4} is true for the case of $k'=s-1$. We consider the case of $k'=s$. Applying Lemma \ref{Toric2}, we can take a plane $Q^1$ defined by a primitive vector $l$ in $\mathbb{Z}^{3}$ such that 
\begin{enumerate}
\item $l$ is contained in the cone $\Theta=\{v \in \mathbb{R}^{3} | (n^{0} \times n^{1}) \cdot l >0, (n^{1} \times n^{2}) \cdot l >0, (n^{0} \times n^{2}) \cdot l <0\}$ and 
\item $l$ satisfies the Delzant condition with $n^0$, $n^2$ and $n^3$ respectively.
\end{enumerate}
Then we have
\begin{enumerate}
\item $Q^1$ intersects $\Delta_1$,
\item $Q^1$ does not intersect $\overline{\Delta}$ and
\item $Q^1$ intersects $P^0,P^2$ and $P^3$ as in the left picture in Figure \ref{Figure} satisfying the Delzant condition with $P^0,P^2$ and $P^3$ respectively.
\end{enumerate}

Let $\overline{\Delta}_{3}$ be the cone defined by normal vectors $\{n^{0}\} \cup \{l\} \cup \{n^i\}_{i=2}^{k}$. Let $\overline{\Delta}_{4}$ be the cone defined by normal vectors $\{n^{0}\} \cup \{l\} \cup \{n^i\}_{i=3}^{k}$. By the Delzant type construction of Boyer-Galicki \cite{BoGa2}, we have contact toric manifolds $(M,\xi)$, $(M_{2},\xi_{2})$, $(M_{3},\xi_{3})$ and $(M_{4},\xi_{4})$ such that the images of the symplectic moment maps of the symplectizations of $(M,\xi)$, $(M_{2},\xi_{2})$, $(M_{3},\xi_{3})$ and $(M_{4},\xi_{4})$ are $\overline{\Delta}$, $\overline{\Delta}_{2}$, $\overline{\Delta}_{3}$ and $\overline{\Delta}_{4}$ respectively. As in the case where $k=2$, we can show that $(M,\xi)$ is obtained from $(M_{3},\xi_{3})$ by a contact blowing up. Similarly $(M_{3},\xi_{3})$ is obtained from $(M_{4},\xi_{4})$ by a contact blowing up. By the induction hypothesis, $(M_{4},\xi_{4})$ is obtained from $(M_{2},\xi_{2})$ by a finite sequence of contact blowing down. Hence $(M_{2},\xi_{2})$ is obtained from $(M,\xi)$ by a finite sequence of contact blowing down.
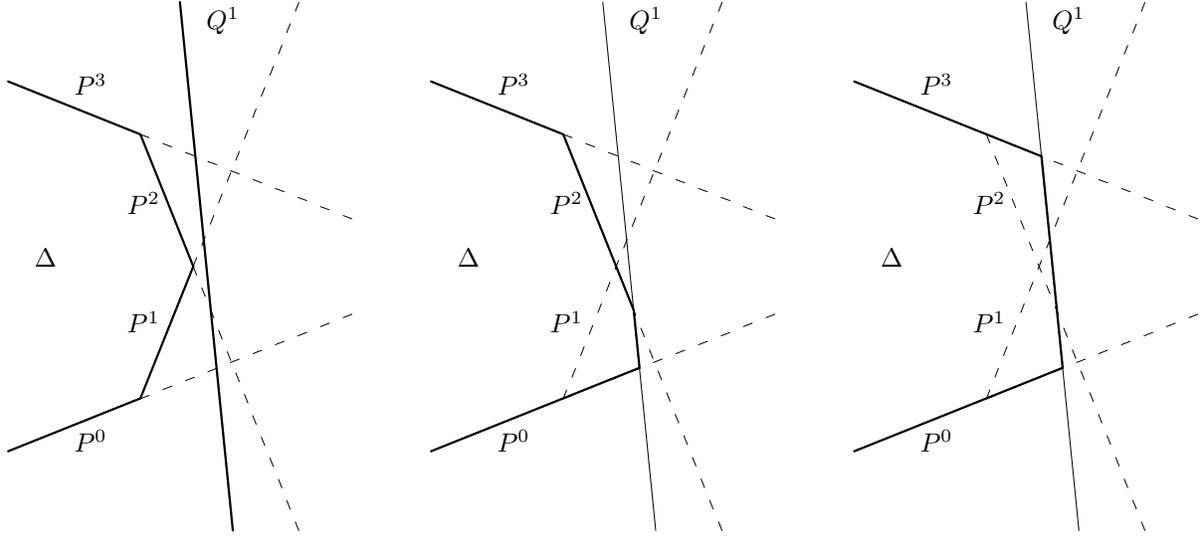
\begin{figure}
\begin{equation*}
\begin{array}{ccc}
\begin{picture}(150,200)(50,-100)
\put(10,0){$\Delta$}
\put(25,65){$P^3$}
\put(45,20){$P^2$}
\put(45,-25){$P^1$}
\put(25,-70){$P^0$}
\put(75,90){$Q^1$}
\dashline{4.000}(50,50)(130,18)
\dashline{4.000}(70,0)(110,-100)
\dashline{4.000}(70,0)(110,100)
\dashline{4.000}(50,-50)(130,-18)
\thicklines
\path(0,70)(50,50)(70,0)(50,-50)(0,-70)
\path(65,100)(85,-100)
\end{picture}
&
\begin{picture}(150,200)(50,-100)
\put(10,0){$\Delta$}
\put(25,65){$P^3$}
\put(45,20){$P^2$}
\put(45,-25){$P^1$}
\put(25,-70){$P^0$}
\put(75,90){$Q^1$}
\path(65,100)(85,-100)
\dashline{4.000}(50,50)(130,18)
\dashline{4.000}(50,50)(110,-100)
\dashline{4.000}(50,-50)(110,100)
\dashline{4.000}(50,-50)(130,-18)
\thicklines
\path(0,70)(50,50)(76.666,-16.666)(78.846,-38.46)(0,-70)
\end{picture}
&
\begin{picture}(150,200)(50,-100)
\put(10,0){$\Delta$}
\put(25,65){$P^3$}
\put(45,20){$P^2$}
\put(45,-25){$P^1$}
\put(25,-70){$P^0$}
\put(75,90){$Q^1$}
\path(65,100)(85,-100)
\dashline{4.000}(50,50)(130,18)
\dashline{4.000}(50,50)(110,-100)
\dashline{4.000}(50,-50)(110,100)
\dashline{4.000}(50,-50)(130,-18)
\thicklines
\path(0,70)(70.833,41.666)(78.846,-38.46)(0,-70)
\end{picture}
\end{array}
\end{equation*}
\caption{Contact blowing down.}
\label{Figure}
\end{figure}
Hence the proof of Lemma \ref{Toric4} is completed.
\end{proof}

We remark on realization of contact blowing up and down for germs of chains in $K$-contact manifolds. We must be careful on realization of desired contact blowing up on a germ of a chain of $(M,\alpha)$ in $M$. It is because we need a large open neighborhood coordinate of closed orbits of the Reeb flow or lens spaces in $M$ where the toric contact form is in the normal form, which contains the necessary hypersurface. On the contrary, we can realize any desired contact blowing down in $M$. 

\subsection{Proof of the classification Theorem \ref{BlowDownToLensSpaceBundles}}

We show Theorem \ref{BlowDownToLensSpaceBundles} as a corollary of Proposition \ref{Fibersum} and Lemma \ref{EmbeddingChains}.

\begin{lem}\label{EmbeddingChains}
Let $C$ be a germ of a chain. Assume that the maximal component and the minimal component of $C$ are of dimension $3$. Then we have
\begin{enumerate}
\item $C$ can be realized as a nontrivial chain in a closed $5$-dimensional $K$-contact toric manifold of rank $2$.
\item $C$ is isomorphic to a germ of a chain with one gradient manifold by a finite sequence of contact blowing down.
\item Assume that $C$ has only one gradient manifold $L$. Then we can perform contact blowing down along $L$ to obtain a trivial germ of a chain.
\end{enumerate}
\end{lem}

\begin{proof}
We show (i). By the assumption, the closure of every gradient manifold in $C$ is a smooth submanifold by Lemma \ref{MorseBottTheory}. Let $\alpha$ be the contact form defined near the chain of $C$. By Lemma \ref{ToricActionNearChains}, there exists an open neighborhood $U$ of the chain and an $\alpha$-preserving $T^3$-action $\tau$ on $U$. Then we have a contact moment map $\Phi_{\alpha}$ for $\tau$ from $U$ to $\Lie(T^3)^{*}$. Let $A$ be the plane $\{ v \in \Lie(T^3)^{*} |v(\overline{R})=1\}$ in $\Lie(T^3)^{*}$ where $\overline{R}$ is the element of $\Lie(T^3)$ whose infinitesimal action is the Reeb vector field of $\alpha$. The image of $\Phi_{\alpha}$ is contained in $A$. Let $P_{\max}$ and $P_{\min}$ be two planes in $\Lie(T^3)^{*}$ such that the germs of the maximal component and the minimal component of $C$ are the inverse images of $P_{\min} \cap A$ and $P_{\max} \cap A$ by $\Phi_{\alpha}$ respectively. Let $\{P^{i}\}_{i=1}^{m}$ be the planes in $\Lie(T^3)^{*}$ such that the inverse images of $P^{1} \cap A, P^{2} \cap A, \cdots, P^{m} \cap A$ of $\Phi_{\alpha}$ are the gradient manifolds in $C$. By Lemma \ref{LastVector}, we have a plane $Q$ in $\Lie(T^3)^{*}$ such that $Q$ intersects $P_{\max}$ and $P_{\min}$ satisfying the Delzant condition and does not intersect $\cup_{i=1}^{m} P^{i} \cap A$. Since we have a good cone with faces $\{P_{\min},P_{\max},Q\} \cup \{P^{i}\}_{i=1}^{m}$ defining a closed contact toric manifold by the Delzant type construction theorem of Boyer-Galicki \cite{BoGa2}.

(ii) is shown by applying Lemma \ref{Toric4} to the set of planes $\{P_{\min},P_{\max}\} \cup \{P^{i}\}_{i=1}^{m}$ where $\{P^{0},P^{1},\cdots,P^{k-1},P^{k}\}$ in Lemma \ref{Toric4} reads as $\{P_{\min},P^{1},\cdots,P^{k-1},P_{\max}\}$.

We prove (iii). Let $G$ be the closure of the Reeb flow in the isometry group for a compatible metric and $\rho$ be the $G$-action associated with $\alpha$. It suffices to show that we can perform contact blowing down along $L$ so that the isotropy group of $\rho$ at the unique gradient manifold of $C$ becomes trivial. We compute the cardinality of the isotropy group of the gradient manifolds which we obtain after contact blowing down. We identify $\Lie(T^3)$ with $\mathbb{R}^{3}$. Let $v^{0}$ be a primitive normal vector defining $\Lie(G)$.

Let $Q$ be the plane in $\Lie(T^3)^{*}$ defined by a primitive normal vector $l$. Let $L_{Q}$ be the lens space obtained by performing contact blowing down for $C$ along $L$ by the plane $Q$ if it is possible. The cardinality of the isotropy group of $\rho$ at $L_{Q}$ is equal to $|v_{0}(l)|$. In fact, since the $T^3$-action is effective, the cardinality of the isotropy group of $\rho$ at $L_{Q}$ is equal to the number of intersection points of $G$ and the $S^1$-subgroup generated by $l$ in $T^3$. This number is equal to $|v_{0}(l)|$. We follow the proof of Lemma \ref{LastVector}.  Let $P^{1},P^{2}$ and $P^{3}$ be the three planes in $\Lie(T^3)^{*}$ corresponding to the germs of the maximal component, the gradient manifold $L$ and the germ of the minimal component respectively. Let $n^i$ be the primitive normal vector which defines $P^i$ for $i=1$, $2$ and $3$. We put the cone $\Theta$ in $\mathbb{R}^{3}$ by
\begin{equation}
\Theta=\{ v \in \mathbb{R}^{3} | \det \begin{psmallmatrix} n^{1} & n^{2} & v \end{psmallmatrix} > 0, \det \begin{psmallmatrix} n^{2} & n^{3} & v \end{psmallmatrix} > 0 \}.
\end{equation}
If $l$ is a vector contained in $\Theta$ which satisfies the Delzant condition with $n^{1}$ and $n^{3}$, then we can perform a contact blowing down along $L$ to $L_{Q}$. To show (iii), it suffices to show that $\Theta \cap \{n \in \mathbb{R}^{3}| v^{0}(n)=1\}$ is nonempty. By the argument of the proof of Lemma \ref{LastVector}, to show that $\Theta \cap \{n \in \mathbb{R}^{3}| v^{0}(n)=1\}$ is nonempty, it suffices to show that $\Theta \cap \{n \in \mathbb{R}^{3}| v^{0}(n)=0\}$ is nonempty. $\Theta \cap \{n \in \mathbb{R}^{3}| v^{0}(n)=0\}$ is nonempty, because $n^{1} + n^{3}$ is contained in $\Theta \cap \{n \in \mathbb{R}^{3}| v^{0}(n)=0\}$. Hence (iii) is proved. Hence the proof of Lemma \ref{EmbeddingChains} is completed.
\end{proof}

We show Theorem \ref{BlowDownToLensSpaceBundles}. Let $(M,\alpha)$ be a closed $5$-dimensional $K$-contact manifold of rank $2$.

\begin{proof}
If (i) is shown, then (ii) and (iii) are shown by Proposition \ref{Fibersum}. We show (i). By Lemma \ref{fib2}, it suffices to show that we can perform contact blowing down to the germ of every nontrivial chain of $(M,\alpha)$ to obtain a trivial germ of a chain. It is possible by Lemma \ref{EmbeddingChains} (ii) and (iii).
\end{proof}

\section{Existence of compatible Sasakian metrics}\label{Sasakian}

In this section, we show Theorem \ref{CompatibleMetric} which asserts the existence of compatible Sasakian metrics for a closed $5$-dimensional $K$-contact manifold $(M,\alpha)$ of rank $2$. We define Sasakian metrics.
\begin{defn}\label{DefinitionOfSasakianMetrics}(Sasakian metrics) A pair $(g,\alpha)$ of a Riemannian metric $g$ on $M$ and a contact form $\alpha$ is a Sasakian metric on $M$ if $(M \times \mathbb{R}_{>0},r^2 g+dr \otimes dr)$ is a K\"{a}hler manifold with K\"{a}hler form $d(r^2 \alpha)$ where $\mathbb{R}_{>0}$ is the set of positive real numbers and $r$ is the standard coordinate on $\mathbb{R}_{>0}$
\end{defn}
We refer \cite{BoGa5} for Sasakian manifolds and basic terminology for orbifolds used in this section.

We denote the torus action associated with $\alpha$ by $\rho$. We put $\Phi=\alpha(X)$ where $X$ is an infinitesimal action of $\rho$ which is not parallel to the Reeb vector field $R$. The maximal component and the minimal component of $\Phi$ are denoted by $B_{\max}$ and $B_{\min}$ respectively.

\textit{Proof of Theorem \ref{CompatibleMetric} in the case where both of $B_{\max}$ and $B_{\min}$ are of dimension $3$.} Let $\{C^i\}_{i=1}^{m}$ be nontrivial chains of $(M,\alpha)$. If $(M,\alpha)$ has no nontrivial chain, then we take a trivial chain as $C^{1}$. By Lemma \ref{EmbeddingChains} (iii), there exist an open neighborhood $U^i$ of $C^i$, an $\alpha$-preserving $T^3$-action on $U^{i}$, a contact toric manifold $(N^{i},\xi^{i})$ and a $T^3$-equivariant embedding $F^{i} \colon U^{i} \longrightarrow N^{i}$ such that $F^{i}_{*} \ker \alpha = \xi^{i}$. By the Delzant type Theorem of Boyer and Galicki \cite{BoGa2}, $(N^{i},\xi^{i})$ has a $T^3$-invariant Sasakian metric $(\tilde{g}^{i},\tilde{\alpha}^{i})$ such that $\ker \tilde{\alpha}^{i}=\xi^{i}$. Let $\tilde{\alpha}^{i}_{1}$ be the $K$-contact form on $N^{i}$ such that $F^{i *} \tilde{\alpha}^{i}_{1}=\alpha$. Then by a theorem of Takahashi \cite{Tak}, there exists a $T^3$-invariant Riemannian metric $\tilde{g}^{i}_{1}$ such that $(\tilde{g}^{i}_{1},\tilde{\alpha}^{i}_{1})$ is Sasakian. Hence $(F^{i *}\tilde{g}^{i}_{1},\alpha|_{U^{i}})$ is a Sasakian metric on $U^{i}$. 

We have a contact form $\alpha'$ on $M$ such that $\ker \alpha=\ker \alpha'$ and every orbit of the Reeb flow of $\alpha'$ is closed. $\alpha'$ is obtained by changing the Reeb vector field to another infinitesimal action of $\rho$. By a theorem of Takahashi \cite{Tak}, if we show the existence of a Sasakian metric $g$ compatible with $\alpha'$ which is invariant under $\rho$, we obtain a Sasakian metric compatible with $\alpha$. We denote the Reeb flow of $\alpha'$ by $\sigma_{\alpha'}$. We denote the $S^1$-action on $M/\sigma_{\alpha'}$ induced from $\rho$ by $\overline{\rho}$. $M/\sigma_{\alpha'}$ is an orbifold with the $\overline{\rho}$-invariant symplectic form $d\alpha'$. To construct a $\rho$-invariant Sasakian metric on $M$ compatible with $\alpha'$, it suffices to construct a $\overline{\rho}$-invariant integrable complex structure on $M/\sigma_{\alpha'}$ compatible with $d\alpha'$ by Lemma 2 of \cite{Noz} which claims a manifold with a contact form $\beta$ whose Reeb flow has a transverse K\"{a}hler structure with K\"{a}hler form $d\beta$ is Sasakian.

Since $(\tilde{g}^{i},\alpha)$ is a Sasakian metric, we have a $\rho$-invariant Riemannian metric $g^{\prime i}$ on $U^{i}$ such that $(g^{i},\alpha'|_{U^{i}})$ is Sasakian by a theorem of Takahashi \cite{Tak}. Then we have a $\overline{\rho}$-invariant complex structure $J_{i}$ on $U^{i}/\sigma_{\alpha'}$ compatible with $d\alpha'$ by Lemma 2 of \cite{Noz}. We show that we have a trivial holomorphic $2$-dimensional orbifold bundle structure on $W^{i}-(C^{i}/\sigma_{\alpha'})$ whose fibers are quotient of gradient manifolds, where $W^{i}$ is an open neighborhood $W^{i}$ of $C^{i}/\sigma_{\alpha'}$ in $U^i/\sigma_{\alpha'}$. 

We show that the product of $\overline{\rho}$ and the gradient flow defines a holomorphic $\mathbb{C}^{\times}$-action on $(U^i-C^i)/\sigma_{\alpha'}$. Let $X$ be the infinitesimal action of $\overline{\rho}$. Then the gradient flow is generated by $J_{i}X$ by \eqref{gradientflow}. Since $L_{X}J_{i}=0$, it suffices to show $L_{J_{i}X}J_{i}=0$ to show that the holomorphic $\mathbb{C}^{\times}$-action is defined. Note that we have $[J_{i}Y,J_{i}Z] - J_{i}([J_{i}Y,Z] + [Y,J_{i}Z]) - [Y,Z]=0$ for any vector fields $Y$ and $Z$ on $(U^i-C^i)/\sigma_{\alpha'}$ by the integrability of $J_{i}$. For any vector field $Y$ on $(U^i-C^i)/\sigma_{\alpha'}$, we have 
\begin{equation}
\begin{array}{rl}
(L_{J_{i}X}J_{i})Y + (L_{X}J_{i})J_{i}Y & = ([J_{i}X,J_{i}Y] - J_{i}[J_{i}X,Y]) + ([X,J_{i}^{2}Y] - J_{i}[X,J_{i}Y]) \\ 
& = [J_{i}X,J_{i}Y] - J_{i}([J_{i}X,Y] + [X,J_{i}Y]) - [X,Y] \\
& =0.
\end{array}
\end{equation}
By $L_{X}J_{i}=0$, we have $L_{J_{i}X}J_{i}=0$.

Let $\tau$ be the holomorphic $\mathbb{C}^{\times}$-action which is the product of $\overline{\rho}$ and the gradient flow. We will define a map which defines the holomorphic fiber bundle structure on $W^i - (C^i/\sigma_{\alpha'})$ for an open neighborhood $W^{i}$ of $(C^i/\sigma_{\alpha'})$ in $M$. Note that the Sasakian metric $(\tilde{g}^{i},\tilde{\alpha}^{i})$ on $N^{i}$ is $T^3$-invariant. Let $\overline{R}'$ be the element of $\Lie(T^3)$ which corresponds to the Reeb vector field of $\alpha'$. We take an infinitesimal action $X_{1}$ of the $T^3$-action so that $\{F^{i}_{*}\overline{R}',\overline{X},\overline{X}_{1}\}$ is a basis of $\Lie(T^3)$ where $\overline{X}$ and $\overline{X}_{1}$ correspond to $X$ and $X_{1}$, respectively. Take a small open neighborhood $W^{i}$ of $C^{i}/\sigma_{\alpha'}$ in $U^{i}/\sigma_{\alpha'}$. The product of the flows generated by $X$, $JX$, $X_{1}$ and $JX_{1}$ defines a local $(\mathbb{C}^{\times})^{2}$-action $\tau_{(\mathbb{C}^{\times})^{2}}$ on $W^{i}$. $\tau_{(\mathbb{C}^{\times})^{2}}$ commutes with $\tau$ and is holomorphic in the same reason as $\tau$. Fix a point $x^{i}$ on $C^{i}/\sigma_{\alpha'}$ which is not a fixed point of $\tau$. Let $\tau_{x^{i}}$ be the $\mathbb{C}^{\times}$-action of the identity component of the isotropy group of $\tau_{(\mathbb{C}^{\times})^{2}}$ at $x^{i}$. We take a transversal $T^{i}$ of $\tau$ which satisfies the following conditions:
\begin{enumerate}
\item $T^{i}$ contains $x^{i}$.
\item $T^{i}$ is a complex submanifold of $U^{i}/\sigma_{\alpha'}$ near $x^{i}$ and invariant.
\item $\tau_{x^{i}}(z)(T^{i}) \cap W^{i}$ is contained in $T^{i}$ for $z$ in an open neighborhood of $1$ in $\mathbb{C}^{\times}$.
\end{enumerate}
We define a map $\pi^{i} \colon W^{i} - (W^{i} \cap (B_{\min} \cup B_{\max})) - (C^{i}/\sigma_{\alpha'}) \longrightarrow T^{i}$ by $\pi^{i}(x)=x_{0}$ where $x_{0}$ is the intersection point of $T^{i}$ and the orbit of $\tau$ which contains $x$. $\pi^{i}$ is holomorphic on $A^{i} - (C^{i}/\sigma_{\alpha'})$ by definition where $A^{i}$ is a small open neighborhood of $x^{i}$ in $U^i/\sigma_{\alpha'}$. $\pi^{i}$ is holomorphic on $W^{i} - (W^{i} \cap (B_{\min} \cup B_{\max})) - (C^{i}/\sigma_{\alpha'})$, since $\pi^{i}$ is written as $\pi^{i}(x) = (\pi^{i}|_{A^{i} - (C^{i}/\sigma_{\alpha'})}) (\tau(z_{0})(x))$ for any point $x$ in $W^{i} - (W^{i} \cap (B_{\min} \cup B_{\max})) - (C^{i}/\sigma_{\alpha'})$ where we take $z_{0}$ so that $\tau(z_{0})(x)$ belongs to $A^{i}$. Take a point $x$ in $B_{\max} \cap (W^{i} - C^{i})$. Then by Lemma \ref{MorseBottTheory} (v), the closure of the gradient manifold whose limit set contains $x$ is smooth. Since $\pi^{i}$ is equivariant with respect to $\tau_{x^{i}}$, $\pi^{i}$ is smooth near $x$. Hence $\pi^{i}$ extends to $W^{i}$ smoothly. Since $\pi^{i}$ is smooth on $W^{i}$ and holomorphic on a dense subset $W^{i} - (W^{i} \cap (B_{\min} \cup B_{\max})) - (C^{i}/\sigma_{\alpha'})$ in $W^{i}$, $\pi^{i}$ is holomorphic on $W^{i}$. Note that the set $\{ x \in W^{i} | \pi^{i}$ is holomorphic at $x \}$ is closed in $W^{i}$.

We show that fibers of $\pi^{i}$ are isomorphic as complex orbifolds. Let $F$ be a fiber of $\pi^{i}$. Since $F - (F \cap (B_{\min} \cup B_{\max}))$ is a principal orbit of $\tau$, $F$ is biholomorphic to $\mathbb{C}^{\times}$. Let $J_{1}$ and $J_{2}$ be two complex structures on $F$. Then both of $(F - (F \cap (B_{\min} \cup B_{\max})),J_{1})$ and $(F - (F \cap (B_{\min} \cup B_{\max})),J_{2})$ are isomorphic to $\mathbb{C}^{\times}$. Hence there exists a biholomorphic map $H \colon (F - (F \cap (B_{\min} \cup B_{\max})),J_{1}) \longrightarrow (F - (F \cap (B_{\min} \cup B_{\max})),J_{2})$. Applying the removable singularity theorem to $H$ on orbifold charts, $H$ extends to $F$. Hence $(F,J_{1})$ and $(F,J_{2})$ are biholomorphic.

Since $\pi^{i}|_{W^{i} - (W^{i} \cap (B_{\min} \cup B_{\max})) - (C^{i}/\sigma_{\alpha'})}$ has a holomorphic section $T^{i}$, $\pi^{i}|_{W^{i} - (W^{i} \cap (B_{\min} \cup B_{\max})) - (C^{i}/\sigma_{\alpha'})}$ is a trivial holomorphic principal $\mathbb{C}^{\times}$-bundle over $T^{i}-\{x^{i}\}$. We denote one of the fibers of $\pi^{i}$ by $F_{0}$. We denote the holomorphic trivialization by $\chi^{i} \colon W^{i} - (W^{i} \cap (B_{\min} \cup B_{\max})) - (C^{i}/\sigma_{\alpha'}) \longrightarrow (T^{i}-\{x^{i}\}) \times \mathbb{C}^{\times}$. $\chi^{i}$ extends to $\chi \colon W^{i} - (C^{i}/\sigma_{\alpha'}) \longrightarrow (T^{i}-\{x^{i}\}) \times F^{0}$ smoothly, since the trivialization coincides with a smooth trivialization of the fiber bundle structure on $(U^{i} - C^{i})/\sigma_{\alpha'}$ whose fibers are the quotient of gradient manifolds by $\sigma_{\alpha'}$ on the dense subset $W^{i} -  (W^{i} \cap (B_{\min} \cup B_{\max})) - (C^{i}/\sigma_{\alpha'})$. Since $\chi^{i}$ is holomorphic on the dense subset $W^{i} -  (W^{i} \cap (B_{\min} \cup B_{\max})) - (C^{i}/\sigma_{\alpha'})$, $\chi^{i}$ is holomorphic on $W^{i}$. Hence $\pi^{i}$ is a trivial holomorphic $F_{0}$-bundle with fiberwise $\mathbb{C}^{\times}$-action by the trivialization $\chi^{i}$.

$(M-\cup_{i=1}^{m}C^i)/\sigma_{\alpha'}$ has a structure of trivial $2$-dimensional orbifold bundle over a surface whose fibers are quotient of gradient manifolds by Proposition \ref{fib2}. We denote the base space by $S$. $\overline{\rho}$ is a fiberwise rotation on $(M-\cup_{i=1}^{m}C^i)/\sigma_{\alpha'}$. We denote the restriction of $\overline{\rho}$ to $F_{0}$ by $\overline{\rho}_{F_{0}}$. Fix a trivialization $\chi_{0} \colon (M-\cup_{i=1}^{m}C^i)/\sigma_{\alpha'} \longrightarrow S \times F_{0}$. Note that the map from $S^1$ to $\overline{\rho}_{F_{0}}$-equivariant diffeomorphism group $\Diff(F,\overline{\rho}_{F_{0}})$ on $F$ defined by the $S^1$-action $\overline{\rho}_{F_{0}}$ is a homotopy equivalence. We denote the homotopy inverse $\Diff(F,\overline{\rho}_{F_{0}}) \longrightarrow S^1$ by $p$. We can assume that $T^{i} - \{x^{i}\}$ is biholomorphic to $D^2 - \{0\}$. Let $\overline{\chi}^{i}$ be the map $T^{i} - \{x^{i}\} \longrightarrow \Diff(F,\overline{\rho}_{F_{0}})$ defined by $\overline{\chi}^{i}(t)(x) = \pr_{2} \circ \chi^{i} \circ \chi^{-1}_{0}(t,x)$ for $x$ in $F_{0}$ and $t$ in $T^{i}$ where $\pr_{2} \colon (T^{i}-\{x^{i}\}) \times F^{0} \longrightarrow F^{0}$ is the second projection. Then the rotation number $r(i)$ of $p \circ \overline{\chi}^{i} \colon T^{i} \longrightarrow S^1$ is defined. Let $\chi^{\prime i}$ be a holomorphic trivialization of $\pi^{i}$ defined by $(\zeta)^{r(i)} \circ \chi^{i}$ where $\zeta^{r(i)}$ is the $r(i)$-times composition of a biholomorphic map $\zeta$ on $T^{i} \times F^{0}$ defined by $\zeta(t,z) = (t, \phi(- h^{i}(t))(z))$ where the $\mathbb{C}^{\times}$-action $\phi(t)$ is the restriction of $\tau^{i}$ on $F_{0}$ and $h^{i} \colon T^{i} - \{x^{i}\} \longrightarrow D^{2} - \{0\} \cong \mathbb{C}^{\times}$ is a biholomorphic map. Then the rotation number of $\chi^{\prime i}$ is zero. Hence we can take a trivialization $\chi \colon (M-\cup_{i=1}^{m}C^i)/\sigma_{\alpha'} \longrightarrow S \times F_{0}$ such that
\begin{enumerate}
\item the restriction of $\chi$ to $\cup_{i=1}^{m}\big( V^i-(C^i/\sigma_{\alpha'}) \big)$ coincides with $\chi^{\prime i}$ for some open neighborhood $V^i$ of $C^i/\sigma_{\alpha'}$ in $W^i$ and 
\item $\chi$ is $S^1$-equivariant with respect to $\overline{\rho}$ and an $S^1$-action on $S \times F_{0}$ defined by the product of the trivial action on the first component and $\overline{\rho}_{F_{0}}$.
\end{enumerate}

We extend the complex structures on $\cup_{i=1}^{m}(V^i-C^i)/\sigma_{\alpha'}$ to a complex structure $J$ on $(M-\cup_{i=1}^{m}C^i)/\sigma_{\alpha'}$ as a product of complex structures on $F_{0}$ and $S$ using the trivialization $\chi$.

We will show that $J$ is an integrable $\overline{\rho}$-invariant complex structure on $(M-\cup_{i=1}^{m}C^i)/\sigma_{\alpha'}$ compatible with $d\alpha'$. The integrability and $\overline{\rho}$-invariance are clear by the construction. Since the fibers of $(M-\cup_{i=1}^{m}C^i)/\sigma_{\alpha'}$ are symplectic and invariant under $J$, the tangent space of $(M-\cup_{i=1}^{m}C^i)/\sigma_{\alpha'}$ at each point is decomposed into two $2$-dimensional $J$-invariant symplectic subspaces. Then $J$ is compatible with $d\alpha'$. Here a complex structure on a $2$-dimensional symplectic linear space is compatible with the symplectic structure, if the orientation given by the complex structure coincides with the orientation given by the symplectic structure.

Hence we have an integrable complex structure on the orbifold $M/\sigma_{\alpha'}$ compatible with $d\alpha'$ and invariant under the $S^1$-action $\overline{\rho}$ induced from $\rho$. The proof in the case where both of $B_{\max}$ and $B_{\min}$ are of dimension $3$ is completed. 

The proof of Theorem \ref{CompatibleMetric} in the other cases is completed by Lemmas \ref{Blowdown} and \ref{fib1}.

\begin{lem}\label{Blowdown}
Let $\Sigma$ be a closed orbit of the Reeb flow on $(M,\alpha)$. Let $(\tilde{M},\tilde{\alpha})$ be the $K$-contact manifold obtained from $(M,\alpha)$ by performing a contact blowing up along $\Sigma$. If there exists a Riemannian metric $\tilde{g}$ on $\tilde{M}$ such that $(\tilde{M},\tilde{\alpha},\tilde{g})$ is a Sasakian manifold, then there exists a Riemannian metric $g$ on $M$ such that $(M,\alpha,g)$ is a Sasakian manifold.
\end{lem}

We follow the proof of the Enriques-Castelnouvo's theorem (See \cite{Bea}).

\begin{proof}
Let $\overline{R}'$ be an element of $\Lie(G)$ such that the infinitesimal actions of $\overline{R}'$ on $\tilde{M}$ generates an $S^1$-action whose orbits are transverse to $\ker \tilde{\alpha}$. Then the infinitesimal action of $\overline{R}'$ on $M$ generates an $S^1$-action whose orbits are transverse to $\ker \alpha$. We define the $1$-form $\alpha'$ on $M$ by $\alpha'=\frac{1}{\alpha(R')}\alpha$ where $R'$ is the infinitesimal action of $\overline{R}'$ on $M$. Then $\alpha'$ is a contact form on $M$ whose Reeb vector field is $R'$ and which defines the contact structure $\ker \alpha$. Similarly we define a $1$-form $\tilde{\alpha}'$ on $\tilde{M}$ by $\tilde{\alpha}=\frac{1}{\tilde{\alpha}(\tilde{R}')}\tilde{\alpha}$ where $\tilde{R}'$ is the infinitesimal action of $\overline{R}'$ on $\tilde{M}$. Then $\tilde{\alpha}'$ is a contact form on $\tilde{M}$ whose Reeb vector field is $\tilde{R}'$ and which define the contact structure $\ker \tilde{\alpha}$. We denote the Reeb flows of $\alpha'$ and $\tilde{\alpha}'$ by $\sigma_{\alpha'}$ and $\sigma_{\tilde{\alpha}'}$, respectively. We denote the $T^2$-actions on $M$ and $\tilde{M}$ associated with $\alpha$ and $\tilde{\alpha}$ by $\rho_{\alpha}$ and $\rho_{\tilde{\alpha}}$, respectively. We denote the $S^1$-actions on $M/\sigma_{\alpha'}$ and $\tilde{M}/\sigma_{\tilde{\alpha}'}$ induced from $\rho_{\alpha}$ and $\rho_{\tilde{\alpha}}$ by $\overline{\rho}_{\alpha}$ and $\overline{\rho}_{\tilde{\alpha}}$, respectively. To prove Lemma \ref{Blowdown}, it suffices to show that there exists a $\overline{\rho}_{\alpha}$-invariant integrable complex structure compatible with $d\alpha'$ on the symplectic orbifold $(M/\sigma_{\alpha'},d\alpha')$. In fact, then $(M,\alpha')$ has a Sasakian metric compatible with $\alpha'$ by Lemma 2 of \cite{Noz}. Then by a theorem of Takahashi \cite{Tak}, $(M,\alpha)$ has a Sasakian metric compatible with $\alpha$.

Let $L$ be the lens space in $\tilde{M}$ obtained by the contact blowing up along $\Sigma$. $L/\sigma_{\tilde{\alpha}'}$ is a weighted projective space (See Godinho \cite{God}). Holomorphic line orbibundles over the weighted projective space $L/\sigma_{\tilde{\alpha}'}$ are classified by the first Chern classes in $H^{2}(L/\sigma_{\tilde{\alpha}'} ; \mathbb{Q})$. We refer to \cite{Roa}. For a holomorphic line orbibundle $E$ over $L/\sigma_{\tilde{\alpha}'}$, we denote the first Chern class by $c_{1}(E)$. We denote $\int_{L/\sigma_{\tilde{\alpha}'}} c_{1}(E)$ by $E \cdot [L/\sigma_{\tilde{\alpha}'}]$. $\mathcal{O}(E)$ denotes the sheaf of sections of a holomorphic line orbibundle $E$ over $\tilde{M}/\sigma_{\tilde{\alpha}'}$. Note that a holomorphic line orbibundle $E$ over $L/\sigma_{\tilde{\alpha}'}$ is positive if and only if $E \cdot [L/\sigma_{\tilde{\alpha}'}]$ is positive in the sense of Baily \cite{Bai} as in the case of holomorphic line bundles over Riemann surfaces. We denote the holomorphic line orbibundle over $\tilde{M}/\sigma_{\tilde{\alpha}'}$ defined by the divisor $L/\sigma_{\tilde{\alpha}'}$ by $[L/\sigma_{\tilde{\alpha}'}]$.

$\tilde{M}$ is an $S^1$-orbibundle over $\tilde{M}/\sigma_{\tilde{\alpha}'}$ associated with a positive holomorphic line orbibundle over $\tilde{M}/\sigma_{\tilde{\alpha}'}$ (See Theorem 7.1.3 of \cite{BoGa5}). We denote the line orbibundle by $E_{\tilde{M}}$. By Baily's theorem \cite{Bai}, there exists a sufficiently large positive integer $n$ such that $\frac{n E_{\tilde{M}} \cdot [L/\sigma_{\tilde{\alpha}'}]}{[L/\sigma_{\tilde{\alpha}'}] \cdot [L/\sigma_{\tilde{\alpha}'}]}$ is an integer and $\tilde{M}/\sigma_{\tilde{\alpha}'}$ is embedded as a complex orbifold into the complex projective space $\mathbb{C}P^{N}$ by the holomorphic map $\iota_{nE_{\tilde{M}}}$ defined as follows:
\begin{equation}
\begin{array}{cccc}
\iota_{nE_{\tilde{M}}} & \colon \tilde{M}/\sigma_{\tilde{\alpha}'} & \longrightarrow & \mathbb{C}P^{N} \\
 & x & \longrightarrow & [s_{0}(x) \colon s_{1}(x) \colon \cdots \colon s_{N}(x)]
\end{array}
\end{equation}
where $\{s_{i}\}_{i=0}^{N}$ is a basis of $H^{0}(\tilde{M}/\sigma_{\tilde{\alpha}'},\mathcal{O}(nE_{\tilde{M}}))$. Note that $n$ is chosen so that the line orbibundle $nE_{\tilde{M}}$ over $\tilde{M}/\sigma_{\tilde{\alpha}'}$ is absolute in the sense of Baily \cite{Bai} (See also \cite{BoGa5}). A line orbibundle $E$ over an orbifold $S$ is absolute if $E$ is a topological fiber bundle over $S$.

We construct a line orbibundle on $\tilde{M}/\sigma_{\tilde{\alpha}'}$ which is positive on the complement of $L/\sigma_{\tilde{\alpha}'}$ and whose restriction to $L/\sigma_{\tilde{\alpha}'}$ is trivial. $[L/\sigma_{\tilde{\alpha}'}] \cdot [L/\sigma_{\tilde{\alpha}'}]$ is equal to the Euler number of the $S^1$-action $\overline{\rho}_{\tilde{\alpha}}$ on the quotient of the total space of the normal $S^1$-orbibundle of $L/\sigma_{\tilde{\alpha}'}$ in $\tilde{M}/\sigma_{\tilde{\alpha}'}$ by $\sigma_{\tilde{\alpha}'}$ by Proposition 2.15 of \cite{God}. Hence $[L/\sigma_{\tilde{\alpha}'}] \cdot [L/\sigma_{\tilde{\alpha}'}]$ is a negative rational number. We put
\begin{equation}
k=-\frac{nE_{\tilde{M}} \cdot [L/\sigma_{\tilde{\alpha}'}]}{[L/\sigma_{\tilde{\alpha}'}] \cdot [L/\sigma_{\tilde{\alpha}'}]}.
\end{equation}
Since $nE_{\tilde{M}} \cdot [L/\sigma_{\tilde{\alpha}'}]$ is positive by the positivity of $E_{\tilde{M}}$, $k$ is a positive integer. Note that for a positive holomorphic line orbibundle $E$ on $\tilde{M}/\sigma_{\tilde{\alpha}'}$, $E|_{L/\sigma_{\tilde{\alpha}'}}$ is positive on $L/\sigma_{\tilde{\alpha}'}$ by definition.

We show that $H^{1}(\tilde{M}/\sigma_{\tilde{\alpha}'} ; \mathcal{O}(nE_{\tilde{M}} + i [L/\sigma_{\tilde{\alpha}'}]))=0$ for $1 \leq i \leq k$. We have the long exact sequence
\begin{equation}\label{CohomologyExactSequence}
\xymatrix{ H^{1}(\tilde{M}/\sigma_{\tilde{\alpha}'} ; \mathcal{O}(nE_{\tilde{M}} + (i-1) [L/\sigma_{\tilde{\alpha}'}])) \ar[r]^{f_{i}} & H^{1}(\tilde{M}/\sigma_{\tilde{\alpha}'} ; \mathcal{O}(nE_{\tilde{M}} + i [L/\sigma_{\tilde{\alpha}'}])) \\
\ar[r] & H^{1}(L/\sigma_{\tilde{\alpha}'} ; \mathcal{O}(nE_{\tilde{M}} + i [L/\sigma_{\tilde{\alpha}'}])).}
\end{equation}
Since $(nE_{\tilde{M}} + i [L/\sigma_{\tilde{\alpha}'}]) \cdot [L/\sigma_{\tilde{\alpha}'}]=-([L/\sigma_{\tilde{\alpha}'}] \cdot [L/\sigma_{\tilde{\alpha}'}]) (k - i) \geq 0$ for $1 \leq i \leq k$, we have
\begin{equation}
H^{1}(L/\sigma_{\tilde{\alpha}'} ; \mathcal{O}(nE_{\tilde{M}} + i [L/\sigma_{\tilde{\alpha}'}]))=0
\end{equation}
for $1 \leq i \leq k$ by the Kodaira-Baily vanishing theorem \cite{Bai}. Hence $f_{i}$ in \eqref{CohomologyExactSequence} is surjective for $1 \leq i \leq k$. Since 
\begin{equation}
H^{1}(\tilde{M}/\sigma_{\tilde{\alpha}'} ; \mathcal{O}(nE_{\tilde{M}}))=0
\end{equation}
by the Kodaira-Baily vanishing theorem \cite{Bai}, we have
\begin{equation}\label{VanishingH1}
H^{1}(\tilde{M}/\sigma_{\tilde{\alpha}'} ; \mathcal{O}(nE_{\tilde{M}} + i [L/\sigma_{\tilde{\alpha}'}])))=0
\end{equation}
for $0 \leq i \leq k$ inductively on $i$. 

By \eqref{VanishingH1} and the exact sequence
\begin{equation}
\xymatrix{ H^{0}(\tilde{M}/\sigma_{\tilde{\alpha}'} ; \mathcal{O}(nE_{\tilde{M}} + i [L/\sigma_{\tilde{\alpha}'}])) \ar[r]^{f'_{i}} & H^{0}(L/\sigma_{\tilde{\alpha}'} ; \mathcal{O}(nE_{\tilde{M}} + i [L/\sigma_{\tilde{\alpha}'}])) \\
 \ar[r] & H^{1}(\tilde{M}/\sigma_{\tilde{\alpha}'} ; \mathcal{O}(nE_{\tilde{M}} + (i-1) [L/\sigma_{\tilde{\alpha}'}])), }
\end{equation}
$f'_{i}$ is surjective for $1 \leq i \leq k + 1$. For $1 \leq i \leq k$, we put
\begin{equation}
d(i)=\dim H^{0}(L/\sigma_{\tilde{\alpha}'} ; \mathcal{O}(nE_{\tilde{M}} + i [L/\sigma_{\tilde{\alpha}'}])) - 1.
\end{equation}
For $1 \leq i \leq k-2$ and $i=k$, let $\{a_{i,0}, a_{i,1}, \cdots, a_{i,d(i)}\}$ be a set of elements of $H^{0}(\tilde{M}/\sigma_{\tilde{\alpha}'} ; \mathcal{O}(nE_{\tilde{M}} + i [L/\sigma_{\tilde{\alpha}'}]))$ which is mapped to a basis of $H^{0}(L/\sigma_{\tilde{\alpha}'} ; \mathcal{O}(nE_{\tilde{M}} + i [L/\sigma_{\tilde{\alpha}'}]))$. For $i=k-1$, we take $\{a_{k-1,0}, a_{k-1,1}, \cdots, a_{k-1,d(k-1)}\}$ as follows: By Lemma \ref{d1} below, there exist elements $b_{k-1,0}$ and $b_{k-1,1}$ of $H^{0}(L/\sigma_{\tilde{\alpha}'} ; \mathcal{O}(nE_{\tilde{M}} + (k-1) [L/\sigma_{\tilde{\alpha}'}]))$ such that the divisor on $L/\sigma_{\tilde{\alpha}'}$ defined by $b_{k-1,0}$ is a rational multiple of one of the fixed points of $\overline{\rho}_{\tilde{\alpha}}$ on $L/\sigma_{\tilde{\alpha}'}$ and the divisor on $L/\sigma_{\tilde{\alpha}'}$ defined by $b_{k-1,1}$ is a rational multiple of the other fixed point of $\overline{\rho}_{\tilde{\alpha}}$ on $L/\sigma_{\tilde{\alpha}'}$. We take $\{a_{k-1,0}, a_{k-1,1}, \cdots, a_{k-1,d(k-1)}\}$ so that $f'_{k-1}(a_{k-1,j})=b_{k-1,j}$ for $j=0$, $1$ and $\{a_{k-1,0}, a_{k-1,1}, \cdots, a_{k-1,d(k-1)}\}$ is mapped to a $H^{0}(L/\sigma_{\tilde{\alpha}'} ; \mathcal{O}(nE_{\tilde{M}} + i [L/\sigma_{\tilde{\alpha}'}]))$. Let $s$ be an element of $H^{0}(\tilde{M}/\sigma_{\tilde{\alpha}'} ; \mathcal{O}([L/\sigma_{\tilde{\alpha}'}]))$. Since $E_{\tilde{M}}$ and $[L/\sigma_{\tilde{\alpha}'}]$ are $\overline{\rho}_{\tilde{\alpha}}$-equivariant, we can take $a_{i,j}$ and $s$ so that they are $\overline{\rho}_{\tilde{\alpha}}$-equivariant by taking the average by $\overline{\rho}_{\tilde{\alpha}}$. We put
\begin{equation}
E_{0}= nE_{\tilde{M}} + k [L/\sigma_{\tilde{\alpha}'}].
\end{equation}
Then $\{s^{k}s_{0},\cdots,s^{k}s_{N}\} \cup \{s^{k-i} a_{i,0}, s^{k-i} a_{i,1}, \cdots, s^{k-i} a_{i,d(i)}\}_{i=1}^{k}$ is a basis of $H^{0}(\tilde{M}/\sigma_{\tilde{\alpha}'} ; \mathcal{O}(E_{0}))$. We define a map
\begin{equation}
\begin{array}{cccc}
\iota_{E_{0}} & \colon \tilde{M}/\sigma_{\tilde{\alpha}'} & \longrightarrow & \mathbb{C}P^{N'} 
\end{array}
\end{equation}
by 
\begin{equation}
\begin{array}{rl}
 \iota_{E_{0}}(x)=[s^{k}s_{0}(x) \colon \cdots \colon s^{k}s_{N}(x) \colon s^{k-1}a_{1,0}(x) \colon \cdots \colon s^{k-1}a_{1,d(1)}(x)\\
& \hspace{-80pt}  \colon s^{k-2}a_{2,0}(x) \colon \cdots \colon a_{k,d(k)}(x)].
\end{array}
\end{equation}
Since $\iota_{nE_{\tilde{M}}}$ is a holomorphic embedding and $s^{k}$ is nowhere vanishing on the complement of $L/\sigma_{\tilde{\alpha}'}$, $\iota_{E_{0}}$ is a holomorphic embedding on the complement of $L/\sigma_{\tilde{\alpha}'}$. Since 
\begin{equation}
[E_{0}|_{L/\sigma_{\tilde{\alpha}'}}] \cdot [L/\sigma_{\tilde{\alpha}'}] = \left( nE_{\tilde{M}} - \frac{nE_{\tilde{M}} \cdot [L/\sigma_{\tilde{\alpha}'}]}{[L/\sigma_{\tilde{\alpha}'}] \cdot [L/\sigma_{\tilde{\alpha}'}]} [L/\sigma_{\tilde{\alpha}'}] \right) \cdot [L/\sigma_{\tilde{\alpha}'}] =0,
\end{equation}
$E_{0}|_{L/\sigma_{\tilde{\alpha}'}}$ is trivial. Hence the restriction of every section of $E_{0}$ to $L/\sigma_{\tilde{\alpha}'}$ is constant. Hence $\iota_{E_{0}}(L/\sigma_{\tilde{\alpha}'})$ is a point. We put
\begin{equation}
\{x_{0}\}=\iota_{E_{0}}(L/\sigma_{\tilde{\alpha}'}).
\end{equation}
Moreover we have 
\begin{equation}
d(k)=0
\end{equation}
by the triviality of $E_{0}|_{L/\sigma_{\tilde{\alpha}'}}$.

We will show that there exists a symplectic form $\omega$ on $\iota_{E_{0}}(\tilde{M}/\sigma_{\tilde{\alpha}'})$ such that $\omega$ is cohomologous to $[E_{0}]$ and $\iota_{E_{0}}(\tilde{M}/\sigma_{\tilde{\alpha}'})$ has a structure of a complex orbifolds whose complex structure is $S^1$-invariant and compatible with $\omega$.

Let $U$ be an open neighborhood of $L/\sigma_{\tilde{\alpha}'}$ in $\tilde{M}/\sigma_{\tilde{\alpha}'}$ defined by the equation $a_{k,0} \neq 0$. On $U$, we have
\begin{equation}\label{iota}
\iota_{E_{0}}(x)=[\frac{s^{k}s_{0}(x)}{a_{k,0}(x)} \colon \cdots \colon \frac{s^{2}a_{k-2,d(k-2)}(x)}{a_{k,0}(x)} \colon \frac{sa_{k-1,0}(x)}{a_{k,0}(x)} \colon \cdots \colon \frac{sa_{k-1,d(k-1)}(x)}{a_{k,0}(x)} \colon 1].
\end{equation}
We consider functions $\frac{sa_{k-1,0}(x)}{a_{k,0}(x)}$ and $\frac{sa_{k-1,1}(x)}{a_{k,0}(x)}$ on $\tilde{M}/\sigma_{\tilde{\alpha}'}$. By \eqref{iota}, $\frac{sa_{k-1,0}(x)}{a_{k,0}(x)}$ and $\frac{sa_{k-1,1}(x)}{a_{k,0}(x)}$ are the restriction of holomorphic functions on $\mathbb{C}P^{N'}$ to $\iota_{E_{0}}(\tilde{M}/\sigma_{\tilde{\alpha}'})$. Since 
\begin{enumerate}
\item $a_{k-1,0}$ and $a_{k-1,1}$ are not constant on $L/\sigma_{\tilde{\alpha}'}$,
\item $a_{k}$ is nonzero constant near $L/\sigma_{\tilde{\alpha}'}$ and
\item $s$ vanishes along $L/\sigma_{\tilde{\alpha}'}$ at degree $1$,
\end{enumerate}
$\frac{sa_{k-1,0}(x)}{a_{k,0}(x)}$ and $\frac{sa_{k-1,1}(x)}{a_{k,0}(x)}$ vanish along $L/\sigma_{\tilde{\alpha}'}$ at degree $1$ on $\tilde{M}/\sigma_{\tilde{\alpha}'}$. Hence $d\left( \frac{sa_{k-1,0}}{a_{k,0}} \right)_{x_{0}}$ and $d\left( \frac{sa_{k-1,1}}{a_{k,0}} \right)_{x_{0}}$ are nonzero elements of the cotangent space of $\iota_{E_{0}}(\tilde{M}/\sigma_{\tilde{\alpha}'})$ at $x_{0}$. Let $D_{j}$ be the divisor on $M/\sigma_{\tilde{\alpha}'}$ defined by $a_{k-1,j}$ for $j=0$ and $1$. Since $a_{k-1,j}$ is $\overline{\rho}_{\tilde{\alpha}}$-equivariant, $D_{j}$ are $\overline{\rho}_{\tilde{\alpha}}$-invariant. Since $a_{k-1,0}$ is not constant on $L/\sigma_{\tilde{\alpha}'}$, $D_{0}$ does not contain $L/\sigma_{\tilde{\alpha}'}$. Then $D_{0}$ contains a $\overline{\rho}_{\tilde{\alpha}}$-invariant complex suborbifold $C_{1}$ which intersects $L/\sigma_{\tilde{\alpha}'}$. Similarly $D_{1}$ contains a $\overline{\rho}_{\tilde{\alpha}}$-invariant complex suborbifold $C_{0}$. By the construction of $a_{k-1,0}$ and $a_{k-1,1}$, $D_{0} \cap (L/\sigma_{\tilde{\alpha}'})$ and $D_{1} \cap (L/\sigma_{\tilde{\alpha}'})$ are two fixed points of $\overline{\rho}_{\tilde{\alpha}}$ on $L/\sigma_{\tilde{\alpha}'}$. Hence $D_{0} \cap (L/\sigma_{\tilde{\alpha}'})$ and $D_{1} \cap (L/\sigma_{\tilde{\alpha}'})$ are disjoint. Since $D_{0}$ does not intersect $D_{1}$ near $L/\sigma_{\tilde{\alpha}'}$, the kernels of $d \left( \frac{sa_{k-1,0}}{a_{k,0}} \right)$ and $d \left( \frac{sa_{k-1,1}}{a_{k,0}} \right)$ are not equal at $x_{0}$. Then
\begin{equation}
d \left( \frac{sa_{k-1,0}}{a_{k,0}} \right)_{x_{0}} \wedge d \left( \frac{sa_{k-1,1}}{a_{k,0}} \right)_{x_{0}} \neq 0.
\end{equation}
Since $D_{0}$ contains $C_{1}$, the restriction of $\frac{sa_{k-1,0}}{a_{k,0}}$ to $C_{1}$ is zero. Similarly the restriction of $\frac{sa_{k-1,1}}{a_{k,1}}$ to $C_{0}$ is zero. Hence we have
\begin{equation}\label{NondegeneracyOfCharts}
d \left( \frac{sa_{k-1,j}}{a_{k,0}} \right) \Big|_{C_{j}} \neq 0
\end{equation}
for $j=0$ and $1$.

We define a $C^{\infty}$ orbifold atlas $\mathcal{U}$ on $\iota_{E_{0}}(\tilde{M}/\sigma_{\tilde{\alpha}'})$ by an orbifold chart $\chi$ around $x_{0}$ and the orbifold charts of the $C^{\infty}$ orbifold atlas $\mathcal{U}_{\tilde{M}/\sigma_{\tilde{\alpha}'}}$ on $\tilde{M}/\sigma_{\tilde{\alpha}'}$ which do not contain $x_{0}$. We define $\chi$ as follows: Take an orbifold chart $(U_{j},\phi_{j},\Gamma_{j})$ of the orbifold $C_{j}$ around a point $C_{j} \cap (L/\sigma_{\tilde{\alpha}'})$ for $j=0$ and $1$. Let $W$ be a small open neighborhood of $x_{0}$ in $\iota_{E_{0}}(\tilde{M}/\sigma_{\tilde{\alpha}'})$. We define the map $\psi_{j} \colon W \longrightarrow U_{j}$ by sending each point $x$ in $W$ to a point $y$ in $U_{j}$ which satisfies $\frac{sa_{k-1,j}(x)}{a_{k,0}(x)}(x)=\frac{sa_{k-1,j}(x)}{a_{k,0}(x)}(y)$ for $j=0$ and $1$. We define a map $\phi \colon W \longrightarrow (U_{0}/\Gamma_{0}) \times (U_{1}/\Gamma_{1})$ by $\phi(x)=(\psi_{0}(x), \psi_{1}(x))$ for $x$ in $W$. By \eqref{NondegeneracyOfCharts}, $\phi$ is a biholomorphic map into $(U_{0}/\Gamma_{0}) \times (U_{1}/\Gamma_{1})$. Since $s$, $a_{k-1,0}$, $a_{k-1,1}$ and $a_{k,0}$ are $\overline{\rho}_{\tilde{\alpha}'}$-equivariant, $\phi$ is $\overline{\rho}_{\tilde{\alpha}'}$-equivariant. In particular, $\phi$ is a $\overline{\rho}_{\tilde{\alpha}'}$-equivariant homeomorphism into $(U_{0}/\Gamma_{0}) \times (U_{1}/\Gamma_{1})$. We change $U_{0}$ and $U_{1}$ so that $(U_{0}/\Gamma_{0}) \times (U_{1}/\Gamma_{1})$ is contained in the image of $\phi$. Let $\psi \colon U_{0} \times U_{1} \longrightarrow W$ be the inverse map of $\phi$. We put $\chi=(U_{0} \times U_{1},\Gamma_{0} \times \Gamma_{1},\psi)$. Then $\chi$ and the orbifold charts in $\mathcal{U}_{\tilde{M}/\sigma_{\tilde{\alpha}'}}$ which do not contain $x_{0}$ form a $C^{\infty}$ orbifold atlas.

Since $\mathcal{U}_{\tilde{M}/\sigma_{\tilde{\alpha}'}}$ defines a complex structure on $\tilde{M}/\sigma_{\tilde{\alpha}'}$ and $\psi$ is a biholomorphic map, $\mathcal{U}$ defines a complex structure $J$ on $\iota_{E_{0}}(\tilde{M}/\sigma_{\tilde{\alpha}'})$. Since the restriction of holomorphic functions on $\mathbb{C}P^{N'}$ to $\iota_{E_{0}}(\tilde{M}/\sigma_{\tilde{\alpha}'})$ are holomorphic with respect to $\mathcal{U}$, $J$ is the complex structure induced from $\mathbb{C}P^{N'}$. Since $\psi$ is $S^1$-equivariant, $J$ is $S^1$-invariant. Let $S$ be a hypersurface in $\mathbb{C}P^{N'}$ defined by $S=\{[ z_0 \colon z_1 \colon \cdots \colon z_{N'}] \in \mathbb{C}P^{N'}| z_0=0 \}$. Then the divisor $\iota_{E_{0}}(\tilde{M}/\sigma_{\tilde{\alpha}'}) \cap S$ on $\iota_{E_{0}}(\tilde{M}/\sigma_{\tilde{\alpha}'})$ is defined by the equation $s^{k}s_{0}=0$ by the definition of $\iota_{E_{0}}$. Since the line bundle defined by $S$ is the dual of the tautological bundle over $\mathbb{C}P^{N'}$, we have $\iota_{E_{0}}^{*}[\omega_{0}]=[E_{0}]$ in $H_{\dR}^{2}(\tilde{M}/\sigma_{\tilde{\alpha}'})$ where $\omega_{0}$ is the symplectic form on $\mathbb{C}P^{N'}$ which is compatible with the complex structure on $\mathbb{C}P^{N'}$ and whose cohomology class is the first Chern class of the dual of the tautological bundle over $\mathbb{C}P^{N'}$. Since $J$ is induced from $\mathbb{C}P^{N'}$ and the complex structure on $\mathbb{C}P^{N'}$ is compatible with $\omega_{0}$, $J$ is compatible with $\omega_{0}|_{\iota_{E_{0}}(\tilde{M}/\sigma_{\tilde{\alpha}'})}$. We put $\omega=\omega_{0}|_{\iota_{E_{0}}(\tilde{M}/\sigma_{\tilde{\alpha}'})}$.

Let $V$ be the open tubular neighborhood of $\Sigma$ in $M$ which is cut off in the contact blowing up along $\Sigma$ to obtain $\tilde{M}$ from $M$. Then $M - V$ is identified with $\tilde{M} - L$. We denote the identification map by
\begin{equation}
q \colon M - V \longrightarrow \tilde{M} - L.
\end{equation}
We denote the tautological map from $\tilde{M}/\sigma_{\tilde{\alpha}'}$ to $\iota_{E_{0}}(\tilde{M}/\sigma_{\tilde{\alpha}'})$ by 
\begin{equation}
\pi \colon \tilde{M}/\sigma_{\tilde{\alpha}'} \longrightarrow \iota_{E_{0}}(\tilde{M}/\sigma_{\tilde{\alpha}'}).
\end{equation}
We show that there exists an $S^1$-equivariant diffeomorphism $f \colon M/\sigma_{\alpha'} \longrightarrow \iota_{E_{0}}(\tilde{M}/\sigma_{\tilde{\alpha}'})$ whose restriction to the complement of an open tubular neighborhood of $\Sigma/\sigma_{\alpha'}$ coincides with $\pi \circ q$. This follows from the argument similar to the last part of the proof of Proposition \ref{Karshon1}. We put 
\begin{equation}
\{x_{1}\}=\Sigma/\sigma_{\alpha'}.
\end{equation}
Since
\begin{enumerate}
\item the orbifold structure groups at $x_{0}$ and $x_{1}$ are isomorphic and 
\item the linear actions of $S^1$ on $T_{x_{0}}(\iota_{E_{0}}(\tilde{M}/\sigma_{\tilde{\alpha}'}))$ and $T_{x_{1}}(M/\sigma_{\alpha'})$ are equal,
\end{enumerate}
by applying the equivariant Darboux theorem \cite{Wei} to orbifold charts, there exist an open neighborhood $V_{0}$ of $x_{0}$ in $\iota_{E_{0}}(\tilde{M}/\sigma_{\tilde{\alpha}'})$, an open neighborhood $V_{1}$ of $x_{1}$ in $M/\sigma_{\alpha'}$ and an $S^1$-equivariant diffeomorphism $h \colon V_{0} \longrightarrow V_{1}$ such that $h(x_{0})=x_{1}$ and $h^{*}\omega=d\alpha'$. Let $\overline{\Phi}_{0}$ be the map on $M/\sigma_{\alpha'}$ induced from $\Phi$. Let $\overline{\Phi}_{1}$ be the hamiltonian function on $\iota_{E_{0}}(\tilde{M}/\sigma_{\tilde{\alpha}'})$ for $\overline{\rho}_{\alpha}$. We choose two level sets $H_{1}$ and $H_{2}$ of $\Phi$ so that $H_{1}/\sigma_{\alpha'}$ is contained in $V_{0}$ and $H_{2}$ is the boundary of an open tubular neighborhood of $\Sigma$ containing $H_{1}$. Note that since they are $S^1$-equivariant and preserve the symplectic forms, $h$ and $\pi$ map the level sets of $\overline{\Phi}_{0}$ to the level sets of $\overline{\Phi}_{1}$. The level sets of $\overline{\Phi}_{0}$ between $H_{1}/\sigma_{\alpha'}$ and $H_{2}/\sigma_{\alpha'}$ are diffeomorphic to $(S^1 \times S^3)/\sigma_{\alpha'}$ which is a $3$-dimensional orbifold with $S^1$-action $\overline{\rho}_{\alpha}$ with at most two exceptional $S^1$-orbits by Lemma \ref{GeneralTorusActions}. Then we can isotope $h|_{H_{1}/\sigma_{\alpha'}}$ to $\pi'|_{H_{2}/\sigma_{\alpha'}}$ as a family of $\overline{\rho}_{\alpha}$-equivariant diffeomorphisms of $(S^1 \times S^3)/\sigma_{\alpha'}$ by the argument in the previous paragraph of the last paragraph of the proof of Proposition \ref{Karshon1}. Hence there is an $S^1$-equivariant diffeomorphism $f \colon \iota_{E_{0}}(\tilde{M}/\sigma_{\tilde{\alpha}'}) \longrightarrow M/\sigma_{\alpha'}$ whose restriction to an open tubular neighborhood of $\iota_{E_{0}}(L/\sigma_{\tilde{\alpha}'})$ coincides with $\pi \circ q$.

By Lemma \ref{EulerClasses} (b) and (c) below, we have
\begin{equation}
\begin{array}{l}
\pi^{*}f^{*}[d\alpha']=\pi^{*}[d\alpha']=[d\tilde{\alpha}'] - \frac{[d\tilde{\alpha}'] \cdot [L/\sigma_{\tilde{\alpha}'}]}{[L/\sigma_{\tilde{\alpha}'}] \cdot [L/\sigma_{\tilde{\alpha}'}]}[L/\sigma_{\tilde{\alpha}'}] \\ 
= [E_{\tilde{M}}] - \frac{[ E_{\tilde{M}} ] \cdot [L/\sigma_{\tilde{\alpha}'}]}{[L/\sigma_{\tilde{\alpha}'}] \cdot [L/\sigma_{\tilde{\alpha}'}]}[L/\sigma_{\tilde{\alpha}'}]=\frac{1}{n}[E_{0}]=\frac{1}{n}\pi^{*}[\omega].
\end{array}
\end{equation}
Hence we have $f^{*}[d\alpha']=\frac{1}{n}[\omega]$ by Lemma \ref{EulerClasses} (a).
We regard $\omega$ as a basic $2$-form on $(M,\mathcal{F}')$, where $\mathcal{F}'$ is the foliation of $M$ defined by the orbits of $\sigma_{\alpha'}$. Since $[d\alpha']$ is the Euler class of the isometric flow $\mathcal{F}'$ on $M$ by the definition of the Euler class, the basic cohomology class of the basic $2$-form $(f^{-1})^{*}\omega$ is the Euler class of $(M,\mathcal{F})$. Hence we have a nonzero constant $c$ and a $1$-form $\beta$ on $M$ such that $\beta(R)=1$ and $d\beta=c(f^{-1})^{*}\omega$ (See \cite{Sar}). By taking the average of $\beta$ by $\rho$, we can assume that $\beta$ is invariant by $\rho$. Then applying Lemma \ref{Karshon3} to $(M,\alpha')$ and $(M,\beta)$, we have a $\rho$-equivariant diffeomorphism $f_{1} \colon M \longrightarrow M$ such that $f_{1}^{*}\beta=\alpha'$. By pulling back $J$ by $f \circ f_{1}$, we have a $\rho$-invariant integrable complex structure on $M/\sigma_{\alpha'}$ compatible with $d\alpha'$.
\end{proof}

\begin{lem}\label{d1}
Let $N$ be a weighted projective space of complex dimension $1$ with a holomorphic $S^1$-action $\sigma$. Let $z_{0}$ and $z_{1}$ be two fixed points of $\sigma$. Let $E$ be a $\sigma$-equivariant holomorphic line orbibundle over $N$ whose first Chern class is positive. For a point $x$ in $N$, $[x]$ denotes the $\mathbb{Q}$-divisor defined by $x$. There exist rational numbers $r_{0}$ and $r_{1}$ such that the divisor $[r_{j}z_{j}]$ defines the holomorphic line orbibundle $E$ for $j=0$ and $1$.
\end{lem}
\begin{proof}
We put
\begin{equation}
\Omega=\{ q_{0}[z_{0}] + q_{1}[z_{1}] | q_{0}, q_{1} \in \mathbb{Q} \}.
\end{equation}
$\Omega$ is the space of $\sigma$-invariant $\mathbb{Q}$-divisors on $N$. We put $\Pic(N)_{\mathbb{Q}}=\Pic(N) \otimes_{\mathbb{Z}} \mathbb{Q}$, where $\Pic(N)$ is the Picard group of $N$. There is a natural map
\begin{equation}
\phi \colon \Omega \longrightarrow \Pic(N)_{\mathbb{Q}}
\end{equation}
which corresponds the holomorphic line orbibundle defined by the $\mathbb{Q}$-divisor $[z]$ to an element $[z]$ of $\Omega$ (See section 3 of \cite{Roa}). Let
\begin{equation}
c \colon \Pic(N)_{\mathbb{Q}} \longrightarrow H^{2}(N;\mathbb{Q})
\end{equation}
be the map which corresponds the first Chern class to a $\sigma$-equivariant line orbibundle. Let $s \colon H^{2}(N;\mathbb{Q}) \longrightarrow \mathbb{Q}$ be the canonical isomorphism. Since $E$ is positive, $s \circ c (E)$ is positive. Since $\dim H^{1}(N;\mathcal{O}_{N}) = H^{2}(N;\mathcal{O}_{N})=0$, $c$ is an isomorphism. Since $s$ and $c$ are isomorphisms, $\Pic(N)_{\mathbb{Q}}$ is identified with $\mathbb{Q}$ so that $E$ is identified with a positive rational number. Hence $\phi^{-1}(E)$ is a $1$-dimensional affine subspace of $\Omega$. Then $\phi^{-1}(E) \cap \{ q_{j}[z_{j}] | q_{j} \in \mathbb{Q} \}$ is not empty and a singleton for $j=0$ and $1$. We put 
\begin{equation}
\{ r_{j} [z_{j}]\} = \phi^{-1}(E) \cap \{ q_{j}[z_{j}] | q_{j} \in \mathbb{Q} \}
\end{equation}
for $j=0$ and $1$. Then $r_{j}$ satisfies the desired condition.
\end{proof}

\begin{lem}\label{EulerClasses}
Under the notation of the proof of Lemma \ref{Blowdown}, we have
\begin{description}
\item[(a)] $\pi$ induces an injection $H_{\dR}^{2}((M - V)/\sigma_{\alpha'}) \longrightarrow  H_{\dR}^{2}((\tilde{M} - L)/\sigma_{\tilde{\alpha}'})$,
\item[(b)] a diffeomorphism $f \colon M/\sigma_{\alpha'} \longrightarrow M/\sigma_{\alpha'}$ whose restriction is to an open tubular neighborhood of $\iota_{E_{0}}(L/\sigma_{\tilde{\alpha}'})$ coincides with $\pi \circ q$ induces the identity on $H_{\dR}^{2}(\tilde{M}/\sigma_{\tilde{\alpha}'})$ and
\item[(c)] \begin{equation}\label{AlphaAndAlphaTilde}
\pi^{*}[d\alpha']=[d\tilde{\alpha}'] - \frac{[d\tilde{\alpha}'] \cdot [L/\sigma_{\tilde{\alpha}'}]}{[L/\sigma_{\tilde{\alpha}'}] \cdot [L/\sigma_{\tilde{\alpha}'}]}[L/\sigma_{\tilde{\alpha}'}].
\end{equation}
\end{description}
\end{lem}

\begin{proof}
Let $p \colon \tilde{M} - L \longrightarrow M - V$ be the inverse of $q$. Let $\mathcal{F}$ and $\tilde{\mathcal{F}}$ be foliations on $M$ and $\tilde{M}$ whose leaves are the orbits of the Reeb flows of $\alpha'$ and $\tilde{\alpha}'$. $p$ maps the leaves of $\tilde{\mathcal{F}}$ in $\tilde{M} - L$ to the leaves of $\mathcal{F}$ in $M - V$. Let $H_{B}^{j}(X,\mathcal{G})$ be the basic cohomology of degree $j$ of a smoothly foliated smooth manifold $(X,\mathcal{G})$. If the leaf space $X/\mathcal{G}$ of $(X,\mathcal{G})$ is an orbifold, then the de Rham cohomology of a smooth orbifold $X/\mathcal{G}$ of degree $j$ is canonically isomorphic to the basic cohomology of $(X,\mathcal{G})$, because their differential complexes are isomorphic by definition. Note that $M/\mathcal{F}=M/\sigma_{\alpha'}$ and $\tilde{M}/\tilde{\mathcal{F}}=\tilde{M}/\sigma_{\tilde{\alpha}'}$. By the definition of $\pi$ and $p$, we have the following diagram:
\begin{equation}\label{PandPi}
\xymatrix{ H_{\dR}^{2}((M - V)/\sigma_{\alpha'}) \ar[r]^{\pi^{*}} \ar[d] & H_{\dR}^{2}((\tilde{M} - L)/\sigma_{\tilde{\alpha}'}) \ar[d] \\
H_{B}^{2}(M - V, \mathcal{F}) \ar[r]^{p^{*}} & H_{B}^{2}(\tilde{M} - L, \tilde{\mathcal{F}})}
\end{equation}
whose vertical arrow are isomorphisms.

By the Mayer-Vietoris exact sequence for basic cohomology of Riemannian foliations \cite{BPR}, we have
\begin{equation}\label{SequenceForM}
\xymatrix{ H_{B}^{1}(W - V, \mathcal{F}) \ar[r] & H_{B}^{2}(M, \mathcal{F}) \ar[r] & H_{B}^{2}(M - V, \mathcal{F}) \oplus H_{B}^{2}(W, \mathcal{F}) \ar[r] & H^{2}(W - V, \mathcal{F})}
\end{equation}
and
\begin{equation}\label{SequenceForTildeM}
\xymatrix{ H_{B}^{1}(\tilde{W} - L,\tilde{\mathcal{F}}) \ar[r] & H_{B}^{2}(\tilde{M},\tilde{\mathcal{F}}) \ar[r] & H_{B}^{2}(\tilde{M} - L, \tilde{\mathcal{F}}) \oplus H_{B}^{2}(\tilde{W},\tilde{\mathcal{F}}) \ar[r] & H^{2}(\tilde{W} - L,\tilde{\mathcal{F}}).}
\end{equation}

We choose a level set $H$ of the contact moment map contained in $W - V$. Let $i_{H} \colon (H,\mathcal{F}) \longrightarrow (W - V,\mathcal{F})$ be the inclusion map. Note that since $H$ is the boundary of an open tubular neighborhood of $\Sigma$, $H$ is diffeomorphic to $S^1 \times S^3$. The leaf space $H/\mathcal{F}$ of $(H,\mathcal{F})$ is an orbifold diffeomorphic to a lens space. In fact, $\sigma_{\alpha'}$ is written as
\begin{equation}
t \cdot (\zeta,z_1,z_2) = (t^{n_{0}} \zeta, t^{n_{1}} z_1, t^{n_{2}} z_2)
\end{equation}
in the standard coordinate on $H$ by Lemma \ref{GeneralTorusActions}. Then we have $H/\mathcal{F}=H/\sigma_{\alpha'} \cong ((S^1 \times S^3)/\rho|_{\mathbb{Z}/n^{0}\mathbb{Z}})/\rho \cong (S^1 \times (S^3/\rho|_{\mathbb{Z}/n^{0}\mathbb{Z}})) /\rho \cong S^3/\rho|_{\mathbb{Z}/n^{0}\mathbb{Z}}$. By Proposition 2.4 of \cite{EKN}, we have 
\begin{equation}\label{BasicCohomology3}
H^{i}_{B}(H,\mathcal{F}) \cong H^{i}_{\dR}(H/\sigma_{\alpha'}) \cong 
\begin{cases}
 \mathbb{R} & i=0, 3, \\
 0          & \text{otherwise}.
\end{cases}
\end{equation}

Let $i_{\Sigma} \colon (\Sigma,\mathcal{F}) \longrightarrow (V,\mathcal{F})$ and $i_{L} \colon (L,\tilde{\mathcal{F}}) \longrightarrow (\tilde{W},\tilde{\mathcal{F}})$ be the inclusion maps. By Proposition 2.4 of \cite{EKN}, $i_{\Sigma}$ and $i_{L}$ induce isomorphisms on basic cohomology. Hence we have
\begin{equation}\label{BasicCohomology1}
H^{i}_{B}(W,\mathcal{F}) \cong H^{i}_{B}(\Sigma,\mathcal{F}) \cong 
\begin{cases}
 \mathbb{R} & i=0, \\
 0          & \text{otherwise}
\end{cases}
\end{equation}
and
\begin{equation}\label{BasicCohomology2}
H^{i}_{B}(\tilde{W},\tilde{\mathcal{F}}) \cong H^{i}_{B}(L,\tilde{\mathcal{F}}) \cong 
\begin{cases}
 \mathbb{R} & i=0, 2, \\
 0          & \text{otherwise}.
\end{cases}
\end{equation}

Substituting \eqref{BasicCohomology1}, \eqref{BasicCohomology2} and \eqref{BasicCohomology3} to sequences \eqref{SequenceForTildeM} and \eqref{SequenceForM}, we have
\begin{equation}\label{SequenceForTildeMandM}
\xymatrix{ H_{B}^{2}(M, \mathcal{F}) \ar[rr]^{r_{M - V} \oplus r_{W} } & & H_{B}^{2}(M - V, \mathcal{F}) \oplus 0 \ar[d]^{p^{*} \oplus 0} \\
H_{B}^{2}(\tilde{M},\tilde{\mathcal{F}}) \ar[rr]^{r_{\tilde{M}-L} \oplus r_{\tilde{W}}} & & H_{B}^{2}(\tilde{M} - L, \tilde{\mathcal{F}}) \oplus \mathbb{R}[L]}
\end{equation}
whose horizontal arrows and $p^{*}$ are isomorphisms where $[L]$ is the Poincar\'{e} dual of $L$. Note that $[L]$ is contained in the image of the canonical map $H^{2}_{B}(\tilde{M},\tilde{\mathcal{F}}) \longrightarrow H^{2}(\tilde{M})$. 

We prove (a) and (b). (a) follows from the following two facts and the diagram \eqref{PandPi}:
\begin{enumerate}
\item $r_{M - V} \colon H_{B}^{2}(M, \mathcal{F}) \longrightarrow H_{B}^{2}(M - V, \mathcal{F})$ is an isomorphism, 
\item $p^{*} \colon H_{B}^{2}(M - V, \mathcal{F}) \longrightarrow H_{B}^{2}(\tilde{M} - L, \tilde{\mathcal{F}})$ is an injection.
\end{enumerate}
(b) follows from the fact that $r_{M - V} \colon H_{B}^{2}(M, \mathcal{F}) \longrightarrow H_{B}^{2}(M - V, \mathcal{F})$ is an isomorphism and $(\pi \circ q)|_{M-V}$ is isotopic to the identity.

To prove (c), we will show  
\begin{equation}\label{AlphaAndAlphaTilde2}
\pi^{*}[d\alpha'] - [d\tilde{\alpha}'] = a [L/\sigma_{\tilde{\alpha}'}]
\end{equation}
for some real number $a$. Since we have $a=-\frac{[d\tilde{\alpha}'] \cdot [L/\sigma_{\tilde{\alpha}'}] \cdot}{[L/\sigma_{\tilde{\alpha}'}] \cdot [L/\sigma_{\tilde{\alpha}'}]}$ by \eqref{AlphaAndAlphaTilde2} and $\pi^{*}[d\alpha'] \cdot [L/\sigma_{\tilde{\alpha}'}]=0$, \eqref{AlphaAndAlphaTilde} follows from \eqref{AlphaAndAlphaTilde2}. By \eqref{PandPi} and \eqref{SequenceForTildeMandM}, we have the commutative diagram
\begin{equation}
\xymatrix{ H_{\dR}^{2}(M/\sigma_{\alpha'}) \ar[r]^{\cong} \ar[d] & H_{\dR}^{2}((M - V)/\sigma_{\alpha'}) \ar[r]^{\pi^{*}} & H_{\dR}^{2}((\tilde{M}-L)/\sigma_{\tilde{\alpha}'}) \ar[r]^{\cong} & H_{\dR}^{2}(\tilde{M}/\sigma_{\tilde{\alpha}'}) \ar[d] \\
H_{B}^{2}(M, \mathcal{F}) \ar[r]_{r_{M - V}} & H_{B}^{2}(M - V, \mathcal{F}) \ar[r]_{p^{*} \oplus 0} & H_{B}^{2}(\tilde{M} - L, \tilde{\mathcal{F}}) \oplus \mathbb{R}[L] \ar[r]_{(r_{\tilde{M}-L} \oplus r_{\tilde{W}})^{-1}} & H_{B}^{2}(\tilde{M},\tilde{\mathcal{F}})}
\end{equation}
which satisfies
\begin{enumerate}
\item the composition of the upper horizontal homomorphism is $\pi^{*}$,
\item $\pi^{*}$ and $p^{*} \oplus 0$ are injective and
\item the other arrows are isomorphisms.
\end{enumerate}
Since we have 
\begin{equation}
p^{*} r_{M - V} [d\alpha] = r_{\tilde{M}-L} [d\tilde{\alpha}]
\end{equation}
by the definition of the contact blowing down, we have \eqref{AlphaAndAlphaTilde2}.
\end{proof}

\begin{lem}\label{Karshon3}
Let $M$ be a closed $5$-dimensional manifold. Let $\alpha_{0}$ and $\alpha_{1}$ be two $K$-contact forms on $M$ of rank $1$. Assume that the Reeb vector fields of $\alpha_{0}$ and $\alpha_{1}$ are the same. If both of $\alpha_{0}$ and $\alpha_{1}$ are invariant under a $T^2$-action $\tau$ on $M$ such that the Reeb flow is a subaction of $\tau$, then there exists a $\tau$-equivariant diffeomorphism $f_{1} \colon M \longrightarrow M$ such that $f_{1}^{*}\alpha_{1}=\alpha_{0}$.
\end{lem}

\begin{proof}
By Lemma \ref{PertubationOfReebVectorFields}, we have $K$-contact forms $\beta_{0}$ and $\beta_{1}$ of rank $2$ on $M$ such that $\ker \beta_{j}=\ker \alpha_{j}$ and the Reeb vector field of $\beta_{j}$ is an infinitesimal action $R'$ of $\tau$ for $j=0$ and $1$. By Proposition \ref{Karshon2}, there exists a diffeomorphism $f_{1} \colon M \longrightarrow M$ such that $f_{1}^{*}\beta_{1}=\beta_{0}$. Since $f_{1}$ is equivariant with respect to $\tau$, we have $f_{1 *}R=R$ where $R$ denotes the Reeb vector field of $\alpha_{0}$ and $\alpha_{1}$. Since $f_{1 *}\ker \alpha_{0}=f_{1 *}\ker \beta_{0}=\ker \beta_{1}=\ker \alpha_{1}$, we have $f^{*}\alpha_{1}=\alpha_{0}$.
\end{proof}

\section{Normal forms of $K$-contact structures}

We show Lemma \ref{NormalFormTheorem} which is the $K$-contact analog of the equivariant normal form theorem of Weinstein \cite{Wei} for symplectic submanifolds in Subsection \ref{subsection:NormalFormTheorem}. By Lemma \ref{NormalFormTheorem}, the equivalence class of open neighborhoods of a $K$-contact submanifold is determined by the $T^2$-equivariant symplectic normal bundle and the element $\overline{R}$ of $\Lie(T^2)$ corresponding to the Reeb vector field. We can construct good coordinates and Lie group actions on open neighborhoods of $K$-contact submanifolds by Lemma \ref{NormalFormTheorem}.

For a group $H$, a set $A$ with an $H$-action and a subset $B$ of $A$ which is invariant under the $H$-action, we denote the isotropy group of the $H$-action at $B$ by $H_{B}$.

\subsection{Normal forms of $K$-contact submanifolds}\label{subsection:NormalFormTheorem}
Let $M$ be a closed smooth manifold. Let $N$ be a smooth submanifold of $M$. Let $\alpha_0$ and $\alpha_1$ be two $K$-contact forms of rank $n$ defined on $M$. Assume that $N$ is a $K$-contact submanifold of $(M,\alpha_{i})$ for $i=0$ and $1$. Let $G_i$ be the closure of the Reeb flow of $\alpha_i$ in $\Isom(M,g_i)$ for a metric $g_i$ invariant under the Reeb flow of $\alpha_i$ for $i=0$ and $1$. Let $\rho_i$ be the $G_i$-action on $M$ for $i=0$ and $1$. The element in $\Lie(G_i)$ corresponding to the Reeb vector field of $\alpha_i$ is denoted by $\overline{R}_i$ for $i=0$ and $1$. The symplectic normal bundle $TN^{d\alpha_{i}}$ of $N$ in $(M,\alpha_{i})$ is defined by $(TN^{d\alpha_{i}})_{x}=\{ v \in T_{x}M | \alpha_i(v)=0, d\alpha_i(v,w)=0$ for every $w$ in $T_{x}N \}$ for $i=0$ and $1$. $TN^{d\alpha_{i}}$ is $\rho_i$-equivariant for $i=0$ and $1$. The normal form theorem of $K$-contact submanifolds is stated as follows:
\begin{lem}\label{NormalFormTheorem}
Assume that 
\begin{enumerate}
\item $\alpha_0|_{N}=\alpha_1|_{N}$ and 
\item The $T^n$-equivariant symplectic vector bundles $(TN^{d\alpha_{0}},d\alpha_0|_{\wedge^{2} TN^{d\alpha_{0}}})$ and $(TN^{d\alpha_{1}},d\alpha_1|_{\wedge^{2} TN^{d\alpha_{1}}})$ over $N$ are isomorphic by a bundle map $q$ which covers $\id_{N}$ and
\item $q_{*}(\overline{R}_{0})=\overline{R}_{1}$ where $q_{*} \colon \Lie(G_0) \longrightarrow \Lie(G_1)$ is the map induced from the map $G_0 \longrightarrow G_1$ associated with the equivariant bundle map $q$.
\end{enumerate}
Then there exist a $\rho_i$-invariant open neighborhood $U_i$ of $N$ for $i=0,1$ and a diffeomorphism $f \colon U_0 \longrightarrow U_1$ such that $f|_{N}=\id_{N}$ and $f^{*}\alpha_1=\alpha_0$.
\end{lem}

\begin{proof}
First, we show the following lemma:
\begin{lem}\label{Equivariant}
There exist a $\rho_i$-invariant open neighborhood $U_i$ of $N$ for $i=0,1$ and a diffeomorphism $f^1 \colon U_0 \longrightarrow U_1$ such that $f^1|_{N}=\id_{N}$, $f^1$ induces $q$ between on the symplectic normal bundles $TN^{d\alpha_{0}}$ and $TN^{d\alpha_{1}}$ of $N$ and $f^1$ is $T^n$-equivariant with respect to $\rho_0$ and $\rho_1$.
\end{lem}
\begin{proof}
For $i=0$ and $1$, fix a metric $g_{i}$ on $M$ invariant under $\rho_{i}$. There exists a small number $\epsilon_i$ such that the exponential map $\exp^{i} \colon (TN)^{d\alpha_{i}} \longrightarrow M$ with respect to $g_{i}$ defined by $\exp^{i}(x,v)=\exp^{i}_{x}(v)$ for $x$ in $N$ and $v$ in $(TN^{d\alpha_{i}})_x$ is a diffeomorphism on a tubular neighborhood $V_{\epsilon_i}$ of the zero section of $TN^{d\alpha_{i}}$ of radius $\epsilon_i$. We take a sufficiently small open neighborhood $U_{0}$ of $N$ so that a diffeomorphism $f^{1}=\exp^{1} \circ q \circ (\exp^{0})^{-1}$ is well-defined on $U_{0}$. We show that $f^{1}$ satisfies the desired conditions. $f^1|_{N}=\id_{N}$ is clear. Since $g_{i}$ is invariant under $\rho_i$, $\exp^{i}$ is $T^n$-equivariant. Then $f^{1}$ is $T^n$-equivariant. Since the differential map of $\exp^{i}$ at $x$ on the zero section of $(TN)^{d\alpha_{i}}$ is the identity through the natural identification of $T_{x}V_{\epsilon_i}$ with $T_{x}M$, the tangent map of $f^{1}$ along $N$ is $q$.
\end{proof}
Pulling back $\alpha_{1}$ by $f^1$ in Lemma \ref{Equivariant} to an open neighborhood $W^{1}$ of $N$, we reduce the proof of Lemma \ref{NormalFormTheorem} to the case where
\begin{enumerate}
\item $(d\alpha_0)_{x}=(d\alpha_1)_{x}$ for $x$ in $N$ and
\item the Reeb flows of $\alpha_0$ and $\alpha_1$ coincide on $W^{1}$.
\end{enumerate}
We identify $G_0$ and $G_1$, and denote them by $G$. The action of $G$ on $W^1$ is denoted by $\rho$. We denote the restriction of the Reeb vector fields of $\alpha_0$ and $\alpha_1$ to $W^{1}$ by $R$.

\begin{lem}\label{EquivariantSymplectic}
There exist $\rho$-invariant open neighborhoods $U_0$ and $U_1$ of $N$ in $W^{1}$ and a $\rho$-equivariant diffeomorphism $f^2 \colon U_0 \longrightarrow U_1$ such that $f^{2}|_{N}=\id_{N}$ and $f^{2 *} d\alpha_1 = d\alpha_0$.
\end{lem}
\begin{proof}
We define the $1$-form $\beta$ on $W^1$ by $\beta=\alpha_1 - \alpha_0$ and the $2$-form $\omega_t$ on $W^1$ by $\omega_t=(1-t)d\alpha_0 + td\alpha_1$ for $t$ in $[0,1]$. Since $(d\alpha_{0})_{x}=(d\alpha_{1})_{x}$ for a point $x$ in $N$ and the Reeb vector fields of $\alpha_{0}$ and $\alpha_{1}$ are equal to $R$ on $W^1$ by the assumption, there exists an open neighborhood $U'_0$ of $N$ in $W^1$ such that $\omega_t$ induces a nondegenerate $2$-form on $TU/\mathbb{R}R$ for every $t$ in $[0,1]$. Since $\beta(R)=\alpha_{1}(R)-\alpha_{0}(R)=0$ by the assumption, there exists a unique smooth family $\{X_t\}_{t \in [0,1]}$ of vector field tangent to $\ker \alpha_0$ on $U'_0$ such that $\iota(X_t)\omega_t + \beta =0$ for $t$ in $[0,1]$. Since $\omega_t,\beta$ and $\ker \alpha_0$ are invariant under $\rho$, $X_t$ is invariant under $\rho$. Note that since $\beta|_{N}=0$, $X_t|_{N}=0$. Then the isotopy $\{\phi_{t}\}_{t \in [0,1]}$ generated by $\{X_t\}_{t \in [0,1]}$ satisfies $\phi_{t}^{*}\omega_t=\omega_0$. In fact, we have
\begin{equation}
\frac{d\phi_{t}^{*}\omega_t}{dt}=\phi_{t}^{*} \Big( d\iota(X_t)\omega_t + \frac{d\omega_t}{dt} \Big) = \phi_{t}^{*}d( \iota(X_t)\omega_t + \beta )=0.
\end{equation}
Since $X_t|_{N}=0$, the isotopy $\{\phi_{t}\}_{t \in [0,1]}$ is well-defined on an open neighborhood $U_{0}$ of $N$ in $U'_{0}$. Put $f^2=\phi_{1}$ and the desired conditions for $f^2$ are satisfied. Note that since $X_t$ is invariant under $\rho$, $f^2$ is $\rho$-equivariant.
\end{proof}
Pulling back $\alpha_{1}$ by $f^2$ in Lemma \ref{EquivariantSymplectic} on an open neighborhood $W^2$ of $N$ in $W^1$, we reduce the proof of Lemma \ref{NormalFormTheorem} to the case where 
\begin{enumerate}
\item $d\alpha_0$ and $d\alpha_1$ are equal and
\item the Reeb flows of $\alpha_0$ and $\alpha_1$ are equal on $W^{2}$.
\end{enumerate}

\begin{lem}\label{Exactness}
There exists a $\rho$-invariant open neighborhood $W^3$ of $N$ in $W^2$ and a $\rho$-invariant function $h$ on $W^3$ such that $\alpha_1|_{W^{3}} - \alpha_0|_{W^{3}}=dh$ and $h|_{N}=0$.
\end{lem}
\begin{proof}
Let $W^{3}$ be a $\rho$-invariant tubular neighborhood of $N$ in $W^{2}$. We put $\beta=\alpha_1|_{W^{3}}-\alpha_0|_{W^{3}}$. Since $N$ is a retract of $W^3$ and $\beta|_{N}=0$, $\beta$ is exact on $W^3$.

We have a function $h'$ on $W^{3}$ such that $dh'=\alpha_1|_{W^{3}}-\alpha_0|_{W^{3}}$. Since the Reeb vector fields of $\alpha_0$ and $\alpha_1$ are equal to $R$ on $W^3$, we have $Rh'=dh'(R)=\alpha_1(R) - \alpha_0(R)=0$ on $W^3$. Since the orbits of the Reeb flow are dense in the orbits of $\rho$, $h'$ is constant on each orbit of $\rho$. Hence $h'$ is invariant under $\rho$. Take a point $x_0$ on $N$ and define $h=h'-h'(x_0)$. Since $h(x_0)=0$ and $dh'|_{N}=0$, we have $h|_{N}=0$. $h$ satisfies the desired conditions.
\end{proof}

We prove Lemma \ref{NormalFormTheorem}. We define the $1$-form $\delta_t=\alpha_0+tdh$ for $t$ in $[0,1]$. $\delta_t$ is a contact form at a point $x$ on $N$ for every $t$ by the assumption. Hence we have an open neighborhood $W^{4}$ of $N$ in $W^{3}$ such that $\delta_t|_{W^{4}}$ is contact for every $t$. Let $R$ be the common Reeb vector field of $\delta_t|_{W^{4}}$ for $t$ in $[0,1]$ and put $Z=-h R$. Let $\{\psi_{t}\}_{t \in [0,1]}$ be the flow generated by $Z$. Since $Z|_{N}=0$, $\{\psi_{t}\}_{t \in [0,1]}$ is well-defined on an open neighborhood $U_0$ of $N$ for $t$ in $[0,1]$. Then we have $\psi_{t}^{*}\delta_{t}=\alpha_{0}$ on $U^{0}$. In fact, we have
\begin{equation}
\frac{d\psi_{t}^{*}\delta_t}{dt}=\psi_{t}^{*} \Big( L_{Z}\delta_t + \frac{d\delta_t}{dt} \Big)=\psi_{t}^{*}(d\iota(Z)\delta_t + dh)=0.
\end{equation}
We put $f=\psi_{1}$. Since $h$ and $R$ are $\rho$-invariant, $f$ is $\rho$-equivariant. $f$ satisfies the desired conditions.
\end{proof}

Let $(M,\alpha)$ be a closed $K$-contact manifold. Let $N$ be a $K$-contact submanifold of $M$. Let $E$ be the symplectic normal bundle of $N$ in $(M,\alpha)$. We denote the zero section of $E$ by $0_{E}$ and identify $0_{E}$ with $N$ naturally. At a point $x$ on $N$, the tangent space $T_{x}M$ is identified with $T_{x}E$ naturally. In fact, we have the decompositions $T_{x}M=T_{x}N \oplus E_{x}$ and $T_{x}E=T_{x}0_{E} \oplus E_{x}$ where $E_{x}$ is the fiber of $E$ at $x$ in $N$. We have the following corollary of Lemma \ref{NormalFormTheorem}:
\begin{cor}\label{NormalBundles}
There exists a $K$-contact form $\beta$ defined on an open neighborhood of $0_{E}$ in $E$ which satisfies 
\begin{enumerate}
\item $(d\beta)_{x}=(d\alpha)_{x}$ for $x$ in $N$,
\item $\beta|_{0_{E}}=\alpha|_{N}$ and 
\item the Reeb flow of $\beta$ is the linear flow on $E$ induced from the Reeb flow of $\alpha$.
\end{enumerate}
If a $K$-contact form $\beta'$ defined on an open neighborhood $W$ of $0_{E}$ satisfies (i), (ii) and (iii), then there exist an open neighborhood $V$ of $0_{E}$ in $E$, an open neighborhood $U$ of $N$ in $M$ and a diffeomorphism $f \colon V \longrightarrow U$ such that $f|_{0_{E}}=\id_{N}$ and $f^{*}\alpha=\beta'$. 
\end{cor}

\begin{proof}
Fix a metric $g$ on $M$ invariant under $\rho$. There exists a tubular neighborhood $V'$ of $0_{E}$ such that the restriction of $\exp \colon E \longrightarrow M$ defined by $\exp(x,v)=\exp_{x}(v)$ to $V'$ is a diffeomorphism into $M$. 

Then $\exp^{*}\alpha$ is a $K$-contact form on $V'$ which satisfies (i), (ii) and (iii). In fact, (i) and (ii) follow from the fact that $(D\exp)_{x}$ is the identity for a point $x$ on $0_{E}$ through the natural identification of $T_{x}V'$ with $T_{x}M$. $\exp$ is $\rho$-equivariant with respect to $\rho$ on $M$ and the linear action on $E$ induced from $\rho$. (iii) follows from the  $\rho$-equivariance of $\exp$.

Assume that a $K$-contact form $\beta'$ on an open neighborhood $W$ of $0_{E}$ satisfies (i), (ii) and (iii). We put $U'=\exp(V' \cap W)$. Then $(\exp^{-1})^{*}\beta'$ is a $K$-contact form on $U'$. The $T^n$-equivariant symplectic normal bundle of $N$ in $(U',(\exp^{-1})^{*}\beta)$ is $E$. In fact, $(D\exp)_{x}$ is the identity for a point $x$ on $0_{E}$ and $\exp$ is equivariant with respect to the torus action $\rho$ on $M$ and the linear action on $E$ induced from $\rho$. Hence $\id_{E}$ satisfies the assumption on $q$ in Lemma \ref{NormalFormTheorem}. By Lemma \ref{NormalFormTheorem}, there exist an open neighborhood $V$ of $0_{E}$ in $E$, an open neighborhood $U$ of $N$ in $M$ and a diffeomorphism $f \colon V \longrightarrow U$ such that $f|_{0_{E}}=\id_{N}$ and $f^{*}\alpha=\beta'$. 
\end{proof}

Let $H$ be a compact Lie group. We have the following lemma:
\begin{lem}\label{LocalTorusActions}
If there exists an $H$-action $\tau$ on $E$ which commutes with $\rho$ and preserves $\alpha|_{0_{E}}$ and the symplectic structures on fibers of $E$, then there exists an $\alpha$-preserving $H$-action on an open neighborhood of $N$.
\end{lem}

\begin{proof}
By Corollary \ref{NormalBundles}, it suffices to show that there exists a $\tau$-invariant $K$-contact form $\beta$ on $E$ which satisfies the conditions (i), (ii) and (iii) in Corollary \ref{NormalBundles}.

Let $\beta'$ be a $K$-contact form defined on an open neighborhood $W$ of $0_{E}$ which satisfies (i), (ii) and (iii). We define $\beta$ by $\beta=\int_{H} h^{*}\beta' dh$ where $dh$ is the Haar measure on $H$ satisfying $\int_{H} dh=1$. Then $\beta$ is invariant under $\tau$ and satisfies (i), (ii) and (iii). The $\tau$-invariance is clear. (i) and (ii) follow from the conditions (i) and (ii) for $\beta'$ and the fact that $\tau$ preserves $\alpha|_{0_{E}}$ and the symplectic structure on the fibers of $E$. Let $R'$ be the Reeb vector field of $\beta'$. We show that $R'$ is the Reeb vector field of $\beta$. Note that since $\tau$ commutes with $\rho$, we have $h_{*}R'=R'$ for every $h$ in $H$. We can calculate $\beta(R')$ and $\iota(R')d\beta$ as follows:
\begin{equation}
\beta(R')=\iota(R') \int_{H} h^{*}\beta' dh= \int_{H} \iota(R') h^{*}\beta' dh= \int_{H} h^{*} \big( \iota(R')\beta' \big) dh=\int_{H} dh=1
\end{equation}
and
\begin{equation}
\iota(R')d\beta=\iota(R')d \int_{H} h^{*}\beta' dh= \int_{H} \iota(R') dh^{*}\beta' dh= \int_{H} h^{*} \big( \iota(R')d\beta' \big) dh=0.
\end{equation}
Hence $\beta$ satisfies (iii).
\end{proof}

We have the following relative version of Lemma \ref{NormalFormTheorem}: Let $M$ be a closed smooth manifold. Let $A_i$ be an open set of $M$ for $i=1$ and $2$. Let $N$ be a subset of $M$ such that $N-\overline{A}_1$ is a smooth submanifold of $M-\overline{A_1}$. $\alpha_0$ and $\alpha_1$ be two $K$-contact forms of rank $2$ defined on an open neighborhood of $N$. We define $G_i$, $\overline{R}_i$ and $\rho_i$ for $i=0$ and $1$ in the same way as in Lemma \ref{NormalFormTheorem}. Assume that $(N-\overline{A_1},\alpha_i|_{N-\overline{A_1}})$ is a $K$-contact submanifold of $(M-\overline{A_1},\alpha_{i}|_{M-\overline{A_1}})$ for $i=0$ and $1$. Assume that $\alpha_0|_{A_2}=\alpha_1|_{A_2}$, $A_{j}$ is invariant under the Reeb flows of $\alpha_{0}$ and $\alpha_{1}$ for $j=1$ and $2$, the closure of $A_1$ is contained in $A_{2}$ and every connected component of $A_{1}$ intersects $N$.
\begin{lem}\label{RelativeNormalFormTheorem}
Assume that 
\begin{enumerate}
\item $\alpha_0|_{N-\overline{A_1}}=\alpha_1|_{N-\overline{A_1}}$ and 
\item $T^n$-equivariant symplectic vector bundles $(T(N-\overline{A_1})^{d\alpha_{0}},d\alpha_0|_{\wedge^{2} T(N-\overline{A_1})^{d\alpha_{0}}})$ and $(T(N-\overline{A_1})^{d\alpha_{1}},d\alpha_1|_{\wedge^{2} T(N-\overline{A_1})^{d\alpha_{1}}})$ over $N-\overline{A_1}$ are isomorphic by a bundle map $q$ which covers $\id_{N-\overline{A_1}}$ and whose restriction $q|_{(N-\overline{A_1}) \cap A_2}$ to $(N-\overline{A_1}) \cap A_2$ is the identity and
\item $q_{*}(\overline{R}_{0})=\overline{R}_{1}$ where $q_{*} \colon \Lie(G_0) \longrightarrow \Lie(G_1)$ is the map induced from the map $G_0 \longrightarrow G_1$ associated with the equivariant bundle map $q$.
\end{enumerate}
Then there exist open neighborhoods $U_0$ and $U_1$ of $N$ and a diffeomorphism $f \colon U_0 \longrightarrow U_1$ such that $f|_{N \cup A_1}=\id_{N \cup A_1}$ and $f^{*}\alpha_1=\alpha_0$.
\end{lem}

\begin{proof}
First, we show that there exists a metric $\tilde{g}_i$ on $M$ for $i=0$, $1$ such that $\tilde{g}_{i}$ is invariant under $\rho_{i}$ for $i=0$, $1$ and $\tilde{g}_{0}|_{A_2}=\tilde{g}_{1}|_{A_2}$. Let $\tilde{g}_{0}$ be a metric on $M$ invariant under $\rho_{0}$. Let $\tilde{g}_{1}$ be the average of $\tilde{g}_{0}$ by $\rho_{1}$. Since $R_{0}|_{A_{2}}=R_{1}|_{A_{2}}$ and $L_{R_{0}}\tilde{g}_{0}=0$, we have $L_{R_{1}}\tilde{g}_{0}=0$ on $A_{2}$. Hence $\tilde{g}_{0}|_{A_{2}}$ is invariant under $\rho_{1}$. $\tilde{g}_{0}|_{A_{2}}$ is equal to the average of $\tilde{g}_{0}|_{A_{2}}$ by $\rho_{1}$. Hence we have $\tilde{g}_0|_{A_2}=\tilde{g}_1|_{A_2}$. 

Take a metric $\tilde{g}_i$ on $M$ for $i=0$, $1$ such that $\tilde{g}_{i}$ is invariant under $\rho_{i}$ for $i=0$, $1$ and $\tilde{g}_{0}|_{A_2}=\tilde{g}_{1}|_{A_2}$. In Lemma \ref{Equivariant}, we define $f^{1}$ by $f^{1}= (\exp^{1})^{-1} \circ q \circ \exp^{0}$ by the exponential map $\exp^{i}$ of $\tilde{g}_{i}$ for $i=0$, $1$. Then we have $f^{1}|_{A'_{1}}= (\exp^{1})^{-1} \circ q \circ \exp^{0}|_{A'_{1}} = (\exp^{0})^{-1} \circ \exp^{0}|_{A'_{1}}= \id_{A'_{1}}$ for an open neighborhood $A'_1$ of $A_1$ by the assumption on $q$. In Lemma \ref{EquivariantSymplectic}, $f^2$ satisfies $f^2|_{N \cup A_1}=\id_{N \cup A_1}$ by $\beta=\alpha_0|_{A'_1}-\alpha_1|_{A'_1}=0$. In Lemma \ref{Exactness}, we have $dh|_{N \cup A_1}=0$ since $\alpha_0|_{A'_1}=\alpha_1|_{A'_1}$. Since $h|_{N}=0$ and every connected component of $A_1$ intersects $N$, we have $h|_{N \cup A_1}=0$. Hence $Z|_{N \cap A_1}=0$ in the proof of Lemma \ref{NormalFormTheorem} and we obtain $f$ which satisfies the desired conditions.
\end{proof}

Let $E$ be the symplectic normal bundle of $N-\overline{A}_{1}$ in $M - \overline{A}_{1}$. We identify $0_{E}$ with $N -\overline{A}_{1}$ naturally. We also have the relative versions of Lemmas \ref{NormalBundles} and \ref{LocalTorusActions}: 
\begin{cor}\label{RelativeNormalBundles}
Let $\beta$ be a $K$-contact form on $E$ which satisfies (i), (ii) and (iii) in Corollary \ref{NormalBundles}. If we have a metric $g$ on $M$ such that the restriction of the exponential map $\exp \colon E|_{N \cap A_{2}} \longrightarrow M - \overline{A}_{1}$ to an open neighborhood $W$ of $0_{E} \cap A_{2}$ in $E$ satisfies $(\exp|_{W})^{*}\alpha=\beta$, then there exist an open neighborhood $V$ of $0_{E}$ in $E$, an open neighborhood $U$ of $N - \overline{A}_{1}$ in $M - \overline{A}_{1}$ and a diffeomorphism $f \colon V \longrightarrow U$ such that $f|_{0_{E}}=\id_{N - \overline{A}_{1}}$, $f^{*}\alpha=\beta$ and $f|_{\exp^{-1}(A'_{1})}=\exp|_{\exp^{-1}(A'_{1})}$ for an open neighborhood $A'_1$ of $\overline{A}_1$.
\end{cor}
\begin{proof}
Take a metric $\tilde{g}$ on $M$ invariant under $\rho$ which satisfies $\tilde{g}|_{\tilde{A}_1}=g|_{\tilde{A}_1}$ for an open neighborhood $\tilde{A}_1$ of $A_1$. There exists an open neighborhood $V'$ of $0_{E}$ such that the restriction $\exp|_{V'}$ of $\exp \colon E \longrightarrow M - \overline{A}_{1}$ to $V'$ is a diffeomorphism into $M$. We put $U'=\exp(V')$. The symplectic normal bundles of $N$ in $(U',(\exp^{-1})^{*}\beta)$ is equal to $E$ by $(D\exp)_{x}=\id_{T_{x}M}$ for $x$ in $N$ and the condition (ii) on $d\beta$. We also have $(\exp^{-1})^{*}(\beta|_{\exp^{-1}(\tilde{A}_1)})=\alpha|_{\tilde{A}_1}$. Take an open neighborhood $A'_1$ of $\overline{A}_1$ in $M$ so that the closure $\overline{A'_{1}}$ is contained in $\tilde{A}_{1}$. By applying Lemma \ref{RelativeNormalFormTheorem} to two $K$-contact forms $(\exp^{-1})^{*}\beta$ and $\alpha$ by putting $q=\id_{E}$, $A_2=\tilde{A}_1$ and $A'_1=A_1$, we can take a diffeomorphism $f$ satisfying the desired conditions.
\end{proof}

\begin{lem}\label{RelativeLocalTorusActions}
Let $\tau$ be an $H$-action on $E$ which commutes with $\rho$ and preserves $\alpha|_{N-\overline{A}_1}$ and the symplectic structures on fibers of $E$. Assume that we have an $\alpha$-preserving $H$-action $\tilde{\tau}$ on $A_{2}$ such that the linear action induced from $\tilde{\tau}$ on $E|_{(A_2 - \overline{A_1}) \cap N}$ coincides with $\tau$. Then there exist an open neighborhood $U$ of $N - \overline{A}_{1}$ in $M - \overline{A}_{1}$ and an $\alpha$-preserving $H$-action on $U$ which coincides with $\tilde{\tau}$ on $A'_1 \cap U$ for an open neighborhood $A'_{1}$ of $\overline{A}_{1}$. 
\end{lem}
\begin{proof}
Fix a metric $g$ on $M$ invariant under $\rho$ and whose restriction $g|_{A_2}$ to $A_{2}$ is invariant under $\tau$. There exists an open neighborhood $V$ of $0_{E}$ such that the restriction $\exp|_{V}$ of $\exp \colon E \longrightarrow M - \overline{A}_{1}$ to $V$ is a diffeomorphism into $M - \overline{A}_{1}$. We put $\beta'=(\exp)^{*}\alpha$ and define $\beta$ by $\beta=\int_{H} h^{*}\beta' dh$ where $dh$ is the Haar measure on $H$ satisfying $\int_{H} dh=1$. Then $\beta$ is invariant under $\tau$ and satisfies (i), (ii) and (iii) in Corollary \ref{NormalBundles}. We have $\beta|_{\pi^{-1}(A_{2} - \overline{A}_{1}) \cap V}=\beta'|_{\pi^{-1}(A_{2} - \overline{A}_{1}) \cap V}$ by the $\tau$-invariance of $\beta'|_{\pi^{-1}(A_{2} - \overline{A}_{1}) \cap V}$. Hence we have $\exp^{*}\alpha|_{(A_{2} - \overline{A}_{1}) \cap \exp(V)}=\beta'|_{(A_{2} - \overline{A}_{1}) \cap\exp(V)}=\beta|_{(A_{2} - \overline{A}_{1}) \cap\exp(V)}$. By Corollary \ref{RelativeNormalBundles}, we have a diffeomorphism $f \colon V \longrightarrow M$ into $M$ such that $f|_{0_{E}}=\id_{N-\overline{A}_{1}}$, $f^{*}\alpha=\beta$ and $f^{*}\alpha=\beta$ and $f|_{\exp^{-1}(A'_{1})}=\exp|_{\exp^{-1}(A'_{1})}$ for an open neighborhood $A'_1$ of $\overline{A}_1$. Hence conjugating $\tau$ on $E$ by $f$, we have an $H$-action on $U$ which satisfies the desired conditions.
\end{proof}

Let $\alpha_0$ and $\alpha_1$ be $K$-contact forms of rank $n$ on a smooth manifold $M$. Let $G_i$ be the closure of the Reeb flow of $\alpha_i$ in $\Isom(M,g_i)$ for a metric $g_i$ invariant under the Reeb flow of $\alpha_i$ for $i=0$ and $1$. Let $\rho_i$ be the $G_i$-action on $M$. The element in $\Lie(G_i)$ corresponding to the Reeb vector field of $\alpha_i$ is denoted by $\overline{R}_i$. Let $N$ be a smooth submanifold of $M$. Assume that $N$ is a $K$-contact submanifold of $(M,\alpha_{i})$ for $i=0$ and $1$. The symplectic normal bundle $TN^{d\alpha_{i}}$ of $N$ in $(M,\alpha_{i})$ is defined by $(TN^{d\alpha_{i}})_{x}=\{ v \in T_{x}M | \alpha_i(v)=0, d\alpha_i(v,w)=0$ for every $w$ in $T_{x}N \}$ for $i=0$ and $1$. $TN^{d\alpha_{i}}$ is $\rho_i$-equivariant for $i=0$ and $1$. Localizing the argument of the proof of Lemma \ref{NormalFormTheorem} to a relatively compact open subset $V$ of $N$ which is invariant under $\rho_{0}$ and $\rho_{1}$, we have the following variant of the normal form theorem:
\begin{lem}\label{NormalFormTheorem2}
Assume that 
\begin{enumerate}
\item $\alpha_0|_{N}=\alpha_1|_{N}$ and 
\item $T^n$-equivariant symplectic vector bundles $(TN^{d\alpha_{0}},d\alpha_0|_{\wedge^{2} TN^{d\alpha_{0}}})$ and $(TN^{d\alpha_{1}},d\alpha_1|_{\wedge^{2} TN^{d\alpha_{1}}})$ over $N$ are isomorphic by a bundle map $q$ which covers $\id_{N}$ and
\item $q_{*}(\overline{R}_{0})=\overline{R}_{1}$ where $q_{*} \colon \Lie(G_0) \longrightarrow \Lie(G_1)$ is the map induced from the map $G_0 \longrightarrow G_1$ associated with the equivariant bundle map $q$.
\end{enumerate}
Then there exist a $\rho_i$-invariant open neighborhood $U_i$ of $V$ for $i=0,1$ and a diffeomorphism $f \colon U_0 \longrightarrow U_1$ such that $f|_{N}=\id_{N}$ and $f^{*}\alpha_1=\alpha_0$.
\end{lem}

\subsection{Local $T^3$-actions}

We construct $\alpha$-preserving $T^3$-actions on open neighborhoods of closed orbits of the Reeb flow, $K$-contact lens spaces and chains.

Let $(M,\alpha)$ be a closed $5$-dimensional $K$-contact manifold of rank $2$. The Reeb vector field of $\alpha$ is denoted by $R$. Let $G$ be the closure of the Reeb flow in $\Isom(M,g)$ for a metric $g$ compatible with $\alpha$. We denote the action of $G$ on $M$ by $\rho$. Let $B_{\max}$ and $B_{\min}$ be the maximal component of the minimal component of the contact moment map for $\rho$ on $(M,\alpha)$.

\subsubsection{Local $T^3$-actions near closed orbits of the Reeb flow}\label{ClosedOrbits}

Let $\Sigma$ be a closed orbit of the Reeb flow of $\alpha$. We will show
\begin{lem}\label{ToricActionNearSingularOrbits}
There exists an $\alpha$-preserving $T^3$-action on an open neighborhood of $\Sigma$.
\end{lem}

Let $E$ be the symplectic normal bundle of $\Sigma$ in $M$. Fix a point $x$ on $\Sigma$ and $E_{x}$ be the fiber of $E$ over $x$. For a vector bundle $V$ over a manifold, $0_{V}$ denotes the zero section of $V$. 

We show Lemma \ref{ToricActionNearSingularOrbits}.

\begin{proof}
By Lemma \ref{LocalTorusActions}, it suffices to show that there exists a $T^3$-action on $E$ which commutes with $\rho$ and preserves $\alpha|_{0_{E}}$ and the symplectic structures on the fibers of $E$.

By Lemma \ref{TwoKContactSubmanifolds2}, we have $\rho$-invariant symplectic vector subbundles $E_{1}$ and $E_{2}$ of $E$ such that $E= E_{1} \oplus E_{2}$. Fix a metric on $E_{j}$ which is compatible with $d\alpha|_{\wedge^{2}E_{j}}$ and invariant under $\rho$ for $j=1$ and $2$. Then we have a complex structure $J_{j}$ on $E_{j}$ invariant under $\rho$. Let $\sigma_{1}$ be the $S^1$-action on $E$ which is the product of the trivial $S^1$-action on $E_{1}$ and the $S^1$-subaction of the linear $\mathbb{C}^{\times}$-action on $E_{2}$ defined by $J_{2}$. Let $\sigma_{2}$ be the $S^1$-action on $E$ which is the product of the trivial $S^1$-action on $E_{2}$ and the $S^1$-subaction of the linear $\mathbb{C}^{\times}$-action on $E_{1}$ defined by $J_{1}$. Then $\sigma_{j}$ preserves $d\alpha$ for $j=1$ and $2$. We will prove that $\sigma_{j}$ commutes with $\rho$ for $j=1$ and $2$. The weight of $\sigma_{j}|_{E_{j x}}$ is $1$. Since $\sigma$ preserves $J_{j}$, $g^{-1} \circ \sigma|_{(E_{j})_{g \cdot x}} \circ g$ is also the $S^1$-action on $E_{x}$ with weight $1$ for every $x$ on $\Sigma$ and every $g$ in $G$. Hence we have $g^{-1} \circ \sigma|_{(E_{j})_{g \cdot x}} \circ g=\sigma|_{E_{j x}}$ for $x$ on $\Sigma$ and $g$ in $G$. Then $\sigma_{j}$ commutes with $\rho$.

$(G_{x})_{0}$ denotes the identity component of the isotropy group $G_{x}$ at $x$. Take an $S^1$-subgroup $G'$ of $G$ which satisfies $G' \times (G_{\Sigma})_{0}=G$. Let $Z$ be the vector field on $E$ generating the linear action of $G'$ on $E$. Fix a point $x$ on $\Sigma$ and let $r$ be the first return map $r \colon E_{x} \longrightarrow E_{x}$ of the flow generated by $Z$. Let $z_i$ be a vector of $E_{j x}$ for $j=1$ and $2$. Since $r$ preserves $E_{j x}$, $J_{j}$ and $g_{j}$ for $j=1$ and $2$, $r$ can be written as $\begin{psmallmatrix} e^{2\pi i \theta_{1}} & 0 \\ 0 & e^{2\pi i \theta_{2}} \end{psmallmatrix}$ for real numbers $\theta_{1}$ and $\theta_{2}$ with respect to the basis $\{z_1,z_2\}$ of $E_{x}$ as a complex vector space. Put $Z'= Z - \theta_{1} Y_{1} -\theta_{2} Y_{2}$ where $Y_{j}$ is the vector field generating $\sigma_{j}$ for $j=1$ and $2$. Then $Z'$ generates a free $S^1$-action $\sigma_{3}$ on $E$ which commutes with $\rho$ and preserves $\alpha|_{0_{E}}$ and the symplectic structures on the fibers of $E$.

The product $\sigma_{1} \times \sigma_{2} \times \sigma_{3}$ gives a $T^3$-action which commutes with $\rho$, preserves $\alpha|_{0_{E}}$ and the symplectic structures on the fibers of $E$.
\end{proof}

\subsubsection{Local $T^3$-actions near $K$-contact lens spaces of rank $2$}

Let $N$ be a $K$-contact submanifold of $(M,\alpha)$. Assume that the rank of $(N,\alpha|_{N})$ is $2$. Then $N$ is a lens space and the Reeb flow of $\alpha|_{N}$ has two closed orbits $\Sigma_{0}$ and $\Sigma_{1}$ by the classification of $3$-dimensional contact toric manifolds by Lerman \cite{Ler2}. Assume that there exist an open neighborhood $U_{j}$ of $\Sigma_{j}$ and an $\alpha$-preserving $T^3$-action $\tilde{\tau}_{j}$ on $U_{j}$ which preserves $N$ for $j=0$ and $1$. We show
\begin{lem}\label{ToricActionNearGradientManifolds}
There exists an $\alpha$-preserving $T^3$-action $\tilde{\tau}$ on an open neighborhood of $N$ which coincides with $\tilde{\tau}_{0}$ on an open neighborhood of $\Sigma_{0}$ and is conjugate to $\tilde{\tau}_{1}$ on an open neighborhood of $\Sigma_{1}$. 
\end{lem}
\begin{proof}
Let $E$ be the symplectic normal bundle of $N$ in $M$ and $\tau_{j}$ be the linear action on $E|_{U_{j} \cap N}$ induced from $\tilde{\tau}_{j}$ for $j=0$ and $1$. Fix a metric $g$ on $E$ which is compatible with $d\alpha|_{\wedge^{2} E}$, invariant under $\rho$ and whose restriction $g|_{V_{j}}$ to $V_{j}$ is invariant under $\tau_{j}$ where $V_{j}$ is an open neighborhood of $\Sigma_{j}$ in $U_{j}$ for $j=0$ and $1$. Let $\sigma$ be the $S^1$-subaction of the linear $\mathbb{C}^{\times}$-action defined by the complex structure determined by $d\alpha|_{\wedge^{2} E}$ and $g$. $\sigma$ commutes with $\rho$ by the invariance of $g$ under $\rho$. We obtain an effective torus action $\tau$ on $E$ generated by $\sigma$ and $\rho$. Since the isotropy group $G_{N}$ of $\rho$ at $N$ is discrete and $\sigma$ fixes $N$, $\sigma$ is not a subaction of $\rho$. Hence $\tau$ is an effective $T^3$-action. Note that $\sigma$ preserves $d\alpha|_{\wedge^{2} E}$ by the compatibility of $g$ with $\alpha$. Hence $\tau$ commutes with $\rho$ and preserves $\alpha|_{0_{E}}$ and the symplectic structures on the fibers of $E$. 

We show that $\tau$ is conjugate to the linear $T^3$-action $\tau_{j}$ induced from $\tilde{\tau}_{j}$ near $\Sigma_{j}$ for $j=0$ and $1$. Assume that $\tau$ is not conjugate to the linear $T^3$-action $\tau_{0}$ induced from $\tilde{\tau_{0}}$ near $\Sigma_{0}$. Since $\tau_{0}$ preserves the complex structure determined by $d\alpha|_{\wedge^{2} E}$ and $g$ on the fibers of $E$ over an open neighborhood of $\Sigma_{0}$, $\tau_{0}$ commutes with $\sigma$ by the argument in the last part of the second paragraph of the proof of Lemma \ref{ToricActionNearSingularOrbits}. Since $\tau$ and $\tau_{0}$ commute and are not conjugate to each other, their orbits do not coincide. Hence $\tau$ and $\tau_{0}$ generate an effective $T^4$-action which commutes with $\rho$, preserves $\alpha|_{0_{E}}$ and the symplectic structures on the fibers of $E$ over an open neighborhood of $\Sigma_{j}$. By Lemma \ref{LocalTorusActions}, we have an effective $T^4$-action which preserves $\alpha$ near $\Sigma_{j}$. This contradicts Lemma \ref{RestrictionOfRank}. Hence $\tau$ restricted to an open neighborhood of $\Sigma_{0}$ is conjugate to $\tau_{0}$. on the fibers of $E$ over an open neighborhood of $\Sigma_{0}$. Similarly $\tau$ restricted to an open neighborhood of $\Sigma_{1}$ is conjugate to $\tau_{1}$ on the fibers of $E$.

We conjugate $\tau$ to $\tilde{\tau}$ so that $\tilde{\tau}$ coincides with $\tau_{0}$ on an open neighborhood of $\Sigma_{0}$. Then $\tilde{\tau}$ is a $T^3$-action on $E$ which commutes with $\rho$ and preserves $\alpha|_{0_{E}}$ and the symplectic structures on the fibers of $E$ over an open neighborhood of $\Sigma_{j}$ for $j=0$, $1$. Moreover $\tilde{\tau}$ coincides with $\tau_{0}$ on an open neighborhood of $\Sigma_{0}$ in $N$, and $\tilde{\tau}$ is conjugate to $\tau_{1}$ on an open neighborhood of $\Sigma_{1}$. By Lemma \ref{LocalTorusActions}, we obtain $\tilde{\tau}$ which satisfies the desired condition.
\end{proof}

\subsubsection{Local $T^3$-actions near chains}

Assume that the closure of every gradient manifold in a chain $C$ is a smooth submanifold of $M$. Let $\Sigma_{0}$ be the bottom closed orbit in $C$ and $\Sigma_{l}$ be the top closed orbit in $C$. Assume that there exist an open neighborhood $U_{j}$ of $\Sigma_{j}$ and an $\alpha$-preserving $T^3$-action $\tilde{\tau}_{j}$ on $U_{j}$ which preserves $N$ for $j=0$ and $l$.
We have
\begin{lem}\label{ToricActionNearChains}
There exists an $\alpha$-preserving $T^3$-action $\tilde{\tau}$ on an open neighborhood of $C$ which coincides with $\tilde{\tau}_{0}$ on an open neighborhood of $\Sigma_{0}$ and is conjugate to $\tilde{\tau}_{l}$ on an open neighborhood of $\Sigma_{l}$.
\end{lem}

\begin{proof}
We denote the lens spaces in $C$ by $L_1$, $L_2$, $\cdots$, $L_l$ from the bottom. Applying Lemma \ref{ToricActionNearGradientManifolds} inductively to $L_{1}$, $L_{2}$, $\cdots$, $L_{l}$, we have a $T^3$-action which satisfies the desired condition.
\end{proof}

\subsubsection{Local $T^3$-actions near $K$-contact lens spaces of rank $1$}

Let $B$ be a $3$-dimensional $K$-contact submanifold of $(M,\alpha)$. Assume that the rank of $(B,\alpha|_{B})$ is $1$ and $B$ is diffeomorphic to a lens space. We show 
\begin{lem}\label{ToricActionNearB}
There exists an $\alpha$-preserving $T^3$-action on an open neighborhood of $B$.
\end{lem}

Let $B/\rho$ is a $2$-dimensional smooth orbifold with a volume form $d\alpha$. $B/\rho$ has at most two singular points. The singular points of $B/\rho$ are denoted by $p_1$ and $p_2$. We refer \cite{BoGa5} for basic terminology on orbifolds.

We show
\begin{lem}
$(B/\rho,d\alpha)$ has a smooth hamiltonian $S^1$-action $\sigma$.
\end{lem}

\begin{proof}
Let $m_1$ and $m_2$ be the multiplicity of the two singular points of $B/\rho$. We put $p=\GCD(m_1,m_2)$ and $m'_{j}=m_{j}/p$ for $j=1,2$. Define a contact form $\alpha_{0}$ on $S^3$ by
\begin{equation}
\alpha_{0}=m'_1 \sqrt{-1}(z_1 d\overline{z}_1 - \overline{z}_1 dz_1) + m'_2 \sqrt{-1}(z_2 d\overline{z}_2 - \overline{z}_2dz_2).
\end{equation}
Then the Reeb flow $\rho_{\alpha_{0}}$ of $\alpha_{0}$ is written as 
\begin{equation}
t \cdot (z_1,z_2)=(t^{m'_1} z_1,t^{m'_2} z_2).
\end{equation}
Hence the multiplicity of the singular points of $S^{3}/\rho_{\alpha_{0}}$ are $m'_1$ and $m'_2$. Let $\sigma$ be a $\mathbb{Z}/p\mathbb{Z}$-action on $S^3$ defined by
\begin{equation}
s \cdot (z_1,z_2)=(e^{\frac{2 \pi i s}{p}} z_1,e^{\frac{2 \pi i s}{p}} z_2)
\end{equation}
for $s$ in $\mathbb{Z}/p\mathbb{Z}$. Since $\rho_{\alpha_{0}}$ commutes with $\sigma$, $\sigma$ acts on $S^{3}/\rho_{\alpha_{0}}$. The multiplicity of the singular points of $(S^{3}/\rho_{\alpha_{0}})/\sigma$ is $m_1$ and $m_2$. $d\alpha_{0}$ induces a volume form $\omega$ on $(S^{3}/\rho_{\alpha_{0}})/\sigma$. Let $\tau$ be an $S^1$-action on $S^3$ defined by 
\begin{equation}
t \cdot (z_1,z_2)=(t^{n_1} z_1,t^{n_2} z_2)
\end{equation}
for $t$ in $S^1$ where $n_1$ and $n_2$ are integers such that $m'_1 n_2 - m'_2 n_1$ is not equal to $0$. Then $\tau$ induces an $S^1$-action $\tau_{1}$ on $(S^{3}/\rho_{\alpha_{0}})/\sigma$ which preserves $d\alpha_{0}$. Let $Y$ be the vector field on $(S^{3}/\rho_{\alpha_{0}})/\sigma$ generating $\tau_{1}$. Since we have $d(\iota(Y)\omega)=L_{Y}\omega=0$, the $1$-form $\iota(Y)\omega$ on $(S^{3}/\rho_{\alpha_{0}})/\sigma$ is closed. By $H^{1}_{\dR}((S^{3}/\rho_{\alpha_{0}})/\sigma)=0$, we have a smooth function $h$ on $(S^{3}/\rho_{\alpha_{0}})/\sigma$ such that $dh=\iota(Y)\omega$. Hence $\tau_{1}$ is hamiltonian.

Since the genus and the multiplicity of the two singular points of $(S^{3}/\rho_{\alpha_{0}})/\sigma$ and $B/\rho$ are equal, we have a diffeomorphism $\phi \colon (S^{3}/\rho_{\alpha_{0}})/\sigma \longrightarrow B/\rho$. Changing $\omega$ to $-\omega$, we can assume that $f$ preserves the orientations determined by $\omega$ and $d\alpha$. By $H^{2}_{\dR}((S^{3}/\rho_{\alpha_{0}})/\sigma)=\mathbb{R}$, we have a real number $c$ and a $1$-form $\beta$ on $B/\rho$ such that $d\beta=cf^{*}\omega-d\alpha$. We define a volume form $\omega_{t}$ by $\omega_{t}=(1-t)d\alpha + t\omega$ on $B/\rho$ and a vector field $Z_{t}$ on $B/\rho$ by $\iota(Z)\omega_{t}+\beta=0$ for $t$ in $[0,1]$. Let $\{\psi_{t}\}_{t \in [0,1]}$ be the flow generated by $\{Z_{t}\}_{t \in [0,1]}$. Then we have $\psi_{1}^{*}\omega=d\alpha$. In fact, we have
\begin{equation}
\frac{d\psi_{t}^{*}\omega_{t}}{dt}=\psi_{t}^{*}(L_{Z_{t}}\omega_{t}+\frac{d\omega_{t}}{dt})=\psi_{t}^{*}d(\iota(Z_{t})\omega_{t} + \beta)=0.
\end{equation}
Conjugating $\phi^{-1} \circ \tau_{1} \circ \phi$ by $\psi_{1}$, we have an $S^1$-action $\sigma$ on $B/\rho$ preserving $d\alpha$. Since $\tau_{1}$ is hamiltonian and $\psi_{1}^{*}\omega=d\alpha$, $\sigma$ is hamiltonian.
\end{proof}

\begin{lem}\label{LensSpacesAreRank2}
There exists an $(\alpha|_{B})$-preserving $T^2$-action on $B$.
\end{lem}

\begin{proof}
Let $X$ be the vector field on $B/\rho$ generating $\sigma$. Let $h$ be a hamiltonian function of $\sigma$. Let $\tilde{X}$ be the lift of $X$ on $B$ tangent to $\ker \alpha$. We put $X'=\tilde{X}-hR$. We prove that $X'$ and $R$ generate an $\alpha$-preserving $T^2$-action.

The flow generated by $X'$ preserves $\alpha$, since we have $L_{X'}\alpha=d(\iota(-hR)\alpha)+\iota(\tilde{X})d\alpha=0$. The flow generated by $X'$ commutes with the Reeb flow. In fact, since we have $\alpha([R,X'])=R(\alpha(X'))-L_{R}\alpha(X')=0$, $[R,X']$ is a section of $\ker \alpha$. For a section of $\ker \alpha$, we have $d\alpha([R,X'],Y)=-X'(d\alpha(R,Y))+d\alpha(R,[X',Y])=0$. Hence we have $[R,X']=0$ by the nondegeneracy of $d\alpha$ on $\ker \alpha$. 

Let $\tau$ be the $(S^1 \times \mathbb{R})$-action on $B$ generated by $R$ and $X'$. We show that $\tau$ induces a $T^2$-action. By $Rh=0$ and $X'h=0$, the orbits of the $(S^1 \times \mathbb{R})$-action generated by $R$ and $X'$ coincide with the fibers of $h$. Let $\Sigma_{0}$ and $\Sigma_{1}$ be the maximal point set and the minimum point set of $h$. $\Sigma_{j}$ is a closed orbit of the Reeb flow of $\alpha$. Since $h$ is a proper submersion on $B - (\Sigma_{0} \cup \Sigma_{1})$, $h|_{B - (\Sigma_{0} \cup \Sigma_{1})}$ is a $T^2$-bundle. We fix a trivialization $B - (\Sigma_{0} \cup \Sigma_{1}) \cong T^{2} \times (0,1)$ such that $R$ and $X'$ are written as $R=\frac{\partial}{\partial x}$ and $X'=\frac{\partial}{\partial y} + f(t)\frac{\partial}{\partial x}$ for some function $f$ with respect to $t$ where $(x,y)$ is the coordinate on $T^2$ as follows: $(B - (\Sigma_{0} \cup \Sigma_{1}))/\rho$ is diffeomorphic to $S^1 \times (0,1)$. Let $s \colon [0,1] \times (0,1) \longrightarrow S^1 \times (0,1)$ be the map defined by the product of the map $[0,1] \longrightarrow S^1=[0,1]/\{0\} \sim \{1\}$ and $\id_{(0,1)}$. Take a lift $\tilde{s}$ of $s$ to $B - (\Sigma_{0} \cup \Sigma_{1})$ such that the image of the lift is invariant under the $(S^1 \times \mathbb{R})$-action generated by $R$ and $X'$. We define a smooth function $f$ on $(0,1)$ by $f(t)=-\tilde{s}(t,1) + \tilde{s}(t,0)$. By putting $\frac{\partial}{\partial x}=R$ and $\frac{\partial}{\partial y}=X' + f(t) R$, we can define a trivialization $B - (\Sigma_{0} \cup \Sigma_{1}) \cong T^{2} \times (0,1)$ such that $R$ and $X'$ are written as $R=\frac{\partial}{\partial x}$ and $X'=\frac{\partial}{\partial y} + f(t)\frac{\partial}{\partial x}$ for some function $f$ with respect to $t$ where $(x,y)$ is the coordinate on $T^2$. 

To show that $\tau$ induces a $T^2$-action, it suffices to show that $f(t)$ is a constant function. We can write $\alpha=dx + h_{1}(t)dy + h_{2}(t)dt$ where $h_{1}(t)$ and $h_{2}(t)$ are smooth functions on $t$. Since $\alpha(X')=f(t)+h_{1}(t)$ and $d(\alpha(X'))=\iota(X')d\alpha=\iota(X')(h'_{1}(t)dt \wedge dy)=dh_{1}$, we have $df(t)=0$. Hence the proof is completed.
\end{proof}

Let $\tau$ be the $T^2$-action on $B$ obtained in Lemma \ref{LensSpacesAreRank2}. Let $E$ be the symplectic normal bundle of $B$ in $M$. Let $\omega$ be the symplectic structure on fibers of $E$. $E$ has a complex structure compatible with  $\omega$.

\textit{Proof of Lemma \ref{ToricActionNearB}}. By Lemma \ref{LocalTorusActions}, it suffices to show that we have a $d\alpha$-preserving $T^3$-action $\tilde{\tau}$ on $E$ such that $\tilde{\tau}$ commutes with $\rho$ and $E$ is equivariant with respect to $\tilde{\tau}$ and $\tau$. Let $\tau_1$ be an $S^1$-subaction of $\tau$ which satisfies $\tau_1 \times \rho'=\tau$ where $\rho'$ is the $S^1$-action on $B$ induced from $\rho$.

First, we show that there exists an $S^1$-action $\tau'_1$ on $E$ so that $\tau'_1$ commutes with $\rho$ and $E$ is $S^1$-equivariant with respect to $\tau_1$ and $\tau'_1$. Let $C$ be a $T^2$-orbit of $\tau$ in $B$. We denote two connected components of $B-C$ by $B_1$ and $B_2$. Since $B$ has a $T^2$-invariant contact structure by Lemma \ref{LensSpacesAreRank2}, $B_1$ and $B_2$ are solid tori which are $\rho$-equivariantly homotopy equivalent to $S^1$ by the classification of $3$-dimensional contact toric manifolds by Lerman \cite{Ler2}. $E|_{B_1}$ and $E|_{B_2}$ are trivial as $\rho$-equivariant vector bundles, since the base spaces $B_1$ and $B_2$ can be $\rho$-equivariantly retracted to $S^1$. We can take $T^2$-equivariant trivializations $E|_{B_i} \cong \mathbb{R}^{2} \times B_i$ for $i=1$ and $2$ where the $T^2$-action on $\mathbb{R}^{2} \times B_i$ is defined by the product of the standard $S^1$-action on $\mathbb{R}^{2}$ and $\rho'$. Then we have an $S^1$-action $\tau'_1$ on $\mathbb{R}^{2} \times B_1$ defined by the product of the trivial $S^1$-action on $\mathbb{R}^{2}$ and the $S^1$-action $\tau$ on $B_1$. $\tau'_1$ induces a rotation on each component of $\mathbb{R}^{2} \times \partial \overline{B_2}$ through the identification $\mathbb{R}^{2} \times \partial \overline{B_1} \longrightarrow \mathbb{R}^{2} \times \partial \overline{B_2}$. Hence we can extend $\tau'_1$ to $E$ so that $\tau'_1$ commutes with $\rho$ and $E$ is $S^1$-equivariant with respect to $\tau_1$ and $\tau'_1$.

We show that we can modify $\tau'_1$ to an $\omega$-preserving $S^1$-action $\tilde{\tau}_1$ on $E$ such that $\tilde{\tau}_1$ commutes with $\rho$ and the vector bundle $E$ is $S^1$-equivariant with respect to $\tau_1$ and $\tilde{\tau}_1$. Let $dh$ be a Haar measure on $S^1$ We define a symplectic structure $\omega'$ on fibers on $E$ by the average $\omega'=\int_{S^1}h^{*}\omega dh$ of $\omega$ by $\tau'_1$. Then $\omega'$ is invariant under $\rho$ by $g^{*}\omega'=\int_{S^1}g^{*}h^{*}\omega dh=\int_{S^1}h^{*}g^{*}\omega dh=\int_{S^1}h^{*}\omega dh=\omega'$. Since the rank of $E$ is $2$, we have a $\rho$-equivariant isomorphism $\psi$ from $(E,\omega')$ to $(E,\omega)$ as oriented vector bundles. Conjugating $\tau'_1$ by $\psi$, we have an $S^1$-action $\tilde{\tau}_{1}$ which satisfies the desired conditions.

We obtain a $T^3$-action $\tilde{\tau}$ on $E$ which satisfies the desired conditions by taking the product of $\tilde{\tau}_1$ and $\rho$.

\subsection{Normal forms of closed orbits of the Reeb flow}

We write normal forms of closed orbits of Reeb flows in this subsection. See \cite{Yam2} for the proof of the normal form theorem of closed orbits of the Reeb flow. 

Let $(M_{i},\alpha_{i})$ be a closed $5$-dimensional $K$-contact manifolds of rank $2$ for $i=0$ and $1$. Let $G_{i}$ be the closure of the Reeb flow in $\Isom(M_{i},g_{i})$ for a metric $g_{i}$ invariant under the Reeb flow. The action of $G_{i}$ on $M_{i}$ is denoted by $\rho_{i}$.

In this subsection, we fix an isomorphism $G_{i} \longrightarrow T^2$ and identify $G_{i}$ with $T^2$ for $i=0$ and $1$. The isotropy groups of $\rho_{i}$ are regarded as subgroups of $T^2$. The element of $\Lie(T^2)$ corresponding to the Reeb vector field of $\alpha_{i}$ through the isomorphism $\Lie(G_{i}) \longrightarrow \Lie(T^2)$ is denoted by $\overline{R}_{i}$. We assume that $\overline{R}_{0}=\overline{R}_{1}$ and denote them by $\overline{R}$. We fix $\overline{X}$ in $\Lie(T^2)$ which is not parallel to $\overline{R}$. We put $\Phi_{i}=\alpha_{i}(X)$ where $X$ is the infinitesimal action of $\overline{X}$. The maximal component and the minimal component of $\Phi_{i}$ are denoted by $B_{\max i}$ and $B_{\min i}$, respectively.

Let $\Sigma_{i}$ be a closed orbit of the Reeb flow of $\alpha_{i}$ for $i=0$ and $1$. A diffeomorphism $f$ from an open neighborhood $U_{0}$ of $\Sigma_{0}$ to an open neighborhood $U_{1}$ of $\Sigma_{1}$ is said to be an isomorphism if $f(\Sigma_{0})=\Sigma_{1}$, $f^{*}\alpha_{1}=\alpha_{0}$ and $f(\rho_{0}(x,g))=\rho_{1}(f(x),g)$ for every $x$ in $\Sigma_{0}$ and $g$ in $T^2$.

We will omit the index $i$ in the proof and the statement of lemmas, when we can fix $i$ in the argument.

\subsubsection{Length of closed orbits of the Reeb flow}

Note that $(\Sigma_{0},\alpha_{0}|_{\Sigma_{0}})$ and $(\Sigma_{1},\alpha_{1}|_{\Sigma_{1}})$ are isomorphic as $K$-contact manifolds if and only if $\int_{\Sigma_{0}}\alpha=\int_{\Sigma_{1}}\alpha$. We show the following:
\begin{lem}\label{LengthOfClosedOrbits}
We have $\int_{\Sigma_{0}}\alpha_{0}=\int_{\Sigma_{1}}\alpha_{1}$ if $(G_{0})_{\Sigma_{0}}=(G_{1})_{\Sigma_{1}}$.
\end{lem}

\begin{proof}
It suffices to compute $\int_{\Sigma}\alpha$ from $\overline{R}$ and $G_{\Sigma}$. Take a parametrization $\gamma \colon S^1 \longrightarrow \Sigma$ so that $\gamma_{*}(\frac{\partial}{\partial \zeta})_{x}=c R_{x}$ for every point $x$ on $\Sigma$ and a positive constant $c$. Then we have $\int_{\Sigma}\alpha=\int_{S^1}\alpha(\gamma_{*}(\frac{\partial}{\partial \zeta}))d\zeta=c$. Fix a point $x$ on $\Sigma$. Then $\Sigma$ is identified with $G/G_{x}$. We denote the projection $\Lie(G) \longrightarrow \Lie(G/G_{x})$ by $\pi$ and identify $\Lie(G/G_{x})$ with $\mathbb{R}$ so that the lattice defined by the kernel of $\Lie(G/G_{x}) \longrightarrow G/G_{x}$ is identified with $\mathbb{Z}$. Then we have $\pi(\overline{R})=\frac{1}{c}$. Hence Lemma \ref{LengthOfClosedOrbits} follows.
\end{proof}

We remark on the values of the contact moment map at closed orbits of the Reeb flow. Let $\Phi_{\alpha_{i}} \colon M_{i} \longrightarrow \Lie(T^2)^{*}$ be the contact moment map for $\rho_{i}$ for $i=0$ and $1$.
\begin{lem}\label{ValuesOfContactMomentMaps}
We have $\Phi_{\alpha_{0}}(\Sigma_{0})=\Phi_{\alpha_{1}}(\Sigma_{1})$ if $(G_{0})_{\Sigma_{0}}=(G_{1})_{\Sigma_{1}}$.
\end{lem}

\begin{proof}
It suffices to compute $\Phi_{\alpha}(\Sigma)$ from $\overline{R}$ and $G_{\Sigma}$. $\Phi_{\alpha}(\Sigma)$ is a covector which defines the Lie algebra of $G_{\Sigma}$. In fact, if we take a vector $\overline{X}$ in $\Lie(T^2)$ which generates the identity component $(G_{\Sigma})_{0}$ of $G_{\Sigma}$ and denote the infinitesimal action of $\overline{X}$ by $X$, we have $\Phi_{\alpha}(x)(\overline{X})=\alpha_{x}(X_{x})=0$ for any point $x$ in $\Sigma$ by $X_{x}=0$. $\Phi_{\alpha}(\Sigma)$ is contained in the affine space $\{v \in \Lie(T^2)^{*} | v(\overline{R})=1\}$ of $\Lie(T^2)^{*}$. Hence $\Phi_{\alpha}(\Sigma)$ is determined by the equation $\Phi_{\alpha}(\Sigma)(\overline{R})=1$ among covectors defining the Lie algebra of $(G_{\Sigma})_{0}$.
\end{proof}

\subsubsection{Connected isotropy group cases}\label{ConnectedIsotropyGroupCases}

We assume that $(G_{i})_{\Sigma_{i}}$ is connected for $i=0$ and $1$ in this subsubsection. 

We show
\begin{lem}\label{LocalFreeActions}
There exists an $S^1$-subaction of $T^2$ which acts freely on an open neighborhood of $\Sigma$.
\end{lem}

\begin{proof}
By the assumption, $G_{x}$ is isomorphic to $S^1$. Hence there exists an $S^1$-subgroup $G'$ of $T^2$ such that $G' \times G_{x} = T^2$. Then $G'$ acts freely on an open neighborhood of $\Sigma$.
\end{proof}

\begin{lem}\label{TwoKContactSubmanifolds}
There exist two $K$-contact submanifolds containing $\Sigma$.
\end{lem}

\begin{proof}
By Lemma \ref{LocalFreeActions}, there exists an $S^1$-subgroup $G'$ of $T^2$ such that $G'$ acts freely on an open neighborhood of $\Sigma$. Fix a point $x$ on $\Sigma$. Let $E$ be the symplectic normal bundle of $\Sigma$ in $(M,\alpha)$. Then the isotropy group of $\Sigma$ acts to the fibers of $E$ preserving the symplectic structure induced from $d\alpha$. $(E_{x},d\alpha|_{E_{x}})$ is $S^1$-equivariantly isomorphic to $\mathbb{C}^{2}$ with the standard symplectic structure and an $S^1$-action defined by
\begin{equation}
s \cdot (z_1,z_2)=(s^{m_1}z_{1}, s^{m_2}z_{2})
\end{equation}
where $(m_1,m_2)$ is a pair of coprime integers. Since the $G'$-action is free on an open neighborhood of $\Sigma$, $E$ has two $T^2$-invariant symplectic vector subbundles defined by the equations $z_1=0$ and $z_2=0$. The proof is completed by Corollary \ref{NormalBundles}.
\end{proof}

We take a coordinate to write $\alpha$ in a simple form near $\Sigma$. Fix an $S^1$-subgroup $G'$ of $T^2$ such that $G'$ acts freely on an open neighborhood of $\Sigma$. We change the identification from $G$ to $T^2$ so that $G'=\{(t,s) \in T^2 | s=1\}$ and $G_{\Sigma}=\{(t,s) \in T^2 | t=1\}$. We put $\overline{R}=\lambda_{0}\frac{\partial}{\partial s} + \lambda_{1}\frac{\partial}{\partial t}$. Since $(M,\alpha)$ is of rank $2$, $\lambda_{0}$ and $\lambda_{1}$ are real numbers which are linearly independent over $\mathbb{Q}$. The length of $\Sigma$ is $\frac{1}{\lambda_{0}}$. Fix a point $x$ on $\Sigma$. Since the action of $G'$ is free, the isomorphism class of the $T^2$-equivariant symplectic normal bundle $E$ of $\Sigma$ is determined by the isomorphism class of $(E_{x},d\alpha|_{E_{x}})$ as a $S^1$-equivariant symplectic vector space. By the equivariant Darboux theorem, $(E_{x},d\alpha|_{E_{x}})$ is isomorphic to $(\mathbb{C}^{2},\sqrt{-1}\sum_{j=1}^{2}dz_{j} \wedge d\overline{z}_{j})$ with the $S^1$-action defined by
\begin{equation}
s \cdot (z_1,z_2)=(s^{m_1}z_{1}, s^{m_2}z_{2})
\end{equation}
for coprime integers $(m_1,m_2)$. 

Define a $1$-form $\alpha_{0}$ on $S^1 \times D^{4}_{\epsilon}$ by 
\begin{equation}\label{NormalFormOfAlpha}
\alpha_{0}=\frac{1-\lambda_1 (m_1 |z_1|^{2} + m_2 |z_2|^{2})}{\lambda_0}d\zeta + \frac{\sqrt{-1}}{2} (z_{1}d\overline{z}_{1}-\overline{z}_{1}dz_{1}) + \frac{\sqrt{-1}}{2} (z_{2}d\overline{z}_{2}-\overline{z}_{2}dz_{2}).
\end{equation}
$\alpha_{0}$ is a $K$-contact form of rank $2$. $S^1 \times \{0\}$ is a closed orbit of the Reeb flow of $\alpha_{0}$ of length $\frac{1}{\lambda_{0}}$. The $T^2$-equivariant symplectic normal bundle of $S^1 \times \{0\}$ is isomorphic to the symplectic normal bundle of $\Sigma$. By Lemma \ref{NormalFormTheorem} we have an open neighborhood $U$ of $\Sigma$, an open neighborhood $V$ of $S^1 \times \{0\}$ and a diffeomorphism $f \colon U \longrightarrow V$ such that $\alpha=f^{*}\alpha_{0}$.

The Reeb vector field $R_{0}$ of $\alpha_{0}$ is written as
\begin{equation}
R_{0}=\lambda_0\frac{\partial}{\partial \zeta}+ \lambda_1 \Big( m_1 \sqrt{-1} \Big( z_{1}\frac{\partial}{\partial z_{1}}-\overline{z}_{1}\frac{\partial}{\partial \overline{z}_{1}} \Big) + m_2 \sqrt{-1} \Big(z_{2}\frac{\partial}{\partial z_{2}}-\overline{z}_{2}\frac{\partial}{\partial \overline{z}_{2}} \Big) \Big).
\end{equation}
The $T^2$-action $\rho_{0}$ associated with $\alpha_{0}$ is given by 
\begin{equation}\label{StandardTorusActionInTheCaseOfConnectedIsotropyGroup}
(t,s) \cdot (\zeta,z_1,z_2)=(t \zeta, s^{m_1}z_{1}, s^{m_2}z_{2}).
\end{equation}
If we take an infinitesimal action 
\begin{equation}\label{FormulaOfX}
X=s_0 \frac{\partial}{\partial \zeta}+ s_1 \Big( m_1 \sqrt{-1} \Big( z_{1}\frac{\partial}{\partial z_{1}}-\overline{z}_{1}\frac{\partial}{\partial \overline{z}_{1}} \Big) + m_2 \sqrt{-1} \Big(z_{2}\frac{\partial}{\partial z_{2}}-\overline{z}_{2}\frac{\partial}{\partial \overline{z}_{2}} \Big) \Big)
\end{equation}
of $\rho_{0}$ which is not parallel to $R_{0}$, then $\Phi_{0}=\alpha_{0}(X_{0})$ is written as 
\begin{equation}\label{NormalFormOfPhi}
\Phi_{0}= m_1 \left( -\frac{s_0 \lambda_1}{\lambda_0} + s_1 \right) |z_1|^{2} + m_2 \left(  -\frac{s_0 \lambda_1}{\lambda_0} + s_1 \right) |z_2|^{2} + \frac{s_0}{\lambda_0}.
\end{equation}

Let $N_{i}^{1}$ and $N_{i}^{2}$ be two $K$-contact submanifolds containing $\Sigma_{i}$ for $i=0$ and $1$. Assume that $N_{i}^{1}$ is contained in $\Phi_{i}^{-1}((-\infty,\Phi_{i}(\Sigma_{i})])$ and $N_{i}^{2}$ is contained in $\Phi_{i}^{-1}([\Phi_{i}(\Sigma_{i}),\infty))$ for $i=0$ and $1$. If $\Sigma_{i}$ is not contained in $B_{\min i}$ nor $B_{\max i}$ for $i=0$ and $1$, we have
\begin{lem}\label{NormalFormOfSingularOrbits}
There exists an isomorphism from an open neighborhood of $\Sigma_{0}$ to an open neighborhood of $\Sigma_{1}$ mapping $N_{0}^{1}$ to $N_{1}^{1}$ if and only if $(G_{0})_{\Sigma_{0}}=(G_{1})_{\Sigma_{1}}$, $I(\rho,N_{0}^{1})=I(\rho,N_{1}^{1})$ and $I(\rho,N_{0}^{2})=I(\rho,N_{1}^{2})$. 
\end{lem}

\begin{proof}
It suffices to show that the normal form of $\Sigma$ is determined by $\overline{R}$, $I(\rho,N^{1})$, $I(\rho,N^{2})$ and $G_{\Sigma}$. By Lemma \ref{LengthOfClosedOrbits}, the isomorphism class of $(\Sigma,\alpha|_{\Sigma})$ is determined by $\overline{R}$ and $G_{\Sigma}$. By Lemma \ref{NormalFormTheorem}, it suffices to show that the given data determine the isomorphism class of $E$ as a $T^2$-equivariant symplectic vector bundle over $\Sigma$. By Lemma \ref{LocalFreeActions}, we take an $S^1$-subgroup $G'$ of $T^2$ so that $G'$ acts freely on an open neighborhood of $\Sigma$. Hence the isomorphism class of $(E_{x},d\alpha|_{E_{x}})$ as a symplectic vector space with an $S^1$-action determines the isomorphism class of $E$. $(E_{x},d\alpha|_{E_{x}})$ is isomorphic to $\mathbb{C}^{2}$ with the standard symplectic structure and an $S^1$-action defined by
\begin{equation}
s \cdot (z_1,z_2)=(s^{m_1}z_{1}, s^{m_2}z_{2})
\end{equation}
where $(m_1,m_2)$ is a pair of coprime integers. We can assume that $N^{1}$ is defined by $z_{2}=0$ and $N^{2}$ is defined by $z_{1}=0$. For $X$ given by \eqref{FormulaOfX}, $\alpha(X)$ is written as 
\begin{equation}
\alpha(X)= m_1 \left( -\frac{s_0 \lambda_1}{\lambda_0} + s_1 \right) |z_1|^{2} + m_2 \left(  -\frac{s_0 \lambda_1}{\lambda_0} + s_1 \right) |z_2|^{2} + \frac{s_0}{\lambda_0}
\end{equation}
near $\Sigma$ in the above coordinate on $S^1 \times D^{4}_{\epsilon}$. Then $m_1 \left( -\frac{s_0 \lambda_1}{\lambda_0} + s_1 \right)$ is positive and $m_2 \left(  -\frac{s_0 \lambda_1}{\lambda_0} + s_1 \right)$ is negative by the assumption. Since the $G'$-action is free near $\Sigma$, we have $|m_1|=I(\rho,N_1)$ and $|m_2|=I(\rho,N_2)$. Hence $m_1$ and $m_2$ are determined by the given data. Since the isomorphism class of $(E_{x},d\alpha|_{E_{x}})$ is determined by $m_1$ and $m_2$, the proof is completed.
\end{proof}

Assume that
\begin{enumerate}
\item $\Sigma_{0}$ is contained in $B_{\min 0}$ and $\Sigma_{1}$ is contained in $B_{\min 1}$ or
\item $\Sigma_{0}$ is contained in $B_{\max 0}$ and $\Sigma_{1}$ is contained in $B_{\max 1}$.
\end{enumerate}
We have the following by the same argument as in the proof of Lemma \ref{NormalFormOfSingularOrbits}:
\begin{lem}\label{NormalFormOfSingularOrbitsInB}
There exists an isomorphism from an open neighborhood of $\Sigma_{0}$ to an open neighborhood of $\Sigma_{1}$ mapping $N_{0}^{1}$ to $N_{1}^{1}$ if and only if $(G_{0})_{\Sigma_{0}}=(G_{1})_{\Sigma_{1}}$, $I(\rho,N_{0}^{1})=I(\rho,N_{1}^{1})$ and $I(\rho,N_{0}^{2})=I(\rho,N_{1}^{2})$. 
\end{lem}

\subsubsection{General cases}

Fix a point $x$ on $\Sigma$. We identify $G/G_{x}$ with $\Sigma$. Define $\tilde{\Sigma}=G/(G_{x})_{0}$ and denote the projection $\tilde{\Sigma} \longrightarrow G/G_{x}$ by $p$. Then $p$ is a $T^2$-equivariant finite cyclic covering with respect to the induced $T^2$-actions from the principal action of $G$ on $G$. Since $G_{\tilde{\Sigma}}=(G_{x})_{0}$, the isotropy group of the $T^2$-action induced on $p^{*}E$ at $0_{p^{*}E}$ is $(G_{x})_{0}$. Hence we have
\begin{lem}\label{FiniteCovering}
There exists a $T^2$-equivariant finite cyclic covering $p \colon \tilde{\Sigma} \longrightarrow \Sigma$ such that the isotropy group of the induced $T^2$-action on $p^{*}E$ at $0_{p^{*}E}$ is connected.
\end{lem}
By Lemma \ref{FiniteCovering}, we can take a good coordinate near a closed orbit if we take a finite covering of an open neighborhood.

The symplectic structure $\omega$ on $\mathbb{C}^{2}$ defined by $\omega=\sqrt{-1}\sum_{j=1}^{2}dz_{j} \wedge d\overline{z}_{j}$ is called standard. We show
\begin{lem}\label{GeneralTorusActions}
The $T^2$-equivariant symplectic vector bundle $E$ is isomorphic to the trivial bundle $S^1 \times \mathbb{C}^{2}$ over $S^1$ with the symplectic structure which is standard on each fiber and the $T^2$-action defined by
\begin{equation}
(t_0,t_1) \cdot (\zeta,z_1,z_2) = (t_{0}^{n_0} \zeta, t_{0}^{n_1} t_{1}^{m_1} z_1, t_{0}^{n_2} t_{1}^{m_2} z_2)
\end{equation}
where $n_0=I(\rho,\Sigma)$ and $n_1,n_2,m_1$ and $m_2$ are some integers which satisfy $\GCD(n_0,n_1,n_2)=1$ and $\GCD(m_1,m_2)=1$.
\end{lem}

\begin{proof}
We denote the projection $G \longrightarrow G/G_{x}=\Sigma$ by $\pi$. Let $g_{*} \colon E \times G \longrightarrow E \times G$ be the map defined as the product of $\rho(g) \colon E \longrightarrow E$ and the left multiplication map of $g$ to $G$. We have the map $\chi \colon G \times E_{x} \longrightarrow \pi^{*}E$ defined by $\chi(g,v)=g_{*}v$ where $g_{*} \colon \pi^{*}E \longrightarrow \pi^{*}E$ is the restriction of $g_{*} \colon E \times G \longrightarrow E \times G$ to $\pi^{*}E$. Since the symplectic structure on the fibers of $E$ are preserved by $\rho$, $\chi$ is an isomorphism as $G$-equivariant symplectic vector bundles. Hence $E$ is obtained by the quotient of the trivial $G$-equivariant vector bundle $G \times E_{x}$ by the action of $G_{x}$.

We fix a subgroup $G'$ of $G$ so that $G' \times (G_{x})_{0}=G$ where $(G_{x})_{0}$ is the identity component of $G_{x}$. Then $G' \times_{G' \cap G_{x}} E_{x}$ has a structure of a $G$-equivariant symplectic vector bundle by the $G$-action defined by $g' h'[g,x]=[g'g,h'x]$ for $g$ in $G$, $g'$ in $G'$, $h'$ in $(G_{x})_{0}$ and $x$ in $E_{x}$. The canonical injection $G' \times_{G' \cap G_{x}} E_{x} \longrightarrow G \times_{G_{x}} E_{x}$ is an isomorphism of $G$-equivariant symplectic vector bundles. Hence $E$ is the quotient of the $T^2$-equivariant vector bundle $G' \times E_{x}$ by a $G' \cap G_{x}$-subaction $\sigma$.

We fix a metric on $E_{x}$ which is compatible with $d\alpha|_{\wedge^{2} E_{x}}$ and invariant under $G_{x}$. We regard $E_{x}$ as a complex vector space with a symplectic form. We fix a basis $\{v_1,v_2\}$ of $E_{x}$ so that $E_{x}$ is identified with $\mathbb{C}^{2}$ with the standard Euclidean metric and the standard symplectic structure. The action of $G_{x}$ on $E_{x}$ is unitary. Since $G_{x}$ is abelian, we can conjugate the action of $G_{x}$ so that the image of $G_{x} \longrightarrow \U(2)$ is contained in the group of diagonal matrices by the simultaneous diagonalization. Hence there exist homomorphisms $\phi_{i} \colon G_{x} \longrightarrow \mathbb{C}$ for $i=0$, $1$ such that the action of $G_{x}$ is written as $g \cdot (w_1,w_2)=( \phi_1(g) w_1, \phi_2(g) w_2)$ for each $g$ in $G_{x}$ and $(w_1,w_2)$ in $\mathbb{C}^{2}$. For $\xi$ in $G' \cong S^1$, the $G$-action on $G' \times E_{x}$ is written as
\begin{equation}
(t_0,t_1) \cdot (\xi,w_1,w_2) = (t_{0} \xi, t_{1}^{m_1} w_1, t_{1}^{m_2} w_2)
\end{equation}
for some integers $m_{1}$ and $m_{2}$ where $G' = G' \times (G_{x})_{0}$ . The $G' \cap G_{x}$-subaction $\sigma$ on $G' \times E_{x}$ is written as
\begin{equation}
s \cdot (\xi,w_1,w_2) = (s^{l_0} \xi, s^{l_1 m_1} w_1, s^{l_1 m_2} w_2)
\end{equation}
for a positive integer $l_0$ and an integer $l_1$ where $s$ is a generator of $\mathbb{Z}/k\mathbb{Z}$. Since $\sigma$ is free, $l_0$ satisfies $\GCD(l_0,k)=1$. By changing the generator of $\mathbb{Z}/k\mathbb{Z}$, we can write $\sigma$ as\begin{equation}
u \cdot (\xi,w_1,w_2) = (u \xi, u^{-n_1} w_1, u^{-n_2} w_2)
\end{equation}
for integers $n_1$ and $n_2$ where $u$ is a generator of $\mathbb{Z}/k\mathbb{Z}$. We have a diffeomorphism $f$ from $G' \times_{G' \cap G_{x}} E_{x}=E$ to $S^1 \times \mathbb{R}^{4}$ induced from a map $f$ defined on $G' \times E_{x}$ by $f(\xi,w_{1},w_{2})=(\xi^{k},\xi^{n_1}w_{1},\xi^{n_2}w_{2})$. Since $\rho$ is induced from $\tilde{\rho}$, $\rho$ is written in the new coordinate $(\zeta,z_1,z_2)=(\xi,\xi^{n_1}w_{1},\xi^{n_2}w_{2})$ on $G' \times_{G' \cap G_{x}} E_{x}$ as
\begin{equation}
(t_{0},t_{1}) \cdot (\zeta,z_1,z_2) = (t_{0}^{k} \zeta, t_{0}^{n_1} t_{1}^{m_1} z_1, t_{0}^{n_2} t_{1}^{m_2} z_2).
\end{equation}
The conditions $\GCD(n_0,n_1,n_2)=1$ and $\GCD(m_1,m_2)=1$ follow from the effectiveness of $\rho$. Note that since $\sigma$ preserves the symplectic structure on $G' \times E_{x}$, the symplectic structure on $E$ is standard with respect to the coordinate $(z_{1},z_{2})$ on each fiber.
\end{proof}

\begin{lem}\label{TwoKContactSubmanifolds2}
There exist two $K$-contact submanifolds containing $\Sigma$.
\end{lem}

\begin{proof}
By Lemma \ref{GeneralTorusActions}, $E$ is isomorphic to the trivial bundle $S^1 \times \mathbb{C}^{4}$ over $S^1$ with the $T^2$-action defined by
\begin{equation}
(t_0,t_1) \cdot (\zeta,z_1,z_2) = (t_{0}^{n_0} \zeta, t_{0}^{n_1} t_{1}^{m_1} z_1, t_{0}^{n_2} t_{1}^{m_2} z_2)
\end{equation}
for some integers $n_0,n_1,n_2,m_1$ and $m_2$. $S^1 \times \mathbb{C}^{4}$ has two $T^2$-invariant symplectic vector subbundles $E_1$ and $E_2$ defined by the equations $z_1=0$ and $z_2=0$. The claim follows from Corollary \ref{NormalBundles}.
\end{proof}

The $S^1$-action on $\mathbb{C}$ defined by $s \cdot z= sz$ for $s$ in $S^1$ is called standard. Let $H$ be a $1$-dimensional Lie group. For a $1$-dimensional complex nondiscrete $H$-action $\sigma$, we define the signature $s(\sigma)$ of the weight of $\sigma$ by $s(\sigma)=1$ if the induced action of $H/\ker \sigma$ on $\mathbb{C}$ is standard, $s(\sigma)=-1$ if the induced action of $H/\ker \sigma$ on $\mathbb{C}$ is not standard. The following lemma is clear:
\begin{lem}\label{ComplexRepresentations}
Two $1$-dimensional complex nondiscrete $H$-actions $\sigma_0$ and $\sigma_1$ are isomorphic if and only if the kernels and the signatures of the weights of $\sigma_{0}$ and $\sigma_{1}$ are equal.
\end{lem}

Let $N_{i}^{0}$ and $N_{i}^{1}$ be two $K$-contact submanifolds containing $\Sigma_{i}$ for $i=0$ and $1$. Assume that $N_{i}^{1}$ is contained in $\Phi_{i}^{-1}((-\infty,\Phi_{i}(\Sigma_{i})])$ and $N_{i}^{2}$ is contained in $\Phi_{i}^{-1}([\Phi_{i}(\Sigma_{i}),\infty))$ for $i=0$ and $1$. If $\Sigma_{i}$ is neither contained in $B_{\min i}$ nor in $B_{\max i}$ for $i=0$ and $1$, we have the following:
\begin{lem}\label{NormalFormOfSingularOrbits2}
There exists an isomorphism from an open neighborhood of $\Sigma_{0}$ to an open neighborhood of $\Sigma_{1}$ mapping $N_{0}^{0}$ to $N_{1}^{0}$ if and only if $(G_{0})_{\Sigma_{0}}=(G_{1})_{\Sigma_{1}}$ and $(G_{0})_{N_{0}^{j}}=(G_{1})_{N_{1}^{j}}$ for $j=0$ and $1$.
\end{lem}
\begin{proof}
It suffices to show that the normal form of $\Sigma_{i}$ is determined by $\overline{R}_{i}$, $(G_{i})_{\Sigma_{i}}$, $(G_{i})_{N^{0}_{i}}$ and $(G_{i})_{N^{1}_{i}}$ for $i=0$ and $1$. Since we argue on each $i=0$ and $1$ separately, we omit the subscript $i$ in the following argument. By Lemma \ref{LengthOfClosedOrbits}, the isomorphism class of $(\Sigma,\alpha|_{\Sigma})$ is determined by $\overline{R}$ and $G_{\Sigma}$. 
 
We will show that the $T^2$-equivariant symplectic normal bundle $E$ of $\Sigma$ is determined by the given data. By the slice theorem, the $T^2$-equivariant normal bundle of $\Sigma$ is determined by $G_{\Sigma}$ and the action of $G_{\Sigma}$ on a fiber $E_{x}$ of $E$ over a point $x$ in $\Sigma$. Hence it suffices to show that the symplectic representation of $G_{\Sigma}$ on $E_{x}$ is determined by the given data. By Lemma \ref{TwoKContactSubmanifolds2}, the $G_{\Sigma}$-action on $E_{x}$ is decomposed into a direct sum of two $2$-dimensional symplectic linear representations as $E_{x}= ((TN^{1})^{d\alpha})_{x} \oplus ((TN^{2})^{d\alpha})_{x}$ where $N^{1}$ and $N^{2}$ are the $K$-contact submanifolds containing $\Sigma$ and $(TN^{j})^{d\alpha}$ is the symplectic normal bundle of $N^{j}$ in $(M,\alpha)$ for $j=1$ and $2$. Fix complex structures on $TN^{1}_{x}$ and $TN^{2}_{x}$ which is compatible with symplectic structures and invariant under $G_{\Sigma}$. Note that the actions of $G_{\Sigma}$ on $TN^{1}_{x}$ and $TN^{2}_{x}$ are not discrete by the assumption, since the kernels of the direct summands are $G_{N^{1}}$ and $G_{N^{2}}$. Since the signatures of weights are determined by the assumption on $N^{j}_{i}$ using the expression of $\alpha$ on a finite covering of an open tubular neighborhood as in the argument in the last part of the proof of Lemma \ref{NormalFormOfSingularOrbits}, the isomorphism classes of direct summands are determined as unitary representations by the given data by Lemma \ref{ComplexRepresentations}.
\end{proof}

Assume that 
\begin{enumerate}
\item $\Sigma_{0}$ is contained in $B_{\min 0}$ of dimension $1$ and $\Sigma_{1}$ is contained in $B_{\min 1}$ of dimension $1$ or
\item $\Sigma_{0}$ is contained in $B_{\max 0}$ of dimension $1$ and $\Sigma_{1}$ is contained in $B_{\max 1}$ of dimension $1$.
\end{enumerate}
In these cases, the action of the isotropy group is decomposed into a direct sum of two $2$-dimensional symplectic nondiscrete linear representations. Hence we have the following by the same argument as in the proof of Lemma \ref{NormalFormOfSingularOrbits2}:
\begin{lem}
There exists an isomorphism from an open neighborhood of $\Sigma_{0}$ to an open neighborhood of $\Sigma_{1}$ mapping $N_{0}^{0}$ to $N_{1}^{0}$ if and only if $(G_{0})_{\Sigma_{0}}=(G_{1})_{\Sigma_{1}}$, $(G_{0})_{N_{0}^{0}}=(G_{1})_{N_{1}^{0}}$ and $(G_{0})_{N_{0}^{1}}=(G_{1})_{N_{1}^{1}}$.
\end{lem}

Assume that 
\begin{enumerate}
\item $\Sigma_{0}$ is contained in $B_{\min 0}$ of dimension $3$ and $\Sigma_{1}$ is contained in $B_{\min 1}$ of dimension $3$ or 
\item $\Sigma_{0}$ is contained in $B_{\max 0}$ of dimension $3$ and $\Sigma_{1}$ is contained in $B_{\max 1}$ of dimension $3$.
\end{enumerate}
We denote $B_{\max i}$ or $B_{\min i}$ containing $\Sigma_{i}$ by $B_i$ for $i=0$ and $1$. Let $N^{i}$ be a $K$-contact submanifold containing $\Sigma_{i}$ other than $B_{i}$ for $i=0$ and $1$. We denote the Seifert fibration on $B_{i}$ defined by the orbits of $\rho_{i}$ by $\mathcal{F}_{i}$ for $i=0$ and $1$.
\begin{lem}\label{ClosedOrbitsInB}
There exists an isomorphism from an open neighborhood of $\Sigma_{0}$ to an open neighborhood of $\Sigma_{1}$ mapping $N^{0}$ to $N^{1}$ if and only if $(G_{0})_{\Sigma_{0}}=(G_{1})_{\Sigma_{1}}$, $(G_{0})_{N_{0}}=(G_{1})_{N_{1}}$ and the Seifert invariant of $\mathcal{F}_{0}$ at $\Sigma_{0}$ is equal to the Seifert invariant of $\mathcal{F}_{1}$ at $\Sigma_{1}$.
\end{lem}
\begin{proof}
The proof is similar to the proof of Lemma \ref{NormalFormOfSingularOrbits2}. It suffices to show that the normal form of $\Sigma_{i}$ is determined by $\overline{R}_{i}$, $(G_{i})_{\Sigma_{i}}$, $(G_{i})_{N}$ and the Seifert invariant of $\mathcal{F}_{i}$ at $\Sigma_{i}$ for each $i=0$ and $1$. Since we argue on each $i=0$ and $1$ separately, we omit the subscript $i$ in the following argument. By Lemma \ref{LengthOfClosedOrbits}, the isomorphism class of $(\Sigma,\alpha|_{\Sigma})$ is determined by $\overline{R}$ and $G_{\Sigma}$. 

We will show that the $T^2$-equivariant symplectic normal bundle $E$ of $\Sigma$ is determined by the given data. By the slice theorem, the $T^2$-equivariant normal bundle of $\Sigma$ is determined by $G_{\Sigma}$ and the action of $G_{\Sigma}$ on the fiber $E_{x}$ of $E$ over a point $x$ in $\Sigma$. Hence it suffices to show that the symplectic representation of $G_{\Sigma}$ on $E_{x}$ is determined by the given data. By the assumption, the $G_{\Sigma}$-action on $E_{x}$ is decomposed into a direct sum of two $2$-dimensional symplectic linear representations as $E_{x}=TN_{x} \oplus TB_{x}$. Note that the actions of $G_{\Sigma}$ on $TN_{x}$ is not discrete by the assumption. Since the signature of the weight of  the actions of $G_{\Sigma}$ on $TN_{x}$ is determined by the assumption on $N^{j}_{i}$ using the expression of $\alpha$ on a finite covering of an open tubular neighborhood as the argument in the last part of the proof of Lemma \ref{NormalFormOfSingularOrbits}, the isomorphism classes of direct summands are determined as unitary representations by the given data by Lemma \ref{ComplexRepresentations}. Hence it suffices to show that the action of $G_{\Sigma}$ on $TB_{x}$ is determined by the given data. First, we show that $G_{B}=(G_{\Sigma})_{0}$. $G_{B}$ is contained in $G_{\Sigma}$ clearly. Since $G_{B}$ acts on a fiber of the normal bundle of $B$ effectively, $G_{B}$ is isomorphic to a subgroup of $\SO(2)$. Since $G_{B}$ is of dimension $1$, $G_{B}$ is isomorphic to $S^1$. Hence we have $G_{B}=(G_{\Sigma})_{0}$. Then the action of $G_{\Sigma}$ on $TB_{x}$ factors the action of $G_{\Sigma}/(G_{\Sigma})_{0}=G_{\Sigma}/G_{B}$ on $TB_{x}$. Since the action of $G_{\Sigma}/G_{B}$ on $TB_{x}$ is determined by the Seifert invariant of $\mathcal{F}$ at $\Sigma$, the proof is completed.
\end{proof}

Assume that $\Sigma$ is neither contained in $B_{\min}$ nor in $B_{\max}$. Let $N^1$ and $N^2$ be the two $K$-contact submanifolds containing $\Sigma$. We have
\begin{lem}\label{GCD}
$I(\rho,\Sigma)=\GCD(I(\rho,N^1),I(\rho,N^2))$.
\end{lem}

\begin{proof}
By Lemmas \ref{NormalBundles} and \ref{GeneralTorusActions}, we have a coordinate on an open neighborhood of $\Sigma$ such that $\rho$ is written as
\begin{equation}
(t,s) \cdot (\zeta,z_1,z_2)=(t^{n_0} \zeta, t^{n_1} s^{m_1} z_{1}, t^{n_2} s^{m_2} z_{2})
\end{equation}
where $n_1,n_2,m_1$ and $m_2$ are integers satisfying $\GCD(n_0,n_1,n_2)=1$, $\GCD(m_1,m_2)=1$ and $n_0=I(\rho,\Sigma)$. $N^1$ and $N^2$ are defined by the equations $z_1=0$ and $z_2=0$ respectively in the above coordinate $(\zeta,z_{1},z_{2})$, because their tangent bundles are invariant under $\rho$. By the assumption, $m_1$ and $m_2$ are nonzero. Since we have $I(\rho,N^1)=|n_0 m_2|$ and $I(\rho,N^2)=|n_0 m_1|$, the proof is completed.
\end{proof}

\subsection{Normal forms of $K$-contact lens spaces and gradient manifolds}

A gradient manifold of a $5$-dimensional $K$-contact manifold of rank $2$ is a $K$-contact manifold of rank $2$. The closures of gradient manifolds are not submanifolds of $M$ in general near $B_{\min}$ or $B_{\max}$ of dimension $1$. Since the limit set of a gradient manifold is the union of two different closed orbits of the Reeb flow, we can obtain normal forms of gradient manifolds by the normal form theorem of closed orbits.

 We fix an isomorphism $G_{i} \longrightarrow T^2$ and identify $G_{i}$ with $T^2$ for $i=0$ and $1$. Let $\overline{R}_{i}$ denote the element of $\Lie(T^2)$ corresponding to the Reeb vector field of $\alpha_{i}$ for $i=0$ and $1$. We assume $\overline{R}_{0}=\overline{R}_{1}$ and denote them by $\overline{R}$. We fix $\overline{X}$ in $\Lie(T^2)$ which is not parallel to $\overline{R}$. Put $\Phi_{i}=\alpha_{i}(X)$ where $X$ is the infinitesimal action of $\overline{X}$.  The maximal component and the minimal component of $\Phi_{i}$ are denoted by $B_{\max i}$ and $B_{\min i}$, respectively, for $i=0$ and $1$.

Let $L_{i}$ be a gradient manifold of $\Phi_{i}$ in $(M_{i},\alpha_{i})$ for $i=0$ and $1$. Let $\Sigma_{i}^{0}$ be the $\alpha$-limit set of $L_{i}$, and $\Sigma_{i}^{1}$ be the $\omega$-limit set of $L_{i}$. Let $N_{i}^{1}$ be the closure of $L_{i}$ in $M_{i}$ for $i=0$ and $1$.

A diffeomorphism $f$ from an open neighborhood $U_{0}$ of $L_{0}$ to an open neighborhood $U_{1}$ of $L_{1}$ is said to be an isomorphism if $f(L_{0})=L_{1}$, $f^{*}\alpha_{1}=\alpha_{0}$ and $f(\rho_{0}(x,g))=\rho_{1}(f(x),g)$ for every $x$ in $U_{0}$ and every $g$ in $T^2$.

\subsubsection{Smooth cases}

Assume that $N_{i}^{1}$ is a smooth submanifold of $M_{i}$ for $i=0$ and $1$. Let $N_{i}^{0}$ be a $3$-dimensional $K$-contact submanifold containing $\Sigma_{i}^{0}$ other than $N_{i}^{1}$ for $i=0$ and $1$. Let $N_{i}^{2}$ be a $3$-dimensional $K$-contact submanifold containing $\Sigma_{i}^{1}$ other than $N_{i}^{1}$ for $i=0$ and $1$. The existence of $N_{i}^{0}$ and $N_{i}^{2}$ for $i=0$ and $1$ follows from Lemma \ref{TwoKContactSubmanifolds}. Note that $N_{i}^{0}$ and $N_{i}^{2}$ can be $B_{\min i}$ or $B_{\max i}$ of dimension $3$ for $i=0$ and $1$. 

We prepare a lemma.
\begin{lem}\label{Isomorphisms}
Let $(N_{0},\alpha_{0})$ and $(N_{1},\alpha_{1})$ be two connected $3$-dimensional $K$-contact manifolds of rank $2$. We fix an isomorphism $G_{i} \longrightarrow T^2$. We assume that $\overline{R}_{0}=\overline{R}_{1}$ where $\overline{R}_{i}$ is the element of $\Lie(T^2)$ corresponding to the Reeb vector field of $\alpha_{i}$ for $i=0$ and $1$. We denote $\overline{R}_{0}$ and $\overline{R}_{1}$ by $\overline{R}$. We denote the contact moment map of $(N_{i},\alpha_{i})$ by $\Phi_{\alpha_{i}}$ for $i=0$ and $1$.

Let $U_{i}$ and $V_{i}$ be open sets of $N_{i}$ invariant under the Reeb flow of $\alpha_{i}$ for $i=0,1$. Let $f$ be an isomorphism from $(V_{0},\alpha_{0}|_{V_{0}})$ to $(V_{1},\alpha_{1}|_{V_{1}})$. Assume that 
\begin{enumerate}
\item $\overline{U}_{i}$ is contained in $V_{i}$, 
\item $N_{i}-U_{i}$ is connected and
\item $U_{i}$ and $V_{i}$ have two connected components.
\end{enumerate}
Then we have an isomorphism $\tilde{f} \colon (N_{0},\alpha_{0}) \longrightarrow (N_{1},\alpha_{1})$ such that $\tilde{f}_{0}|_{U_{0}}=f$.
\end{lem}

\begin{proof}
First, we show that each fiber of $\Phi_{\alpha_{i}}$ is an orbit of $\rho_{i}$. Note that each connected component of the fibers of $\Phi_{\alpha_{i}}$ is an orbit of $\rho_{i}$. It is because the orbits of $\rho_{i}$ are of codimension $1$ in $N_{i}$ and the contact moment map $\Phi_{\alpha_{i}}$ is a submersion on the union of orbits with trivial isotropy group of $\rho_{i}$ by Lemma \ref{MomentMapsAreSubmersive}. Hence it suffices to show that each fiber of $\Phi_{i}$ is connected. Fix an infinitesimal action $Y_{i 0}$ of $\rho_{i}$ so that every orbit of the flow generated by $Y_{i 0}$ is closed. We denote the $S^1$-action generated by $Y_{i 0}$ by $\rho'_{i}$. We put $\alpha'_{i}=\frac{1}{\alpha'_{i}(Y_{i 0})}\alpha$. Then $d\alpha'_{i}$ induces a symplectic form on the orbifold $N_{i}/\rho'_{i}$. $\rho_{i}$ induces an $S^1$-action $\overline{\rho}_{i}$ on $N_{i}/\rho'_{i}$. We show that $\overline{\rho}_{i}$ is a hamiltonian $S^1$-action on $(N_{i}/\rho'_{i},d\alpha'_{i})$. Let $\overline{Z}_{i}$ be an infinitesimal action of $\overline{\rho}_{i}$. Then there exists an infinitesimal action $Z_{i}$ of $\rho_{i}$ such that $\pi_{*}Z_{i}=\overline{Z}_{i}$, where $\pi$ is the canonical projection $N_{i} \longrightarrow N_{i}/\rho'_{i}$. Note that $\alpha(Z_{i})$ is constant on the orbits of $\rho_{i}$ and can be regarded as a smooth function on $N_{i}/\rho'_{i}$. The $1$-form $\iota(Z_{i})d\alpha'$ is basic with respect to $\rho'$ and can be regarded as a $1$-form on $N_{i}/\rho'_{i}$. $\iota(Z_{i})d\alpha'$ is equal to $\iota(\overline{Z}_{i})d\alpha'$ as a $1$-form on $N_{i}/\rho'_{i}$. Since $L_{Z_{i}}\alpha'=0$, we have $d(\alpha'(Z_{i}))=-\iota(\overline{Z}_{i})d\alpha'$, which implies that $\overline{\rho}$ is hamiltonian. The fibers of the hamiltonian function $\alpha'(Z_{i})$ on $N_{i}/\rho'_{i}$ is the quotient of the fibers of the contact moment map on $(N_{i},\alpha_{i})$ by $\rho'$. Hence to show that a fiber of $\Phi_{i}$ is connected, it suffices to show that the the fiber of the hamiltonian function $\alpha'(Z_{i})$ on $N_{i}/\rho'_{i}$ is connected. This follows from the connectivity of the fibers of the moment map of symplectic orbifolds. We refer \cite{LeTo}. 

Since $U_{i}$ is invariant under the Reeb flow and the fiber of $\Phi_{\alpha_{i}}$ is an orbit of $\rho_{i}$,
\begin{equation}\label{MomentMapImage0}
\Phi_{\alpha_{i}}(U_{i}) \cap \Phi_{\alpha_{i}}(N_{i}-U_{i})=\emptyset
\end{equation}
for $i=0$ and $1$. Assume that $\Phi_{\alpha_{i}}(U_{i}) \cap \Phi_{\alpha_{i}}(N_{i}-U_{i})$ contains a point $x_{i 0}$. Since $\Phi_{\alpha_{i}}^{-1}(x_{i 0})$ is an orbit of $\rho_{i}$, $\Phi_{\alpha_{i}}^{-1}(x_{i 0})$ is contained in $U_{i} \cap (N_{i}-U_{i})$ by the invariance of $U_{i}$ under the Reeb flows. It is contradiction. Hence \eqref{MomentMapImage0} is proved.

We show
\begin{equation}\label{MomentMapImage2}
\Phi_{\alpha_{0}}(N_{0})=\Phi_{\alpha_{1}}(N_{1}).
\end{equation}
In fact, we show that the assumption implies that $\Phi_{\alpha_{i}}(N_{i})$ is the convex hull of $\Phi_{\alpha_{i}}(U_{i})$ for $i=0$ and $1$. Note that the image of $\Phi_{\alpha_{0}}$ and $\Phi_{\alpha_{1}}$ are contained in a $1$-dimensional affine space $\{v \in \Lie(T^2)^{*}| v(\overline{R})=1\}$ of $\Lie(T^2)^{*}$. $\Phi_{\alpha_{i}}(N_{i}-U_{i})$ is a connected interval in $\Phi_{\alpha_{i}}(N_{i})$. Hence the number of the connected components of $\Phi_{\alpha_{i}}(U_{i})$ is one or two. If the number of the connected components of $\Phi_{\alpha_{i}}(U_{i})$ is two, $\Phi_{\alpha_{i}}(N_{i})$ is the convex hull of $\Phi_{\alpha_{i}}(U_{i})$ clearly. Assume that the number of the connected component of $\Phi_{\alpha_{i}}(U_{i})$ is one. Let $U_{i}^{1}$ and $U_{i}^{2}$ be the connected components of $U_{i}$. Since a fiber of $\Phi_{\alpha_{0}}$ and $\Phi_{\alpha_{1}}$ is an orbit of $\rho_{i}$ for $i=0$ and $1$, $U_{i}^{1} \cap U_{i}^{2} = \emptyset$, and it is contradiction. Hence \eqref{MomentMapImage2} is proved.

Since $f|_{U_{0}} \colon U_{0} \longrightarrow U_{1}$ is an isomorphism, we have $\Phi_{\alpha_{1}} \circ f=\Phi_{\alpha_{0}}$. Then we have
\begin{equation}\label{MomentMapImage1}
\Phi_{\alpha_{0}}(U_{0})=\Phi_{\alpha_{1}}(U_{1}).
\end{equation}

By \eqref{MomentMapImage0}, \eqref{MomentMapImage2} and \eqref{MomentMapImage1}, we have
\begin{equation}\label{MomentMapImage}
\Phi_{\alpha_{0}}(N_{0}-U_{0})=\Phi_{\alpha_{1}}(N_{1}-U_{1}).
\end{equation}

We put $S=\Phi_{\alpha_{0}}(N_{0}-U_{0})=\Phi_{\alpha_{1}}(N_{1}-U_{1})$. We have a $3$-dimensional $K$-contact manifold $(T^2 \times S, \beta)$ of rank $2$ where $\beta$ is defined by
\begin{enumerate}
\item $\beta(\frac{\partial}{\partial s})=0$ where $\frac{\partial}{\partial s}$ is the basis of the tangent space of $S$ in the product $T^2 \times S$ and
\item $\beta|_{p^{-1}(s)}=s$ for $s$ in $S$ where we regard $s$ as a $1$-form on $T^2$.
\end{enumerate}
Then the image of the contact moment map of $(T^2 \times S, \beta)$ is $S$. By Lemma 4.9 and Proposition 5.2 of \cite{Ler2}, two $3$-dimensional $K$-contact manifolds of rank $2$ are isomorphic if the images of the contact moment maps are the same convex sets. Hence $(N_{0}-U_{0},\alpha_{0}|_{N_{0}-U_{0}})$ and $(N_{1}-U_{1},\alpha_{1}|_{N_{1}-U_{1}})$ are isomorphic to $(T^2 \times S, \beta)$. We put $\overline{R}_{0}=\overline{R}_{1}=a \frac{\partial}{\partial x} + b \frac{\partial}{\partial y}$ for real numbers $a$ and $b$. We put a coordinate $s=s_{1} dx + s_{2} dy$ on $\Lie(T^2)^{*}$. Then $S$ is defined by the equation $a s_{1} + b s_{2}=0$. If we put $s_{1}=\frac{\cos t}{a\cos t + b\sin t}$ and $s_{2}=\frac{\sin t}{a\cos t + b\sin t}$, $\beta$ is written as $\beta=\frac{dx(\cos t)+dy(\sin t)}{a\cos t + b\sin t}$.

We identify  $(N_{0}-U_{0},\alpha_{0}|_{N_{0}-U_{0}})$ and $(N_{1}-U_{1},\alpha_{1}|_{N_{1}-U_{1}})$ with $(T^2 \times S, \beta)$ respectively. Then $f$ is regarded as an isomorphism $\overline{f}$ from an open neighborhood of the boundary of $(T^2 \times S,\beta)$ to an open neighborhood of the boundary of $(T^2 \times S,\beta)$. To prove Lemma \ref{Isomorphisms}, it suffices to show that $\overline{f}$ extends to $(T^2 \times S,\beta)$ as an isomorphism. Since $\overline{f}$ commutes with contact moment maps and is $T^2$-equivariant, $\overline{f}$ is written as $\overline{f}(x,y,t)=(x+\overline{f}_1(t),y+\overline{f}_2(t),t)$ where $\overline{f}_1$ and $\overline{f}_2$ are functions on $S$ defined only near the boundary of $S$. Note that for a diffeomorphism $F$ on $T^2 \times S$ written as $F(x,y,t)=(x+F_1(t),y+F_2(t),t)$, we have $F^{*}\beta=\beta$ if and only if $F'_1(t) \cos t + F'_2(t) \sin t=0$. Hence to extend $\overline{f}$, it suffices to show that $\overline{f}_1$ and $\overline{f}_2$ extend to $S$ satisfying $\overline{f}'_1(t) \cos t + \overline{f}'_2(t) \sin t=0$. 

We denote the endpoints of $S$ by $s_0$ and $s_1$. We put the values of the parameter $t$ at $s_{j}$ by $t_{j}$ for $j=0$ and $1$. First, take a smooth function $\tilde{f}_{1}$ on $S$ whose restriction to an open neighborhood of $t=t_{0}$ coincides with $\overline{f}_{1}$ and whose restriction to open neighborhoods of the points $t=0$ and $t=\pi$ is equal to $\cos t$ if $t=0$ or $t=\pi$ is contained in the interior of the domain of $\tilde{f}_{1}$. Then there exists a unique smooth function $\tilde{f}_{2}$ which satisfies 
\begin{enumerate}
\item $\tilde{f}'_1(t) \cos t + \tilde{f}'_2(t) \sin t=0$ and
\item $\tilde{f}_{2}|_{W_{0}}=\overline{f}_{2}|_{W_{0}}$
\end{enumerate}
where $W_{0}$ is an open neighborhood of $s_{0}$. Since the derivation of $\tilde{f}_{2}|_{W_{1}}$ and $\overline{f}_{2}|_{W_{1}}$ are equal, we have a constant $c$ such that $\tilde{f}_{2}|_{W_{1}}-\overline{f}_{2}|_{W_{1}}=c$ where $W_{1}$ is an open neighborhood of $s_{1}$ contained in the domain of $\overline{f}_{2}$. $c$ is determined by $\tilde{f}_{1}$. To extend $\overline{f}_1$ and $\overline{f}_2$ on $S$, it suffices to choose $\tilde{f}_1$ so that the constant $c$ is zero. We show that there exists a smooth function $r$ on $S$ such that $\tilde{f}_1 + r$ satisfies the desired condition. Choose an interval $[w_{0},w_{1}]$ in $S$ which does not contain parameters $t=0$ nor $t=\pi$ and does not intersect $W_{0}$ nor $W_{1}$. Take a nonzero smooth function $r_{1}$ on $S$ whose support is compact in $(w_{0},w_{1})$ and satisfies $\int_{t_{0}}^{t_{1}}r'_{1}(t)\frac{\cos t}{\sin t}dt=1$. Put $r(t)=-cr_{1}(t)$. Then we have $\int_{t_{0}}^{t_{1}}r'(t)\frac{\cos t}{\sin t}dt=-c$. Let $\hat{f}_{2}$ be the unique smooth function on $S$ which satisfies $(\tilde{f}'_1(t) + r(t) )\cos t + \hat{f}'_2(t) \sin t=0$ and $\hat{f}_{2}|_{W_{0}}=f_{2}|_{W_{0}}$. Then we have $\tilde{f}_{2}|_{W_{1}}-\overline{f}_{2}|_{W_{1}}=c - \int_{t_{0}}^{t_{1}}r'(t)\frac{\cos t}{\sin t}dt=0$. Hence the proof is completed.
\end{proof}

We have
\begin{lem}\label{IsomorphismClassOfL}
There exists an isomorphism from $(N_{0}^{1},\alpha|_{N_{0}^{1}})$ to $(N_{1}^{1},\alpha|_{N_{1}^{1}})$ mapping $\Sigma_{0}^{0}$ to $\Sigma_{1}^{0}$ if and only if $(G_{0})_{\Sigma_{0}^{j}} = (G_{1})_{\Sigma_{1}^{j}}$ for $j=0$ and $1$.
\end{lem}

\begin{proof}
It suffices to show that the isomorphism class of $(N_{i}^{1},\alpha|_{N_{i}^{1}})$ is determined by $\overline{R}_{i}$, $G_{\Sigma_{i}^{0}}$ and $G_{\Sigma_{i}^{1}}$ for $i=0$ and $1$. Since we argue on each $i=0$ and $1$ separately, we omit the subscript $i$ in the following argument. $(N^{1},\alpha|_{N^{1}})$ is a $3$-dimensional contact toric manifold. By the classification theorem of $3$-dimensional contact toric manifold by Lerman \cite{Ler2}, the isomorphism class of a $3$-dimensional contact toric manifold is determined by the maximum value and the minimum value of the contact moment map. By Lemma \ref{ValuesOfContactMomentMaps}, the maximum and minimum values of contact moment map on $N^{1}$ are determined by $\overline{R}$, $G_{\Sigma_{0}}$ and $G_{\Sigma_{1}}$. Hence Lemma \ref{IsomorphismClassOfL} is proved.
\end{proof}

Let $A_{j}$ be an open tubular neighborhood of $\Sigma_{0}^{0}$ in $M_{0}$ for $j=1$ and $2$ and $f^{0} \colon A_{2} \longrightarrow M_{1}$ be a diffeomorphism into $M_1$. Assume that 
\begin{enumerate}
\item $f^{0}(A_2 \cap N^{0})=f^{0}(A_2) \cap N^{1}$,
\item $\alpha_0|_{A_2}=f^{0 *}\alpha_1|_{A_2}$, 
\item $A_{j}$ is invariant under the Reeb flows of $\alpha_0|_{A_2}$ and $f^{0 *}\alpha_1|_{A_2}$ and
\item the closure of $A_1$ is contained in $A_{2}$.
\end{enumerate}
We show
\begin{lem}\label{NormalFormOfGradientManifolds2}
There exists an isomorphism $f$ from an open neighborhoods of $N_{0}^{1}$ to an open neighborhood of $N_{1}^{1}$ which maps $\Sigma_{0}^{0}$ to $\Sigma_{1}^{0}$ and satisfies $f|_{A_{1}}=f^{0}|_{A_1}$ if and only if $(G_{0})_{\Sigma_{0}^{j}}=(G_{1})_{\Sigma_{1}^{j}}$ for $j=0$, $1$ and $(G_{0})_{N_{0}^{j}}=(G_{1})_{N_{1}^{j}}$ for $j=0$, $1$, $2$.
\end{lem}

\begin{proof}
By Lemmas \ref{RelativeNormalFormTheorem} and \ref{RelativeNormalBundles}, it suffices to construct an isomorphism $\phi \colon N_{0}^{1} \longrightarrow N_{1}^{1}$ and the isomorphism $q$ between $T^2$-equivariant symplectic normal bundles of $N_{0}^{1}$ and $N_{1}^{1}$ which covers $\phi$ such that $\phi|_{A'_1}=f^{0}|_{A'_1}$ and the restriction of $q$ on $A'_1$ is induced from $f^{0}$ for an open neighborhood $A'_1$ of $A_1$.

We construct $\phi$. By Lemma \ref{IsomorphismClassOfL}, we have an isomorphism $\phi' \colon N_{0}^{1} \longrightarrow N_{1}^{1}$. $f \circ \phi^{\prime -1}$ is an isomorphism from an open neighborhood $\phi(U)$ to $\phi'(U)$ where $U$ is an open neighborhood of $A_{1}$ in $N_{0}^{1}$. By Lemma \ref{Isomorphisms}, there exists an isomorphism $\phi'' \colon N_{1}^{1} \longrightarrow N_{1}^{1}$ such that $\phi''|_{A'_{1}}=f \circ \phi^{\prime -1}|_{U}$. We put $\phi = \phi' \circ \phi''$. Then $\phi$ is an isomorphism $\phi \colon N_{0}^{1} \longrightarrow N_{1}^{1}$ such that $\phi|_{A'_1}=f^{0}|_{A'_1}$.

We construct $q$. Since the normal bundles of $N_{0}^{1}$ and $N_{1}^{1}$ are vector bundles of rank $2$, the symplectic structures on fibers of the normal bundles of $N_{0}^{1}$ and $N_{1}^{1}$ are volume forms. Hence to construct an isomorphism as symplectic vector bundles between the normal bundles of $N_{0}^{1}$ and $N_{1}^{1}$, it suffices to construct an isomorphism $q$ between their underlying oriented $T^2$-equivariant vector bundles. We have an $\alpha_{0}$-preserving $T^3$-action $\tau_{0}$ on an open neighborhood of $N_{0}^{1}$ by Lemma \ref{ToricActionNearGradientManifolds}. Conjugating the restriction of $\tau_{0}$ to $A'_1$ by $f$, we have an $\alpha_{1}$-preserving $T^3$-action $\tau_{1}$ on an open neighborhood of $\Sigma_{1}^{0}$ such that $\tau_{0}=f^{-1} \circ \tau_{1} \circ f$. The generic orbit of $\tau_{i}$ in $N_{i}^{1}$ is diffeomorphic to $T^2$. Let $C_{0}$ be a $T^2$-orbit of $\tau_{0}$ in $N_{0}^{1} \cap A'_1$. We fix the Heegaard decomposition of $N_{0}^{1}$ of genus $1$ which is invariant under $\tau_{0}$ such that $C_{0}$ is the boundary of two solid tori. We fix the Heegaard decomposition of $N_{1}^{1}$ of genus $1$ by mapping the Heegaard decomposition of $N_{0}^{1}$ by $f$ as in the following diagram:
\begin{equation}\label{CoordinatesForDecomposition}
\xymatrix{ S^1 \times D^2 \ar[rd] \ar[rrd] &              &           \\
                          & N_{0}^{1} \ar[r]^<{\,\,\,\,\,f} & N_{1}^{1}.         \\
           S^1 \times D^2 \ar[ru] \ar[rru] &              &               }
\end{equation}

Cutting the total space $E_{0}$ of the normal bundle of $N_{0}^{1}$ along the union of fibers over $C_{0}$, we have a $\tau$-invariant decomposition of $E_{0}$ defined by
\begin{equation}
E_{0}= (S^1 \times D^2)_{0} \times \mathbb{R}^2 \cup_{\tilde{A}_{0}} (S^1 \times D^2)_{1} \times \mathbb{R}^{2}
\end{equation}
such that $\tau_{0}$ is written as 
\begin{equation}
\mathbf{t} \cdot (\zeta,z_1,z_2)=( l_{00}^{j}(\mathbf{t}) \zeta, l_{01}^{j}(\mathbf{t}) z_{1}, l_{02}^{j}(\mathbf{t}) z_{2})
\end{equation}
 on $(S^1 \times \partial D^2)_{j} \times \mathbb{R}^2$ for every $\mathbf{t}$ in $T^3$ where $l_{00}^{j}$, $l_{01}^{j}$ and $l_{02}^{j}$ are homomorphisms from $T^3$ to $S^1$ for $j=0$ and $1$. The attaching map $\tilde{A}_{0} \colon (S^1 \times \partial D^2)_{0} \times \mathbb{R}^2 \longrightarrow (S^1 \times \partial D^2)_{1} \times \mathbb{R}^2$ is written as
\begin{equation}\label{tildeA}
\tilde{A}_{0} (\zeta,z_1,z_2) = (\zeta^{a} z_{1}^{b}, \zeta^{c} z_{1}^{d} , \zeta^{e_{0}} z_{1}^{e_{0}} z_{2}^{f_{0}})
\end{equation}
with respect to the standard coordinates. Note that the map $A_{0} \colon (S^1 \times \partial D^2)_{0} \longrightarrow (S^1 \times \partial D^2)_{1}$ defined by
\begin{equation}
A_{0} (\zeta,z_1) = (\zeta^{a} z_{1}^{b}, \zeta^{c} z_{1}^{d})
\end{equation} 
is the attaching map of the Heegaard decomposition of $N_{0}^{1}$ of genus $1$. The determinant of $A_{0}$ is $-1$.

We decompose the total space $E_{1}$ of the normal bundle of $N_{0}^{1}$ along the union of fibers over $f(C_{0})$ as
\begin{equation}
E_{1} = (S^1 \times D^2)_{0} \times \mathbb{R}^2 \cup_{\tilde{A}_{1}} (S^1 \times D^2)_{1} \times \mathbb{R}^2.
\end{equation}
Using the coordinate
\begin{equation}\label{CoordinatesForDecomposition2}
\xymatrix{(S^1 \times D^2)_{0} \times \mathbb{R}^2 \ar[r] & E_{0} \ar[r]^{f_{*}} & E_{1}}
\end{equation}
as the coordinate near $\Sigma_{0}^{1}$, $\tau_{1}$ is written as 
\begin{equation}
\mathbf{t} \cdot (\zeta,z_1,z_2)=(l_{10}^{j}(\mathbf{t}) \zeta, l_{11}^{j}(\mathbf{t}) z_{1}, l_{12}^{j}(\mathbf{t}) z_{2})
\end{equation}
on $(S^1 \times \partial D^2)_{j} \times \mathbb{R}^2$ for every $t$ in $T^3$ where $l_{10}^{j}(\mathbf{t})$, $l_{11}^{j}(\mathbf{t})$ and $l_{12}^{j}(\mathbf{t})$ are homomorphisms from $T^3$ to $S^1$ for $j=0$ and $1$.

Since the Heegaard decomposition of $N_{0}^{1}$ is mapped to the Heegaard decomposition of $N_{1}^{1}$ by $\phi$, the attaching map $\tilde{A}_{1} \colon (S^1 \times \partial D^2)_{0} \times \mathbb{R}^2 \longrightarrow (S^1 \times \partial D^2)_{1} \times \mathbb{R}^2$ is written as
\begin{equation}
\tilde{A}_{1} = 
\begin{pmatrix}
a & b & 0 \\
c & d & 0 \\
e_{1} & f_{1} & 1
\end{pmatrix}
\end{equation}
with respect to the standard coordinates where $A_{1} = A_{0} = \begin{psmallmatrix} a & b \\ c & d \end{psmallmatrix} \colon (S^1 \times \partial D^2)_{0} \longrightarrow (S^1 \times \partial D^2)_{1}$ is the attaching map of the Heegaard decomposition of $N_{1}^{1}$ of genus $1$. The determinant of $A_{1}$ is $-1$.

Since $f$ is $T^3$-equivariant and we took the decomposition by the coordinates \eqref{CoordinatesForDecomposition} and \eqref{CoordinatesForDecomposition2}, the following equations are satisfied:
\begin{equation}\label{CommonT3Actions}
l_{00}^{0}=l_{10}^{0}, l_{01}^{0}=l_{11}^{0}, l_{02}^{0}=l_{12}^{0}, l_{00}^{1}=l_{10}^{1}, l_{01}^{1}=l_{11}^{1}.
\end{equation}
Let $\rho_{i}$ be given as 
\begin{equation}
(t_0,t_1) \cdot (z_1,\zeta,z_2)=(t_{0}^{q_{i}^{j}} t_{1}^{p_{i}^{j}} z_{1}, t_{0}^{s_{i}^{j}} t_{1}^{r_{i}^{j}} \zeta, t_{0}^{u_{i}^{j}} t_{1}^{v_{i}^{j}} z_{2})
\end{equation}
on $(S^1 \times \partial D^2)_{j} \times \mathbb{R}^2$ so that $\det \begin{pmatrix} s_{0}^{j} & r_{0}^{j} \\ q_{0}^{j} & p_{0}^{j} \end{pmatrix}$ is positive for $j=0$, $1$ and $i=0$, $1$. Then we have
\begin{equation}\label{CommonComponents}
q_{0}^{0}=q_{1}^{0}, p_{0}^{0}=p_{1}^{0}, s_{0}^{0}=s_{1}^{0}, t_{0}^{0}=t_{1}^{0}, u_{0}^{0}=u_{1}^{0}, v_{0}^{0}=v_{1}^{0}
\end{equation}
by \eqref{CommonT3Actions}.

 For $i=0$ and $1$, the $T^2$-equivariant normal bundle of $N_{i}^{1}$ is determined by $\begin{psmallmatrix} s_{i}^{j} \\ q_{i}^{j} \\ u_{i}^{j} \end{psmallmatrix}$, $\begin{psmallmatrix} r_{i}^{j} \\ p_{i}^{j} \\ v_{i}^{j} \end{psmallmatrix}$ and $\tilde{A}_{i}$ for $j=0,1$. Note that since the attaching map $\tilde{A}_{i}$ respects $\rho_{i}$ for $i=0$ and $1$ we have
\begin{equation}\label{WelldefinednessOfRho}
\begin{pmatrix}
s_{i}^{1} \\ q_{i}^{1} \\ u_{i}^{1} 
\end{pmatrix}
=\tilde{A}_{i}
\begin{pmatrix}
s_{i}^{0} \\ q_{i}^{0} \\ u_{i}^{0} 
\end{pmatrix}, 
\begin{pmatrix}
r_{i}^{1} \\ p_{i}^{1} \\ v_{i}^{1} 
\end{pmatrix}
=\tilde{A}_{i}
\begin{pmatrix}
r_{i}^{0} \\ p_{i}^{0} \\ v_{i}^{0}
\end{pmatrix}.
\end{equation}
Hence by \eqref{WelldefinednessOfRho}, $\begin{psmallmatrix} s_{i}^{1} \\ q_{i}^{1} \\ u_{i}^{1} \end{psmallmatrix}$ and $\begin{psmallmatrix} r_{i}^{1} \\ p_{i}^{1} \\ v_{i}^{1} \end{psmallmatrix}$ are determined by $\begin{psmallmatrix} s_{i}^{0} \\ q_{i}^{0} \\ u_{i}^{0} \end{psmallmatrix}$, $\begin{psmallmatrix} r_{i}^{0} \\ p_{i}^{0} \\ v_{i}^{0} \end{psmallmatrix}$ and $\tilde{A}_{i}$. Hence to construct $q$, by \eqref{CommonComponents}, it suffices to show $e_{0}=e_{1}$ and $f_{0}=f_{1}$. Then Lemma \ref{NormalFormOfGradientManifolds2} follows from Lemma \ref{LemmaEquationfork} below.
\end{proof}

\begin{lem}\label{LemmaEquationfork}
We drop the index $i$ in the notation in the proof of Lemma \ref{NormalFormOfGradientManifolds2}. We put $k^{0}=I(\rho,N^{0})$, $k^{1}=I(\rho,N^{1})$ and $k^{2}=I(\rho,N^{2})$.
 We put
\begin{equation}\label{ReductionToCircleActions}
h^{0}=\det \begin{pmatrix} s_{0} & r_{0} \\ u_{0} & v_{0} \end{pmatrix}, h^{1}=\det \begin{pmatrix} s_{1} & r_{1} \\ u_{1} & v_{1} \end{pmatrix}.
\end{equation}
 We have equations
\begin{equation}\label{Equationfork}
\begin{array}{l}
(be-af)k^{1}=-k^{2}+ak^{0}-bh^{0}, \\
(de-cf)k^{1}=h^{1}+ck^{0}-dh^{0}.
\end{array}
\end{equation}
\end{lem}
\begin{proof}
We have
\begin{equation}\label{DefinitionOfk}
-k^{0}=\det \begin{pmatrix} q_{0} & p_{0} \\ u_{0} & v_{0} \end{pmatrix}, k^{1}=\det \begin{pmatrix} s_{0} & r_{0} \\ q_{0} & p_{0} \end{pmatrix}=-\det \begin{pmatrix} s_{1} & r_{1} \\ q_{1} & p_{1} \end{pmatrix}, -k^{2}= \det \begin{pmatrix} q_{1} & p_{1} \\ u_{1} & v_{1} \end{pmatrix}
\end{equation}
by the definition of $k^{0}$, $k^{1}$ and $k^{2}$.

Substituting \eqref{WelldefinednessOfRho} to the third equation of \eqref{DefinitionOfk}, we have 
\begin{equation}\label{ks1}
-k^{2}=(be-af)k^{1}-ak^{0}-bh^{0}.
\end{equation}
Substituting \eqref{WelldefinednessOfRho} to the second equation of \eqref{ReductionToCircleActions}, we have
\begin{equation}\label{ks2}
h^{1}=(de-cf)k^{1}-ck^{0}+dh^{0}.
\end{equation}
\end{proof}

\subsubsection{Nonsmooth cases}

Assume that the closure $N_{i}^{1}$ of $L_{i}$ is not a smooth submanifold of $M_{i}$ for $i=0$ and $1$. Then we have $I(\rho_{i},L_{i})=1$ for $i=0$ and $1$ by Lemma \ref{KContactSubmanifolds}. 

Assume that there exist a $\rho_{i}$-invariant open tubular neighborhood $U_{i}^{j}$ of $\Sigma_{i}^{j}$ for $i=0,1$ and $j=0,1$ and an isomorphism $f^{j} \colon (U_{0}^{j},\alpha_{0}|_{U_{0}^{j}}) \longrightarrow (U_{1}^{j},\alpha_{1}|_{U_{1}^{j}})$ such that $f^{j}(L_{0} \cap U_{0}^{j})=L_{1} \cap U_{1}^{j}$ for $j=0$ and $1$. 
\begin{lem}\label{RelativeNormalFormTheoremForL}
There exist an open neighborhood $U_{i}$ of $L_{i}$ for $i=0$ and $1$ and an isomorphism $f \colon (U_{0},\alpha_{0}|_{U_{0}}) \longrightarrow (U_{1},\alpha_{1}|_{U_{1}})$ such that $f(L_{0})=L_{1}$ and $f|_{W^{j} \cap U_{0}}=f^{j}|_{W^{j} \cap U_{0}}$ for an open neighborhood $W^{j}$ of $\Sigma_{0}^{j}$ for $j=0$ and $1$. 
\end{lem}

\begin{proof}
By Lemma \ref{Isomorphisms}, there exist an open neighborhood $U_{i}$ of $L_{i}$ for $i=0,1$ and a diffeomorphism $f \colon U_{0} \longrightarrow U_{1}$ such that $f(L_{0})=L_{1}$ and $(f|_{L_{0} \cup U_{0}^{0} \cup U_{1}^{0}})^{*}\alpha_{1}=\alpha_{0}$. 

Let $E_{i}$ be the $T^2$-equivariant symplectic normal bundles of $L_{i}-\overline{U_{i}^{0} \cup U_{i}^{1}}$ in $(M_{i}-\overline{U_{i}^{0} \cup U_{i}^{1}},\alpha_{i})$ for $i=0$ and $1$. By Lemma \ref{RelativeNormalFormTheorem}, to prove Lemma \ref{RelativeNormalFormTheoremForL} it suffices to show that $E_{0}$ and $E_{1}$ are isomorphic by an isomorphism whose restriction near $\Sigma_{0}^{j}$ covers $f^{j}$. Since $E_{0}$ and $E_{1}$ are vector bundles of rank $2$, the symplectic structures on fibers of $E_{0}$ and $E_{1}$ are volume forms. Hence to construct an isomorphism as $T^2$-equivariant symplectic vector bundles between $E_{0}$ and $E_{1}$, it suffices to construct an isomorphism $q$ between their underlying oriented $T^2$-equivariant vector bundles. $L_{i} - (U_{i}^{0} \cup U_{i}^{1})$ is $T^2$-equivariantly diffeomorphic to $T^2 \times [0,1]$ for $i=0$ and $1$. Since every $T^2$-equivariant vector bundle over $T^2 \times [0,1]$ is trivial, $E_{i}|_{L_{i}-\overline{U_{i}^{0} \cup U_{i}^{1}}}$ is isomorphic to the trivial $T^2$-equivariant vector bundle over $T^2 \times [0,1]$ for $i=0$ and $1$. Using these trivializations, we can extend the isomorphism $E_{0}|_{U_{0}^{1}} \longrightarrow E_{0}|_{U_{1}^{1}}$ induced from $f^{1}$ to an isomorphism $\phi \colon E_{0}|_{L_{0} - U_{0}^{0}} \longrightarrow E_{1}|_{L_{1} - U_{1}^{0}}$. We have an isomorphism $\psi \colon E_{0}|_{U_{0}^{0}} \longrightarrow E_{1}|_{U_{1}^{0}}$ induced from $f^{0}$. To construct an isomorphism from $E_{0}$ to $E_{1}$ from $\phi$ and $\psi$, it suffices to show that the diagram
\begin{equation}
\xymatrix{ E_{0}|_{\partial (L_{0} - U_{0}^{0})} \ar[r]^{\phi} \ar[d]_{A_{0}} & E_{1}|_{\partial (L_{1} - U_{1}^{0})} \ar[d]^{A_{1}} \\
           E_{0}|_{\partial U_{0}^{0}}  & E_{1}|_{\partial U_{1}^{0}}  \ar[l]^{\psi^{-1}} }
\end{equation}
commutes where $A_{i}$ is the attaching map of $E_{i}$ for $i=0$ and $1$. By the assumption, $A_{0}$ and $ \psi^{-1} \circ A_{1} \circ \phi$ is $T^2$-equivariant bundle maps from a trivial $T^2$-equivariant vector bundle over $T^2$ to itself. Since such map is unique, we have $A_{0}=\psi^{-1} \circ A_{1} \circ \phi$. The proof is completed.
\end{proof}

Assume that $B_{\min i}$ is a closed orbit of the Reeb flow for $i=0$ and $1$. Let $g_{i}$ be a metric on $M_{i}$ compatible with $\alpha_{i}$ for $i=0$ and $1$. Let $\Sigma_{i}^{j}$ is a closed orbit of the Reeb flow of $\alpha_{i}$ for $j=1$, $2$, $\cdots$, $l$ and $i=0$ and $1$. Let $L_{i}^{j}$ be the gradient manifold with respect to $g_{i}$ whose $\omega$-limit set is $\Sigma_{i}^{j}$. Assume that $I(\rho_{i},L_{i}^{j})=1$ and the $\alpha$-limit set of $L_{i}^{j}$ is $B_{\min i}$. Assume that $g_{0}$ is Euclidean with respect to a coordinate on a finite cover of an open neighborhood of $B_{\min 0}$ which represents $\alpha$ as \eqref{NormalFormOfAlpha}. Assume that we have an open neighborhood $V_{i}$ of $\cup_{j=1}^{l} \Sigma_{i}^{j} \cup B_{\min i}$ for $i=0$, $1$ and an isomorphism $f_{\min} \colon (V_{0}, \alpha_{0}|_{V_{0}}) \longrightarrow (V_{1}, \alpha_{1}|_{V_{1}})$ which satisfies $g_{0}=f^{*}g_{1}$. 
\begin{lem}\label{Octopus}
There exists an open neighborhood $U_{0}$ of $\cup_{j=1}^{l}L_{0}^{j}$, an open neighborhood $U_{1}$ of $\cup_{j=1}^{l}\Sigma_{1}^{j} \cup B_{\min 1}$ and an isomorphism $f \colon U_{0} \longrightarrow U_{1}$ such that $f|_{W_{0}}=f_{\min}|_{W_{0}}$ for an open neighborhood $W_{0}$ of $\cup_{j=1}^{l} \Sigma_{0}^{j} \cup B_{\min 0}$.
\end{lem}

\begin{proof}
First, we show that there exists a metric $g_{1}^{\prime}$ on $M_1$ compatible with $\alpha_{1}$ such that $g_{0}=f_{\min}^{*}(g'_1|_{W_{0}})$ and $L_{0}^{j} \cap f_{\min}^{-1}(W_{0}) = f_{\min}^{-1}(L_{1}^{\prime j} \cap W_{0})$ for an open neighborhood $W_{0}$ of $\cup_{j=1}^{l} \Sigma_{1}^{j} \cup B_{\min 1}$, where $L_{1}^{\prime j}$ is the gradient manifold with respect to $g'_{1}$ whose $\omega$-limit set is $\Sigma_{1}^{j}$. 

Since $f_{\min}$ is an isometry on $V_{0}$, there exists an open neighborhood $V'_{1}$ of $\cup_{j=1}^{l} \Sigma_{1}^{j}$ in $V_{1}$ for $i=0$, $1$ such that $L_{0}^{j} \cap f_{\min}^{-1}(V'_{1}) = f_{\min}^{-1}(L_{1}^{\prime j} \cap V'_{1})$. We modify $g_{1}$ on an open neighborhood of the boundary of a tubular neighborhood of $B_{\min 1}$ to obtain $g'_1$ satisfying the above conditions. We define an $\mathbb{R}_{>0}$-action $\sigma$ on an open neighborhood of $B_{\min i}$ conjugating the linear $\mathbb{R}_{>0}$-action on the orthogonal normal bundle of $B_{\min i}$ by the exponential map with respect to $g_{i}$ for $i=0$ and $1$. Since $g_{1}$ is Euclidean with respect to a coordinate on a finite cover of an open neighborhood of $B_{\min 0}$ which represent $\alpha$ as \eqref{NormalFormOfAlpha}, the gradient manifolds are invariant under $\sigma_{i}$. Hence the space of gradient manifold whose $\alpha$-limit set is $B_{\min 1}$ is $(S^1 \times S^3)/\rho_{1}$, which is diffeomorphic to a $2$-dimensional orbifold $S$ with at most two singular points. We denote the points in $S$ corresponding to $f_{\min}(L_{0}^{j})$ and $L_{1}^{j}$ by $[f_{\min}(L_{0}^{j})]$ and $[L_{1}^{j}]$ respectively. There exist small positive numbers $\delta$, $\epsilon$, $\epsilon'$ and an isotopy $\{\psi_{t}\}_{t \in [\epsilon,\epsilon']}$ on $S$ such that $\epsilon < \epsilon + \delta < \epsilon' - \delta < \epsilon'$ and $\psi_{t}$ is the identity for $t$ in $[\epsilon,\epsilon+\delta]$, $\psi_{t}$ is a diffeomorphism $\psi$ such that $\psi([L_{1}^{j}])=[f_{\min}(L_{0}^{j})]$ for $j=1$, $2$, $\cdots$, $l$ for $t$ in $[\epsilon'-\delta,\epsilon']$. Let $\{Y_{t}\}_{t \in [\epsilon,\epsilon']}$ be vector fields on $S$ which generate $\{\psi_{t}\}_{t \in [\epsilon,\epsilon']}$. We fix a trivialization $\phi \colon (S^1 \times S^3 \times [\epsilon,\epsilon'])/\rho_{1} \longrightarrow S \times [\epsilon,\epsilon']$ as a family of orbifolds diffeomorphic to $S$ by using $\sigma_{i}$. Let $Y$ be a vector field on $S \times [\epsilon,\epsilon']$ defined by $Y_{(x,t)}=(Y_{t})_{x}$ for $x$ in $S$ and $t$ in $[\epsilon,\epsilon']$. We define a vector field $Y'$ on $(S^1 \times S^3 \times [\epsilon,\epsilon'])/\rho_{1}$ by $Y'=(\phi_{*})^{-1}Y$. We take a lift $\tilde{Y}$ of $Y'$ on $S^1 \times S^3 \times [\epsilon,\epsilon']$ so that $\tilde{Y}$ is tangent to $\ker \alpha_{1}$ and invariant under $\rho_{1}$. We put $Z=\tilde{Y} + \grad \Phi_{1}$. We show that $Z$ is realized as the vector field generating the gradient flow with respect to a metric $g'_1$ on $M$ compatible with $\alpha_{1}$ such that $g_{0}=f_{\min}^{*}(g'_1|_{W_{1}})$ for an open neighborhood $W_{1}$ of $\cup_{j=1}^{l} \Sigma_{1}^{j} \cup B_{\min 1}$. It suffices to show that there exists a complex structure $J'$ on $\ker \alpha_{1}$ which satisfies $J'|_{W_{1}}=J|_{W_{1}}$ for an open neighborhood $W_{1}$ of $\cup_{j=1}^{l} \Sigma_{1}^{j} \cup B_{\min 1}$ and defines a metric with the symplectic structure $d\alpha_{1}$ and $J'X=Z$ by \eqref{gradientflow}. We have $d\alpha_{1}(Z,X)=Z\Phi_{1}$ on $S^1 \times S^3 \times [\epsilon,\epsilon']$. Since $Z\Phi_{1}$ is positive by the construction of $Z$,  the vector bundle $\mathbb{R}Z \oplus \mathbb{R}X$ is symplectic on $S^1 \times S^3 \times [\epsilon,\epsilon']$. Hence we have $J'$ satisfying the desired conditions.

Then Lemma \ref{Octopus} is proved by applying Lemma \ref{RelativeNormalFormTheoremForL} to $L_{0}^{j}$ and $L_{1}^{j}$ for $j=1$, $2$, $\cdots$, $l$.
\end{proof}

\subsubsection{Normal forms of chains}

Let $C_{i}$ be a chain in $(M_{i},\alpha_{i})$ for $i=0$ and $1$. Let $L_{i}^{1}$, $L_{i}^{2}$, $\cdots$, $L_{i}^{l}$ be the gradient manifolds which form $C_{i}$ in the order of the superscripts. We denote the $\alpha$-limit set of $L_{i}^{j}$ by $\Sigma_{i}^{j-1}$ for $j=1$, $2$, $\cdots$, $l$ and $i=0$, $1$. We denote the $\omega$-limit set of $L_{i}^{l}$ by $\Sigma_{i}^{l}$ for $i=0$ and $1$. We assume that the closure $N_{i}^{j}$ of $L_{i}^{j}$ in $M_{i}$ is a smooth submanifold of $M_{i}$ for $j=1$, $2$, $\cdots$, $l$. We denote a $K$-contact submanifold containing $\Sigma_{i}^{0}$ other than $N_{i}^{1}$ by $N_{i}^{0}$ for $i=0$ and $1$. We denote a $K$-contact submanifold containing $\Sigma_{i}^{l}$ other than $N_{i}^{l}$ by $N_{i}^{l+1}$ for $i=0$ and $1$. Let $\Sigma_{i}$ be a closed orbit of the Reeb flow of $\alpha_{i}$ for $i=0$ and $1$. A diffeomorphism $f$ from an open neighborhood $U_{0}$ of $C_{0}$ to an open neighborhood $U_{1}$ of $C_{1}$ is said to be an isomorphism if $f(C_{0})=C_{1}$, $f^{*}\alpha_{1}=\alpha_{0}$ and $f(\rho_{0}(x,g))=\rho_{1}(f(x),g)$ for every $x$ in $\Sigma_{0}$ and $g$ in $T^2$. We have the following:
\begin{lem}\label{NormalFormOfChains}
There exists an isomorphism $f$ from an open neighborhood of $C_{0}$ to an open neighborhood of $C_{1}$ mapping $\Sigma_{0}^{0}$ to $\Sigma_{1}^{0}$ if and only if $(G_{0})_{\Sigma_{0}^{j}}=(G_{1})_{\Sigma_{1}^{j}}$ for $j=0$, $1$, $\cdots$, $l$ and $(G_{0})_{N_{0}^{j}}=(G_{1})_{N_{1}^{j}}$ for $i=0$, $1$, $\cdots$, $l+1$.
\end{lem}
\begin{proof}
We show Lemma \ref{NormalFormOfChains} inductively on $i$. By Lemma \ref{NormalFormOfGradientManifolds2}, there exists an isomorphism $f^{0}$ from an open neighborhood of $N_{0}^{0}$ to an open neighborhood of $N_{1}^{0}$ mapping $\Sigma_{0}^{0}$ to $\Sigma_{1}^{0}$. Lemma \ref{NormalFormOfChains} is true in the case where $l=s$. Consider the case where $l=s+1$. By the induction hypothesis, there exists an isomorphism $f^{s}$ from an open neighborhood of $\cup_{j=0}^{l-1} N_{0}^{j}$ to an open neighborhood of $\cup_{j=0}^{l-1} N_{1}^{j}$ mapping $\Sigma_{0}^{0}$ to $\Sigma_{1}^{0}$. Applying Lemma \ref{NormalFormOfGradientManifolds2} to $N_{0}^{l}$ and $N_{1}^{l}$, we extend the isomorphism $f^{s}$ to an isomorphism $f$ from an open neighborhood of $C_{0}$ to an open neighborhood of $C_{1}$ mapping $\Sigma_{0}^{0}$ to $\Sigma_{1}^{0}$.
\end{proof}

\subsubsection{Normal forms of a $T^2$-orbit with trivial isotropy group}

Let $C$ be a $T^2$-orbit of $\rho$ with trivial isotropy group. Let $L$ be a subset in $M$ such that  for an open neighborhood $U'$ of $C$ in $M$, $L \cap U'$ is a $K$-contact submanifold of $(U',\alpha|_{U'})$ containing $C$. Then we have
\begin{lem}\label{NormalFormOfC}
There exists an open neighborhood $U$ of $C$ in $U'$ and a diffeomorphism $f \colon T^2 \times [t^{0},t^{1}] \times \mathbb{C} \longrightarrow U$ such that 
\begin{equation}\label{CoordinateNearC}
f^{*}\alpha=\frac{dx \cos t + dy \sin t}{a \cos t + b \sin t} + \frac{\sqrt{-1}}{2} (z_1 d\overline{z}_1 - \overline{z}_1dz_1)
\end{equation}
and $U \cap L$ is defined by the equation $z_1=0$ where $(x,y)$ is the coordinate on $T^2$ induced from the Euclidean coordinate on $\mathbb{R}^{2}$, $t$ and $z_1$ are the standard coordinate on $[t^{0},t^{1}]$ and $\mathbb{C}$ respectively.
\end{lem} 

\begin{proof}
By the argument in the second paragraph of Lemma \ref{Isomorphisms}, we have an open neighborhood $U''$ of $C$ in $M$ such that $(L \cap U'',\alpha|_{L \cap U''})$ is isomorphic to $(T^2 \times [t^{0},t^{1}], \beta)$ for some real numbers $t^{0}$ and $t^{1}$ where $\beta$ is defined by $\beta=\frac{dx \cos t + dy \sin t}{a \cos t + b \sin t}$ for some real numbers $a$ and $b$. Since the $T^2$-orbits contained in $L$ have trivial isotropy group of $\rho$, the normal bundle of $L \cap U'$ is trivial as a $T^2$-equivariant vector bundle. By Lemma \ref{NormalFormTheorem2}, we have an open neighborhood $U$ of $C$ in $M$ and a diffeomorphism $f \colon T^2 \times [t^{0},t^{1}] \times \mathbb{R}^{2} \longrightarrow U$ which satisfies the desired conditions.
\end{proof}

\subsection{Normal forms of the minimal component and the maximal component of the contact moment map of dimension $3$}

Let $B_{i}$ be the maximal component or the minimal component of the contact moment map of dimension $3$ in $(M_{i},\alpha_{i})$ for $i=0$ and $1$. We show
\begin{lem}\label{NormalFormOfB}
There exists an isomorphism from an open neighborhood of $B_{0}$ to an open neighborhood of $B_{1}$ if and only if $G_{B_{0}}=G_{B_{1}}$, the oriented Seifert fibrations defined by the Reeb flows on $B_{0}$ and $B_{1}$ are isomorphic and the Euler classes of the normal bundles of $B_{0}$ and $B_{1}$ are equal.
\end{lem}
Lemma \ref{NormalFormOfB} follows from Lemma \ref{NormalFormTheorem} and the following lemma:

\begin{lem}\label{ReebFlowsDetermineForms}
Let $\alpha_0$ and $\alpha_1$ be two $K$-contact forms of rank $1$ on a closed $3$-dimensional manifold $M$. $(M,\alpha_0)$ and $(M,\alpha_1)$ are isomorphic if and only if the Reeb flows of $\alpha_0$ and $\alpha_1$ are isomorphic as $S^1$-actions with transverse orientations.
\end{lem}

\begin{proof}
The ``only if'' part is trivial. We show the ``if'' part. Assume that $\alpha_0$ and $\alpha_1$ have the common Reeb vector field $R$. Let $\rho$ be the Reeb flow. We define $\beta_{t}=(1-t)\alpha_0 + t\alpha_1$ for $t$ in $[0,1]$. Since $(1-t)d\alpha_{0}+td\alpha_{1}$ is a volume form on the base space by the assumption, $\beta_{t}$ is a contact form for every $t$. Hence $\{\ker \beta_{t}\}_{t \in [0,1]}$ defines an isotopy connecting $\ker \alpha_{0}$ and $\ker \alpha_{1}$ as $\rho$-invariant contact structures. By the equivariant version of Gray's stability theorem \cite{Gra}, we have a $\rho$-equivariant diffeomorphism $f \colon M \longrightarrow M$ such that $f_{*}(\ker \alpha_{0})=\ker \alpha_{1}$. Since $f$ is $\rho$-equivariant, we have $f_{*}R=R$. Hence $f^{*}\alpha_{1}=\alpha_{0}$.
\end{proof}

We denote the symplectic normal bundle of $B$ in $M$ by $E$. We show that 
\begin{lem}\label{Linearization}
There exists an open neighborhood $V$ of the zero section of $E$ in $E$, an open neighborhood $U$ of $B$ in $M$ and a diffeomorphism $f \colon V \longrightarrow U$ such that $f|_{B}=\id_{B}$, the restriction of $f^{*}d\alpha$ to each fiber of $E$ is linear and $f$ is equivariant with respect to $\rho$ on $U$ and the linear action on $V$ induced from $\rho$.
\end{lem}

\begin{proof}
Let $g$ be a metric on $M$ invariant under the Reeb flow. There exist an open tubular neighborhood $V'$ of the zero section of $E$ in $E$ and an open tubular neighborhood $U'$ of $B$ in $M$ such that $\exp \colon V' \longrightarrow U'$ is a diffeomorphism. Let $p \colon E \longrightarrow N$ be the projection. We denote the kernel of $Dp \colon TE \longrightarrow TB$ by $FE$. Let $\beta_{0}$ be the element of $C^{\infty}(\wedge^{2} FE^{*}|_{V'})$ obtained by the restriction of $\exp^{*}d\alpha$ to $\wedge^{2} FE|_{V'}$. Let $\beta_{1}$ be the element of $C^{\infty}(\wedge^{2} FE^{*}|_{V'})$ obtained by the linear symplectic form defined by $d\alpha$. Since $(D\exp)_{x}=\id_{T_{x}M}$ for $x$  on $B$ under the natural identification of $T_{x}E$ with $T_{x}M$, we have 
\begin{equation}\label{CoincidenceOfForms}
(\beta_{0})_{x}=(\beta_{1})_{x}
\end{equation}
 for $x$ on $B$. Let $\delta$ be the element of $C^{\infty}(\wedge^{2} FE^{*}|_{V'})$ obtained as the restriction of $\exp^{*}d\alpha - d\alpha$ to $\wedge^{2} FE|_{V'}$. Then we have that $\delta|_{B}=0$ and $d\delta=\beta_0 - \beta_1$. We define $\omega_t=(1-t)\beta_0 + t\beta_1$ for $t$ in $[0,1]$. By \eqref{CoincidenceOfForms}, there exists an open neighborhood $V''$ of $B$ in $E$ such that $\omega_t$ is nondegenerate on $V''$ for every $t$ in $[0,1]$. Define the element $Z_t$ in $C^{\infty}(FE|_{V''})$ by $\iota_{Z_t}\omega_t+\delta=0$ for $t$ in $[0,1]$. Then the isotopy $\{\psi_{t}\}_{t \in [0,1]}$ generated by $\{Z_t\}_{t \in [0,1]}$ satisfies $\psi_{t}^{*}\omega_t=\omega_0$. In fact, we have
\begin{equation}
\frac{\partial \psi_{t}^{*}\omega_t}{\partial t}=\psi_{t}^{*} \Big( \frac{d}{dt}\omega_t + d\iota(Z_t)\omega_t \Big) = \psi_{t}^{*}d(\delta + \iota(Z_t)\omega_t)=0.
\end{equation}
Note that $Z_t|_{B}=0$ since $\delta|_{B}=0$. Hence $\psi_{1}$ is well-defined on an open neighborhood $V$ of $B$ in $E$ and satisfies $\psi_{1}|_{B}=\id_{B}$. Since $Z_t$ is invariant by the linear action induced from $\rho$, $\psi_{1}$ is equivariant with respect to the linear action induced from $\rho$. Hence if we put $f=\exp \circ \psi_{1}$, then $f$ satisfies the desired conditions. In fact, we have $f^{*} d\alpha= \psi_{1}^{*}\exp^{*}d\alpha=\psi_{1}^{*}\beta_{1}=\beta_{0}$.
\end{proof}

\end{document}